%% file: main.tex
\documentclass[english,oneside,a4paper]{book}
\usepackage[T1]{fontenc}
\usepackage[utf8]{inputenc}
\usepackage{graphicx}
\usepackage{geometry}
\usepackage{upgreek}
\usepackage{minitoc} 
\usepackage{babel}
\usepackage{verbatim}
\usepackage{float}
\usepackage{mathrsfs}
\usepackage{lipsum}
\usepackage{amsmath}
\usepackage{amsthm}
\usepackage{amssymb}
\usepackage{stmaryrd}
\usepackage{tikz}
\usetikzlibrary{shapes,arrows,automata,matrix,fit,patterns,cd}
 \usepackage{fancybox}
\usepackage[
backend=biber,
style=alphabetic,
sorting=nty,
isbn=false,
eprint=false,
url=false,
doi=false
]{biblatex}
\DeclareUnicodeCharacter{1D48}{$d$}
\usepackage{hyperref}
\hypersetup{
    pdftoolbar=false,
    colorlinks=true,
    linkcolor=black,
    citecolor=black,
    filecolor=black,
    urlcolor=black,
    filecolor=magenta,      
    urlcolor=cyan,
    pdfpagemode=FullScreen,
    }
\DefineBibliographyStrings{english}{%
  backrefpage = {page},
  backrefpages = {pages},
}
\usepackage[nameinlink]{cleveref}

\theoremstyle{plain}
\newtheorem{theorem}{Theorem}[chapter]
\newtheorem{thm}[theorem]{Theorem}
\newtheorem{lem}[theorem]{Lemma}
\newtheorem{cor}[theorem]{Corollary}
\newtheorem{prop}[theorem]{Proposition}
\newtheorem{question}[theorem]{Question}
\newtheorem{problem}[theorem]{Problem}
\newtheorem{conjecture}[theorem]{Conjecture}
\theoremstyle{definition}
\newtheorem{defn}[theorem]{Definition}
\newtheorem*{defn*}{Definition}
\newtheorem{rem}[theorem]{Remark}
\newtheorem*{rem*}{Remark}
\newtheorem{example}[theorem]{Example}
\newtheorem*{example*}{Example}
\newtheorem{claim}{Claim}
\newtheorem*{claim*}{Claim}
\theoremstyle{plain}
\newtheorem{lemma}[theorem]{Lemma}
\newtheorem{corollary}[theorem]{Corollary}
\newtheorem{proposition}[theorem]{Proposition}

\theoremstyle{definition}
\newtheorem{observation}[theorem]{Observation}
\newtheorem{definition}[theorem]{Definition}
\newtheorem{remark}[theorem]{Remark}

\renewcommand{\mid}{ \ : \ }
\input{Notation.tex}

\begin{document}
\pagestyle{plain}
\dominitoc
\frontmatter
\input{titlepage}

\section*{Abstract}
Subshifts are a fundamental class of topological dynamical systems. The study of subshifts on groups different from $\mathbb{Z}$, such as $\mathbb{Z}^d$, $d\geq 2$, has been a subject of intense research in recent years. These investigations have unveiled a
remarkable connection between dynamics and recursion theory. That is, different questions about the dynamics of these systems have been answered in recursion-theoretical terms. 

In this work we further explore this connection. We use the framework of computable analysis to explore the class of effective dynamical systems on metric spaces, and relate these systems to subshifts of finite type (SFTs) on groups. We prove that every effective dynamical system on a general metric space is the topological factor of an effective dynamical system with topological dimension zero. We combine this result with existing simulation results to obtain new examples of systems that are factors of SFTs

We also study a conjugacy invariant for subshifts on groups called Medvedev degree. This invariant is a complexity measure of algorithmic nature. We develop the basic theory of these degrees for subshifts on arbitrary finitely generated groups. Using these tools we are able to classify the values that this invariant attains for SFTs and other classes of subshifts on several groups. Furthermore,  we establish a connection between these degrees and the distribution of isolated points in the space of all subshifts. 

Motivated by the study of Medvedev degrees of subshifts, we also consider translation-like actions of groups on graphs. We prove that every connected, locally finite, and infinite graph admits a translation by $\mathbb{Z}$, and that this action can be chosen transitive exactly when the graph has one or two ends. This generalizes a result of Seward about translation-like action of $\mathbb{Z}$ on finitely generated groups. Our proof is constructive, and allows us to prove that under natural hypotheses, translation-like actions by $\mathbb{Z}$ on groups and graphs can be effectively computed. 

\tableofcontents{}

\newpage
\section*{Acknowledgements}
I would like to thank my advisors Cristóbal Rojas and Sebastián Barbieri for their  guidance, for their numerous suggestions, and for their generosity with their time. I am also very grateful to Mathieu Sablik for his hospitality during the time I spent in Toulouse. 

This research received support from ANID 21201185 doctorado nacional, ANID/Basal National Center for Artificial Intelligence CENIA FB210017, and the European Union’s Horizon 2020 research and innovation program under the Marie Sklodowska-Curie grant agreement No 731143. I also thank the  organizers of the events ``Postgraduate school UFRO - Lican Ray '',  ``Thematic month at CIRM, France: Discrete Mathematics $\&$ Computer Science: Groups, Dynamics, Complexity, Words'', ``16th International Conference on Computability, Complexity and Randomness. Kochel, Germany '', and ``Journées annuelles SDA2, Toulouse, France''.

\newpage 
\vfill
\null\hfill \textit{To Aurelia.} 
\mainmatter
\chapter{Introduction}
\input{Chapter_introduction.tex}

\chapter{Background}\label{chap:Background}

\input{Chapter_background.tex}

\chapter{Subshifts on groups and computable
analysis}\label{chap:computable_analysis_subshifts}

\input{Chapter_computable_analysis_and_subshifts}

\chapter{Effective dynamical systems and factors of SFTs}\label{chap:EDS}

\input{Chapter_effective_dynamical_systems.tex}

\chapter{Medvedev degrees of subshifts}\label{chap:Medvedev}
\input{Chapter_Medvedev.tex}

\chapter{Computable analysis on the space of $G$-subshifts}\label{Chap:computable-analysis-on-S(G)}
\input{Chapter_computable_analysis_on_the_space_of_subshifts.tex}

\chapter{On a Rice theorem for dynamical properties of SFTs and sofic subshfits on groups}\label{Chap:rice}
\input{Chapter_rice_sfts.tex}

\chapter{Translation-like actions by $\Z$ on locally finite graphs}\label{chap:translation-like-actions}

\input{Chapter_translation_like_actions.tex}

\chapter{Computable translation-like actions by $\Z$, and the orbit membership problem}\label{chap:computable-translation-like-actions}

\input{Chapter_computable_translation_like_actions}

\chapter{Conclusion}
\input{Chapter_conclusion.tex}

\printbibliography[notkeyword={carrasco-vargas}]
\printbibliography[keyword={carrasco-vargas},title={Personal Bibliography}]
\end{document}

%% file: notation.tex
\global\long\def\Q{\mathbb{Q}}%
\global\long\def\R{\mathbb{R}}%
\global\long\def\N{\mathbb{N}}%
\global\long\def\NN{\mathbb{N}}%
\global\long\def\ZZ{\Z}%
\global\long\def\CC{\mathbb{C}}%
\global\long\def\RR{\mathbb{R}}%
\global\long\def\Z{\mathbb{Z}}%
\global\long\def\varprojlim{\underset{\longleftarrow}{\lim \ }}%

\global\long\def\G{\Gamma}%
\global\long\def\L{\Lambda}%
\global\long\def\cay{\text{Cay}}%
\global\long\def\path{\operatorname{\textbf{Path}}}%
{} 
\global\long\def\F{\mathcal{F}}%
\global\long\def\Dp{\text{Dp}}%

\global\long\def\P{\mathscr{P}}%
 
\global\long\def\h{\mathcal{I}}%

\global\long\def\action{\curvearrowright}%
\global\long\def\supp{\text{supp}}%
\global\long\def\GL{\text{GL}}%

\global\long\def\ZZ {\mathbb{Z}}
\global\long\def\RR {\mathbb{R}}
\global\long\def\QQ {\mathbb{Q}}
\global\long\def\NN {\mathbb{N}}
\global\long\def\F{\mathcal{F}}%
\newcommand{\define}[1]{\textbf{#1}}

\global\long\def\M{\mathcal{\mathfrak{M}}}

\newcommand{\meff}[1]{\M_{\texttt{EFF}}(#1)}
\newcommand{\msft}[1]{\M_{\texttt{SFT}}(#1)}
\newcommand{\msof}[1]{\M_{\texttt{SOF}}(#1)}

\renewcommand{\varnothing}{\emptyset}

%% file: titlepage.tex
\begin{titlepage}
\begin{center} 
	\begin{figure}
    \centering\includegraphics[width=0.4\textwidth]{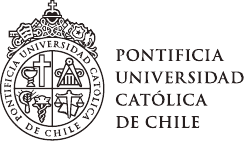}
	\end{figure}
\end{center}

\begin{center}
        \vspace*{1cm}
		\Huge

        \textbf{Subshifts on groups and computable analysis}
            
        \vspace{0.5cm}
        \LARGE      
        \textbf{Nicanor Carrasco-Vargas}
   
        \vspace{1.5cm}

        \vfill
        \Large    
        A thesis presented for the degree of\\
        Doctor of Philosophy
            
        \vspace{0.5cm}
        \Large
        Faculty of Mathematics\\
        Pontifical Catholic University of Chile\\
        Chile\\
        2024\\
        
    \end{center}
        \vfill
        \textit{Advisor}:\\\\ Cristóbal Rojas - Pontifical Catholic University of Chile
     \\\\
        \textit{Co-advisor}:\\\\  Sebastián Barbieri - University of Santiago, Chile\\\\
        \textit{Committee}: \\\\
        Alvaro Daniel Coronel Soto - Pontifical Catholic University of Chile
        \\\\
        Benjamin Hellouin de Menibus - Université 
        Paris Saclay
        \\\\
        Jarkko Kari - University of Turku\\\\
        Tom Meyerovitch - Ben Gurion University of the Negev\\\\
        \\
        \\
        \\
        \tiny{©Nicanor Carrasco-Vargas}

\end{titlepage}

%% file: Chapter_introduction.tex
The mathematical notion of dynamical system is motivated by real-world processes that evolve with time. Interesting examples are the evolution of wealth distribution, carbon emissions, the sea level, the internet, or the change of climate in some region. 

The mathematical object is defined as follows. A dynamical system is a pair $(X,T)$ of a set $X$ and a transformation $T\colon X\to X$. The idea is that $X$ represents all possible states, and $T$ represents their evolution after one unit of time. We can enrich the system with different mathematical structures: we can assume that $T$ is continuous for a certain topology on $X$, measurable for a certain sigma-algebra on $X$, measure-preserving for a  certain measure on $X$, computable for a certain computable structure on $X$, etc. We will often consider a finite collection of transformations instead of a single one. When we consider invertible transformations, then the collection of all possible compositions of them forms a group, and it is convenient to instead speak about a group acting on $X$.

This work is mostly concerned with a class of topological dynamical systems called subshifts. These systems were conceived as tools to understand other systems, but their study evolved into a subject on its own right; see the seminal work \cite{morse_symbolic_1938}.  Let $A$ be a finite set of symbols, let $G$ be a group, and let $A^G=\{x\colon G\to A\}$ be endowed with the prodiscrete topology. The group $G$ acts on $A^G$ by left shift translations. A $G$-subshift is a closed subset $X\subset A^G$ that is invariant by shift translations, and it has finite type (SFT) when it can be  defined by a finite collection of local rules expressed as forbidden patterns.  

\begin{figure}[H]
    \centering
    \includegraphics[width=0.6\textwidth]{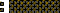}
    \caption{Many decorative tilings can be seen as $\Z^2$-SFTs. Given a finite set of decorated square tiles of the same size, we can define an SFT whose alphabet is the set of tiles, and with the local rule that adjacent tiles preserve the intended decoration. Conversely, every SFT on $\Z^2$ is topologically conjugate to one defined by square decorated tiles. This identification between SFTs and tilings is also valid on finitely generated groups. One just needs to fix a generating set for the group, and then use it to express the adjacency relations. For this purpose it is also convenient to consider Wang tiles, a standard model for decorated tiles proposed by Wang \cite{wang_proving_1961}.}\label{fig:enter-label}
\end{figure}

Subshifts and SFTs were studied first on $\Z$, and this is the case where the theory is best understood; see the monograph \cite{lind_introduction_1995}. In recent years there has been considerable effort to  better understand SFTs on $\Z^d$ for $d\geq 2$, also called multidimensional SFTs. These investigations show that these systems can be radically more complex than in one dimension. Moreover, it has been found that dynamical aspects of these systems are closely connected with recursion-theoretical ones. 

This situation can be illustrated by the characterization of the values of different dynamical invariants. Hochman and Meyerovitch proved that the set of topological entropies of multidimensional SFTs is the set of upper-semicomputable non-negative real numbers \cite{hochman_characterization_2010}. This characterization was later extended to SFTs on other amenable groups \cite{barbieri_entropies_2021,bartholdi_salo_shift_lamplighter_2024}. Moreover, the class of entropy-dimensions of multidimensional SFTs was also characterized by a recursion-theoretical property \cite{meyerovitch_growthtype_2011,gangloff_characterizing_2022}. Some dynamical restrictions reduce the possible values for these invariants, but the characterization is still recursion-theoretical  \cite{gangloff_quantified_2021,gangloff_effect_2019,pavlov_entropies_2015,esnay_parametrization_2023}.

Another example of the interplay between dynamical and recursion-theoretical aspects of SFTs is given by simulation results \cite{hochman_dynamics_2009,aubrun_simulation_2013,durand_effective_2010,durand_fixedpoint_2012,barbieri_generalization_2019,barbieri_geometric_2019,barbieri_groups_2022, bartholdi_shifts_2024}.  These results relate SFTs on different groups to effective dynamical systems on the Cantor set $\{0,1\}^\N$. Roughly speaking, a dynamical system is effective when both the transformations and the collection of states can be described by a Turing machine. These results have been used to produce SFTs with a highly specific behaviour. This behaviour is given by a Turing machine, which is then ``simulated'' by the SFT. 

In this thesis we explore further the interaction between dynamical and recursion-theoretical aspects of SFTs over different groups, with an emphasis on using the framework of computable analysis to formalize and systematize these connections.

\section{Main contributions}

A common feature of simulation results is that they only apply to systems with topological dimension zero. It is therefore natural to ask whether this restriction can be lifted. In \Cref{chap:EDS} we show that under natural hypotheses the answer is ``yes''. This is done using the framework of computable analysis, which allows to define computable functions and related concepts on separable metric spaces.  Using these tools, we introduce the class of \textit{effective dynamical systems} on general metric spaces, as those containing a computable representative in their topological conjugacy class. Many natural systems one can think of are effective in this sense, such as affine actions on the torus, rational maps on the Riemann sphere, and finitely presented algebraic actions of recursively presented groups.

The main result of \Cref{chap:EDS}, \Cref{thm:zero_dim_effective_extension}, states that every effective dynamical system given by a recursively presented group is the topological factor a computable action on an effectively closed subset of $\{0,1\}^\N$. In other words, \Cref{thm:zero_dim_effective_extension} is an effective version of the well-known result that every topological dynamical system admits an extension with dimension zero. The contribution of our result is that, together with simulation results, it can be used to exhibit new examples of systems that are topological factors of SFTs. 
 
In \Cref{chap:Medvedev} we study Medvedev degrees of subshifts. These degrees are a topological conjugacy invariant that measures algorithmic complexity. This invariant is related to a classic result of Hanf and Myers \cite{hanf_nonrecursive_1974,myers_nonrecursive_1974} that shows the existence of a $\Z^2$-SFT where every configuration is uncomputable. Medvedev degrees quantify this behaviour, and in fact a subshift has zero Medvedev degree exactly when it contains computable configurations. Thus Hanf and Myers proved that $\Z^2$ admits SFTs with nonzero Medvedev degrees. This raises the problem of classifying the possible values that this invariant attains for different classes of subshifts. Simpson \cite{simpson_medvedev_2014} proved that Medvedev degrees of $\Z^2$-SFTs are exactly the $\Pi_1^0$ Medvedev   degrees. The goal of  \Cref{chap:Medvedev}  is developing the basic theory of these degrees for subshifts on arbitrary finitely generated groups. These tools allow us to classify the possible Medvedev degrees of SFTs and other classes of subshifts on different groups. Among other results, we extend Simpson's classification to all virtually polycyclic groups that are not virtually cyclic. 

The study of Medvedev degrees of SFTs is closely related to the domino problem on groups. The domino problem for a finitely generated group $G$ is the algorithmic problem of determining whether a $G$-SFT is empty. This problem is decidable for $\Z$, and undecidable for $\Z^d$, $d\geq 2$  \cite{berger_undecidability_1966}. The proof of this result came along with the first aperiodic SFT on $\Z^2$, which was also the first aperiodic tiling of $\Z^2$. But what happens in other groups? this question has been considered in several works \cite{aubrun_domino_2018,aubrun_domino_2019,aubrun_tiling_2013,ballier_domino_2018,bartholdi_domino_2023,bitar_contributions_2023,cohen_large_2017,jeandel_translationlike_2015,jeandel_undecidability_2020,margenstern_domino_2008,aubrun_strongly_2023}. For many groups, the techniques that have allowed to show the undecidability of the domino problem are the same techniques that have allowed to construct aperiodic SFTs. We employ precisely these techniques to construct SFTs with nonzero Medvedev degree on different groups. We also observe that an SFT with nonzero Medvedev degree must be weakly aperiodic, and that the existence of SFTs with nonzero Medvedev degree implies the undecidability of the domino problem for the corresponding group.  

In \Cref{Chap:computable-analysis-on-S(G)} we prove some basic results about the space $S(G)$ of all $G$-subshifts, where $G$ is  a finitely generated group with decidable word problem. We observe a relation between the distribution of isolated points in $S(G)$, and the Medvedev degrees of $G$-SFTs. That is, we show that if $G$ admits SFTs with nonzero Medvedev degree, then isolated points are not dense in $S(G)$. This shall be compared with the recent result of Pavlov and Schmieding \cite{pavlov_structure_2023} asserting that isolated points are dense in $S(\Z)$. 

\Cref{Chap:rice} concerns the algorithmic undecidability of dynamical properties of SFTs. Let us recall that $\Z$-SFTs are very well-understood systems from the algorithmic point of view. There is a number of algorithms that allow to perform computations on presentations of these systems \cite{lind_introduction_1995}. The case of $\Z^2$-SFTs is very different, and many dynamical properties of these systems are known to be undecidable. Lind called this the ``swamp of undecidabiliy'' \cite{lind_multidimensional_2004}.

The goal of  \Cref{Chap:rice} is formalizing the ``swamp of undecidabiliy'', with a result similar to Rice's undecidability theorem for computer programs. Given a group $G$  with undecidable domino problem -such as $\Z^2$- we prove that all nontrivial dynamical properties for sofic $G$-subshifts are undecidable. Although such a result is not possible for SFTs, we still prove a general undecidability result for dynamical properties of $G$-SFTs that can be applied to most properties of common interest. With the extra hypothesis of amenability, we prove that topological entropy is not computable from presentations of SFTs. We also obtain a similar uncomputability result for abstract dynamical invariants taking values in partially ordered sets.

In \Cref{chap:translation-like-actions} we abandon the realm of symbolic dynamics. We consider actions of groups on graphs that are \textit{translation-like} actions, a notion introduced in \cite{whyte_amenability_1999}. Seward \cite{seward_burnside_2014} proved that every infinite and finitely generated group admits a translation-like action by $\Z$, and that this action can be chosen transitive exactly when the group has either one or two ends. The main result of  \Cref{chap:translation-like-actions} is a generalization of this result to locally finite graphs. This result answers partially  a problem left open in \cite[Problem 3.5]{seward_burnside_2014}. 

In \Cref{chap:computable-translation-like-actions} we build on the results from \Cref{chap:translation-like-actions} to prove some computability results for translation-like actions. The main result is an effective version of Seward's theorem for groups with decidable word problem. We prove that in this case the action can be taken computable, and with the property of having decidable \textit{orbit membership problem}. 

Translation-like actions have proven to be a very useful tool in the study of SFTs. Indeed, given two finitely generated grups $H$ and $G$, the collection of translation-like actions $H\curvearrowright G$ with fixed parameters can be described by a subshift. This fact has been used several times to transfer constructions on subshifts from one group to another. Indeed, in \Cref{chap:Medvedev} we use translation-like actions to transfer Medvedev degrees of SFTs from one group to another. For some of our constructions the existence of translation-like actions is not sufficient, we need these translation-like actions to be  computable and have decidable orbit membership problem. This is the main motivation behind the results in \Cref{chap:computable-translation-like-actions}.

The present document is based on the works \cite{barbieri_effective_2024,barbieri_medvedev_2024,carrasco-vargas_translationlike_2024,carrasco-vargas_rice_2024}.
\newpage
\section{Organization of this document}
In \Cref{chap:Background} we review some basic facts and terminology that will be used in this document. Furthermore, we provide a self-contained presentation of the results from computable analysis that will be needed. In \Cref{chap:computable_analysis_subshifts} we apply these notions to shift spaces, and review the class of effective subshifts. The content of Chapters \ref{chap:Background} and \ref{chap:computable_analysis_subshifts} will be used in Chapters \ref{chap:EDS}, \ref{chap:Medvedev} and \ref{Chap:computable-analysis-on-S(G)}. \Cref{Chap:rice} is independent, and \Cref{chap:computable-translation-like-actions} is a continuation of \Cref{chap:translation-like-actions}. This is represented in the following diagram. Dashed arrows mean that a result of one chapter is used in the other, but that they could be read independently.  
\vfill
\begin{tikzcd}
{\ovalbox{\begin{minipage}[t]{0.4\textwidth}
\centering
\textbf{\Cref{chap:Background}:
\\\nameref{chap:Background}}
\\
+
\\
\textbf{\Cref{chap:computable_analysis_subshifts}: 
\\\nameref{chap:computable_analysis_subshifts}}
\end{minipage}}} \arrow[r] \arrow[rd, bend right] \arrow[rdd, bend right=60] &\ovalbox{\begin{minipage}[t]{0.4\textwidth}
\centering
\textbf{\Cref{chap:EDS}: 
\\
\nameref{chap:EDS}} 
\end{minipage}}                          \\
&\ovalbox{\begin{minipage}[t]{0.4\textwidth}
\centering
\textbf{\Cref{chap:Medvedev}: 
\\
\nameref{chap:Medvedev}} 
\end{minipage}} \arrow[d, dashed]        \\
&\ovalbox{\begin{minipage}[t]{0.4\textwidth}
\centering
\textbf{\Cref{Chap:computable-analysis-on-S(G)}: 
\\
\nameref{Chap:computable-analysis-on-S(G)}}
\end{minipage}}              \\
& \ovalbox{\begin{minipage}[t]{0.4\textwidth}
\centering
\textbf{\Cref{Chap:rice}: 
\\
\nameref{Chap:rice}}
\end{minipage}}                                    \\
& \ovalbox{\begin{minipage}[t]{0.4\textwidth}
\centering
\textbf{\Cref{chap:translation-like-actions}: \\
\nameref{chap:translation-like-actions}}
\end{minipage}}  \arrow[d]                \\
& \ovalbox{\begin{minipage}[t]{0.4\textwidth}
\centering
\textbf{\Cref{chap:computable-translation-like-actions}: 
\\
\nameref{chap:computable-translation-like-actions}}
\end{minipage}}  \arrow[uuuu, dashed, bend right=90]
\end{tikzcd}

%% file: Chapter_background.tex
In this chapter we review some basic terminology and notation that will be used along this document. We review basic facts about groups, topological dynamics, subshifts, and computability theory. We review the basic theory of computable analysis, and prove some results that will be used in other chapters.

\minitoc
\section{Conventions}
The composition of functions is written $f\circ g(x)=f(g(x))$, the cardinality of a set $X$ is denoted $|X|$, set difference is denoted $\smallsetminus$ or $-$, and  $0\in\N$.
\section{Words and finitely generated groups}

Here we review some terminology and notation on words, alphabets, and
finitely generated groups. An \textbf{alphabet }is a set, which in most cases is assumed to be finite. A \textbf{word }on alphabet $A$ of
\textbf{length} $n$ is an element in $A^{\{0,\dots,n-1\}}$. The
\textbf{empty word }is denoted by $\epsilon$. The set of all words
on alphabet $A$ is denoted by $A^{\ast}$.

Let $G$ be a group, and let $S\subset G$ be a finite set. Given a word $w\in S^\ast$, we denote the corresponding group element by $w_G$ or simply $\underline w$. We say that $S\subset G$ is a \textbf{generating set} for $G$ when  every element in $G$ can be written as a word in $S^\ast$. Unless specified, we will always consider generating sets that are symmetric, that is, closed by inverses. A group that admits a finite generating set is called \textbf{finitely generated}. Most of the groups considered here are finitely generated.

Let $G$ be a group finitely generated by $S$. We denote by $|g|_S$ the length of the shortest word in $S^\ast$ with $g=\underline w$. The function $|\cdot |_S\colon G\to \N$ is called a \textbf{word length}.  The set $S$ also induces the metric $d_S$ on $G$ given by $d_S(g,h)=|g^{-1}h|_S$.  We call $d_S$ a \textbf{word metric} on $G$. Let us note that $d_S$ is invariant by multiplication on the left, that is $d_S(g,h)=d_S(tg,th)$ for all $g,h,t\in G$. 

We now recall some algorithmic properties of groups and subgroups, algorithmic concepts will be defined in \Cref{preliminaries-computability}.  Given a group $G$ finitely generated by $S$, the \textbf{word problem} for $G$ is the set $\{w\in S^\ast : \underline{w}=1_G\}$. Thus $G$ has decidable word problem when $\{w\in S^\ast : \underline{w}=1_G\}$ is a decidable set. It is easily seen that this definition does not depend of the chosen generating set. More generally, a subgroup $H\leqslant G$ has \textbf{decidable subgroup membership problem} when $\{w\in S^{*}\mid \underline w \in H\}$ is a decidable subset of $S^\ast$.  

We will also be interested in groups whose word problem is recursively enumerable. It is well known that this is the case exactly when $G$ is a recursively presented group, meaning that it admits a recursive presentation (see \cite{lyndon_combinatorial_2001}).

\section{Topological dynamics}
Here we recall some basic concepts from topological dynamics. Let $G$ be a finitely generated group. A $G$-\textbf{topological dynamical system} is a pair $(X,T)$ of a compact
metrizable space $X$ and a continuous left group action $T\colon$$G\times X\to X$.
We often write $G\curvearrowright X$ for short. A \textbf{morphism} of dynamical
systems $G\curvearrowright X$ and $G\curvearrowright Y$ is a continous
map $\phi\colon X\to Y$ which conmutes with the corresponding group
actions. A morphism of dynamical systems which is surjective (resp.
injective, resp. bijective) is called a topological \textbf{factor} (resp.
\textbf{embedding}, resp. \textbf{conjugacy}). When there exists a topological factor
map $\phi\colon X\to Y$ we also say that $Y$ is a topological factor
of $X$, $X$ is a topological extension of $Y$, and that $X$ factors
over $Y$.%
{} %
%

\section{Subshifts}\label{preliminaries-shift-spaces}

Here we review some basic notions about subshifts, the reader is referred to the book \cite{ceccherini-silberstein_cellular_2018}. Let $G$ a group and let $A$ be a finite alphabet. The set  $A^{G}=\{x\colon G\to A\}$ is called a \textbf{fullshift}, and its elements are called \textbf{configurations}. We endow $A^G$  with the pro-discrete topology. If $G=\{g_i : i\in \N\}$ is countable, we can define a metric for this topology by setting \[d(x,y)=\inf \{2^{-n} : n\in \N ,\text{ and } x(i)=y(i)\text{ for all }i\leq n\}.\]
If  $G$ is finitely generated and endowed with a word length $|\cdot|$, then another metric for the pro-discrete topology on $A^G$ is 

\[d(x,y)=\inf \{2^{-n} : n\in \N ,\text{ and } x(g)=y(g)\text{ for all $g$ with } |g|\leq n\}.\]

A \textbf{pattern }is an element $p\in A^F$, where  $F\subset G$ is finite, and it determines the \textbf{ cylinder } $[p]=\{x\in A^{G}\ : \ x|_{F}=p\}$. Cylinders are a sub-basis for the prodiscrete topology on $A^G$.  If $gx\in[p]$ for some $g\in G$, we say that $p$ \textbf{appears}
on $x$.

We define the \textbf{shift} action $G\curvearrowright A^G$ by $(gx)(h)=x(g^{-1}h)$. A \textbf{subshift
} is a subset $X\subset A^{G}$ which is closed and shift-invariant. The \textbf{language} of $X$ is the set of patterns that appear in $X$. A subshift $X$ is completely determined by its language, or equivalently, by the set of patterns that do \textit{not} appear in $X$. 

It is customary to define a subshift by a set of \textbf{forbidden} patterns. Given a set of forbidden patterns  $\F$, we define the subshift $X_{\mathcal{\F}}$ of all elements $x\in A^{G}$ such that no pattern of $\F$ appears in $x$. If we need to make the alphabet explicit we may write  $X_{(A,\mathcal F)}$. Every subshift is determined by a maximal set of forbidden patterns, but it can have more than one defining set of forbidden patterns. A subshift is\textbf{ of finite
type }(SFT) if it can be defined with a finite set of forbidden patterns.

We always regard subshifts as topological dynamical systems with the shift action. Thus a morphism of subshifts $X\subset A^G$ and $Y\subset B^G$ is a continuous map $\phi\colon X\to Y$ that  commutes with the shift actions. The Curtis-Hedlund-Lyndon theorem states that the map $\phi$ is a morphism if and only if there is a finite set $F\subset G$ and a local function $\mu\colon A^F\to B$ such that $\phi(x)(g)=\mu (g^{-1}x|_F)$. A subshift is \textbf{sofic} when it is the topological factor of an SFT. A third class of interest is that of effective subshifts. We will discuss this class in detail in \Cref{chap:computable_analysis_subshifts}, as we need to introduce some algorithmic notions first.

\section{Computability theory on countable sets}\label{preliminaries-computability}
Here we review some basic facts from computability or recursion theory on countable sets. We use the word \textbf{algorithm} as a synonym of the Turing machine. We will use other common synonyms such as ``effective
procedure''. 

Let $f$ be a function from a subset of $\N$ to $\N$. We say that $f$ is a \textbf{partial} function $\N\to\N$. If the domain of $f$ is equal to $\N$, then we emphasize this fact by calling it a \textbf{total} function. A partial function $f\colon\N\to\N$ with domain $D$ is \textbf{computable} if there is an algorithm that on input $n$ halts if and only if $n\in D$, and which for each $n\in D$ outputs $f(n)$. A set $D\subset\N$ is called \textbf{semi-decidable} when it is the domain of a partial computable function, and \textbf{recursively enumerable} when it is the range of a partial computable function.  A set $D\subset\N$ is \textbf{decidable} when its characteristic function is computable. 

Let us review these concepts in more intuitive terms. A set $D\subset\N$ is decidable when there is an algorithm which for each $n\in\N$, determines correctly whether $n\in D$ or not. $D$ is semi-decidable when there is an algorithm which on input $n\in D$ will detect correctly that $n\in D$, but which on input $n\not\in D$ runs forever. Moreover, $D$ is recursively enumerable when there is an algorithm that enumerates all elements in $D$. It is easily seen that $D\subset\N$ is semi-decidable if and only if it is recursively enumerable, but we will use the first term when we refer to decision processes and the second term to refer to enumerations.

A sequence $(f_n)_{n\in\N}$ of functions is \textbf{uniformly computable} if there is a computable function $f\colon \N^2\to \N$ such that $f(n,m)=f_n(m)$. More generally, when we have a family of objects that is computable in some sense (decidable sets, recursively enumerable sets,...), we say that they are uniformly computable when there is a single algorithm for all of them. 

All these notions extend directly to products $\N^{p}$, $p\geq1$,
and sets of words $A^{\ast}$, as these objects can be represented
by natural numbers in a canonical way. In order to extend these notions
further, we consider numberings:
\begin{defn}
A (bijective)\textbf{ numbering }of a set $X$ is a bijective map
$\nu:N\to X$, where $N$ is a decidable subset of $\N$. We call
$(X,\nu)$ a \textbf{numbered set}. When $\nu(n)=x$, we say that
$n$ is a \textbf{description} or an \textbf{index} for $x$. 
\end{defn}

A numbering of $X$ is used to translate computability notinos from $\N$ to $X$, in the same
manner that charts are used to define continuous or differentiable
function on manifolds. For instance, a function $f\colon X\to X$
is \textbf{computable} on $(X,\nu)$ when the ``function in charts'' $\nu^{-1}\circ f\circ\nu$
is computable. There is a notion of equivalence for numberings: two
\textbf{numberings} $\nu$, $\nu'$ of $X$ are equivalent when the identity
function $(X,\nu)\to(X,\nu')$ is computable. The Cartesian product
$X\times X'$ of two numbered sets $(X,\nu)$, $(X',\nu')$ admits
a unique numbering -up to equivalence- for which the projection functions
to $(X,\nu)$, $(X',\nu')$ are computable. This allows us to speak, without ambiguity, 
about computable functions and relations between numbered sets.

\begin{rem}Of course, any countable set admits a bijection with $\N$. Given a set $X$ and some structure in the form of functions or relations, we are interested in the existence of a numbering that makes these functions or relations computable.  
\end{rem}
\begin{example}\label{numbering-of-partial-computable-functions}
    A numbering $(\varphi_e)_{e\in\N}$ for the set of partial computable functions $\N\to\N$ can be obtained as follows. Each Turing machine is determined by the finite information of a finite alphabet, and a transition rule expressed as a matrix. We just have to code this finite information in a natural number $e$, which becomes the description of the corresponding function $\varphi_e$. One should think that  that the natural number $e$ represents the text of a computer program that computes $\varphi_e$. 
\end{example}
\begin{example}\label{numberinf-of-r.e.-sets}
    A numbering $(W_e)_{e\in \N}$ for the set of 
    recursively enumerable subsets of $\N$ can be obtained from \Cref{numbering-of-partial-computable-functions}. That is, for each $e\in\N$ we let $W_e$ be the domain of $\varphi _e$.
\end{example}
\begin{remark}
    The set of partial computable functions can be given other numberings that are not equivalent to the one in \Cref{numbering-of-partial-computable-functions}, and it is natural to ask what are their relations. This topic goes far beyond this short background chapter. We just mention that  numberings equivalent to the one in \Cref{numbering-of-partial-computable-functions} are called \textit{admissible} numberings, and they can be characterized by abstract properties.   
\end{remark}

Here we will be particularly interested in numberings of countable groups. 
\begin{defn}[\cite{rabin_computable_1960}]
A numbering of a countable group $(G,\cdot)$
is called \textbf{computable } when it makes the group operation $G\times G\to G$
 computable. A group that admits a
computable numbering is called a \textbf{computable group}. 

\end{defn}
For instance, $\mathbb Q$ is a computable group. In the case of finitely generated groups, we have the following well-known fact:
\begin{prop}
\label{prop:computabilitiy-groups-1}Let $G$ be a finitely generated
group. Then:
\begin{enumerate}
\item $G$ admits a computable numbering if and only if it has decidable
word problem.
\item If $G$ admits a computable numbering, then all computable numberings
of $G$ are equivalent.
\item Every group
homomorphism $f\colon G\to H$ between computable groups is computable. 
\end{enumerate}
\end{prop}

\begin{proof}[Proof sketch]
We start with the first item. Suppose that $G$ has decidable word problem, let $S\subset G$
 be a finite and symmetric generating set, and let $\pi\colon S^{\ast}\to G$
be the function that sends a word to the corresponding group element.
Using the decidability of the word problem, we can compute a set $N\subset S^{\ast}$
such that the restriction of $\pi$ to $N$ is a bijection. Being
$N$ a decidable subset of $S^{\ast}$, it admits a computable
bijection with $\N$. The composition of these functions gives a bijection
$\nu\colon\N\to G$, and it is easy to verify that it is a computable
numbering. The reverse implication is left to the reader. Items $2$
and $3$ are also left to the reader: the relevant functions are determined
by the finite information of letter-to-word substitutions, and this
allows to prove that they are computable. 
\end{proof}

\begin{rem}
    In many cases we will endow sets with numberings without providing details, as there is only one natural choice (up to equivalence). This is the case for the set of  all finite subsets of $\N$, the collection of all functions $f\colon F \to A$ where $F$ and $A$ are finite subsets of $\N$, or the collection of all eventually constant sequences in $\N^\N$. 
    
\end{rem}

\section{Computability theory on separable metric spaces}\label{preliminaries-computable-analysis}

In this section we review some basic results for computable metric spaces. We refer the reader to \cite{brattka_handbook_2021,weihrauch_computable_2000,brattka_tutorial_2008} for a more detailed exposition. We will also  prove some results that will be used later. 

Turing \cite{turing_computable_1937} defined a real number to be computable when its decimal expansion is  computable as a function  $f\colon\N\to\{0,\dots,9\}$. It is easy to see that then a number $x\in \R$ is computable when there is a computable sequence of rational numbers $(q_n)_{n\in\N}$ such that for all $n$ we have $|x-q_n|\leq 1/n$. In other words, a real number is computable when we can approximate it with arbitrary precision, and for this purpose we can use a countable set where we can apply computability notions from $\N$. The notion of computable metric space is based on this idea. 

\begin{defn}
A \textbf{computable metric space }is a triple $(\mathcal{X},d,\mathcal{S})$,
where $(\mathcal{X},d)$ is a separable metric space, $\mathcal{S}=\{s_{i}: i\geq1\}$
is a countable and dense subset of $\mathcal{X}$, and the distance function
$d$ is uniformly computable on elements from $\mathcal{S}$. That is, there
is a total computable function $f\colon\N^3\to \Q$ such that for every $(i,j,n)\in\N^3$,  $|d(s_i,s_j)-f(i,j,n)|\leq 1/n$. 

We say that $\mathcal S$ is a \textbf{computable metric space structure}, or simply a \textbf{computability structure}, for $(\mathcal X, d)$. 
\end{defn}

\begin{example}
The following are computable metric spaces:
\begin{enumerate}
\item $\R^{n}$ with the Euclidean metric, and $\mathcal{S}=\Q^{n}$.
\item $[0,1]^{n}$ with the Euclidean metric, and $\mathcal{S}=[0,1]^{n}\cap\Q^{n}$. 
\item $\{0,1\}^{\N}$ with the metric $d(x,y)=\inf\{2^{-n} : n\in\N, \ x(i)=y(i) \ \text{for all} \ i\leq n\}$,
and $\mathcal{S}$ as the set of sequences that are eventually constant.
\item A group $G$ with decidable word problem, with a word-length metric, and with $\mathcal S=G=\{g_i:i\in\N\}$ given by a computable numbering of $G$ (see \Cref{prop:computabilitiy-groups-1}).
\end{enumerate}
\end{example}

A computable metric space admits a countable basis for the metric topology given by balls with center in $\mathcal S$, and with rational radius. In this section, we review a number of effective versions of notions from topology. A computable metric space is a particular case of the more general notion of computable $T_0$ topological space (see \cite[Section 2]{amir_strong_2023}).

\subsection{Open, closed, and compact sets}
Here we define effectively open and closed sets, and recursively compact sets. We also show some basic facts about them.

We start introducing some notation. Let $(\mathcal X,d,\mathcal S)$ be a computable metric space,  which will be fixed during this section. We write $B(x,r)=\{y\in \mathcal{X}\mid d(y,x)<r\}$. We denote by $\overline {B(x,r)}$ the topological closure of $B(x,r)$. If $x\in\mathcal{S}$ and $r\in\Q$, then $B(x,r)$ is called a\textbf{ basic ball}. It is easily seen that the collection of basic balls is a basis for the metric topology on $\mathcal X$. 

For the following definitions, we fix an enumeration $(B_{n})_{n\in\N}$ of all basic balls in $\mathcal X$. By this we mean that from the index $n$ we can compute both an index for the center $s_{i}$ and $r\in\Q$, and vice versa. For this purpose, we take a computable bijection $\N\to\N\times\Q$, $n\mapsto(\varphi_{1}(n),\varphi_{2}(n))$, and define $B_{n}=B(s_{\varphi_{1}(n)},r_{\varphi_{2}(n)})$.
Given a subset $I\subset\N,$ we define $U_{I}=\bigcup_{i\in I}B_{i}.$ Furthermore, we consider a computable enumeration $(F_n)_{n\in\N}$ of all finite subsets of $\N$. We are now ready to give the main definitions of this subsection:
\begin{defn}
A set $V\subset \mathcal{X}$ is \textbf{effectively open }if there is a recursively enumerable set $I\subset \N$ such that $V=U_I$.
\end{defn}
\begin{defn}
    A set $C\subset \mathcal X$ is \textbf{effectively closed} if it is the complement of an effectively open set.
\end{defn}
\begin{defn}
A compact set $K\subset \mathcal{X}$ is \textbf{recursively compact} when it is semi-decidable whether a finite set $F\subset\N$ satisfies $K\subset U_F$. That is, the set $\{n\in \N : K\subset U_{F_n}\}$ is recursively enumerable. 
\end{defn}
Observe that if $K$ is recursively compact, then we can also semi-decide whether $K\subset U$ given an effectively open set $U$. To see this, let $I$ be a recursively enumerable set with $U=\bigcup_{i\in I} B_i$. As $K$ is compact we have that $K\subset U_I$ if and only if $K\subset U_F$ for some finite set $F\subset I$. In order to discover whether $K\subset U$, we enumerate all finite sets $F\subset \N$, we run for each $F$ the algorithms to semi-decide $F\subset I$ and $K\subset U_F$, and halt the whole process once we find that some $F$ satisfies both conditions. 

\begin{rem} 
By definition, effectively open and closed sets, and recursively compact sets, can be specified with finite information, which we call a \textbf{description} for them (see \Cref{preliminaries-computability} and \Cref{numberinf-of-r.e.-sets}). In what follows, when say that we compute an effectively open set $U$ or we decide whether $U$ has a certain property (as in the previous paragraph), we mean that such computation occurs at the level of descriptions. The same remark applies to effectively closed sets, recursively compact sets, and other objects that will be defined in this chapter. 
\end{rem}
We will now prove some basic properties about open and closed sets, and recursively compact sets.

\begin{prop}
\label{closure-of-basic-balls-is-effective}The topological closure of a basic ball is effectively closed.  
\end{prop}
\begin{proof}
    Let $B(x_0,r_0)$ be a basic ball. Given $i\in \N$ and $r\in \Q$, it is semi-decidable whether $d(s_i,x_0)>r+r_0$. This simply follows from the computability of the metric. It follows that then $I=\{n\in \N : B_n=B(s_i,r) \ \text{and} \ d(s_i,x_0)>r+r_0\}$ is recursively enumerable. It is an easy verification that the set $I$ satisfies  $\mathcal X\smallsetminus U_I=\overline{B(x_0,r_0)}$, so the claim follows.
\end{proof}
\begin{rem}
    In general, the topological closure of an effectively open set is not always an effectively closed set. A counterexample will be given in \Cref{halting-constant}.
\end{rem}

\begin{prop}\label{maximal-set-effectively-closed}
    Suppose that $\mathcal X$ is recursively compact. Then a set $C\subset \mathcal X$ is effectively closed if and only if $I=\{i\in \N : \overline{B_i}\cap C=\emptyset\}$ is recursively enumerable. 
\end{prop}
\begin{proof}
    We only prove the forward implication. As $C$ is effectively closed there is a recursively enumerable set $I\subset \N$ with $\mathcal X\smallsetminus C=U_I$. We show that it is semidecidable whether a basic ball $B$ satisfies $\overline B\cap X=\emptyset$, uniformly on $B$. Indeed, given a basic ball $B$, its topological closure $\overline B$ is effectively closed by \Cref{closure-of-basic-balls-is-effective} and then recursively compact by \Cref{prop:effective-closed-subset-of-compact-is-compact}. Now observe that $\overline B\cap X=\emptyset$ when $\overline B\subset U_I$, which is semi-decidable by recursive compactness of $\overline B$. 
\end{proof}

\begin{prop}\label{recursively-compact-is-closed}
    A recursively compact set   is effectively closed.
\end{prop}
\begin{proof}
    Let $C$ be a recursively compact subset of the computable metric space $\mathcal X$. As $C$ is closed, the set $I=\{n\in\N : \overline{B_n}\cap C=\emptyset\}$ satisfies $\mathcal X\smallsetminus U_I=C$. We claim that $I$ is recursively enumerable. In other words, that given a basic ball $B$, it is semi-decidable whether $\overline B\cap C=\emptyset$. This is shown by the following algorithm. On input $B$, we use \Cref{closure-of-basic-balls-is-effective} to compute an effectively open set $U$ with $\mathcal X\smallsetminus \overline B=U$. Observe now that $\overline B\cap C=\emptyset$ exactly when $C\subset U$, and we can semi-decide this condition by recursive compactnes of $C$.
\end{proof}
\begin{prop}\label{prop:effective-closed-subset-of-compact-is-compact}\label{prop.equivalence.comp.closed.subset}
    An effectively closed set contained in a recursively compact set is also recursively compact.
\end{prop}
\begin{proof}
    Let $\mathcal X$ be a computable metric space, let $K\subset\mathcal X$ be recursively compact, and let $C\subset \mathcal X$ be an effectively closed set. Note that $\mathcal X \smallsetminus C$
    is effectively open, and that a finite open cover $U_{I}$ covers $C$ if and
    only if $(\mathcal X\smallsetminus K)\cup U_{I}$ covers $K$. But this condition is
    semi-decidable, as $K$ is recursively compact and $(X\smallsetminus K)\cup U_{I}$
    is effectively open. It follows that $K$ is recursively compact.
\end{proof}
\begin{prop}\label{union-of-effectively-open-sets}
    Let $(U_n)_{n\in\N}$ be a uniformly computable sequence of effectively open sets. Then their union $\bigcup_{n\in\N} U_n$ is effectively open as well. 
\end{prop}
\begin{proof}
    Let $(J_n)_{n\in\N}$ be a uniform sequence of recursively enumerable subsets of $\N$ with $U_n=U_{J_n}$. Then the union $J=\bigcup_{n\in\N}J_n$ is recursively enumerable, as we just have to enumerate all $J_n$ simultaneously and with the same algorithm. As $\bigcup_{n\in\N} U_n=U_J$, our claim follows. 
\end{proof}
\subsection{Computable functions}
Here we review some facts about computable functions on computable metric spaces. For this subsection we fix two computable metric spaces $\mathcal X$ and $\mathcal X'$. As before, we let $(B_n')_{n\in\N}$ be a computable enumeration of basic balls in $\mathcal X'$, and write $U_I'=\bigcup_{i\in I}  B_i'$ for $I\subset \N$.

\begin{defn}\label{def:computable-function}
A function $f\colon \mathcal X\to \mathcal X'$ is \textbf{computable} when there is an algorithm which, given an effectively open $U'\subset \mathcal X$, computes an effectively open set $U\subset\mathcal X$ such that \[f^{-1}(U')=U.\]
\end{defn}
That is, the pre-image of an effectively open set is effectively open, uniformly on the descriptions of the sets. For our purposes we will need a sligthly more general definition, where we have no computability assumption on the domain of the function. 

\begin{defn}\label{computable-function-relative}
    Let $X\subset \mathcal X$ and $X'\subset \mathcal X'$ be arbitrary sets. A function $f\colon X\to X'$ is \textbf{computable} if  there is an algorithm which, given an effectively open set $U'\subset \mathcal X'$ computes an effectively open set $U\subset\mathcal X$ such that  
    \[f^{-1}(U'\cap X')=U\cap X.\]
    
    If moreover $f$ admits a computable inverse, then it is a \textbf{computable homeomorphism} between $X$ and $X'$.
\end{defn}

\begin{remark}
    A computable function must be continuous on its domain. Indeed, given $X\subset\mathcal X$, we have that $\{B_i\cap X : i\in\N\}$ is a basis for the subspace topology for $X$. 
\end{remark}
We now prove some basic results about computable functions.

\begin{prop}\label{composition-of-computable-functions-is-computable}
    Let  $f\colon X\to X'$ and $g\colon X'\to X''$ be computable functions, where $X$, $X'$, and $X''$ are subsets of computable metric spaces. Then the composition $g\circ  f \colon X\to X''$ is computable.
\end{prop}
\begin{proof}
    Given an effectively open set $U''$, we can use the computability of $g$ to compute an effectively open set $U'$ with with $g^{-1}(U''\cap X'')=U'\cap X'$, and then the computability of $f$ to compute an effectively open set $U$ with $f^{-1}(U'\cap X')=U\cap X$. This shows that $g\circ f$ verifices \Cref{computable-function-relative}.
\end{proof}
\begin{prop}
\label{prop:effective-continuous-image-of-compact-is-compact} Let
$f\colon X\to  X'$ be a computable function, and let $K\subset X'$ be recursively compact. Then $f(K)$ is also recursively compact.
\end{prop}

\begin{proof}
We exhibit an algorithm which, given an effectively open set $U'\subset\mathcal X'$, halts if and only if $f(K)\subset U'$. Given $U'$, we compute an effectively open set $U\subset \mathcal X$ such that 
\[f^{-1}(U'\cap X')=U\cap X.\]
The fact that $K\subset X$ ensures that $f(K)\subset U'$ if and only if $K\subset U$. We can semi-decide the last condition by recursive compactness of $K$.
\end{proof}

\begin{prop}
\label{prop:fixed-points-of-computable-function}Let $f\colon X\to X'$
be a computable function.  Then there is an effectively closed set  $C\subset \mathcal X$ such that 
\[\{x\in X : f(x)=x\}=C\cap X\]
\end{prop}

\begin{proof}
Let $F\colon X\to\R$ be defined by  $x\mapsto d(x,f(x))$. It can be shown using \Cref{composition-of-computable-functions-is-computable} that $F$ is a computable function. As $\{x\in \R : x\ne 0\}$ is an effectively open set and $f$ is computable, there is an effectively open set $U\subset \mathcal X$ such that 
$F^{-1}(\{x\in \R : x\ne 0\})=U\cap X.$
Thus the efffectively closed set $C=\mathcal X\smallsetminus U$ satisfies
$F^{-1}(\{x\in \R : x=0\})=C\cap X.$
But $F^{-1}(\{x\in \R : x=0\})=\{x\in X : f(x)=x\}$, so the claim follows.
\end{proof}

\begin{prop}\label{prop:computable-functions-are-invertible-on-compact-domains}Let
$X$ be recursively compact, and let $f\colon X\to X'$ be a computable
and injective function. Then the inverse of $f$ is computable. 
\end{prop}
\begin{proof}  
We exhibit an algorithm which on input an effectively closed set $C\subset \mathcal X$ computes an effectively closed set $C'\subset\mathcal X'$ with 
\[f(C\cap X)=C\cap X'.\]
This shows that $f^{-1}$ verifies \Cref{computable-function-relative} by taking complements. On input $C$, observe that $C\cap X$ is recursively compact by \Cref{prop:effective-closed-subset-of-compact-is-compact}. Moreover, the image  $f(C\cap X)$ is recursively compact by \Cref{prop:effective-continuous-image-of-compact-is-compact}. A recursively compact set is effectively closed set by \Cref{recursively-compact-is-closed}, so we can compute an effectively closed set $C'\subset\mathcal X'$ with $f(C\cap X)=C'$. But then $C'\subset X'$, so we also have $f(C\cap X)=C'\cap X'$. This is what we wanted.
\end{proof}
\subsection{Totally bounded sets and compact sets}
In this subsection we review the following notion. 
\begin{defn}\label{effectively-totally-bounded}
    A set $X\subset \mathcal X$ is \textbf{effectively totally bounded} if there is a computable function $f\colon\N\to\N$ such that for all $n\in \N$, \[X\subset \bigcup_{i=0}^{f(n)} B(s_i,1/n).\] 
\end{defn}
The following result will be useful to prove the recursive compactness of certain spaces.
\begin{prop}\label{effecttively-compact-vs-totally-bounded}
    A computable metric space is recursively compact if and only if it is complete and  effectively totally bounded.
\end{prop}
\begin{proof}
    Let $\mathcal X$ be a computable metric space. We start with the forward implication, so we assume that $\mathcal X$ is recursively compact. Then $\mathcal X$ is compact, and a compact metric space is complete. In order to verify that $\mathcal X$ is effectively totally bounded, observe that for every $n$, the set \[C_n=\{t\in\N : X\subset \bigcup _{i\leq t} B(s_i,1/n)\}\] is recursively enumerable, uniformly on $n$. Moreover, $C_n$ is always nonempty because $\mathcal X$ is compact. Thus we can let $f(n)$ be the first  element enumerated in  $C_n$, and then $n\to f(n)$ is a computable function. This proves the forward implication.

    We now prove the backward implication. We assume that there is a computable function $f$ as in the statement, and we prove that $\mathcal X$ is recursively compact. For this purpose we describe an algorithm $\mathcal A(I,n,m)$ that takes as input a finite set $I\subset \N$, and $n,m\in\N$. This algorithm stops for every input, it may conclude that $\mathcal X\subset U_I$, or it may conclude nothing. 
    
    The algorithm $\mathcal A$ tries to find  whether every ball $B(s_j,1/n)$, $j\leq f(n)$ is contained in  $B_i$ for some $i\in I$. This is done by just computing the distance between the centers of the balls with precision $1/m$, comparing it with the radious of each ball, and using triangle inequality to see if an inclusion is guaranteed. More formally, the process followed by the algorithm $\mathcal A$ on input $(I,n,m)$ is as follows. For each $i\in I$ we compute $\varphi_1(i)$ and $\varphi_2(i)$, so that $B_i=B(s_{\varphi_1(i)},r_{\varphi_2(i)})$. Then we compute $f(n)$, and for each pair $(i,j)$ with $i\in I$ and $j\leq f(n)$, we compute a rational number $r_{i,j}$ so that $|r_{i,j}-d(s_{\varphi_1(i)},s_j)|<1/m$. If for all $j\leq f(n)$ there is $i\in I$ with $r_{i,j}<r_{\varphi_2(i)}+1/n$,  then the algorithm $\mathcal A$ stops and concludes that $\mathcal X\subset U_I$. Otherwise, the algorithm $\mathcal A$ gives no output.

    We now describe an algorithm $\mathcal B$ which takes as input a finite set $I\subset \N$, and halts if and only if $\mathcal X\subset U_I$. In input $I$, the algorithm $B$ runs $\mathcal A(I,n,m)$, for every pair $n,m\in\N^2$ in an ordered manner. If for some $n,m$ the algorithm $\mathcal A$ concludes that  $X\subset U_I$, then $\mathcal B$ stops, and otherwise keeps running. 

    We claim that $\mathcal B(I)$ stops if and only if $\mathcal X\subset U_I$. The forward implication is clear from the definitions of $\mathcal A$ and $\mathcal B$. Suppose now that $\mathcal X\subset U_I$. If we ignore the computability assumptions,  $\mathcal X$ is a complete and totally bounded metric space, which implies that $\mathcal X$ is compact. In particular, Lebesgue covering lemma  ensures that, if $U_I$ is an open cover of $\mathcal X$, there is a number $\delta_I$ such that every subset of $\mathcal X$ with diameter at most $\delta_I$ is contained in an element from the cover $U$. Inspecting the definition of $\mathcal A$ we see that if both $n,m$ are greater than $2/\delta_I$, then $\mathcal A (I,n,m)$ stops, so $\mathcal B (I)$ also stops. 
\end{proof}
\subsection{Computable points}
In this subsection we review computable elements in a computable metric space. 
\begin{defn}
    An element $x\in\mathcal X$ is \textbf{computable} when the set $\{n\in\N : x\in B_n\}$ is recursively enumerable.
\end{defn} 

\begin{prop}\label{characterization-of-computable-points-via-sequences}
    An element $x\in\mathcal X$ is computable if and only if there is a total computable function $f\colon\N\to\N$ such that for all $n\in\N$, \[d(x,s_{f(n)})\leq 1/n. \ \ \]
\end{prop}
\begin{proof}
    We start with the forward implication, so assume that $\{n\in\N : x\in B_n\}$ is recursively enumerable. For each $n$, we compute a natural number $f(n)$ as follows. We enumerate the set $\{n\in\N : x\in B_n\}$ until we find a number $k$  such that $B_k$ has radius at most $1/n$. This implies that the center $s$ of $B_k$ verifies $d(x,s)\leq 1/n$, so we let $f(n)$ be the index of $s$. Then $f$ verifies the claim. 

    We now prove the backward direction. Let $f$ be as in the statement, and let $B(s,r)$ be an arbitrary basic ball. We exhibit an algorithm that semi-decides whether $x\in B(s,r)$. Observe that $x\in B(s,r)$ when for some $n\in\N$ we have $d(s_{f(n)},s)<r+1/n$. Also observe that this inequality is semi-decidable. For the algorithm, we just have to semi-decide the inequality for all $n$, and stop the procedure if we find that it is verified. This shows that  $\{n\in\N : x\in B_n\}$ is recursively enumerable.
\end{proof}

\begin{prop}\label{isolated-points-are-computable}
    Let $K\subset \mathcal X$ be recursively compact. If $x$ is an isolated point in $K$, then it is a computable point.
\end{prop}
\begin{proof}
    As $x$ is isolated in $K$, there is a basic ball $B(s,r)$ with $B(s,r)\cap K=\{x\}$. We let $B=B(s,r/2)$, so $\overline B\cap X=\emptyset$. As $\overline B$ is effectively closed, there is a recursively enumerable set $I\subset \N$ such that $\mathcal X\smallsetminus U_I=\overline B$. But now a basic ball $B_n$ contains $x$ if and only if $K\subset B_n\cup U_I$. This is a semi-decidable condition, so $\{n\in\N : x\in B_n\}$ is recursively enumerable. This proves the claim.
\end{proof}

\begin{defn}
A closed subset $C$ of $\mathcal{X}$ is \textbf{recursively enumerable }(\textbf{r.e.})\textbf{
}if there is a sequence $(x_{i})_{i\in\N}$ of uniformly computable points
that is dense in $C$. If $C$ is both effectively closed and r.e.,
then it is called \textbf{computably closed}.
\end{defn}

\subsection{Semi-computable real numbers and semi-computable functions}\label{preliminaries-semicomputable}
In this section we review semi-computable real numbers and semi-computable functions.
\begin{defn} A real number $c$ is \textbf{upper-semicomputable} when $(c,+\infty)$ is effectively open, and is \textbf{lower-semicomputable} when $(-\infty,c)$ is effectively open.
\end{defn}
The following characterization is well-known, and sometimes used as definition. The proof follows the same idea as \Cref{characterization-of-computable-points-via-sequences}, and is left to the reader.
\begin{prop}\label{characterization-of-semicomputable-numbers}
A real number is upper-semicomputable when it is the infimum of a computable sequence of rational numbers, and lower-semicomputable when it is the supremum of a computable sequence of rational numbers.
\end{prop}
\begin{example}\label{halting-constant}
    Let $(\varphi_e)_{e\in\N}$ be a computable enumeration of all partial computable functions as in \Cref{numbering-of-partial-computable-functions}. Let $H=\{e\in \N : \varphi_e(e)\text{ halts}\}$ be the Halting set. This set is recursively enumerable, and a classic diagonal argument shows that it is not decidable. Let \[c=\sum_{e\in H}2^{-e}.\]
    We claim that $c$ is lower-semicomputable. Indeed, as $H$ is recursively enumerable, there is a recursively enumerable sequence of finite sets $(H_n)_{n\in\N}$ with $H=\bigcup_{n\in\N} H_n$. For each $H_n$ we define a computable real number $c_n=\sum_{e\in H_n}2^{-e}$. Thus $(c_n)_{n\in\N}$ is a computable sequence of rational numbers with $c=\sup_{n} c_n$, so it follows from \Cref{characterization-of-semicomputable-numbers} that $c$ is lower-semicomputable. A simple argument shows that if $c$ was upper-semicomputable, then $H$ would be a decidable set. This shows that $(-\infty, c)$ is an effectively open set whose topological closure is not effectively closed. 
\end{example}
\begin{defn}
    Let $\mathcal X$ be a computable metric space, and let $X\subset\mathcal X$ be a subset. A function $f\colon X\to \R$ is \textbf{upper-semicomputable} when given an effectively open set $(-\infty,c)$ we can compute an effectively open set $U\subset \mathcal X$ such that $f^{-1}(-\infty,c)=U\cap X$. Similarly, $f$ is \textbf{lower-semicomputable} when given an effectively open set $(c,\infty)$ we can compute an effectively open set $U\subset \mathcal X$ such that $f^{-1}(c,\infty)=U\cap X$.
\end{defn}
\begin{prop}\label{semicomputable-functions-are-closed-by-infimum}
    Let $\mathcal X$ be a computable metric space, let $X\subset \mathcal X$, and let $(f_n)_{n\in\N}$ be a uniform sequence of upper-semicomputable functions from $X$ to $\R$. Then their infimum $\inf_{n\in\N}f_n$ is also upper-semicomputable. 
\end{prop}
\begin{proof}
    Let $f=\inf_{n\in\N}f_n$, and let $(-\infty,c)$ be effectively open. As every $f_n$ is upper-semicomputable, for each $n\in\N$, we can compute an effectively open set $U_n$ with $f_n^{-1}(-\infty,c)=X\cap U_n$. But then $(U_n)_{n\in\N}$ is a uniformly computable sequence of effectively open sets, and we have \[f^{-1}(-\infty,c)=(\bigcup_{n\in\N}U_n)\cap X.\]
    \Cref{union-of-effectively-open-sets} shows that $\bigcup_{n\in\N}U_n$ is effectively open, uniformly on  $(-\infty,c)$. Our claim that $f$ is upper-semicomputable follows. 
\end{proof}
\subsection{New computable metric spaces from known ones}

Here we review how to obtain new computable metric spaces from known ones. We also introduce the notion of equivalence for computability structures. 
\begin{defn}
Let $\mathcal S_1$ and $\mathcal S_2$ be computable metric space structures for the metric space $(\mathcal X,d)$. If the identity function from $(\mathcal X,d,\mathcal S_1)$ to $(\mathcal X,d,\mathcal S_2)$ is computable, then we say that $\mathcal S_1$ and $\mathcal S_2$ are \textbf{equivalent}. 
\end{defn}

For instance, let $\mathcal X$ and $\mathcal X'$ be two computable metric spaces. Then their product $\mathcal X\times \mathcal X'$ inherits a natural computable metric space structure from the ones for $\mathcal X$ and $\mathcal X'$. This will be defined in detail in \Cref{thm:Tychonoffcomputable}. This is the unique structure -up to equivalence- that makes the projection functions $\mathcal X\times \mathcal X'\to \mathcal X$ and $\mathcal X\times \mathcal X'\to \mathcal X'$ computable. Although we do not put much attention to this, we mention that many of the constructions that we describe in this section can be  uniquely characterized in a similar manner. 

\subsubsection*{Direct products}
\begin{thm}
\label{thm:Tychonoffcomputable}\label{thm:Tychonoff_computable} Let $(\mathcal X _{n},d_{n},\mathcal{S}_{n})$
be a uniform sequence of computable metric spaces, such that the diameter of all them is bounded by $M$ for some $M\in\N$.  Let $\mathcal{S}$ be the set of elements
in $\prod \mathcal{S}_{n}$ having eventually constant value. Then the product
space $\prod \mathcal X_{n}$, endowed with $\mathcal{S}$ and $d=\sum_{n}2^{-n}d_{n}$,
is a computable metric space too. Moreover, given a sequence $K_{n}\subset \mathcal X_{n}$
of recursively compact sets, uniformly on $n$, the set $\prod K_{n}$
is recursively compact in $\prod \mathcal X_{n}$.
\end{thm}

\begin{proof}
It is straightforward that the metric $d$ is uniformly computable
on $\mathcal{S}$, so $(\prod_{n} \mathcal X_n,d,\mathcal{S})$ is a computable
metric space. We now verify the claim about recursively compact sets.
Let $U$ be a finite union of basic balls in $\prod \mathcal X_{n}$. Then
we can compute a number $N$, and finite sets $F_{n}\subset \N$, such
that $U=U_{F_{0}}\times\dots\times U_{F_{N}}\times\prod_{n>N}\mathcal X_{n}$.
Here each $U_{F_{n}}$ is a finite union of basic balls in $\mathcal X_{n}$,
$n\leq N$. In order to check whether $\prod K_{n}$ is contained
in $U$, it suffices to use the uniform recursive compactness of $K_{n}$
to check the finitely many inclusions $K_{0}\subset U_{I_{0}},\dots,K_{N}\subset U_{I_{N}}.$
\end{proof}
\begin{example}
For each $n\in\N$ let $\mathcal X_{n}$ be the computable metric space $\N$, endowed with
the discrete metric and the set of simple points $\N$. Then the Baire space
\[
\mathcal X=\prod_{n}\N=\N^{\N}
\]
is a computable metric space, with the metric $d$ and set $\mathcal S$ as in \Cref{thm:Tychonoffcomputable}. We can alternatively use the metric
\[d'(x,y)=\inf\{2^{-n} : \N\in\N, \ x(i)=y(i) \ \text{for all} \ i\leq n\}.\] As $d'$ is
uniformly computable on the set $\mathcal{S}$, we have that the identity
function $(\mathcal X,d,\mathcal{S})\to(\mathcal X,d',\mathcal{S})$ is computable.
In other words, the metrics $d$ and $d'$ are computably equivalent. 
\end{example}

\begin{example}
\label{exa:computable-infinite-product-of-finite-alphabets}Let $(a_{n})$
be a computable sequence of natural numbers, with $a_{n}\geq2$ for
all $n$, and endow $\{1,\dots,a_{n}\}$ with the discrete metric.
These are uniformly recursively compact computable metric spaces.
It follows from \Cref{thm:Tychonoffcomputable} that 
\[
\mathcal X=\prod_{n}\{1,\dots,a_{n}\}
\]
is a recursively compact metric space. In particular, $\prod_{n}\{1,\dots,a_{n}\}$
is a recursively compact subset of the computable metric space $\N^{\N}$.
\end{example}

\subsubsection*{Inverse limits}
\begin{prop}
Let $(\mathcal X_{n},d_{n},\mathcal{S}_{n})$ be a uniform sequence of computable
metric spaces, where the diameter of all them is bounded by $M$ for some $M\in\N$.
Let $(f_{n})_{n\geq1}$ be a uniformly computable sequence of computable
functions, with $f_{n+1}\colon \mathcal  X_{n+1}\to \mathcal X_{n}$. Then the inverse
limit 
\[
\varprojlim  \mathcal X_{n}=\{x\in\prod \mathcal X_{n}\mid\text{ for all \ensuremath{n\in\N} we have }f_{n+1}(x_{n+1})=x_{n}\}
\]
is an effectively closed subset of the product space $\prod \mathcal X_{n}$.
\end{prop}

\begin{proof}
Let $f\colon\prod \mathcal X_{n}\to\prod \mathcal  X_{n}$ be the function defined by
$(x_{n})_{n}\mapsto(f_{n+1}(x_{n+1}))_{n}$. We will verify that $f$
is computable, so let $U\subset\prod \mathcal X_{n}$ be an open set with the
form 
\[
U=U_{I_{0}}\times U_{I_{1}}\times\dots\times U_{I_{N}}\times\prod_{n>N}\mathcal X_{n}.
\]
The preimage $f^{-1}(U)$ is 
\[
f^{-1}(U)=f_{1}^{-1}(U_{I_{0}})\times f_{2}^{-1}(U_{I_{1}})\times f_{3}^{-1}(U_{I_{2}})\times\dots\times f_{N+1}^{-1}(U_{I_{N}})\times\prod_{n>N}\mathcal X_{n}.
\]
As $(f_{n})_{n\geq1}$ is a uniformly computable sequence of computable
functions, this shows that the preimage of an effectively open subset
$U$ of $\prod \mathcal  X_{n}$ is effectively open, with an algorithm that
is uniform on $U$. It follows from \Cref{prop:fixed-points-of-computable-function}
that the set of fixed points of $f$ in $\prod \mathcal  X_{n}$ is efffectively
closed. Finally, observe that the set of fixed points of $f$ in $\prod \mathcal X_{n}$
is equal to $\varprojlim \mathcal X_{n}$, so the claim follows. 
\end{proof}
\begin{prop}
\label{prop:inverse-limit-of-finite-alphabets-is-Cantor}Let $(A_{n})_{n\geq0}$
be a uniformly computable sequence of finite subsets of $\N$, and
let $(f_{n})_{n\geq1}$ be a uniformly computable sequence of functions
$f_{n+1}\colon A_{n+1}\to A_{n}$ such that every element in $A_{n}$
has at least two preimages by $f_{n+1}$. Then $\varprojlim A_{n}$
is recursively homeomorphic to $\{0,1\}^{\N}$. 
\end{prop}

\begin{proof}
Let $T\subset\N^{\ast}$ be the decidable subset of $\N^{\ast}$ defined
by $T=\{\epsilon\}\cup\bigcup_{n\in\N}A_{0}\times\dots\times A_{n}$.
We define a function $h\colon T\to\{0,1\}^{\ast}$ recursively. We
first set $h(\epsilon)=\epsilon.$ Now let  $w=w_{0}w_{1}\dots w_{n}$,
with $w_{i}\in A_{i}$ for $i=1,\dots,n$, and assume that $h(w)$ has been defined. We compute
the set $f_{n+1}^{-1}(w_{n})=\{w_{n+1}^{1},\dots,w_{n+1}^{m}\}$,
labeled in such a manner that $i<j$ when $w_{n+1}^{i}<w_{n+1}^{j}$
in lexicographical order. Recall that by hypothesis $m=|f_{n+1}^{-1}(w_{n})|\geq2$
. Now, we compute the set of words $v_{0}=0,\ v_{1}=10,\ ,v_{2}=110,\dots,v_{m-1}=1^{m-1}0,\ v_{m}=1^{m}$.
For each $i$ with $1\leq i\leq m$ we define $h(w_{0}\dots w_{n}w_{n+1}^{i})$
as the concatenation $h(w_{0}w_{1}\dots w_{n})v_{i}$. These conditions
define $h(w)$ for every $w\in T$. From our construction, it is clear
that $h$ is a computable function, it is monotone for the prefix
order, and the length of $h(w)$ tends to infinity with the length
of $w$. Thus $h$ induces a function $H\colon\varprojlim A_{n}\to\{0,1\}^{\N}$
given by $H(y)(k)=h(y_{0}\dots y_{n})(k)$, the $k$-th element in
the word $h(y_{0}\dots y_{n})$, for some $k$ big enough. Then $H$
is computable, bijective, and its inverse is easily seen to be computable.
\end{proof}
\begin{prop}
\label{prop:infinite-product-of-finite-alphabets-is-Cantor}Let $(a_{n})_{n\in\N}$
be a computable sequence of natural numbers, with $a_{n}\geq2$ for
all $n$. Then $\prod_{n}\{1,\dots,a_{n}\}$ is recursively homeomorphic
to $\{0,1\}^{\N}$.
\end{prop}
\begin{proof}
The proof is a simplified version of the proof for \Cref{prop:inverse-limit-of-finite-alphabets-is-Cantor}, and is left
to the reader. %
\end{proof}

\subsubsection*{Hyperspaces}
Let $(\mathcal X,d)$ be a metric space. Given $x\in \mathcal X$ and a compact set $K\subset \mathcal X$, we write \[d(x,K)=\inf_{y\in K} d(x,y).\]
The \textbf{Hausdorff distance} $d_H$ between two compact subsets of $\mathcal X$ is defined by
\[d_H (K,K')=\max\{\sup_{x\in K} d(x,K'), \sup_{x\in K'} d(x,K) \}.\]

\begin{defn}
    The \textbf{hyperspace} associated to a metric space $(\mathcal X,d)$ is the collection $K(\mathcal X)$ of nonempty compact subsets of $\mathcal X$, endowed with the metric $d_H$. 
\end{defn}

The hyperspace associated to a metric space is always a complete metric space. Moreover, it inherits some properties from the original space, such as compactness and separability. Indeed, if $\mathcal S$ is a countable and dense subset of $\mathcal X$, then the collection of finite subsets of $\mathcal S$ is a countable and dense subset of $(K(\mathcal X),d_H)$. We shall prove now that it also inherits a computable metric space structure from the original space.

For the following result, we write $s_F=\{s_i\mid i\in F\}$ for $F\subset \N$ finite. We also recall that $(F_n)_{n\in\N}$ is a computable enumeration of all finite subsets of $\N$.

\begin{prop}\label{Hyperspace-inherits-computable-metric-space-structure}
    Let $(\mathcal X,d,\mathcal{S})$ be a computable metric space.  Define a dense subset of $K(\mathcal X)$ by $\mathcal S'=\{s_{F_n}:n\in\N\}$. Then $(K(\mathcal X),d_H,\mathcal{S}')$ is a computable metric space. If $\mathcal X$ is recursively compact, then the same holds for $K(\mathcal X)$. 
\end{prop}
\begin{proof}
     First observe that $d_H$ is uniformly computable on elements from $\mathcal S'$. In other words, that we can compute $d_H(s_I,s_J)$ from finite sets $I,J\subset \N$. This follows from the definition of $d_H$, that  \[\max\colon\R^2\to \R,\ \ \ \min\colon\R^2\to\R\]
    are computable functions, that $d\colon \mathcal X\times \mathcal X\to \mathcal  X$ is a computable function, and that the composition of computable functions is computable. This proves that $(K(\mathcal X),d_H,\mathcal{S}')$ is a computable metric space.

    We now prove that if $\mathcal X$ is recursively compact, then the same is true for $K (\mathcal X)$. First observe that the space $K (\mathcal X)$ is complete, so by \Cref{effecttively-compact-vs-totally-bounded} it suffices to prove that $K(\mathcal X)$ is effectively totally bounded. Observe that by \Cref{effecttively-compact-vs-totally-bounded} the set $\mathcal X$ is effectively totally bounded, so there is a computable function $f\colon\N\to\N$ such that for each $n\in\N$, \[\mathcal X\subset \bigcup_{i=0}^{f(n)} B(s_i,\frac{1}{2n}).\] 
    For each $n$, we let $g(n)$ be the least natural number so that all finite subsets of $\{0,\dots,f(n)\}$ appear in $\{I_0,\dots,I_{g(n)}\}$. Thus $g$ is a computable function, and $g$ shows that $K(\mathcal X)$ verifies \Cref{effectively-totally-bounded}. This follows from the fact that for every compact set $K\subset \mathcal X$, there a subset  $I\subset \{0,\dots,f(n)\}$ such that $d_H(K,s_I)<1/n$. This can be seen taking $I$ as  $\{i\in\{0,\dots,f(n)\} : B(s_i,\frac{1}{n})\cap K\ne\emptyset\}$.   
\end{proof}
The following result follows easily from the definition, and its proof is left to the reader.
\begin{prop}\label{hyperfunction-of-a-computable-function-is-computable}
    Let $(\mathcal X,d,\mathcal S)$ be a computable metric space, and let $f\colon \mathcal  X\to \mathcal X$ be a computable function. Then the function  $K(\mathcal X)\to K(\mathcal X)$, $Y\to f
    (Y)$ is computable.
\end{prop} 
\subsection{Zero-dimensional sets in computable metric spaces}

Here we will prove some results regarding subsets of computable metric
spaces whose topological dimension is zero. These results will be used in \Cref{chap:EDS}.

We say that a subset of a topological space is \textbf{clopen} when it is both open and closed. We recall that a topological
space has topological dimension zero if and only if every open cover admits a refinement of clopen sets. 
\begin{prop}
\label{prop:computable-clopens-dimension-zero}\label{prop.clopens}Let $K$ be a zero-dimensional
recursively compact subset of a computable metric space. Then we can
compute a collection $(C_{n})$ of finite unions of basic balls in
$X$, such that $(C_{n}\cap K)$ forms a basis of clopen sets for
$K$ (with the subspace topology).
\end{prop}

\begin{proof}
Since $K$ is zero-dimensional, it has a basis of clopen sets. Let $C$ be clopen set in
$K$. Since $C$ is open in $K$, there is a set $U$ that is open
in $X$ with $C=U\cap K$. Since the set $U$ is open and the sets
$B_{i}$ are a basis for the topology on $X$, we can write $C=U_{I}\cap K$.
Since $C$ is compact and $\{B_{i}\cap K\mid i\in I\}$ is an open
cover for $C$ in the subspace topology, it follows that we can write
$C=U_{I}\cap K$ for a finite set $I$. 

Let $V_{1},V_{2},\dots$ be the collection of all finite unions of
basic balls. We have shown that this collection contains a basis of
clopen sets for $K$. We claim that given a finite union $V$, we
can semi-decide whether $C=V\cap K$ is clopen in $K$.  Indeed, this
occurs exactly when $V\cap K=\overline{V}\cap K$, where $\overline{V}$
is the union of the closures the finitely many balls defining $V$.
But this is the case exactly when $\overline{V}\cap K\subset V$.
As $\overline{V}\cap K$ is recursively compact and $V$ is effectively
open, this inclusion is semi-decidable. 
\end{proof}
\begin{prop}
\label{thm:compacto-y-cero-dimensional-tiene-copia-en-el-cantor}
Let $X$ be a nonempty recursively compact zero-dimensional subset
of a computable metric space. Then $X$ is computably homeomorphic
to an effectively closed subset $E$ of $\{0,1\}^{\N}$. Moreover:
\begin{enumerate}
\item If $X$ is computably closed, then $E$ can be taken to be computably
closed.
\item If $X$ is computably closed and has no isolated points, then $E$
can be taken to be $\{0,1\}^{\N}$. 
\end{enumerate}
\end{prop}

\begin{proof}
We start by proving the first item. If $X$ is a singleton then the
claim clearly holds, so we assume from now on that $X$ is not a singleton.
It follows from \Cref{prop:computable-clopens-dimension-zero} that
one can compute a sequence $(\mathcal{P}^{n})$, $n\in\N$, of partitions
of $X$ made by clopen sets, where:
\begin{enumerate}
\item $\mathcal{P}^{0}$ has at least two elements. This is possible because
$X$ is not a singleton. 
\item Each set in $\mathcal{P}^{n}$ has diameter at most $2^{-n}$.
\item The elements in $\mathcal{P}^{n}$ are unions of elements in $\mathcal{P}^{n+1}$.
\end{enumerate}
Indeed, it suffices to observe that given any $n\in\NN$, by using \Cref{prop:computable-clopens-dimension-zero}, one can uniformly compute a basis of clopen sets whose diameter is less than $2^{-n}$ (recall that from the index of a basic ball one can compute its radius). Now, given any clopen and effectively closed set, call it $C$, we can use recursive compactness of $C$ to find a finite collection of basic clopen sets of small diameter that covers $C$. By iteratively applying this observation, we can compute the sequence of partitions  $(\mathcal P^n)_{n\in\N}$ as needed.

We index each $\mathcal{P}^{n}$ as $\{P_{0}^{n},\dots,P_{k_{n}}^{n}\}$,
and for each $n\in\N$ we let $B_{n}=\{0,\dots,k_{n}\}$. Let $Y=\prod B_{n}$.
It follows from \Cref{exa:computable-infinite-product-of-finite-alphabets}
that $Y$ is an effectively closed and recursively compact subset
of $\N^{\N}$. We define a function $f\colon X\to Y$ by defining
$f(x)(n)$ as the index of the element of $\mathcal{P}^{n}$ that
contains $x$. It is clear that $f$ is a computable function. Then
the image of $f$ is effectively compact by \Cref{prop:effective-continuous-image-of-compact-is-compact},
and then effectively closed by \Cref{prop:effective-closed-subset-of-compact-is-compact}.
Moreover, it follows from \Cref{prop:computable-functions-are-invertible-on-compact-domains}
that $f$ is recursively homeomorphic to its image in $Y$. But $Y$
is recursively homeomorphic to $\{0,1\}^{\N}$ by \Cref{prop:infinite-product-of-finite-alphabets-is-Cantor},
so the claim follows by composing with this homeomorphism. 

We now consider the case where $X$ is computably closed and has no
isolated points. Then we need to add two extra conditions to the sequence
$(\mathcal{P}^{n})_{n}$:
\begin{enumerate}
\item Each element in $\mathcal{P}^{n}$ actually intersects $X$. 
\item The collection $\mathcal{P}^{n+1}$ satisfies that every element in
$\mathcal{P}^{n}$ is the union of at least two elements from $\mathcal{P}^{n+1}$. 
\end{enumerate}
Indeed, starting from a sequence $(\mathcal{P}^{n})_{n}$ as before,
the first new condition can be satisfied by removing those elements
whose intersection with $X$ is empty. This is possible because $X$
is assumed to be computably compact, so we can decide whether a clopen
set intersects $X$ or not. Then, we can satisfy the second new condition
by taking a subsequence. Here we are using the fact that $X$ has
no isolated points, so we know that an exhaustive search will stop. 

Having the sequence $(\mathcal{P}^{n})$, with $\mathcal{P}^{n}=\{P_{0}^{n},\dots,P_{k_{n}}^{n}\}$,
we let $B_{n}=\{P_{0}^{n},\dots,P_{k_{n}}^{n}\}$, and define a sequence
of maps $f_{n+1}\colon B_{n+1}\to B_{n}$ by the condition that $P_{i}^{n+1}\subset P_{f_{n+1}(i)}^{n}$,
$i\in\{0,\dots,k_{n+1}\}$. Then we let $Y=\varprojlim B_{n}$. The
rest of the argument is similar to the one in the first case, but
now we use the fact that $Y$ is recursively homeomorphic to $\{0,1\}^{\N}$
by \Cref{prop:inverse-limit-of-finite-alphabets-is-Cantor}.
\end{proof}
\begin{cor}
Let $K\subset X$ be a recursively compact set that is homeomorphic
to (a subset of) $\{0,1\}^{\N}$. Then the homeomorphism can be taken
to be computable. 
\end{cor}

%% file: Chapter_computable_analysis_and_subshifts.tex
The goal of this chapter is describing a computable metric space structure for subshifts on finitely generated groups.  This will be used in Chapters \ref{chap:EDS}, \ref{chap:Medvedev} and \ref{Chap:computable-analysis-on-S(G)}. This exposition is an expanded version of \cite[Section 3.3]{barbieri_effective_2024}.

\minitoc
\section{Introduction}
The space $A^\Z$ admits a natural computable metric space structure. With this structure, basic balls are cylinder sets, and a set $X\subset A^\Z$ is effectively closed when the collection of all cylinders having empty intersection with $X$ is recursively enumerable (see \Cref{maximal-set-effectively-closed}). 

It is easy to see that an SFT $X\subset A^\Z$ is effectively closed. A finite set of forbidden words that defines $X$ can be easily used to define an algorithm that enumerates all cylinder sets having empty intersection with $X$.  Using the Curtis Lindon Hedlund Theorem, it is also clear that a sofic subshift is effectively closed. It is natural then to define the class of effective subshifts as those subshifts that are  effectively closed sets. This notion was introduced in \cite{hochman_dynamics_2009}, and since then it has been considered several times in the literature 
\cite{cenzer_computable_2008,hochman_characterization_2010,aubrun_simulation_2013,durand_effective_2010}. Generalizing the  notion of effective subshift from $\Z$ to a finitely generated group is not a straightforward task. One reason is that the group under consideration may have an undecidable word problem, and this is reflected in the algorithmic tools available to handle patterns and cylinder sets in $A^G$. 

The purpose of this chapter is presenting a generalization of the notion of effective subshift from $\Z$ to
an arbitrary finitely generated group form the viewpoint of computable analysis. That is, we will define a subshift as effective when it is an effectively closed subset of a suitable computable metric space. The content of this chapter is closely related to the results in \cite{aubrun_notion_2017}, where the authors consider a generalization based on combinatorial and algorithmic properties of forbidden patterns. As we shall see, we arrive at equivalent notions for recursively presented groups. In the general case, however, the notions are not equivalent.
\section{A computable metric space structure for $A^G$ when $G$ has decidable word problem}\label{computable-metric-space-structure-fullshift-decidable-word-problem}

In this section we endow $A^G$ with a computable metric space structure, when $G$ is a group with decidable word problem and $A$ is a finite alphabet.

\begin{definition}\label{computable-metric-space-structure-on-A^G}
    Let $G$ be a finitely generated group with decidable word problem, and let $\nu\colon \N\to G$ be a computable bijection (see \Cref{prop:computabilitiy-groups-1}). Given a finite alphabet $A$, we define the homeomorphism 
    \[ \psi\colon A^{\N}\to A^{G},\ \ \ \psi(x)(g)=x(\nu^{-1}(g)).\]
    The sets $A^\N$ and $A^G$ are endowed with the prodiscrete topology, and $A^\N$ is given the standard computable metric space structure (see \Cref{preliminaries-computable-analysis}). We use $\psi$ to transport this structure to $A^G$. That is, we define a metric $d$ on $A^G$ by declaring $\psi$ to be an isometry, and we let $\mathcal{S}$ be the set of images by $\psi$ of eventually constant sequences. Thus $\mathcal{S}$ is a dense subset of $A^{G}$ where the distance $d$ is uniformly computable. Then $(A^G,d,\mathcal S)$ is a computable metric space. 
\end{definition}
A basic but important feature of this structure is that it makes the shift action $G\curvearrowright A^G$ computable:
    
    \begin{prop}\label{shift-action-is-computable}
    Let $G$ be a group with decidable word problem. The shift map $G\times A^G\to A^G$,  $(g,x)\mapsto gx$ (where $gx(h)=x(g^{-1}h)$), is computable. 
    \end{prop}

    \begin{proof}
    The computability of the shift map reduces to the computability of functions on $\N$ and $A^\N$ as follows. Given a computable bijection $\nu\colon\N\to G$, we can define a group operation $\star$ on $\N$ by declaring $\nu$ to be a group isomorphism. Then $\star\colon\N^2\to \N$ is a computable map. It follows that $i\colon\N\to\N$ defined by $i(n)\star n=\nu ^{-1}(1_G)$ is computable. Then $f\colon \N\times A^\N\curvearrowright A^\N\to A^\N$ given by $f(n,x)(m)=x(i(n)\star m)$ is a computable map. The fact that $f$ is computable, plus the fact that the computable homeomorphism $\psi$ from \Cref{computable-metric-space-structure-on-A^G} conjugates $f$ with the shift map $G\times A^G\to A^G$, shows that the last one is a computable map.  
\end{proof}

Let us now make a few observations about these definitions.
\begin{remark}\label{computable-metric-space-on-A^G-infinite-alfabet}
    Our focus is on finitely generated groups, but \Cref{computable-metric-space-structure-on-A^G} and \Cref{shift-action-is-computable} are valid for the larger class of countable computable groups (such as $\Q$). Similarly, these results do not actually use the finiteness of the alphabet, and we can give a computable metric space structure to the space $\N^G$ following the same procedure.
\end{remark}
\begin{remark}    Recall from \Cref{preliminaries-computable-analysis} that two computable metric space structures $(X,d, \mathcal S)$ and  $(X,d', \mathcal S')$ are called equivalent when the identity function is computable between them. This allows us to express that the structure defined in  \Cref{computable-metric-space-structure-on-A^G} is in a certain way inherent to $G$. For instance, if we take two different numberings for the group, then the structures given by \Cref{computable-metric-space-structure-on-A^G} are equivalent. The elementary proof is left to the reader (see \Cref{prop:computabilitiy-groups-1}). 
\end{remark}
    \begin{remark}
    For finitely generated groups, \textit{all} numberings that make the multiplication action $G\curvearrowright G$ computable are equivalent (see \Cref{prop:computabilitiy-groups-1}). We will prove in \Cref{chap:EDS} that shift spaces enjoy a similar property. That is, it follows from \Cref{prop:EDS_and_effective_subshift_are_equivalent} that \textit{all} computable metric space structures for $A^G$ for which the shift action $G\curvearrowright A^G$ is computable,  are equivalent.  
\end{remark}
\begin{remark}\label{compatibility-of-metrics-for-A^G}
    For a finitely generated group $G$, it is customary to consider a metric for $A^G$ associated to a word metric on $G$ (see \Cref{preliminaries-shift-spaces}). All these metrics are compatible with  \Cref{computable-metric-space-structure-on-A^G}. That is, they define equivalent computable metric space structures on $A^G$. The elementary proof is left to the reader.  
\end{remark}
\section{Effective subshifts}\label{subsec:pullback-subshift}
The goal of this section is defining effective subshifts.  For some finitely generated infinite groups, the corresponding fullshift
$A^{G}$ can \emph{not }be endowed with a computable metric space
structure for which the action $G\action A^{G}$ by translations is
computable. However, it can be naturally identified with a subset
of a computable metric space. More precisely, the set $A^{G}$
can always be identified with a closed subset of $A^{F}$, for a suitable
finitely generated free group $F$, and the space $A^{F}$ can be given a natural computable
metric space structure as described in \Cref{computable-metric-space-structure-on-A^G}. We will need the following notion.

\begin{definition}
    Let $G$ be a finitely generated group, and let $X\subset A^G$ be a subshift. We define a \textbf{pullback subshift} as follows. Let $S$ be a finite set of generators for $G$, let $F(S)$ be the free group generated by $S$, and let $\phi\colon F(S)\to G$ be the group homomorphism that maps every reduced word on $S$  to its corresponding element
    in $G$. Let $\phi^*\colon A^{G}\to A^{F(S)}$ be the
    map given by $\phi^*(x)(w)=x(\phi(w))$ for every $w\in F(S)$. We define the pullback subshift $\widehat{X}\subset A^{F(S)}$  by 
\[
\widehat{X}=\{\phi^*(x)\in A^{F(S)}\mid x\in A^{G}\}.
\]

\end{definition}
From a dynamical point of view, the pullback subshift contains all the information that we need, as the actions $G\curvearrowright X$ and $G\curvearrowright \widehat X$ are topologically conjugate.  The interest of replacing a subshift by its pullback is that the last one is contained in a computable metric space.
\begin{defn}
\label{def:effectivesubshiftthroughfreegroup}A subshift $X\subset A^{G}$
is \textbf{effective }if for some finite symmetric generating set
$S$ for $G$, the pullback subshift $\widehat{X}$ is an effectively
closed subset of $A^{F(S)}$.
\end{defn}

\begin{rem}\label{pullbacks-are-recursively-homeomorphic}
A subshfit is effective if and only if its pullback to \textit{every} generating set is an effectively closed set. Indeed, suppose that $\widehat X\subset A^{F(S)}$ is effectively closed, let $S'$ be another set of generators, and denote by $\widehat{X}'$ the pullback of $X$ in $A^{F(S')}$. For
each $s\in S$, let $\phi(s)$ be a word in $S'$ such that $s$ is
equal to $\psi(s)$ in $G$. Then $\psi$ extends to a computable
injective homomorphism from $F(S)$ to $F(S')$. It follows that the
map $\Psi\colon\widehat{X}\to\widehat{X}'$ given by $\Psi(x)(w)=x(\psi(w))$
is a computable homeomorphism. In particular, $\widehat{X}'$ is also
effectively closed. 
\end{rem}
\begin{rem}If $G$ has decidable word problem, then a subshift $X\subset A^G$ is an effectively closed subset of $A^G$  with the structure for $A^G$, if and only if $\widehat{X}$ is effectively
closed in $A^{F(S)}$. This follows from the fact that $A^G$ is recursively homeomorphic to $\widehat{A^{G}}$ with the corresponding structures.  
\end{rem}

\section{All morphisms of subshifts are computable}
In this section we observe that all morphisms of subshifts are computable. In the case where we have no assumption on the word problem of the group, this is formally expressed through pullbacks. We use this to show that the class of effective subshifts is closed by factor maps. 

\begin{prop}\label{prop:morphisms-of-subshifts-are-computable-wp-dec}
    Let $G$ be a finitely generated group with decidable word problem, and let $\psi\colon X\to Y$ be a morphism of subshifts. Then $\psi$ is a computable function. 
\end{prop}
\begin{proof}
    Suppose that $X\subset A^G$ and $Y\subset B^G$. By the Curtis-Lyndon-Hedlund theorem, there is a finite set $F\subset G$ and a local function $\mu\colon A^F\to B$ such that $\psi(x)(g)=\mu(g^{-1}x|_F)$ for all $x$ and $g$. It is clear that such a function is computable. 
\end{proof}
\begin{prop}\label{prop:morphisms-of-subshifts-are-computable-in-pullbacks}
    Let $G$ be a finitely generated group, and let $\psi\colon X\to Y$ be a morphism of subshifts. Then $\psi$  is computable on pullbacks, in the sense that there is a computable function $\widehat \psi$ making the following diagram commute:
    \begin{center}
    \begin{tikzcd}
    X \arrow[r, "\psi"]\arrow[d]  & Y \arrow[d] 
    \\
    \widehat{X} \arrow[r,"\widehat{\psi}"]  & \widehat{Y}
    \end{tikzcd}
    \end{center}
    Vertical arrows represent the conjugacy between a subshift and its pullback from \Cref{def:effectivesubshiftthroughfreegroup}. 
\end{prop}
\begin{proof}
    After we fix bijective maps $X\to \widehat{X}$ and $Y\to\widehat{Y}$  as in \Cref{def:effectivesubshiftthroughfreegroup}, the diagram determines the function $\widehat{\psi}$. But now $\widehat \psi$ is a morphism of subshifts on a free group, so it is computable by \Cref{prop:morphisms-of-subshifts-are-computable-wp-dec}.
\end{proof}

A consequence of this result is the following. 
\begin{prop}\label{effective-subshifts-are-closed-by-factors}
    Let $G$ be a finitely generated group. The class of effective $G$-subshifts is closed by topological factor maps.
\end{prop}
\begin{proof}
    Let $\psi\colon X\to Y$ be a topological factor map, and let $\widehat \psi$ as in \Cref{prop:morphisms-of-subshifts-are-computable-in-pullbacks}. Let us recall that the fullshift on a free group is recursively compact. Joining this with \Cref{prop:effective-continuous-image-of-compact-is-compact}, we have that  $\widehat X$ is recursively compact. Recursive compactness is preserved by computable images (\Cref{prop:effective-continuous-image-of-compact-is-compact}), so $\widehat Y$ is also recursively compact. But this implies that $\widehat Y$ is effectively closed (\Cref{prop:effective-closed-subset-of-compact-is-compact}), so we conclude that $Y$ is an effective subshift. 
\end{proof}
\begin{remark}
    Some analogies have been observed between finitely generated groups and subshifts, and particularly an analogy between simulation results for subshifts and Highman's embedding theorem for groups \cite{jeandel_enumeration_2017}. This is closely related to the fact that morphisms are automatically computable. Indeed, for finitely generated groups, all morphisms are computable on words, as they are determined by the finite information of a letter-to-word substitution. Analogously, the Curtis-Hedlund-Lyndon theorem shows that all morphisms of subshifts are determined by the finite information of a local function, and it follows from this fact that a morphism of subshifts is always computable on pullbacks. This is the fundamental reason for which algorithmic properties of subshifts are conjugacy invariants. 
\end{remark}
\section{Effective subshifts and patterns}
In \cite{aubrun_notion_2017} there is a notion of effective subshift defined in terms of patterns and pattern codings. In this section we compare them with \Cref{def:orbit-membership-problem}.  

Let $G$ be a finitely generated group, and let $A$ be an alphabet. Let us recall that a pattern is a function $p\colon F\to A$, where $F\subset G$ is finite. Patterns are used to define subshifts and SFTs. If we do not assume $G$ to have decidable word problem then it is natural to represent group elements by words, and then use this to represent patterns. The definition of effective subshift from \cite{aubrun_notion_2017} is based on this idea. Let $S$ be a finite and symmetric generating set for $G$. A \textbf{pattern coding} is a function $c\colon W\to A$, where $W\subset S^\ast$ is a finite set. The support $W$ of $c$ is denoted $\operatorname{supp}(c)$. The cylinder set induced by a pattern coding $c$ is given by \[ [c] = \{ x \in A^{G} : x(\underline{w}) = c(w) \mbox{ for every } w \in \operatorname{supp}(c) \}.      \]

A pattern coding is \define{inconsistent} if there are $u,v \in \operatorname{supp}(c)$ such that $\underline{u}=\underline{v}$ but $c(u)\neq c(v)$. Notice that for inconsistent pattern codings we have $[c] = \emptyset$. The set of all pattern codings with generating set $S$ and alphabet $A$ admits a natural numbering, and thus we can speak about computability of sets of pattern codings. Given a set $\mathcal{C}$ of pattern codings for alphabet $A$, we can define a subshift $X_{\mathcal{C}}\subset A^{G}$ by  \[ X_{\mathcal{C}} = A^{G} \setminus \bigcup_{g \in G, c \in \mathcal{C}} g[c].  \]

\begin{definition}[\cite{aubrun_notion_2017}]
\label{def:ecp}
    A subshift $X \subset A^{G}$ is \define{effectively closed by patterns} (ECP) if there exists a recursively enumerable set of pattern codings $\mathcal{C}$ such that $X = X_{\mathcal{C}}$.
\end{definition}

When we define a subshift $X=X_{\mathcal C}$ by a set of forbidden pattern codings, we must be aware that the set $\mathcal C$ is not unique.  For instance, we can add or remove inconsistent pattern codings without changing the subshift. Moreover, every subshift $X$ admits a defining set of pattern codings that is maximal for inclusion, which is the set of all pattern codings with $[c]\cap X=\emptyset$.

In \cite{aubrun_notion_2017}, the authors observe that given a subshift $X$, there is a difference between the existence of \textit{some} defining set of pattern codings that is recursively enumerable, and having the same property for the maximal one. The second condition clearly implies the first. Moreover, they are equivalent for recursively presented groups:
\begin{proposition}[\cite{aubrun_notion_2017}] 
\label{ecp-and-maximal-set}
    Let $G$ be a recursively presented group, and let $X$ be a $G$-subshift. Then the following two conditions are equivalent:
    \begin{enumerate}
        \item $X$ admits some defining set of pattern codings that is recursively enumerable (i.e. $X$ is ECP).
        \item The maximal defining set of pattern codings for $X$, $\{c : [c]\cap X=\emptyset\}$, is recursively enumerable.
    \end{enumerate}
\end{proposition}
The relation between the notion of effective subshifts (\Cref{def:effectivesubshiftthroughfreegroup}), and effectively closed by patterns subshifts (\Cref{def:ecp}), is that a subshift is effective exactly when the maximal defining set of pattern codings is recursively enumerable. 

\begin{proposition}\label{prop_effective_implies_max_forbidden_is_re}\label{prop:effective_implies_max_forbidden_is_re}
    Let $G$ be a finitely generated group. A subshift $X\subset A^{G}$ is effective if and only if the maximal defining set of pattern codings for $X$, $\{c : [c]\cap X=\emptyset\}$, is recursively enumerable.
\end{proposition}

\begin{proof}
    Let $\mathcal C_{\mbox{max}}$ be the set of all pattern codings $c$ such that $[c]\cap X = \varnothing$, and let $\mathcal C_{\mbox{max}}'$ be the set of elements in $\mathcal C_{\mbox{max}}$ that are consistent for the free group $F(S)$. As the set of pattern codings which are consistent for the free group is decidable, it follows that  $\mathcal C_{\mbox{max}}(X)$ is recursively enumerable if and only if $\mathcal C_{\mbox{max}}'$ is recursively enumerable. Observe now that $\mathcal C_{\mbox{max}}'$ is recursively enumerable if and only if $\widehat{X}\subset A^{F(S)}$ is effectively closed. This simply follows from the fact that $\mathcal C_{\mbox{max}}'$ corresponds to all cylinders in $A^{F(S)}$ that do not intersect $\widehat X$. 
\end{proof}

Joining \Cref{ecp-and-maximal-set} and \Cref{prop_effective_implies_max_forbidden_is_re}, we have the following.
\begin{corollary}\label{effective-subshifts-and-ecp}\label{cor:ECP_plus_RP_implies_EDS}
    Every effective subshift is effectively closed by patterns. Moreover, the two notions are equivalent for recursively presented groups. 
\end{corollary}
As all SFTs are ECP, we have the following.
\begin{prop}\label{prop:sft-and-sofic-subshifts-are-effective}
    Let $G$ be recursively presented. Then $G$-SFTs and sofic $G$-subshifts are effective.
\end{prop}
\begin{proof}
    Let $X$ be a $G$-SFT. Then $X$ is effectively closed by patterns. It follows from \Cref{effective-subshifts-and-ecp} that $X$ is effective. This proves the claim regarding SFTs. Now  the claim for sofic subshifts follows from this and \Cref{effective-subshifts-are-closed-by-factors}.
\end{proof}
The following example shows that for groups that are not recursively presented, the class of ECP is strictly smaller than the class of effective subshifts.  
\begin{prop}\label{rem:fullshiftRPiseffective}\label{example:fullshift_RP_is_effective}
    Let $G$ be a finitely generated group, and let $A$ be an alphabet with $|A|\geq 2$. Then $A^G$ is effectively closed by patterns. However, $A^G$ is effective exactly when $G$ is recursively presented.
\end{prop}
\begin{proof}
    Observe that $A^G$ is effectively closed by patterns because because it can be defined by the empty set of pattern codings. Thus for $G$ recursively presented, it follows from \Cref{effective-subshifts-and-ecp} that $A^G$ is effective. 

    For the converse, let us fix a symmetric and generating set $S$, and suppose that $A^G$ is effective. Then  \Cref{prop_effective_implies_max_forbidden_is_re} implies that the maximal defining set of pattern codings for $A^G$ is recursively enumerable. Observe that this is exactly the set of inconsistent pattern codings,  so the set of inconsistent pattern codings with alphabet $A$ and generating set $S$ is recursively enumerable. This implies that $G$ is recursively presented by \cite[Proposition 2.2]{aubrun_notion_2017}.  
\end{proof}

%% file: Chapter_effective_dynamical_systems.tex
\renewcommand{\Gamma}{{G}} 
In this chapter we present the results obtained with Sebastián Barbieri
and Cristóbal Rojas in \cite{barbieri_effective_2024}. We develop and study the notion of effective dynamical system on general metric spaces.  Our main result states that under suitable hypotheses, these systems admit effective zero-dimensional topological extensions. The main application of this result is that, combined with simulation results, it can be used to construct new examples of systems that are factors of SFTs. 

%
\minitoc
\section{Results}

Starting with the work of Hadamard~\cite{hadamard_surfaces_1898}
and the highly influential article of Morse and Hedlund~\cite{morse_symbolic_1938},
symbolic dynamical systems have quite often played a pivotal role
in the understanding of more general dynamics. A celebrated instance
of this is the prominent role of subshifts of finite type (SFT) in
the study of Anosov, and more generally of Axiom A diffeomorphisms,
through Markov partitions of their non-wandering sets~\cite{bowen_axiom_1978}.
Another well-known example of this tight relationship is the fact
that the natural action of a word-hyperbolic group on its boundary
is a very well behaved topological factor of an SFT on the same group~\cite{coornaert_symbolic_1993}.

These and similar results raise the question of understanding precisely
which dynamical systems are topological factors of SFTs. An observation
made by Hochman~\cite{hochman_dynamics_2009} is that subactions
of multidimensional SFTs satisfy strong computability constraints. More precisely, they are conjugated to an action by computable homeomorphisms on an effectively closed subset of $\{0,1\}^\N$ (see \Cref{preliminaries-computable-analysis}). The truly remarkable discovery of Hochman
is that, up to a difference in the dimension of the acting group and
a topological factor, this is the only constraint: every computable
homeomorphism on an effectively closed subset of $\{0,1\}^\N$ is the
topological factor of a subaction of a $\ZZ^{3}$-SFT.

Results linking computable maps on effectively closed subsets of $\{0,1\}^\N$ have been called ``simulation results'', as they express
that very explicit and simple models such as SFTs are capable of universally
encoding this considerably larger class of dynamical systems. These
simulation results along with further developments~\cite{aubrun_simulation_2013,durand_effective_2010}
led to a new understanding of classical results in the theory of symbolic dynamics of group actions, such
as the undecidability of the domino problem~\cite{wang_proving_1961,berger_undecidability_1966},
the existence of aperiodic tilesets~\cite{berger_undecidability_1966,robinson_undecidability_1971}
and more generally the existence of two-dimensional tilings without
computable orbits~\cite{hanf_nonrecursive_1974,myers_nonrecursive_1974}.
Moreover, they also provided the tools to obtain new and long sought-after
results, such as the classification of topological entropies of multidimensional
SFTs~\cite{hochman_characterization_2010}.

Further work has extended the initial result of Hochman to the context
of actions of discrete groups on zero-dimensional spaces~\cite{barbieri_generalization_2019,barbieri_geometric_2019,barbieri_soficity_2023}.
Notably, it was shown that for a large class of non-amenable groups
called self-simulable, the class of zero-dimensional topological factors
of SFTs contains every possible computable action on an effectively
closed zero-dimensional set~\cite{barbieri_groups_2022}. These works
have led to new results about the dynamics of such groups. For instance,
they have provided new examples of groups that can act freely, expansively
and with shadowing on a zero-dimensional space.

A common theme among all of the previous simulation results is that
they apply to actions on zero-dimensional spaces. The reason behind
this is rooted in the fact that the Cantor space $\{0,1\}^{\NN}$
admits a natural computable structure, where the cylinders are described
by finite words, and that allows the application of algorithmic techniques.
The main objective of this work is to explore a generalization to
groups acting on non-symbolic spaces, such as compact subsets of $\RR^{n}$,
$\GL_{n}(\CC)$ or compact abelian groups such as $(\RR/\ZZ)^{n}$.
Thus, the goal of this work  is to explore the following question.
\begin{question}
Can the simulation results be extended to group actions on spaces
that are not zero-dimensional? 
\end{question}

The theory of computable analysis offers the means to endow separable
metric spaces with computable structures that allow the notion of
computable map to make sense, and the application of algorithmic techniques
possible. In this chapter we explore this approach and study its connections
with the existing simulation results. In particular, we introduce
a very general notion of effective dynamical system (EDS) for systems on general metric spaces. We show that all known simulation results can be extended to this class.
By an effective dynamical system we mean one that is topologically
conjugate to a computable action over a recursively compact subset
of a computable metric space. We do not require the topological conjugacy
to be computable, so an EDS does not need to be computable itself.
The class of EDS is therefore quite large and encompasses virtually
all natural examples (although artificial non-examples can be constructed,
see \Cref{sec:factors}). Our main tool will be the following
result:
\begin{thm}
\label{thm:zero_dim_effective_extension} Let $G$ be a finitely generated
and recursively presented group. For any effective dynamical system
$G\curvearrowright X$ there exists an effectively closed zero-dimensional
space $\widetilde{X}\subset\{0,1\}^{\N}$ and a computable action
$\Gamma\curvearrowright\widetilde{X}$ such that $\Gamma\curvearrowright X$
is a topological factor of $\Gamma\curvearrowright\widetilde{X}$. 
\end{thm}

We remark that is it a well-known fact that every group action by
homeomorphisms on a compact metrizable space admits a zero-dimensional
extension. This extension, however, carries a priory no computable
structure. Our contribution is that, for recursively presented groups,
the computable structure stemming from the effective nature of the
system is preserved by a well-chosen zero-dimensional extension.

Given a finitely generated group $\Gamma$ and an epimorphism $\psi \colon \Gamma \to H$, we say that $\Gamma$ simulates $H$ if given any computable action of $H$ on an effectively closed subset of the Cantor space, the corresponding action of $\Gamma$ induced through $\psi$, is a factor of a $\Gamma$-SFT. If $\psi$ is an isomorphism, then $\Gamma$ is called self-simulable. The class of self-simulable groups includes several interesting examples, see \Cref{thm:simulation_results}.

The main application of \Cref{thm:zero_dim_effective_extension}
is that for recursively presented groups the simulation results in
the literature also apply to EDS.
\begin{thm}
\label{thm:simulation-by-SFTs-enhanced-version-GOTY} Let $\Gamma,H$
be finitely generated groups and $\psi\colon\Gamma\to H$ be an epimorphism.
Suppose that $H$ is recursively presented, then $\Gamma$ simulates
$H$ if and only if for every effective dynamical system $H\curvearrowright X$
the induced action of $\Gamma$ is a topological factor of a $\Gamma$-SFT. 
\end{thm}

Combining this result with the simulations result in the literature,
we obtain that many dynamical systems on non zero-dimensional spaces
can be realized as factors of subshifts of finite type. This includes
the natural action of $\GL_{n}(\ZZ)\curvearrowright\RR^{n}/\ZZ^{n}$
by left matrix multiplication for $n\geq5$, actions of $F_{2}\times F_{2}$
on the circle by computable rotations on each generator, actions of
the braid groups $B_{n}$ for $n\geq7$ on compact computable groups,
and more. 

To better illustrate the power of \Cref{thm:simulation-by-SFTs-enhanced-version-GOTY},
we mention a particular application which we believe could
be special interest. An action of a countable group on a compact abelian
topological group by automorphisms is called algebraic. This is a
rich class of group actions which has been the object of much study
in the literature, for instance in \cite{Lind1990,Schmidt1995,chung_2013,lind_homoclinic_1999}.
Using our results we obtain that a large class of algebraic actions
of non-amenable groups can be presented as topological factors of
SFTs.
\begin{thm}
\label{thm:finitely_presented_algebraic_action} Let $\Gamma$ be
a finitely generated and recursively presented self-simulable group.
Every finitely presented algebraic action of $\Gamma$ is the topological
factor of a $\Gamma$-subshift of finite type.
\end{thm}
The details concerning the proof of this result go beyond this thesis chapter, and the reader is referred to \cite{barbieri_effective_2024} for the proof.

We complete our discussion with several observations regarding the class of topological factors of EDS (see \Cref{sec:factors}). While it is true that computable factors of EDS are EDS (\Cref{prop:computable_factor_EDS_is_EDS}), the class of EDS is in general not closed under topological factor maps, even if we restrict ourselves to zero-dimensional systems. For zero dimensional systems on recursively presented groups, we propose a class that we call weak effective
dynamical system (WEDS). This class is closed by topological factors and contains every EDS. However, we do not know whether this equals the class of factors of EDS, or it is a bigger class.

\section{Effective dynamical systems}\label{sec.eds}

Let $X$ be a subset of a computable metric space, let $\Gamma$ be a finitely generated group, and let $\Gamma\curvearrowright X$ be a group action. We say that $\Gamma \curvearrowright X$ is \define{computable action} if for some finite and symmetric set of generators $S$, we have that for each $s\in S$ the group action map $f_s \colon X \to X$ given by $f_s(x) = sx$ is a computable function. Notice that in this case, $f_g\colon X \to X$ is in fact uniformly computable for all $g\in\Gamma.$ In particular, if an action is computable, it will satisfy the definition with respect to every finite set of generators.

We are interested in the behavior of actions as topological dynamical systems, and thus will consider them up to topological conjugacy. Our focus will be on those conjugacy classes that contain some computable representative.   

\begin{definition} 
    Let $X$ be a compact metrizable space and $\Gamma$ a finitely generated group. We say that an action $\Gamma \curvearrowright X$ is an \define{effective dynamical system (EDS)} if it is topologically conjugate to a computable action $\Gamma \curvearrowright \widehat{X}$, where $\widehat{X}$ is a recursively compact subset of a computable metric space.
    \end{definition}

We insist in that our definition of effective dynamical system \emph{does not} require the space $X$ to have any computable structure, or the topological conjugacy to be computable. Given an EDS $\Gamma \curvearrowright X$, we will refer to a topologically conjugate instance of a computable action $\Gamma \curvearrowright \widehat{X}$ over a recursively compact subset $\widehat{X}$ of a computable metric space as a \define{computable representative}.

\begin{example}\label{example:effective_subshift_is_EDS}
      An effective subshift $X\subset A^{\Gamma}$ is an EDS. Indeed, if $A^{\Gamma}$ is an effective subshift, then $\Gamma\curvearrowright X$ is topologically conjugate to the action $\Gamma\curvearrowright \widehat{X}$, where $\widehat{X}$ is the subshift from \Cref{def:effectivesubshiftthroughfreegroup} and the action $G\curvearrowright\widehat{X}$ is computable by \Cref{shift-action-is-computable}. We will see later that, conversely, the only subshifts which are EDS are the effective ones, thus justifying the name (\Cref{prop:EDS_and_effective_subshift_are_equivalent}).    
\end{example}

The following example shows that the conjugacy class of an effective dynamical system can contain uncountably many systems. In particular, it contains representatives that are not computably homeomorphic.

\begin{example}
Fix some computable angle $\alpha \in [0,2\pi]$ and consider $C_r=\{x\in \RR^2: \|x\|=r\}$ to be the circle of radius $r>0$. Note that  the space $C_r$ is a recursively compact subset of $\RR^2$ only for computable $r$. Despite this, the conjugacy class of the action $\ZZ \curvearrowright C_r$ induced by the rotation by $\alpha$ has a computable representative, namely  $\ZZ \curvearrowright C_1$, and is therefore an EDS regardless of the value of $r$. 
\end{example}

\begin{example}\label{example:torus_matrix_multiplication}
The action $\operatorname{GL}_n(\ZZ) \curvearrowright \RR^n/\ZZ^n$ of the general linear group on the $n$-dimensional torus by left matrix multiplication is an EDS.
\end{example}

\begin{example}
    Recall the Thompson's groups $F$ and $T$. These groups are often defined in the literature as the collection of piecewise linear homeomorphisms of $[0,1]$ (in the case of $T$, of $\RR/\ZZ$) that preserve orientation and whose non-differentiable points are dyadic rationals and whose slopes are all powers of 2. These maps are clearly computable and thus it follows that the natural actions $F\curvearrowright [0,1]$ and $T \curvearrowright \RR/\ZZ$ are EDS.
\end{example}

\section{Effective versions of classical constructions on effective dynamical
systems}

In this section we prove effective versions of a few classical constructions
commonly used in dynamical systems, and which may therefore be of
independent interest.

\subsection{Computable topological factors}

Here we show that the class of EDS is closed under computable topological
factor maps. We stress the fact that in this result we need both
actions to be represented as computable actions in computable metric
spaces, and the factor to be computable as well. Later we will show
that the class of EDS is in general not closed under topological factor
maps.\begin{proposition}\label{prop:computable_factor_EDS_is_EDS}
        Let $\Gamma$ be a finitely generated group and $\mathcal{X}$, $\mathcal{Y}$ be computable metric spaces. Let $X\subset \mathcal{X}$ be a recursively compact set and $\Gamma\curvearrowright X$ be a computable action. Let $\Gamma\curvearrowright Y \subset \mathcal Y$ be a topological factor of $\Gamma\curvearrowright X$ by a computable function $f\colon X\to Y$. Then $Y$ is a recursively compact set and $\Gamma\curvearrowright Y$ is a computable group action.    
    \end{proposition}

    \begin{proof}
         As $X$ is recursively compact and $f$ is computable, it follows that $Y$ is recursively compact. Let $(B^{\mathcal X}_i)_{i \in \NN}$ and $(B^{\mathcal Y}_i)_{i \in \NN}$ be recursive enumerations of the basic balls in $\mathcal X$ and $\mathcal Y$ respectively. Let $s \in \Gamma$ and let $i \in \NN$. As $f$ is computable it follows that, uniformly in $i$, there is a recursively enumerable set $I_i \subset \NN$ such that $f^{-1}(B^{\mathcal Y}_i) = X \cap \bigcup_{j \in I_i} B^{\mathcal X}_j$. As the action is computable, uniformly for each $j \in \NN$ there is a recursively enumerable set $I'_j\subset \NN$ such that \[s^{-1}f^{-1}(B^{\mathcal Y}_i) = X \cap \bigcup_{j \in I_i}\bigcup_{k \in I'_j} B_k^{\mathcal X}.\]
        Finally, as $f$ is computable and $X$ recursively compact, it follows by~\Cref{prop:computable-functions-are-invertible-on-compact-domains} that uniformly for $k \in \NN$ there is a recursively enumerable set $I''_k\subset \NN$ such that \[ s^{-1}(B_i^{\mathcal Y}\cap Y) = f(s^{-1}f^{-1}(B_i^{\mathcal Y})) = \bigcup_{j \in I_i}\bigcup_{k \in I'_j}\bigcup_{m \in I''_k} B_m^{\mathcal Y}.  \]
        Where the first equality holds because $f$ is $\Gamma$-equivariant. It follows the map $y \mapsto sy$ is computable and as $s$ is arbitrary, we obtain that $\Gamma \curvearrowright Y$ is computable.
    \end{proof}

\subsection{Products and inverse limit constructions}

Now we show that the operations of countable product and inverse limits
are computable under mild computability assumptions. For the next
two results, $(\mathcal X_{n})_{n\in\N}$ will be a sequence of uniformly computable
metric spaces as in \Cref{thm:Tychonoffcomputable}, and we will consider
$\prod_{n\in\N}\mathcal X_{n}$ with its product computable metric space structure.
\begin{prop}
\label{prop:countable-product-of-EDS} Let $X_{n}\subset \mathcal X_{n}$ be
a sequence of subsets and $\Gamma\curvearrowright X_{n}$ be a uniform
sequence of computable actions of the finitely generated group $\Gamma$.
Then the component-wise action 
\[
\Gamma\curvearrowright\prod_{n\in\N}X_{n}\subset\prod_{n\in\N}\mathcal X_{n}
\]
is computable. 
\end{prop}

\begin{proof}
Let us fix $g\in \Gamma$. Let $n\in\N$ and let $U_{i}$ be an effectively
open subset of $X_{i}$ for $i\in\{0,\dots,n\}$. We verify that the
preimage by $g$ of 
\[
U=U_{0}\times\dots\times U_{n}\times  X_{n+1}\times  X_{n+2}\times\dots
\]
is effectively open in $\prod_{n\in\N} X_{n}$ and that this process
is uniform in $n$ and the $U_{i}$. Indeed, 
\[
g^{-1}(U)=g^{-1}(U_{0})\times\dots\times g^{-1}(U_{n})\times X_{n+1}\times  X_{n+2}\times\dots
\]
This set is effectively open in $\prod_{n\in\N}X_{n}$, and it can
be uniformly computed from $n\in\N$ and a description of $U_{0},\dots,U_{n}$
because each of the $g^{-1}(U_{i})$ is effectively open in $X_{i}$,
and can be uniformly computed from a description of $U_{i}\subset Y_{i}$. 
\end{proof}
\begin{prop}
\label{prop:inverse-limit-of-EDS} Let $\Gamma\curvearrowright X_{n}$
be as in the previous statement, where each $X_{n}$ is now assumed
to be an effectively closed subset of $\mathcal X_{n}$. Let $(\pi_{n})_{n\geq1}$
be a sequence of uniformly computable functions, where each $\pi_{n+1}\colon X_{n+1}\to X_{n}$
is a topological factor map from $\Gamma\curvearrowright X_{n+1}$ to $\Gamma\curvearrowright X_{n}$.
Then the inverse limit system
\[
\varprojlim X_{n}=\{(x_{n})\in\prod_{n\in\N}X_{n}\mid\pi_{n+1}(x_{n+1})=x_{n},n\geq1\}
\]
is an effectively closed subset of $\prod_{n\in\N}\mathcal X_{n}$. 
\end{prop}

\begin{proof}
Let $f\colon\prod_{n\in\N}X_{n}\to\prod_{n\in\N}X_{n}$ be the function
defined by $(x_{n})_{n\in\N}\mapsto(\pi_{n+1}(x_{n+1}))_{n\in\N}$.
We claim that $f$ is computable. We verify that the preimage by $f$
of a set of the form 
\[
U=U_{0}\times\dots\times U_{n}\times X_{n+1}\times X_{n+2}\times\dots
\]
is effectively open in $\prod_{n\in\N}Y_{n}$, uniformly on $n$ and
the $U_{i}$. Indeed, the preimage $f^{-1}$ can be written as 
\[
f^{-1}(U)=\pi_{1}^{-1}(U_{0})\times\dots\times\pi_{n+1}^{-1}(U_{n})\times X_{n+1}\times X_{n+2}\times\dots
\]
This set is effectively open uniformly on the $U_{i}$ because $(\pi_{n})_{n\geq1}$
is a sequence of uniformly computable functions. On the other hand,
it is clear that $\varprojlim X_{n}$ equals the set of fixed points
of $f$ in $\prod_{n\in\N}X_{n}$. Then it follows from \Cref{prop:fixed-points-of-computable-function}
that $\varprojlim X_{n}$ is effectively closed. 
\end{proof}

\subsection{Effective covers and partitions, and their associated subshifts}\label{subsec:subshift-covers}

    In this subsection we fix a recursively compact set $X$ which lies in a computable metric space $\mathcal{X}$. We start by reviewing some standard terminology regarding covers and the construction of a subshift from a cover. A \textbf{cover} $\mathcal P$ of $X$ is a finite collection of sets whose union equals $X$. A cover $\mathcal P$ is said to be open (resp.~closed) if it consists on sets which are open in $X$ (resp.~closed in $X$). We say that the cover $\mathcal P'$ \textbf{refines} $\mathcal P$ if every element in $\mathcal P'$ is contained in some element of $\mathcal P$. The join of two covers $\mathcal P \lor \mathcal P'$ is the cover $\{P\cap P'\mid P\in \mathcal P, P'\in\mathcal P'\}$. Given a group action $\Gamma\curvearrowright X$ and a finite subset $F\subset \Gamma$, we write $\bigvee _{g\in F} g^{-1}\mathcal P$ for the join of all $g^{-1}\mathcal{P}$ for $g\in F$. The \textbf{diameter} of the cover $\mathcal P$ is the maximum of the diameters of its elements. A cover $\mathcal P$ of $X$  is said to be \textbf{generating} for the group action $\Gamma\curvearrowright X$ if for each $\varepsilon >0$ there is a finite subset $F\subset \Gamma$ such that $\bigvee _{g\in F} g^{-1}\mathcal P$ has diameter at most $\varepsilon$. 
	
    Given a cover of $X$ labeled as $\mathcal{P}=\{P_0,\dots,P_n\}$, we define the \textbf{subshift cover} $Y(\Gamma\curvearrowright X, \mathcal{P})$ by \[ Y(\Gamma\curvearrowright X, \mathcal{P}) = \{ y \in \{0,\dots,n\}^{\Gamma}  : \mbox{ there exists } x \in X \mbox{ such that for every } g \in \Gamma, g^{-1}x \in \overline{P_{y(g)}}\}.   \] 
    
	The idea behind this definition is that every configuration in the subshift labels the orbit of some element $x\in X$ under the action $\Gamma\curvearrowright X$ by indicating the elements of $\mathcal{P}$ it hits. A compactness argument shows that a configuration $y$ lies in $Y(\Gamma \curvearrowright X, \mathcal P)$ if and only if for every finite $F\subset \Gamma$ and pattern $p\colon F\to \{0,\dots,n\}$ which occurs in $y$, the set
	  \[ D(p):=\bigcap_{g\in F} g^{-1}(\overline{P_{p(g)}}). \]
	is nonempty. In particular, $Y(\Gamma\curvearrowright X, \mathcal{P})$ is indeed a subshift. Now we prove that this well known construction is computable when the cover is made by effectively closed sets.
     \begin{definition}
        A cover $\mathcal P$ of $X$ is called \textbf{effective} if its elements are effectively closed sets.
    \end{definition}
	\begin{proposition}\label{prop:effective-covers-give-effective-subshifts}
        Let $\Gamma \curvearrowright X$ be a computable action of a finitely generated group, $S$ be a finite generating set of $\Gamma$, and let $F(S)\curvearrowright X$ be the induced action. Then the associated subshift $Y(F(S) \curvearrowright X, \mathcal{P})$ is effective uniformly for all effective covers $\mathcal{P}$.
	\end{proposition}
    \begin{proof} It suffices to show that we can semi-decide, given an effective cover $\mathcal{P}$ and a pattern $p \colon W \to F(S)$ with $W\subset F(S)$ finite, whether $D(p)=\varnothing$. Observe that $D(p)$ is an effectively closed subset of $X$, uniformly in $p$ and $\mathcal{P}$. Indeed, each $P_i$  is an effectively closed subset of $X$, the preimage of an effectively closed set by a computable function is effectively closed, and the finite intersection of effectively closed sets is effectively closed. Thus, we can uniformly semi-decide whether ${D(p)}$ is empty using the recursive compactness of $X$. 
    \end{proof}    
     
    \begin{remark}
    In what follows we focus on the construction of effective covers. These will be obtained intersecting $X$ with unions of closures of basic balls in $\mathcal{X}$. Note that $X$ is not assumed to be a computable metric space, and this is the reason why we appeal to the space $\mathcal X$.   
    \end{remark}
    \begin{proposition}\label{prop:Covers-from-basic-balls-are-effective}
        A cover of $X$  whose elements are finite unions of elements of the form $\overline{B}\cap X$, where $B$ are basic balls in $\mathcal{X}$, is effective.
    \end{proposition}
    \begin{proof}
        This follows from the fact that $X$ is effectively closed, and that the topological closure of a basic ball in $\mathcal{X}$ is effectively closed (see \Cref{closure-of-basic-balls-is-effective}). 
    \end{proof}

   The following construction provides a uniform sequence of refined effective covers at any desired resolution. 
   
	\begin{lemma}\label{lem:lemma-cubiertas-efectivas}
        There is an algorithm which on input $n\in\N$ computes an effective cover $\mathcal{P}_n$ of $X$ with diameter at most  $2^{-n}$, and such that for every $n\in \N$, $\mathcal P_ {n+1}$ refines $\mathcal{P}_n$ and the inclusion of elements in $\mathcal P_ {n+1}$ to elements in $\mathcal P _n$ is decidable uniformly on $n$. Moreover, if $X$ is a zero dimensional set, we can add the extra condition that every $\mathcal P_n$ is a partition of $X$.
	\end{lemma}
	\begin{proof} Using the recursive compactness of $X$, one can compute a cover $\mathcal{P}_0$ of $X$ made by the intersection of $X$ with (the closure of) basic balls of radius one. The algorithm now proceeds inductively on $n \in \NN$. After it has computed $\mathcal P_0,\dots,\mathcal P_n$ with the desired properties, it computes $\mathcal P_{n+1}$ as follows. Let $\{V_i\}_{i \in \NN}$ be a computable enumeration of all finite unions of basic balls. First observe that there is a finite subset $I$ of $\N$ such that the set  $\{V_i : i\in I\}$ satisfies the following conditions:
        \begin{enumerate}
            \item $V_i$ has diameter at most $2^{-(n+1)}$. 
            \item $\overline{V_i}$ is contained in the interior of some of the elements of $\mathcal P_n$
            \item The union of $V_i$ for $i\in I$ contains $X$
            \item If $X$ is zero dimensional, then the sets $\overline{V_i}\cap X$ are disjoint.
        \end{enumerate}
        The existence of such a finite set of basic balls satisfying the first three conditions follows from the fact that $X$ is compact and that the set $\{V_i : i\in \N\}$ is a basis for the topology. If $X$ is zero dimensional, then the last condition can be added because by \Cref{prop.clopens} we can take the basic sets to be clopen. We now observe that it is semi-decidable whether a finite subset $I$ of $\N$ satisfies the conditions above:
        \begin{enumerate}
        \item Finite unions of (the closure of) basic balls are computable sets, so the first condition is routine,
        \item the semi-decidability of the second follows from the fact that $\overline{V}_i\cap X$ is recursively compact and the interior of the elements in the partition are effectively open,
        \item the semi-decidability of the third condition follows directly from the recursive compactness of $X$,
        \item finally, for the fourth condition it suffices to note that in a recursively compact space it is semi-decidable whether two effectively closed sets have empty intersection (recall that in a recursively compact space we can semi-decide whether an effectively closed subset is empty, the intersection of two effectively closed sets is effectively closed, and this is a computable operation on the descriptions of the sets). 
        \end{enumerate}
        Thus we can compute a set $I$ as desired by an exhaustive search, and finally set $\mathcal {P}_{n+1}=\{V_i\cap X : i\in I\}$.  
	\end{proof}
    \begin{proposition}\label{prop:generating-covers-can-be-taken-effective}
        If $\Gamma\curvearrowright X$ admits a generating open cover, then it admits a generating effective cover.
    \end{proposition}
    \begin{proof}
        Let $\mathcal P$ be a generating open cover of $\Gamma\curvearrowright X$. We construct a new cover $\mathcal P'$ as follows. Let $\mathcal J$ be the collection of all sets of the form $B_i\cap X$ which are contained in some element of $\mathcal P$, where $B_i$ is a basic ball in $\mathcal X$. The fact that basic balls constitute a basis for a topology of $\mathcal X$, together with the compactness of $X$, imply that there is a finite subset $\mathcal J'\subset \mathcal J$ which is an open cover of $X$.  As $\mathcal J'$ refines $\mathcal P$, it follows that it is also a generating cover of $\Gamma\curvearrowright X$.  
        Finally we define $\mathcal P'$ as the closed cover obtained by taking the topological closure of all elements in $\mathcal J'$. By construction, $\mathcal P'$ is a closed cover of $X$ which is generating for the action $\Gamma\curvearrowright X$, and it follows from  \Cref{prop:Covers-from-basic-balls-are-effective} that $\mathcal P'$ is an effective cover. 
    \end{proof}

We remark that in the previous proposition we do not claim the cover $\mathcal P'$ to be computable uniformly from a description of $\Gamma\curvearrowright X$, just its existence. 
    
    \begin{proposition}\label{prop:generating-covers-can-be-taken-effective-and-partition-in-zero-dimension}
        If  $\Gamma\curvearrowright X$ admits a generating open cover and $X$ is zero dimensional, then $\Gamma\curvearrowright X$  admits an effective generating cover which is also a partition. 
    \end{proposition}
    \begin{proof}
        
    It follows from the hypotheses that $X$ admits a generating open cover $\mathcal P$  which is a partition of $X$. Indeed, let $\mathcal P_0$ be an open cover for $X$  which is generating for $\Gamma\curvearrowright X$. As $X$ is zero-dimensional, there is an open cover $\mathcal P$ of $X$ which refines $\mathcal P_0$ and constitutes a partition of $X$. Observe that the cover $\mathcal P$ is also generating for $\Gamma\curvearrowright X$ because it refines $\mathcal P_0$.

  Repeating the construction in the proof of~\Cref{prop:generating-covers-can-be-taken-effective} with the cover $\mathcal P=\{P_0,\dots,P_n\}$, we obtain a finite open cover $\mathcal J'$ of $X$ which refines $\mathcal P$, and such that every element in $\mathcal J'$ has the form $B\cap X$, where $B$ is a basic ball in $\mathcal X$.
    We define a new cover $\mathcal P'=\{P'_0,\dots,P'_n\}$ by the following condition. For every $i \in \{0,\dots,n\}$, $P'_i$ is the union of all elements in $\mathcal J'$ which are contained in $P_i$. The fact that $\mathcal P$  is a partition of $X$ implies that the same holds for $\mathcal P'$. It is also clear that $\mathcal P'$ is a generating cover. We now verify that $\mathcal P'$  is an effective cover. Indeed, the fact that $\mathcal P'$ is a partition implies that the elements of $\mathcal P'$ are closed in $X$. In particular we can write \[P'_i=\overline{P'_i}=\bigcup_{m\in M} \overline{B_m}\cap X\] where $M$ is a finite subset of $\N$ and $B_m$  are rational balls. Here topological closures can be taken in $X$ and in $\mathcal X$. Both topological closures coincide because $X$ is a closed subset of $\mathcal X$ and we consider $X$ with the subspace topology. It follows from the previous set equality and~\Cref{prop:Covers-from-basic-balls-are-effective} that $\mathcal P'$ is also an effective cover. 
        \end{proof}

\section{Proof of the main results}\label{sec:resultado}

In this section we prove \Cref{thm:zero_dim_effective_extension}.
Let us briefly recall the standard procedure to build a zero-dimensional
topological extension of a dynamical system $\Gamma\curvearrowright X$.
\begin{enumerate}
\item First we extract a sequence $(\mathcal{P}_{n})_{n\in\N}$ of open
covers of $X$ with diameters tending to zero, and associate to each
one of them a subshift $Y_{n}$. 
\item Define the sequence of subshifts $(Z_{n})_{n\in\N}$ where $Z_{n}=\prod_{k\leq n}Y_{k}$
with the coordinate-wise shift, and consider the projection factor
maps $\pi_{n+1}\colon Z_{n+1}\to Z_{n}$. 
\item Define the inverse limit $Z=\varprojlim Z_{n}$, this is a topological
extension of $\Gamma\curvearrowright X$. 
\end{enumerate}
In our proof we effectivize this construction. For this reason we have proved effective versions of some of the steps, such as infinite products, inverse limits, and subshifts associated to covers. We start with the following observation for computable actions on
zero dimensional spaces.
\begin{prop}
\label{cor:cantor-representative} Let $\Gamma\curvearrowright X$
be a computable action, where $X$ is zero-dimensional subset of a
computable metric space. Then $\Gamma\curvearrowright X$ is topologically
conjugate to a computable action $\Gamma\curvearrowright Y$, where
$Y$ is an effectively closed subset of $\{0,1\}^{\N}$. 
\end{prop}

\begin{proof}
This follows by conjugating the action $\Gamma\curvearrowright X$
with a computable homeomorphism between $X$ and an effectively closed
subset of $\{0,1\}^{\N}$, whose existence is guaranteed by \Cref{thm:compacto-y-cero-dimensional-tiene-copia-en-el-cantor}. 
\end{proof}
We are now ready to present the proof of our main result.

\begin{theorem}[\Cref{thm:zero_dim_effective_extension}]
\label{prop:main-result}
		Let $\Gamma\curvearrowright X$ be an EDS, where $\Gamma$ is a finitely generated and recursively presented group. Then $\Gamma\curvearrowright X$ is the topological factor of a computable action of an effectively closed subset of $\{{0},{1}\}^{\NN}$. 
	\end{theorem} 
    \begin{proof}
        We start by replacing $\Gamma \curvearrowright X$ by a computable representative, as defined in~\Cref{sec.eds}.  As explained at the beginning of this section, we will construct a topological extension of $\Gamma\curvearrowright X$ as the inverse limit of a sequence of subshifts. 
        
        Let $S$ be a finite set of generators of $\Gamma$ and $F(S)$ be the free group generated by $S$. Let $F(S)\curvearrowright X$ be the action induced from $\Gamma \curvearrowright X$ and notice it is computable. We apply the tools from~\Cref{subsec:subshift-covers} to the action $F(S)\curvearrowright X$. Let $(\mathcal P ^n)_{n\in\N}$ be a uniform sequence of effective covers as in  \Cref{lem:lemma-cubiertas-efectivas}. For each $\mathcal P^n=\{P^n_0,\dots,P^n _{m_n}\}$ we set $A_n$ to be the alphabet $\{0,\dots,m_n\}$, and we let $Y(F(S)\curvearrowright X, \mathcal P ^n)$ be the subshift cover associated to the action $F(S)\curvearrowright X$ and the cover $\mathcal P ^n$. By \Cref{prop:effective-covers-give-effective-subshifts}, each $Y(F(S)\curvearrowright X, \mathcal P ^n)$ is an effective subshift, and the sequence $(Y(F(S)\curvearrowright X, \mathcal P ^n))_{n\in\N}$ is uniformly effective.

        It is possible that some elements in $F(S)$ which are mapped to the identity element in $\Gamma$ (and thus act trivially on $X$), act non-trivially on $Y(F(S)\curvearrowright X, \mathcal P ^n)$. Indeed, if $w\in F(S)$ is mapped to the identity element in $\Gamma$, and there are distinct $i,j$ such that $P^n_{i} \cap P^n_{j} \neq \varnothing$, we could associate distinct elements to $1_{F(S)}$ and $w$, and thus $w$ acts non-trivially on $Y(F(S)\curvearrowright X, \mathcal P ^n)$.
        This is solved as follows. For each $n$, we define $Y_n$ by  
        \[Y_n = Y(F(S)\curvearrowright X, \mathcal P ^n)\cap \widehat{A_n^\Gamma}\subset A_n^{F(S)},\]
        where $\widehat{A_n^\Gamma}$ is the pullback of $A_n^\Gamma$ to $F(S)$. As $\Gamma$ is recursively presented, the set $\widehat{A_n^\Gamma}$ is effectively closed (\Cref{example:fullshift_RP_is_effective}). As $Y_n$ is intersection of two effectively closed sets, then $Y_n$ is effectively closed. Furthermore, as $Y(F(S)\curvearrowright X, \mathcal P ^n)$ is uniform and the second set in the intersection ($\widehat{A_n^\Gamma}$) is the same for each $n$ it follows that $(Y_n)_{n \in \NN}$ is uniformly effectively closed. 

        We now build a zero-dimensional topological extension of $\Gamma \curvearrowright X$, and then we verify the computability of the construction. We define a $F(S)$-subshift $Z_n$ on alphabet $B_n=A_0\times\dots\times A_n$ by the following two conditions:
        \begin{enumerate}
            \item $Z_n \subset Y_0 \times \dots \times Y_n$. 
            \item For each $z\in Z_n$ and for each $w\in F(S)$ with $z(w)=(e_0,\dots,e_n)$ we have \[P^n_{e_n}\subset\dots\subset P^0_{e_0}.\]
        \end{enumerate}


        For each $n$, let $\pi_{n+1}\colon Z_{n+1}\to Z_n$ be the map which removes the last component of the tuple. Then let $Z$ be the inverse limit of dynamical systems 
                  \[Z=\varprojlim Z_n=\{(z_n)\in\prod_{n \in \NN} Z_n\mid\pi_{n+1}(x_{n+1})=x_n, n\geq 1\}\subset\prod_{n \in \NN} B_n^{F(S)}.\]

        We now go into some computability considerations.  First observe that for each $n\in \NN$, $Z_n$ is an effectively closed subset of $B_n^{F(S)}$. As $(Y_n)$ is a uniform sequence of effectively closed sets, we can semi-decide whether the first condition fails. Moreover the second defining condition of $Z_n$ is decidable, uniformly on $n$, by construction of the sequence $(\mathcal P^n)_{n\in\N}$ (\Cref{lem:lemma-cubiertas-efectivas}). Thus $(Z_n)_{n \in \NN}$ is a uniform sequence of effective $F(S)$-subshifts. It is also clear that $(\pi_n)_{n\geq 1}$ is a uniform sequence of computable functions, which are also factor maps. Now we claim that $Z$ is a recursively compact subset of $\prod_{n \in \NN} B_n^{F(S)}$. This follows from three facts:
        \begin{enumerate}
            \item $Z$ is an effectively closed subset of $\prod_{n \in \NN} B_n^{F(S)}$ by \Cref{prop:inverse-limit-of-EDS}.
            \item The product $\prod_{n \in \NN} B_n^{F(S)}$ is recursively compact  by \Cref{thm:Tychonoff_computable}. 
            \item An effectively closed subset of a recursively compact computable metric space is recursively compact (\Cref{prop.equivalence.comp.closed.subset}).
        \end{enumerate}
        The action $F(S)\curvearrowright \prod_{n \in \NN} B_n^{F(S)}$ is computable by \Cref{prop:countable-product-of-EDS}, and then the same holds for $F(S)\curvearrowright Z$. By construction, every element $w\in F(S)$ which corresponds to the identity element in $\Gamma$  acts trivially on $Z$, and thus we can define an action $\Gamma\curvearrowright Z$ by the expression $gz=wz$, where $w\in F(S)$ is any element with $\underline w = g$. It remains to verify that $\Gamma\curvearrowright Z$ is a topological extension of $\Gamma\curvearrowright X$. This is well known, but we sketch the argument for completeness. 
        
        We define a topological factor $\phi\colon Z\to X$ as follows. First, for each $n\in\N$ we define $\nu_n\colon Z\to Y_n$ as the function which sends $z=(z_n)_{n\in\N}$ to the configuration $\nu_n(z)$ given by $\nu_n(z)(w) = (z_n)(w)(n)$, the $n$-th component of $(z_n)(w)\in B_n$. We now define $\phi(z)$ by the expression
            \[\{\phi(z)\}=\bigcap_{n\in\N} P^n_{\nu_n(z)(1_{F(S)})}.\]

        The function $\phi$ is well defined, as a nested intersection of nonempty compact sets with diameter tending to zero must be a singleton. It is straightforward to verify that $\phi$ is continuous, surjective, and it commutes with the actions $\Gamma\curvearrowright X$ and $\Gamma\curvearrowright Z$.

        We have proved that $\Gamma\curvearrowright Z$ is a computable action on a recursively compact and zero-dimensional subset of a computable metric space, and that it is a topological extension of $\Gamma\curvearrowright X$. Now by   \Cref{cor:cantor-representative}, the action $\Gamma\curvearrowright Z$ is topologically conjugate to a computable action on an effectively closed subset of $\{{0},{1}\}^\N$. This finishes the proof. \end{proof}
        
We end this section by proving that when the action admits a generating
partition, the extension can be taken as an effective subshift. In
the zero-dimensional regime this is the computable version of the
fact that subshifts are precisely the expansive actions on zero-dimensional
sets.
\begin{thm}
\label{prop:expansive-actions-are-factors-of-effective-subshifts}
Let $\Gamma\curvearrowright X$ be an EDS, where $\Gamma$ is a finitely
generated and recursively presented group and the action admits a
generating open cover. Then $\Gamma\curvearrowright X$ is a topological
factor of an effective $\Gamma$-subshift.

Moreover, if $X$ is zero-dimensional, then $\Gamma\curvearrowright X$
is topologically conjugate to an effective $\Gamma$-subshift and
the recursively presented hypothesis on $\Gamma$ can be removed.
\end{thm}
\begin{proof}
 We first prove the result assuming that $\Gamma$ is recursively presented. By \Cref{prop:generating-covers-can-be-taken-effective}, $\Gamma\curvearrowright X$ admits an effective generating cover $\mathcal P$. Now let $S$ be a finite set of generators of $\Gamma$ and let $F(S)$ be the free group generated by $S$. Let $F(S)\curvearrowright X$ be the action induced by $\Gamma\curvearrowright X$, and consider the subshift $Y(F(S) \curvearrowright X, \mathcal P)$ defined in~\Cref{subsec:subshift-covers}. By \Cref{prop:effective-covers-give-effective-subshifts} we have that $Y(F(S) \curvearrowright X, \mathcal P)$ is effective. 
        
        Observe that some elements in $F(S)$ which are mapped to the identity element in $\Gamma$ and thus act trivially on $X$, have a non-trivial action on $Y(F(S) \curvearrowright X, \mathcal P)$. In order to solve this, we define $Y$ by  
        \[Y = Y(F(S) \curvearrowright X, \mathcal P)\cap \widehat{A^\Gamma},\]
        
        where $\widehat{A^\Gamma}$ is the pullback of $A^\Gamma$ in $F(S)$. As $\Gamma$ is recursively presented, the set $\widehat{A^\Gamma}$ is effectively closed (\Cref{example:fullshift_RP_is_effective}), and thus $Y$ is an effective subshift. 
        
        It is a standard fact that $X$ is a topological factor of $Y$ thanks to the fact that $\mathcal P$ is a generating cover. We sketch the argument for completeness.  Let $(F_n)_{n\in\mathbb{N}}$ be an increasing sequence of finite subsets of $F(S)$ whose union is $F(S)$. For each $y\in Y(F(S) \curvearrowright X, \mathcal P)$, we define $D(y)$ as the following subset of $X$: 
		\[ D(y) = \bigcap_{n\in\N } D(y|_{F_n}). \]
        Recall that $D(p)$, for a pattern $p$, was defined in \Cref{subsec:subshift-covers}. For each $y\in Y$ the set $D(y)$ is a singleton, being a nested intersection of nonempty compact sets with diameter tending to zero. Indeed, for each $n$ the set $D(y|_{F_n})$ is contained in the cover $\bigvee _{g\in F_n} g^{-1}\mathcal P$, and the diameters of these elements tend to zero by definition of generating cover. 

        We can now define a topological factor map $\phi\colon Y\to X$   by the expression
		\[\{\phi(y)\}=D(y).\]

        It is straightforward to verify that $\phi$ is continuous, surjective, and it commutes with the actions $\Gamma\curvearrowright X$ and $\Gamma\curvearrowright Y$.  

        We have proved the claim for recursively presented groups. We now prove that if $X$ is zero-dimensional, the hypothesis that $\Gamma$ is recursively presented can be removed. Note that if $X$ is zero-dimensional, then by \Cref{prop:generating-covers-can-be-taken-effective-and-partition-in-zero-dimension} we can assume that $\mathcal P$ is also a partition, which makes $\phi$ injective, and therefore a topological conjugacy. In particular, it preserves the set of group  elements which act trivially. More precisely, recall that $F(S)\curvearrowright X$ is the action induced by $\Gamma\curvearrowright X$. As $F(S)$ has decidable word problem, the previous construction shows that the action $F(S)\curvearrowright X$ is topologically conjugate to an effectively closed $F(S)$-subshift $Y$. Being a topological conjugacy, every element in $F(S)$ which maps to the trivial element of $\Gamma$ must act trivially on $Y$. Thus we can define a $\Gamma$-subshift $Z$ given by \[Z= \{z\in A^\Gamma: \mbox{ there exists } y \in Y \mbox{ where } y(w) = z(\underline{w}) \mbox{ for every } w \in F(S) \}.\] 
        It is straightforward that the pullback $\widehat{Z}$ to $F(S)$ is precisely $Y$, and thus $Z$ is an effective $\Gamma$-subshift as required. 
\end{proof}

\section{Simulation by subshifts of finite type}\label{sec:simu}

The main motivation behind \Cref{thm:zero_dim_effective_extension}
is that it can be used to enhance several results in the literature
that characterize zero-dimensional effective dynamical systems as
topological factors of subshifts of finite type. We shall provide
a general definition that will help us to summarize most of these
results.

Consider two groups $\Gamma,H$ and suppose there exists an epimorphism
$\psi\colon\Gamma\to H$, that is, that $H$ is a quotient of $\Gamma$.
Given an action $H\curvearrowright X$, recall that the \textbf{induced
action} of $H\curvearrowright X$ to $\Gamma$ is the action $\Gamma\curvearrowright X$
which satisfies that $g\cdot x=\psi(g)\cdot x$ for every $g\in\Gamma$.
\begin{defn}
Let $\Gamma$ and $H$ be finitely generated groups and $\psi\colon\Gamma\to H$
be an epimorphism. We say $\Gamma$ \textbf{simulates} $H$ (through
$\psi$) if for every effectively closed set $X\subset\{0,1\}^{\NN}$
and every computable action $H\curvearrowright X$ there exists a
$\Gamma$-SFT $Y$ such that the induced action $\Gamma\curvearrowright X$
is a topological factor of $\Gamma\curvearrowright Y$.
\end{defn}

If the map $\psi$ in the definition above is clear then we don't
mention it (for instance, if $\Gamma$ is a direct product of $H$
with another group we assume that $\psi$ is the projection to $H$).
Furthermore, if $\psi$ is an isomorphism, we say that $\Gamma$ is
a \textbf{self-simulable} group. Let us briefly summarize a few of
the simulation results available in the literature
\begin{thm}
\label{thm:simulation_results} The following simulation results hold
\begin{enumerate}
\item \cite{hochman_dynamics_2009} $\Gamma=\ZZ^{3}$ simulates $\ZZ$. 
\item \cite{barbieri_generalization_2019} Let $H$ be a finitely generated
group, $d\geq2$ and $\varphi\colon H\to\GL_{d}(\ZZ)$. Then the semidirect
product $\Gamma=\ZZ^{d}\rtimes_{\varphi}H$ simulates $H$ (through
the projection epimorphism). 
\item \cite{barbieri_geometric_2019} Let $H,G_{1},G_{2}$ be finitely generated
groups. Then $\Gamma=H\times G_{1}\times G_{2}$ simulates $H$. 
\item \cite{barbieri_soficity_2023} Let $H,K$ be finitely generated groups,
suppose $H$ is nonamenable and that $K$ has decidable word problem,
then $\Gamma=H\times K$ simulates $H$. 
\item \cite{barbieri_groups_2022} The following groups are self-simulable: 
\begin{itemize}
\item The direct product of any two finitely generated nonamenable groups. 
\item The general linear group $\GL_{n}(\ZZ)$ and the special linear group
$\text{SL}_{n}(\ZZ)$ for $n\geq5$. 
\item Thompson's group $V$ and the Brin-Thompson's groups $nV$ for $n\geq1$. 
\item Finitely generated non-amenable Branch groups. 
\item The automorphism group $\text{Aut}(F_{n})$ and the outer automorphism
group $\text{Out}(F_{n})$ of the free group $F_{n}$ on $n$ generators
for $n\geq5$. 
\item The braid groups $B_{n}$ for $n\geq7$ strands. 
\end{itemize}
\end{enumerate}
\end{thm}
We can enhance all of the above results using \Cref{thm:zero_dim_effective_extension}. 
\begin{thm}
Let $\Gamma,H$ be finitely generated groups and $\psi\colon\Gamma\to H$
be an epimorphism. Suppose further that $H$ is recursively presented. Then
$\Gamma$ simulates $H$ if and only if for every effective dynamical
system $H\curvearrowright X$ the induced action of $\Gamma$ is a
topological factor of a $\Gamma$-SFT. 
\end{thm}

\begin{proof}
Suppose first that for every EDS $H\curvearrowright X$ the induced
action of $\Gamma$ is a topological factor of a $\Gamma$-SFT. As
a computable action on an effectively closed zero-dimensional set
is in particular an EDS, it follows that $\Gamma$ simulates $H$.

Conversely, let $H\curvearrowright X$ be an EDS. As $H$ is recursively
presented, \Cref{thm:zero_dim_effective_extension} implies that
there exists an effectively closed set $\widetilde{X}\subset\{0,1\}^{\NN}$
and a computable action $H\curvearrowright\widetilde{X}$ such that
$H\curvearrowright X$ is a topological factor of $H\curvearrowright\widetilde{X}$.
Taking the induced action of $\Gamma$ of these two actions it follows
that $\Gamma\curvearrowright X$ is a topological factor of $\Gamma\curvearrowright\widetilde{X}$.
As $\Gamma$ simulates $H$, there exists a $\Gamma$-SFT $Y$ which
factors onto $\Gamma\curvearrowright\widetilde{X}$ and thus onto
$\Gamma\curvearrowright X$. 
\end{proof}
Let us illustrate how \Cref{thm:zero_dim_effective_extension}
can be applied with a few examples.
\begin{example}
The action $\GL_{n}(\ZZ)\curvearrowright\RR^{n}/\ZZ^{n}$ by left
matrix multiplication is an EDS. For $n\geq5$ the group $\GL_{n}(\ZZ)$
is both recursively presented and self-simulable, thus it follows
by \Cref{thm:simulation-by-SFTs-enhanced-version-GOTY} that
there exists a $\GL_{n}(\ZZ)$-SFT which topologically factors onto
$\GL_{n}(\ZZ)\curvearrowright\RR^{n}/\ZZ^{n}$. 
\end{example}

\begin{example}
   Let $\{\beta_i\}_{i=1}^4$ be four computable points in $\RR/\ZZ$. Consider $F_2 \times F_2 = \langle a_1,a_2 \rangle \times \langle a_3,a_4 \rangle$ and the action $F_2 \times F_2 \curvearrowright \RR/\ZZ$ given by $a_i\cdot x = x+\beta_i \bmod{\ZZ}$ for $i \in \{1,\dots,4\}$ and remark that this action is in fact the pullback of an action $\ZZ^4 \curvearrowright \RR/\ZZ$.  As $F_2\times F_2$ is recursively presented and self-simulable, it follows by~\Cref{thm:simulation-by-SFTs-enhanced-version-GOTY} that there exists an $(F_2\times F_2)$-SFT which topologically factors onto $F_2 \times F_2 \curvearrowright \RR/\ZZ$. 
\end{example}

\begin{example}
Let $G$ be a group, $n\geq2$ and consider $\Delta_{n}(G)\subset G^{n}$
the space of all $n$-tuples of elements of $G$ whose product is
trivial, that is 
\[
\Delta_{n}(G)=\{(x_{1},\dots,x_{n})\in G^{n}:x_{1}\cdots x_{n}=1_{G}\}
\]
The braid group $B_{n}$ on $n$ strands acts on $\Delta_{n}(G)$
by 
\[
\sigma_{i}(x_{1},\dots,x_{i-1},x_{i},x_{i+1},\dots,x_{n})=(x_{1},\dots,x_{i-1},x_{i+1},x_{i+1}^{-1}x_{i}x_{i+1},x_{i+2},\dots,x_{n}).
\]
Where $1\leq i\leq n-1$ and $\sigma_{i}$ is the element of $B_{n}$
which crosses strands $i$ and $i+1$ (for the definition of braid
group, refer to \cite{Birman_Brendle_2005_braidgroupssurvey}). If
we let $G$ be any compact topological group which admits a computable
structure such that group multiplication is a computable map, for
instance the group of unitary complex matrices $G=U(k)\leqslant\GL_{k}(\CC)$
for some $k\geq1$, then $\Delta_{n}(G)$ is an effectively closed
set of $G^{n}$ and the action $B_{n}\curvearrowright\Delta_{n}(G)$
is computable. It follows that if $n\geq7$ then there exists a $B_{n}$-SFT
which topologically factors onto $B_{n}\curvearrowright\Delta_{n}(G)$. 
\end{example}

In~\cite{barbieri_groups_2022} it was proven that Thompson's group
$F$ %
{} is self-simulable if and only if $F$ is non-amenable, which is a
longstanding open question. We can therefore strengthen the characterization
of the potential amenability of $F$ in the following way:
\begin{cor}
Thompson's group $F$ is amenable if and only if there exists an EDS
$F\curvearrowright X$ which is not the topological factor of an $F$-SFT. 
\end{cor}

In the literature, there are a few simulation results which apply
exclusively to expansive maps. More precisely, let $\Gamma$ and $H$
be finitely generated groups and $\psi\colon\Gamma\to H$ be an epimorphism.
We say $\Gamma$ \textbf{simulates expansive actions of} $H$ (through
$\psi$) if for every effective subshift $X\subset A^{H}$ there exists
a $\Gamma$-SFT $Y$ such that the pullback subshift $\Gamma\curvearrowright X$
is a topological factor of $\Gamma\curvearrowright Y$. The most famous
results are due to Aubrun and Sablik \cite{aubrun_simulation_2013}
and Durand, Romaschenko and Shen~ which state that $\ZZ^{2}$ simulates
expansive actions of $\ZZ$ (through the projection epimorphism).

In light of \Cref{prop:expansive-actions-are-factors-of-effective-subshifts},
we can improve the expansive simulation theorems in the following
way:
\begin{thm}
Let $\Gamma,H$ be finitely generated groups and $\psi\colon\Gamma\to H$
be an epimorphism. Suppose that $H$ is recursively presented, then
$\Gamma$ simulates expansive actions of $H$ if and only if for every
effective dynamical system $H\curvearrowright X$ which admits a generating
cover (in particular, every expansive EDS) the induced action of $\Gamma$
is a topological factor of a $\Gamma$-SFT. 
\end{thm}

\section{Subshifts that are EDS}\label{sec:computability-on-shift-spaces}
In this section we compare the notion of EDS with that of effective subshift, defined in \Cref{chap:computable_analysis_subshifts}. The following result shows that given a subshift $X$, the action $G\curvearrowright A^G$ admits some computable representative exactly when the pullback defined in \Cref{chap:computable_analysis_subshifts} is a computable representative.  
\begin{prop}
\label{prop:EDS_and_effective_subshift_are_equivalent} Let $\Gamma$
be a finitely generated group. A subshift $X\subset A^{\Gamma}$ is
an EDS if and only if it is effective. 
\end{prop}

\begin{proof}
Let us recall that for a finitely generated group $\Gamma$ a subshift
$X\subset A^{\Gamma}$ is effective if its pullback into the full
shift $A^{F(S)}$ is an effectively closed set. We already argued
in \Cref{rem:fullshiftRPiseffective} that every effective subshift
is an EDS. Let us fix a finite set of generators $S$ of $\Gamma$
and denote by $\widehat{X}\subset A^{F(S)}$ the subshift induced
by $X$ on the free group $F(S)$. Suppose that $\Gamma\curvearrowright X$
is topologically conjugate to a computable action on an effectively
closed subset $Y$ of a recursively compact metric space. By \Cref{cor:cantor-representative},
we may assume without loss of generality that $Y$ is an effectively
closed subset of $\{0,1\}^{\NN}$ with the canonical computable structure.
As $\Gamma\curvearrowright X$ is topologically conjugate to $\Gamma\curvearrowright\widehat{X}$,
there is a topological conjugacy $\phi\colon Y\to\widehat{X}$. For
every $a\in A$, let $[a]=\{\widehat{x}\in\widehat{X}:\widehat{x}(1_{F(S)})=a\}$.
As $\phi$ is a homeomorphism, we have that $C_{a}=\phi^{-1}([a])$
is a finite union of cylinder sets in $Y$ and thus an effectively
closed subset of $Y$.

We claim that there exists an algorithm that given $W\Subset F(S)$
and a pattern $p\colon W\to A$, accepts if and only if $[p]\cap\widehat{X}=\varnothing$.
Indeed, given such a $p$, we may compute a description of $C_{p}=\bigcap_{w\in W}w\cdot C_{p(w)}$.
As the action $\Gamma\curvearrowright Y$ is computable, it follows
that each $w\cdot C_{p(w)}$ is an effectively closed set and thus,
as the intersection of finitely many effectively closed sets is effectively
closed, one may algorithmically detect whether $C_{p}=\varnothing$,
which occurs precisely when $[p]\cap\widehat{X}=\varnothing$. It
follows that the collection $\mathcal{F}_{\mathrm{max}}=\{p\mbox{ is a pattern }:[p]\cap\widehat{X}=\varnothing\}$
is recursively enumerable. As $\mathcal{F}_{\mathrm{max}}$ is the
maximal set of patterns which defines $\widehat{X}$, it follows that
$\widehat{X}$ is an effectively closed set. 
\end{proof}

We remark that for an effective cover $\mathcal{P}$ of a recursively
compact metric space $X$, the subshift $Y(\Gamma\curvearrowright X,\mathcal{P})$
from \Cref{prop:effective-covers-give-effective-subshifts} is
ECP but not necessarily effective (see \Cref{def:ecp}). This is fundamentally the reason
why we need to ask that $\Gamma$ is recursively presented in \Cref{thm:zero_dim_effective_extension}.

\section{Topological factors of effective dynamical systems}\label{sec:factors}

We begin by showing that the class of EDS is not closed by topological
factors. This shows that subshifts of finite type on
certain groups may have factors that are not EDS. However, we shall
show that under certain conditions such as having topological zero
dimension, we can still provide computational restrictions that said
factors must satisfy.

Given an action $\Gamma\curvearrowright X$, we define its set of
periods as 
\[
\mathtt{Per}(\Gamma\curvearrowright X)=\{g\in\Gamma:\mbox{ there is }x\in X\mbox{ such that }gx=x\}
\]

The following lemma holds in general for any finitely generated
group with decidable word problem. However, we shall only need it
for the special case of $\ZZ$-actions.
\begin{lem}
Let $\ZZ\curvearrowright X$ be an EDS. Then $\mathtt{Per}(\ZZ\curvearrowright X)$
is a co-recursively enumerable set. 
\end{lem}

\begin{proof}
We will prove that the complement of $\mathtt{Per}(\ZZ\curvearrowright X)$
is a recursively enumerable subset of $\ZZ$. Without loss of generality,
replace $\ZZ\curvearrowright X$ by a computable representative. Then
the product space $X^{2}$ is recursively compact and thus the diagonal
$\Delta^{2}=\{(x,x)\in X^{2}:x\in X\}$ is an effectively closed subset
and thus also recursively compact. As $\ZZ\curvearrowright X$ is
computable, it follows that the collection of diagonal maps $f_{n}\colon X^{2}\to X^{2}$
given by $(x,y)\mapsto(T^{n}(x),y)$ for $n\in\ZZ$ is uniformly computable,
and thus the collection of sets $Y_{n}=f_{n}^{-1}(\Delta^{2})$ is
uniformly recursively compact. It follows that we can recursively
enumerate the integers $n$ such that $Y_{n}\cap\Delta^{2}=\varnothing$.
In other words, the set 
\[
\ZZ\setminus\mathtt{Per}(\ZZ\curvearrowright X)=\{n\in\ZZ:Y_{n}\cap\Delta^{2}=\varnothing\},
\]
is recursively enumerable. 
\end{proof}
The set of periods is an invariant of topological conjugacy, therefore
any action $\ZZ\curvearrowright X$ for which the set of periods is
not co-recursively enumerable cannot be topologically conjugate to
an EDS.
\begin{prop}
The class of EDS is not closed under topological factor maps. 
\end{prop}

\begin{proof}
Let us consider 
\[
X=\{0\}\cup\left\{ z\in\CC:|z|=\frac{1}{n}\mbox{ for some integer }n\geq1\right\} .
\]
And let $T\colon X\to X$ be given by $T(z)=z\exp\left(2\pi i|z|\right)$.
It is clear that $X$ is a recursively compact subset of $\CC$ with
the standard topology and that $T$ is a computable homeomorphism
which thus induces a computable action $\ZZ\curvearrowright X$. The
map $T$ consists of rational rotations on each circle where the period
is given by the inverse of the radius, and thus it is easy to see
that $\mathtt{Per}(\ZZ\curvearrowright X)=\ZZ$.

\begin{figure}[!ht]
\centering \begin{tikzpicture}
    \begin{scope}[shift = {(0,0)}, scale = 3]
        \foreach \ii in {1,...,20}{
            \draw[thick] (0,0) circle (1/\ii);
        }
        \draw[thick, fill = black] (0,0) circle (1/21);
        \node[red] at (0,0.65) {$T$};
        \draw [red,thick,domain=0:180, ->] plot ({0.55*cos(\x)},{0.55*sin(\x)});
        \draw [red,thick,domain=0:120, ->] plot ({0.36*cos(\x)},{0.36*sin(\x)});
        \draw [red,thick,domain=0:90, ->] plot ({0.28*cos(\x)},{0.28*sin(\x)});
        \draw [red,thick,domain=0:72, ->] plot ({0.22*cos(\x)},{0.22*sin(\x)});
        \node at (0,-1.2) {$X$};
    \end{scope}
    \begin{scope}[shift = {(9,0)}, scale = 3]
        \foreach \ii in {1,2,4,8,13,17}{
            \draw[thick] (0,0) circle (1/\ii);
        }
        \draw[thick, fill = black] (0,0) circle (1/21);
        \node[red] at (0,0.65) {$S$};
        \draw [red,thick,domain=0:180, ->] plot ({0.55*cos(\x)},{0.55*sin(\x)});
        \draw [red,thick,domain=0:90, ->] plot ({0.28*cos(\x)},{0.28*sin(\x)});
        \draw [red,thick,domain=0:45, ->] plot ({0.15*cos(\x)},{0.15*sin(\x)});
        \node at (0,-1.2) {$Y$};
    \end{scope}
    \draw[thick, ->] (3.5,0) to (5.5,0);
    \node at (4.5,0.5) {$f$};
    \end{tikzpicture} \caption{A representation of the actions $\protect\ZZ\curvearrowright X$ and
$\protect\ZZ\curvearrowright Y$.}
\label{fig:circulito}
\end{figure}
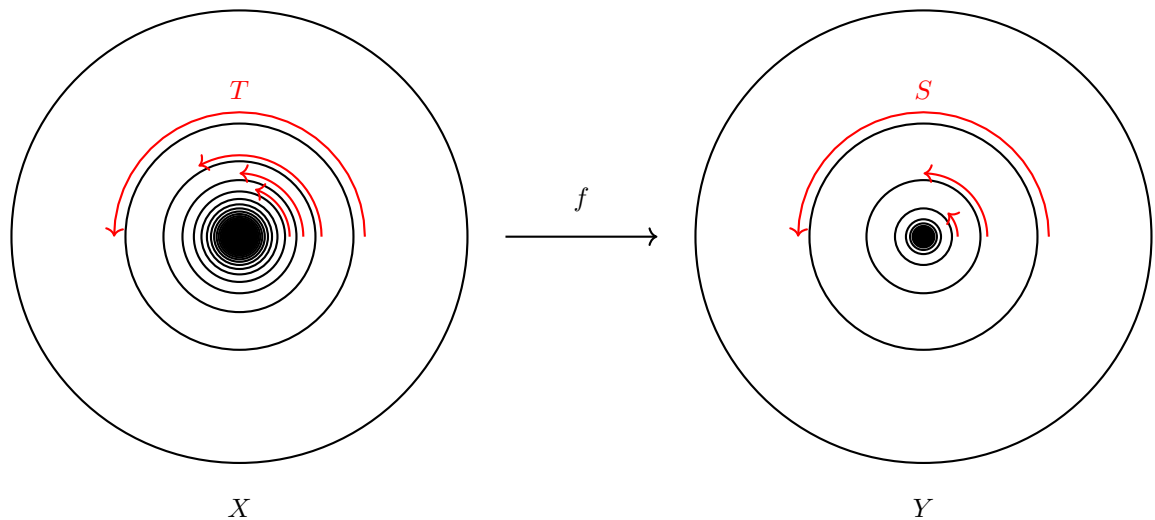

Now let $A\subset\NN\setminus\{0\}$ be some infinite set and let
$Y\subset\CC$ be given by 
\[
Y=\{0\}\cup\left\{ z\in\CC:|z|=\frac{1}{n}\mbox{ for some }n\in A\right\} .
\]
Similarly, the map $S\colon Y\to Y$ given by $S(z)=z\exp\left(2\pi i|z|\right)$
is a homeomorphism which induces an action $\ZZ\curvearrowright Y$.

Let $f\colon X\to Y$ be the map given by 
\[
f(z)=\begin{cases}
z & \mbox{ if }|z|\in A\mbox{ for some }n\in A,\\
0 & \mbox{ if }|z|\notin A.
\end{cases}
\]

It is clear from the definition that $f$ is a continuous, surjective
map such that $f\circ T=S\circ T$ and thus a topological factor map
from $\ZZ\curvearrowright X$ to $\ZZ\curvearrowright Y$. However,
notice that $\mathtt{Per}(\ZZ\curvearrowright Y)=\ZZ A$, that is,
the set of integers which are integer multiples of some element of
$A$.

Notice that as the set of prime numbers is decidable, then a subset $A$ of prime numbers is co-recursively enumerable if and only if $\ZZ A$ is co-recursively enumerable. In particular, if we take $A$ which is not co-recursively enumerable we would have that $\mathtt{Per}(\ZZ\curvearrowright Y)$ is not co-recursively enumerable either and thus $\ZZ \curvearrowright Y$ cannot be an EDS.
\end{proof}
\begin{prop}
\label{prop:factor-of-eds-which-is-not-eds} The class of EDS is not
closed under topological factor maps, even if we restrict to zero-dimensional
spaces. 
\end{prop}

\begin{proof}
For each $n\in\N$, let $X_{n}=\{0,\dots,n\}$, and let $t_{n}=(1\ 2\ \dots\ n)$
be the permutation on $X_{n}$ which fixes $0$ and cyclically permutes
$1\mapsto2\mapsto3\mapsto\dots\mapsto n\mapsto1$. Now define $Y_{n}$
as the product $X_{0}\times\dots\times X_{n}$, and $s_{n}$ as the
map on $Y_{n}$ which applies $t_{0},\dots,t_{n}$ component-wise.
We define $\pi_{n+1}$ as the map $Y_{n+1}\to Y_{n}$ which removes
the last component of the tuple, this is a factor map from $Y_{n+1}$
to $Y_{n}$ with their respective actions. We now define $Y$ as the
inverse limit of sequence 
\[
\dots\xrightarrow{\pi_{n+2}}\ Y_{n+1}\xrightarrow{\pi_{n+1}}Y_{n}\xrightarrow{\pi_{n}\ \ }\dots\xrightarrow{\pi_{1}\ \ }Y_{0}.
\]

It is straightforward to verify that the set $Y$ endowed with the
component-wise action is a zero dimensional EDS. Let $A\subset\NN$.
We now define a topological factor of $Y$ called $Y_{A}$.

For each $n$ define $g_{n}\colon X_{n}\to X_{n}$, as follows. For
$n\not\in A$, $g_{n}$ sends everything to $0$. For $n\in A$, $g_{n}$
is the identity function on $X_{n}$. Now define $f_{n}\colon Y_{n}\to Y_{n}$
as the function which on component $n$ applies $g_{n}$. Then $f_{n}$
is an endomorphism of $Y_{n}$ with the action induced by $s_{n}$.
Finally, define $F$ as the endomorphism of $Y$ which in component
$n$ applies $f_{n}$ and let $Y_{A}$ be the image of $F$, which
is a compact subset of $\prod_{n\in\NN}Y_{n}$ by the continuity of
$F$. We endow $Y_{A}$ with the componentwise action inherited as
a subsystem of $Y$.

As in the previous construction, it suffices to take $A$ as a set
of primes which is not co-recursively enumerable. In this case, $Y_{A}$
is a zero-dimensional topological factor of the zero dimensional EDS
$Y$, but it cannot be an EDS because its set of periods is exactly
$\ZZ A$ which is not co-recursively enumerable. 
\end{proof}
Even if the class of EDS is not closed under topological factor maps,
it is reasonable to assume that a factor of an EDS should still satisfy
some sort of computability constraint. With the aim of understanding
this class in the case of zero-dimensional spaces, we introduce a
notion of weak EDS which will turn out to be stable under topological
factor maps.
\begin{defn}
\label{def:WEDS} An action $\Gamma\curvearrowright X$ with $X\subset A^{\NN}$
is a \textbf{weak effective dynamical system} (WEDS) if for every
clopen partition $\mathcal{P}=(P_{i})_{i=1}^{n}$ of $X$ the subshift
\[
Y(\Gamma\curvearrowright X,\mathcal{P})=\{y\in\{1,\dots,n\}^{\Gamma}:\mbox{ there is }x\in X\mbox{ such that for every }g\in\Gamma,\ g^{-1}x\in P_{y(g)}\}.
\]
is effective. 
\end{defn}

Notice that for an action $\Gamma\curvearrowright X$ with $X\subset A^{\NN}$
to be a WEDS, it suffices to check the condition on the clopen partitions
$\mathcal{P}_{n}=\{[w]:w\in A^{n}\mbox{ and }[w]\cap X\neq\varnothing\}$
as this family of clopen partitions refines any other clopen partition.
As $\Gamma\curvearrowright X$ can be obtained as the inverse limit
of the shift action on the $Y(\Gamma\curvearrowright X,\mathcal{P}_{n})$,
we may also define a WEDS as an action that can be obtained as the
inverse limit of a sequence of effective subshifts. Notice that we
do not require the sequence to be uniform in this definition.
\begin{rem}
Notice that any expansive WEDS $\Gamma\curvearrowright X$ is automatically
an EDS, as in this case $\Gamma\curvearrowright X$ is topologically
conjugate to $Y(\Gamma\curvearrowright X,\mathcal{P})$ for any partition
with small enough diameter and thus we have that $\Gamma\curvearrowright X$
is an EDS. 
\end{rem}

\begin{prop}
\label{prop:0d-EDS_is_WEDS} If $\Gamma\curvearrowright X$ is a zero
dimensional EDS, then it is a WEDS.
\end{prop}

\begin{proof}
Let $\mathcal{P}$ be a cover of $X$ formed of disjoint clopen sets
and let $S$ be a finite set of generators for $\Gamma$. Let us consider
the action of the free group $F(S)\curvearrowright X$ induced by
the action $\Gamma\curvearrowright X$. We proved in \Cref{prop:effective-covers-give-effective-subshifts}
that the subshift $Y(F(S)\curvearrowright X,\mathcal{P})$ is effective
and thus an EDS. We now prove that $Y(\Gamma\curvearrowright X,\mathcal{P})$
is an EDS. For this it suffices to prove that every $w\in F(S)$ with
$\underline{w}=1_{\Gamma}$ acts trivially on $Y(\Gamma\curvearrowright X,\mathcal{P})$,
as then we have a topological conjugacy between $\Gamma\curvearrowright Y(\Gamma\curvearrowright X,\mathcal{P})$
and $\Gamma\curvearrowright Y(F(S)\curvearrowright X,\mathcal{P})$.

Assume that $wy\neq y$ for some $y$ in $Y(F(S)\curvearrowright X,\mathcal{P})$
in order to obtain a contradiction. We can assume that $wy$ and $y$
differ in $1_{F(S)}$ by shifting $y$. By definition, this means
that there is some element $x\in X$ such that $x\in P_{y(1_{F(S)})}$
and $x\in P_{wy(1_{F(S)})}$. This is a contradiction, as the computable
map associated to $w$ is the identity and thus the sets $P_{y(1_{F(S)})}$,
$P_{wy(1_{F(S)})}$ are disjoint. 
\end{proof}
Next we show that the class of WEDS is closed under topological factors.
\begin{prop}
\label{prop:factor_of_WEDS_is_WEDS} Let $\Gamma\curvearrowright X$
be a WEDS and $\Gamma\curvearrowright Y$ be a zero-dimensional topological
factor. Then $\Gamma\curvearrowright Y$ is a WEDS. 
\end{prop}

\begin{proof}
Denote by $f\colon X\to Y$ the topological factor map and let $\mathcal{P}=(P_{i})_{i=1}^{n}$
be a clopen partition of $Y$. Then $\mathcal{Q}=(f^{-1}(P_{i}))_{i=1}^{n}$
is a clopen partition of $X$. As $X$ is WEDS, we obtain that $Y(\Gamma\curvearrowright X,\mathcal{Q})$
is an effective subshift. As $Y(\Gamma\curvearrowright X,\mathcal{Q})=Y(\Gamma\curvearrowright Y,\mathcal{P})$,
we obtain that $Y$ is a WEDS. 
\end{proof}
As an immediate corollary of and \Cref{prop:factor_of_WEDS_is_WEDS}
we obtain:
\begin{cor}
Every zero-dimensional topological factor of a zero-dimensional EDS
is a WEDS. 
\end{cor}

Finally, putting this result together with \Cref{thm:zero_dim_effective_extension}
we obtain the following general result.
\begin{prop}
\label{prop:factor-of-EDS-is-WEDS} Let $\Gamma$ be a finitely generated
and recursively presented group, $\Gamma\curvearrowright X$ be an
EDS and $\Gamma\curvearrowright Y$ be a zero-dimensional topological
factor. Then $\Gamma\curvearrowright Y$ is a WEDS.
\end{prop}

\begin{proof}
As $\Gamma$ is recursively presented, by \Cref{thm:zero_dim_effective_extension}
there is a zero-dimensional EDS extension of $\Gamma\curvearrowright X$.
Then $\Gamma\curvearrowright Y$ is a topological factor of this zero-dimensional
extension and thus by \Cref{prop:0d-EDS_is_WEDS} we obtain that
it is a WEDS. 
\end{proof}
We have proved that for zero dimensional systems, the class of EDS
is different to the class of WEDS. Indeed, the example constructed
in \Cref{prop:factor-of-eds-which-is-not-eds} is not an EDS,
but it follows from \Cref{prop:factor-of-EDS-is-WEDS} that is
a WEDS.
\begin{cor}
Let $\Gamma\curvearrowright X$ be an EDS and $\Gamma\curvearrowright Y$
be an expansive and zero-dimensional topological factor. If $X$ is
zero-dimensional or $\Gamma$ is recursively presented then $\Gamma\curvearrowright Y$
is an EDS. 
\end{cor}

\begin{cor}
\label{cor:EDSsubshifts_closed_under_top_factors} The class of effective
subshifts is closed under topological factor maps. 
\end{cor}
We remark that \Cref{cor:EDSsubshifts_closed_under_top_factors}
can also be obtained through the Curtis-Hedlund-Lyndon theorem (\Cref{effective-subshifts-are-closed-by-factors})
We finish this section with the following question which we were unable
to answer.
\begin{question}
Is it true that any WEDS is the topological factor of some EDS? 
\end{question}

%% file: Chapter_Medvedev.tex
In this chapter we present the results obtained by the
author and Sebastián Barbieri in \cite{barbieri_medvedev_2024}, with the addition of one result proved by the author in \cite{carrasco-vargas_translationlike_2024}. These results concern a dynamical invariant for subshifts of  recursion-theoretical origin called Medvedev degrees. This invariant can be used to compare the complexity of subshifts that contain only uncomputable configurations. We study how these degrees can be transferred from one group to another through algebraic and geometric relations, such as quotients, subgroups, translation-like actions and quasi-isometries. We use these tools to study the possible values taken by this invariant on subshifts of finite type on some finitely generated groups. We obtain a full classification for some classes, such as virtually polycyclic groups and branch groups with decidable word problem. We also show that all groups which are quasi-isometric to the hyperbolic plane admit SFTs with nonzero Medvedev degree. Furthermore, we provide a classification of the degrees of sofic subshifts for several classes of groups. 

\minitoc

\section{Introduction}
Hanf and Myers proved the existence of a $\mathbb{Z}^2$-SFT for which every configuration is an uncomputable function~\cite{hanf_nonrecursive_1974,myers_nonrecursive_1974}. This result can be further refined using the language of Medvedev degrees. Informally, the Medvedev degree of a set measures the algorithmic complexity of computing one of its elements, and can be used to meaningfully distinguish sets without computable points. A formal definition of Medvedev degrees is given in~\Cref{section:Medvedev}.

In this work we propose a systematic study of Medvedev degrees as a topological conjugacy invariant for subshifts on groups. These degrees form a lattice and capture the algorithmic complexity of a subshift. For instance, the Medvedev degree of a subshift is zero exactly when it has computable configurations. As a measure of complexity, the Medvedev degree of a subshift shares some properties with topological entropy for amenable groups: it does not increase under factor maps, and behaves nicely with direct products and disjoint unions. These properties hold for subshifts on any finitely generated group, with no assumption on the complexity of its word problem or amenability. 

The main goal of this project is to study the following classification problem:

\begin{problem}\label{problem}
    Given a finitely generated group $G$, what is the class of Medvedev degrees of $G$-SFTs?
\end{problem}

Remarkable effort has been put into classifying the possible values of topological entropy for SFTs. Lind~\cite{Lind1984} provided a full classification of the entropies of $\ZZ$-SFTs, Hochman and Meyerovitch provided a classification for $\ZZ^d$-SFTs for $d\geq 2$, the first author extended their classification to several classes of amenable groups~\cite{barbieri_entropies_2021}, while Bartholdi and Salo recently extended it to Baumslag-Solitar and Lamplighter groups~\cite{bartholdi_salo_shift_lamplighter_2024}. This work can be thought of as an analogous effort for Medvedev degrees.

\subsubsection*{State of the art}

There is a complete answer for $\ZZ^d$, $d\geq 1$. Indeed, it is well-known that every nonempty $\mathbb{Z}$-SFT has finite orbits which implies that its Medvedev degree is zero. On the other hand, Simpson proved that the class of Medvedev degrees of nonempty SFTs on $\mathbb{Z}^{d}$, $d\geq2$ is the class of $\Pi_{1}^{0}$ degrees \cite{simpson_medvedev_2014}. This result implies and refines greatly the result of Hanf and Myers~\cite{hanf_nonrecursive_1974,myers_nonrecursive_1974}. In addition, nonempty SFTs with nonzero Medvedev degree have been constructed on Baumslag Solitar groups $\operatorname{BS}(n,m)$, $n,m\geq 1$ \cite{aubrun_tiling_2013}, and the lamplighter group~\cite{bartholdi_salo_shift_lamplighter_2024}. 

We also mention that for the more general class of effective subshifts on $\ZZ$, Medvedev degrees have been characterized as those $\Pi_1^0$ degrees by Miller~\cite{miller_two_2012}, and that Medvedev degrees of subshifts have been considered in \cite{hochman_note_2009,ballier2013universality} in relation to the existence of systems that are universal for factor maps.   

\subsubsection*{Main results}

Given a finitely generated group $G$, we denote by $\msft{G}$ the class of Medvedev degrees of $G$-SFTs. First we observe that if $G$ is recursively presented, then $\msft{G}$ must be contained in the class of $\Pi_1^0$ degrees (see~\Cref{prop:msft-and-msoff-are-contained-in-Pi-1-degrees}). In order to study~\Cref{problem}, we study the behaviour of $\msft{G}$ and some variants with respect to different group theoretical relations, such as subgroups (\Cref{cor:subgroups_decidable_membership}), commensurability (\Cref{prop:commensurable_same_degrees}), quotients (\Cref{prop:quotients_fg}), translation-like actions (\Cref{prop:medvedev-degrees-and-translation-like-actions,prop:translation-like-action-subshift-is-effective} and~\Cref{prop:translation-like-actions-and-medvedev-degrees}), and quasi-isometries (\Cref{lem:QI_medvedev} and~\Cref{coro_QI}).

We provide a full classification of $\msft{G}$ for virtually polycyclic groups (\Cref{thm:polycyclic}), namely, we show that for a virtually polycyclic group $G$ which is virtually cyclic then $\msft{G}$ consists of only the zero degree, and otherwise $\msft{G}$ is the class of all $\Pi_1^0$ Medvedev degrees. We also show that $\msft{G}$ is this latter class for direct products of infinite groups with decidable word problem (\Cref{thm:direct_products_clasification_medvedev_degrees}), and branch groups with decidable word problem (\Cref{cor:branch_groups}). Without hypotheses on the word problem, we are still able to show that all direct products of infinite groups, and infinite branch groups admit SFTs with nonzero Medvedev degree (\Cref{thm:direct_products_nontrivial_medvedev_degrees}). Furthermore, we prove that all groups which are quasi-isometric to the hyperbolic plane admit SFTs with nonzero Medvedev degree (\Cref{thm:hyperbolic_qi_medvedev}). 

For some groups where we are not able to compute $\msft{G}$ we can still prove that the class $\msof{G}$ of Medvedev degrees of sofic subshifts equals the class of $\Pi_1^0$ Medvedev degrees. This is done coupling our results with existing results about Medvedev degrees of effective subshifts, and existing simulation results. Simulation results relate sofic subshifts on one group with effective subshifts on a different group (see \Cref{subsec:simulation} for a precise definition).

We prove that $\msof{G}$ equals the class of $\Pi_1^0$ Medvedev degrees for all infinite finitely generated groups with decidable word problem that simulate some other group verifying the same hypotheses (\Cref{prop:simulation}). This result covers self-simulable groups with decidable word problem, such as Thompson's $V$, $\operatorname{GL}_n(\ZZ)$, $\operatorname{SL}_n(\ZZ)$, $\operatorname{Aut}(F_n)$ and $\operatorname{Out}(F_n)$ for $n\geq 5$, where $F_n$ denotes the free groups on $n$ generators~\cite{barbieri_groups_2022}. By the results in~\cite{bartholdi_shifts_2024}, this also applies to the Baumslag-Solitar groups $\operatorname{BS}(1,n)$, $n\geq 1$, and the lamplighter group.   

\subsubsection*{Open problems}

Let us note that a finitely generated group admits a nonempty SFT with nonzero Medvedev degree if and only if it admits a sofic subshift with nonzero Medvedev degree. However, this does not necessarily imply that $\msft{G}=\msof{G}$. This raises the question of whether every sofic $G$-subshift admits an SFT extension with equal Medvedev degree (\Cref{question:equal_degree_extension}). This is a Medvedev-degree version of the question of whether every sofic subshift on an amenable group admits an equal entropy SFT extension (see \cite[Problem 9.4]{hochman_characterization_2010}).

The classical observation that $\ZZ$-SFTs can only achieve the zero degree can be easily generalized to virtually free groups (\Cref{prop:virtually_free_have_0_degree}). In every other finitely generated and recursively presented group where the set $\msft{G}$ is known, it is the set of all $\Pi_1^0$ degrees. This leads us to conjecture (\Cref{conjecture_allSFTSarePi10completeclasses}) that for every infinite, finitely generated and recursively presented group which is not virtually free, then $\msft{G}$ is the set of all $\Pi_1^0$ Medvedev degrees.

For finitely generated groups which are not recursively presented, the $\Pi_1^0$ bound may not hold. In that context we pose the less ambitious conjecture that is $G$ is finitely generated and not virtually free, then there exists a nonempty $G$-SFT with nontrivial Medvedev degree (\Cref{conj:uncomputableconfig}).

Overall, we expect that positive solutions to any of these two conjectures would be rather distant with the existing techniques. \Cref{conj:uncomputableconfig} would imply a positive solution to Carroll and Penland's conjecture about groups admitting weakly aperiodic SFTs \cite{carroll_periodic_2015}, and to Ballier and Stein's conjecture about groups with undecidable domino problem~\cite{ballier_domino_2018} (see \Cref{Medvedev-conjecture-is-pulent}). 
\section{Medvedev degrees and basic properties}\label{section:Medvedev}
    \subsection{The lattice of Medvedev degrees and the sublattice of $\Pi_1^0$ degrees}
    Here we provide a brief review of the lattice of Medvedev degrees $\M$. These degrees were introduced in \cite{medvedev_degrees_1955} with the purpose of relating propositional formulas with mass problems. Important sources are the surveys \cite{hinman_survey_2012}, \cite{sorbi_medvedev_1996}, see also \cite{lewis_topological_2011}. 

    Intuitively, a mathematical problem $P$ has a higher Medvedev degree than a mathematical problem $Q$ if every solution to the problem $P$ can be used to compute a solution to the problem $Q$. This intuition can be made precise by defining a pre-order relation $\leq$ on subsets of $\{0,1\}^\NN$,  where each set is interpreted as the set of solutions of a fixed mathematical problem. Given two sets $P,Q\subset \{0,1\}^\NN$, we say that $P$ is \define{Medvedev reducible} to $Q$ when there is a partial computable function $\Psi$ on $\{0,1\}^\NN$, defined on $Q$ and such that $\Psi(Q)\subset P$. We abbreviate this relation by $P\leq Q$. We say that $P$ is \define{Medvedev equivalent} to $Q$ if both $P\leq Q$ and $Q\leq P$. The set $\M$ of Medvedev degrees is the set of equivalence classes of $\{0,1\}^\NN$ modulo Medvedev equivalence. Notice that Medvedev reduction induces a partial order on $\M$, and we will use the same symbol $\leq$ to compare Medvedev degrees. We denote the Medvedev degree of a set $P$ by $m(P)$. 

    The partially ordered set $(\M,\leq)$  is indeed a distributive lattice: there is a minimal element denoted $0_\M$ called the trivial degree, an operation $\wedge$ of infimum, an operation $\vee$ of supremum, and a maximal element that here will be denoted $1_\M$. Given two sets $P$ and $Q$, we have the following:
    \begin{itemize}
        \item $m(P)=0_\M$ when $P$ has a computable element. In intuitive terms, a mathematical problem is easy with this complexity measure when it has at least one computable solution.
        \item $m(P)=1_\M$ when $P$ is empty. In intuitive terms, a mathematical problem has maximal complexity with this complexity measure when it has no solution.
        \item $m(P)\vee m(Q)$ equals the degree of the set $\{x\in\{0,1\}^\NN : (x_{2n})_{n\in\NN}\in P\text{ and } (x_{2n+1})_{n\in\NN}\in Q \}$. In intuitive terms, $m(P)\vee m(Q)$ is the difficulty of solving both problems $P$ and $Q$ simultaneously. 
        \item $m(P)\wedge m(Q)$ equals the degree of the set $\{0x :x\in P \}\cup \{1x :x\in Q \}$. Here $0x$ stands for the concatenation of the finite word $0$ with the infinite sequence $x$.   In intuitive terms, $m(P)\wedge m(Q)$ is the difficulty of solving at least one of the problems $P$ and $Q$. 
    \end{itemize}

    In this chapter we will use the following terminology:
    \begin{definition}
        A subset of $\{0,1\}^\NN$ is called $\Pi_1^0$ when it is an effectively closed set in the sense of \Cref{chap:Background}.
    \end{definition}
        
    A sublattice of $(\M,\leq)$ that will be relevant for us is that of $\Pi_1^0$ degrees.  These are the degrees of nonempty $\Pi_1^0$ subsets of $\{0,1\}^\NN$. This lattice contains the minimal element $0_{\M}$, but also contains a maximal element which corresponds to the degree of the class of all PA-complete sets. See~\cite{Cenzer_1999_survey_in_Pi_10_degrees} for a survey on these results.
    

    \subsection{Medvedev degrees as a dynamical invariant for shift spaces}

    The goal of this section is to define Medvedev degrees of subshifts of $A^G$, where $G$ is an arbitrary finitely generated group. This will be done through the pullback subshift, defined in \Cref{chap:computable_analysis_subshifts}.

    \begin{definition}
        Let $G$ be a finitely generated group and let $X\subset A^G$ be a subshift. Let $\widehat{X}$ be the pullback subshift of $X$ defined in \Cref{def:effectivesubshiftthroughfreegroup}.  We define the Medvedev degree $m(X)$ as $m(\widehat{X})$.
    \end{definition}
       
    This definition is independent of the chosen generating set of $G$ by \Cref{pullbacks-are-recursively-homeomorphic}. We now prove some simple facts about $m$ as a dynamical invariant for subshifts. Given two subshifts $X,Y$, denote by $X\times Y$ their direct product as dynamical systems and by $X\sqcup Y$ their disjoint union.
    
    \begin{proposition}\label{prop:basic-properties}
        Let $G$ be a finitely generated group and let $X$, $Y$ be $G$-subshifts. 
        \begin{enumerate}
            \item If there is a topological morphism from $X$ to $Y$, then $m(X)\geq m(Y)$. In particular, the Medvedev degree of a subshift does not increase under factors and is a topological conjugacy invariant. 
            \item $m(X\times Y)=m(X)\vee m(Y)$.
            \item $m(X\sqcup Y)=m(X) \wedge m(Y)$.
        \end{enumerate}
    \end{proposition}
    \begin{proof}
        If there is a topological morphism $\psi\colon X\to Y$, then there is also a topological morphism $\widehat \psi\colon \widehat X\to \widehat  Y$. But then $\widehat \psi$ is computable by \Cref{prop:morphisms-of-subshifts-are-computable-in-pullbacks}, so we have $m(\widehat X)\geq m(\widehat Y)$. It follows that $m(X)\geq m(Y)$. We now prove the second claim. We identify $X\times Y$ with the subshift \[Z = \{(x(g),y(g))_{g\in G} : x,y\in X,Y\}.\]
        By our discussion above, $\widehat X$ and $\widehat Y$ are recursively homeomorphic to sets $P,Q\subset\{0,1\}^\NN$. Using this fact, it is straightforward that $\widehat  Z$ is recursively homeomorphic to $\{x\in\{0,1\}^\NN : (x_{2n})_{n\in\NN}\in P\text{ and } (x_{2n+1})_{n\in\NN}\in Q \}$. As the Medvedev degree of this set equals $m(P)\vee m(Q)$, it follows that $m(Z)=m(X)\vee m(Y)$.

        Again, by taking a topologically conjugate version of $X\sqcup Y$, we can assume that $A$ and $B$ are disjoint alphabets, and that $Z=X\cup Y$ is a subshift on alphabet $A\cup B$. By our discussion above, $\widehat X$ and $\widehat Y$ are recursively homeomorphic to sets $P,Q\subset\{0,1\}^\NN$. Using this fact, it is  straightforward that $\widehat Z$ is recursively homeomorphic to  $\{0x :x\in P \}\cup \{1x :x\in Q \}$. As the Medvedev degree of this set equals $m(P)\wedge m(Q)$, it follows that $m(Z)=m(X)\wedge m(Y)$. 
    \end{proof}

    We recall that effective subshifts are defined in \Cref{def:effectivesubshiftthroughfreegroup}.
    \begin{definition}
        Let $G$ be a finitely generated group. We denote \begin{enumerate}
            \item $\msft{G} = \{ m(X) : X \mbox{ is a nonempty $G$-SFT}\}.$
            \item $\msof{G} = \{ m(X) : X \mbox{ is a nonempty sofic $G$-subshift} \}.$
            \item $\meff{G} = \{ m(X) : X \mbox{ is a nonempty effective $G$-subshift}\}.$
        \end{enumerate}
        \end{definition}
    It is immediate that all three of the above classes are invariants of isomorphism, as subshifts in two isomorphic groups can be lifted through pullbacks to the same free group. Moreover, each of these three sets is a sublattice: it is closed under the operations $\vee$ and $\wedge$ and always contains $0_{\M}$. This follows from \Cref{prop:basic-properties} and the fact that the classes of SFTs, sofic, and effective subshifts are closed by direct products and disjoint unions.

    Next we will review the basic relations between these three sets. As every SFT is sofic, we have $\msft{G}\subset \msof{G}$. If $G$ is recursively presented then sofic subshifts are effective (\Cref{prop:sft-and-sofic-subshifts-are-effective}), so $\msof{G}\subset\meff{G}$.  It follows from the definition that an effective subshift has a $\Pi_1^0$ Medvedev degree. As sofic subshifts and SFTs are effective when $G$ is recursively presented, we have the following:
    \begin{observation}\label{prop:msft-and-msoff-are-contained-in-Pi-1-degrees}
        Let $G$ be a recursively presented group. Then all three classes $\msft{G}$, $\msof{G}$ and $\meff{G}$ are contained in the class of $\Pi_1^0$ degrees. 
    \end{observation}

    We remark that some non-recursively presented groups only admit the trivial effective action \cite[Proposition 2.10]{barbieri_groups_2022}). For these groups we have $\meff{G}= \{ 0_{\M}\}$. This shows that the inclusion $\msof{G}\subset \meff{G}$ requires the hypothesis of recursive presentation.

        
\section{Transference results}\label{section:transference}

Next we shall study how the space of Medvedev degrees of subshifts on groups behaves with respect to basic relations in group theory.

\subsection{Elementary constructions}

Let $\rho \colon G \to H$ be a group homomorphism. Then $\rho$ induces a map $\rho^* \colon A^H \to A^G$ through \[ \rho^*(x)(g) = x(\rho(g)) \mbox{ for every } x\in A^H, g\in G. \]
It is clear from the above definition that for any subshift $X\subset A^H$, then $\rho^*(X)\subset A^G$ is also a subshift.

    \begin{lemma}\label{lema:medvedev_degree_of_pullback}
        For any finitely generated group $G$, epimorphism $\rho \colon G \to H$ and subshift $X\subset A^H$, we have \[ m(X) = m(\rho^*(X)).  \]
    \end{lemma}

    \begin{proof}
        Note first that as $G$ is finitely generated and $\rho$ is an epimorphism, then $H$ is also finitely generated. Fix finite symmetric generating sets $S,U$ for $G$ and $H$ respectively. Let $\widehat{X}$ be the pullback of $X$ to $F(U)$. For each $s \in S$, let $\bar{\rho}(s) \in U^*$ be a word which represents $\rho(s)$ in $H$. For $\widehat{x} \in \widehat{X}$, we can define $\widehat{y}\colon F(S) \to A$ by \[\widehat{y}(s_1\dots s_n) = \widehat{x}(\bar{\rho}(s_1) \dots \bar{\rho}(s_n)).\]
        It is clear that $\widehat{y} \in \widehat{\rho^*(X)}$ and that the map which sends $\widehat{x}$ to $\widehat{y}$ is computable. It follows that $m(\rho^*(X)) \leq m(X)$. 

        Conversely, for every $u \in U$, fix $w_u \in S^*$ such that $\rho(\underline{w_u}) = u$. Given as input $\widehat{y} \in \widehat{\rho^*(X)}$ we define $\widehat{x}\colon F(U)\to A$ by \[\widehat{x}(u_1\dots u_n) = \widehat{y}(w_{u_1}\dots w_{u_n}).  \]
        This map again is clearly computable and $x \in \widehat{X}$. It follows that $m(X) \leq m(\rho^*(X))$.
    \end{proof}

    \begin{definition}
        Let $H \leqslant G$ be a subgroup and $X\subset A^H$ a subshift. The \define{free extension} of $X$ to $G$ is the subshift \[  \widetilde{X} = \{ x \in A^{G} : \mbox{ for every } g \in G, (x(gh))_{h \in H} \in X\}.   \]
    \end{definition}

    It is clear from the definition that any set of forbidden patterns that defines $X$ in $H$ also defines $\widetilde{X}$ in $G$. In particular it follows that if $X$ is an SFT (respectively sofic, effectively closed by patterns) then so is $\widetilde{X}$.
    
\begin{lemma}\label{free-extensions}
    Let $H\leqslant G$ be finitely generated groups and let $X\subset A^H$ be a subshift. Let $\widetilde{X}$ be the free extension of $X$ to $G$. Then $m(X) \leq m(\widetilde{X})$. Furthermore, if either
    \begin{enumerate}
        \item[(1)] $G$ is recursively presented and $H$ has decidable membership problem in $G$,
        \item[(2)] $H$ has finite index in $G$,
    \end{enumerate}then $m(X) = m(\widetilde{X})$. 
\end{lemma}

\begin{proof}
   Let $S,U\subset G$ be two finite symmetric sets with $S\subset U$ such that $S$ generates $H$ and $U$ generates $G$. As $F(S)\subset F(U)$, the map which on input $\widetilde{x} \in \widehat{\widetilde{X}}$ returns its restriction to $F(S)$ is computable and yields an element $x\in \widehat{X}$. Thus $m(X) \leq m(\widetilde{X})$.

    Conversely, suppose (1) holds. Consider the algorithm which on input $x\colon F(S)\to A$ lists all elements of $F(U)$ in lexicographical order $(u_i)_{i \in \NN}$. When it lists $u_n$, it checks in increasing order for every $k < n$ (using the algorithm for the membership problem of $H$ in $G$) whether $\underline{u_k^{-1}u_n} \in H$. If it is the case for some $k$, then it computes (using that $G$ is recursively presented) a word $w_n \in F(S)$ such that $\underline{u_k^{-1}u_nw_n^{-1}} = 1_G$ and sets $\widetilde{x}(u_n) = x(w_n)$; Otherwise, if $\underline{u_k^{-1}u_n} \notin H$ for every $k < n$, it sets $\widetilde{x}(u_n)=x(1_{F(S)})$. This yields a map $\widetilde{x} \colon F(U) \to A$ which satisfies that $\widetilde{x} \in \widehat{\widetilde{X}}$ if and only if $x \in \widehat{X}$, thus $m(\widetilde{X}) \leq m(X)$.

    Finally, suppose (2) holds. Let $T$ be a finite set such that $TH=G$. Without loss of generality, we may suppose $T\subset U$ and that each $u \in U$ is written as $u = t_uw_u$ with $t_u \in T$ and $w_u \in S^*$. Given such a set $U$, there is a computable map which on input $v \in U^*$ returns $t_v \in T$ and $u_v \in S^*$ such that $\underline{v} = \underline{t_v u_v}$ (to do this, we only need to store the values of all multiplications $tk$ and $s t$ for $s \in S, t,k \in T$, of which there are finitely many). Using this map, it follows that the function which on input $x\colon F(S)\to A$ returns $\widetilde{x} \colon F(U) \to A$ given by $\widetilde{x}(v)=x(u_v)$ is computable and $\widetilde{x} \in \widehat{\widetilde{X}}$ if and only if $x \in \widehat{X}$. It follows that $m(\widetilde{X}) \leq m(X)$.
\end{proof}

\begin{corollary}
    Let $H \leqslant G$ be finitely generated groups. If there exists an $H$-SFT with nontrivial Medvedev degree, then there exists a $G$-SFT with nontrivial Medvedev degree.
\end{corollary}

\begin{corollary}\label{cor:subgroups_decidable_membership}
    Let $H\leqslant G$ be finitely generated, recursively presented groups and suppose $H$ has decidable membership problem in $G$, then $\msft{H} \subset \msft{G}$ and $\msof{H} \subset \msof{G}$. 
\end{corollary}

\subsection{Commensurability and quotients}

    Here we prove that $\msft{G}$ and $\msof{G}$ are commensurablility invariants. For this we apply a construction of Carroll and Penland~\cite{carroll_periodic_2015}. Let $H\leqslant G$ be finitely generated groups such that $H$ has finite index in $G$, and let $T=\{t_1=1_G,t_2,\dots,t_r\}\subset G$ be a finite set such that $HT=G$. Given an alphabet $A$, we let $B=A^T$ and $\Psi\colon A^G\to B^H$ be the map given by $\Psi(x)(h)=(h^{-1}x)|_T $. The image of a $G$-subshift $X$ by $\Psi$ is an $H$-subshift. Moreover, $X$ is a $G$-SFT (respectively sofic subshift) if and only if $\Psi(X)$ is an $H$-SFT (respectively sofic $H$-subshift), see~\cite[Section 3.1]{carroll_periodic_2015} for the proof.

    \begin{lemma}\label{higher-block-preserves-m}
        Given a subshift $X\subset A^G$, $m(X)=m(\Psi(X))$.
    \end{lemma}
    \begin{proof}
        Let $U,S \subset G$ be symmetric finite generating sets with $U\subset S$ for $H$ and $G$ respectively. Fix a set of elements $\{h_1,\dots,h_r\}$ in $F(S)$ such that $h_i$ corresponds to $t_i$ in $G$.
        
        For the inequality $m(X)\geq m(\Psi(X))$, we define a computable function $\widehat \Psi\colon A^{F(S)}\to B^{F(U)}$ that represents $\Psi$ on the pullback subshifts. Given  $x\in A^{F(S)}$, $\widehat \Psi (x)\colon F(U)\to B$ is defined by $\widehat{\Psi}(x)(w)(t_i) = x(w h_i)$ for $w\in F(U)$ and $t_i \in T$. It is clear that $\widehat \Psi$ is computable and that every element in $\widehat X$ is mapped by $\widehat \Psi$ to $\widehat{\Psi(X)}$, thus we obtain $m(X)\geq m(\Psi(X))$. 

        For the remaining inequality, we proceed analogously as in~\Cref{free-extensions}. There is an algorithm which on input $v \in S^*$, returns $t_v \in T$ and $u_v\in U^*$ such that $\underline{v} = \underline{u_v t_v}$. Consider the map which on input an element $x\in \widehat{\Psi(X)}$, returns $y\in \widehat X$ defined as follows: for $v\in F(S)$ set $y(v)=x(u_v)(t_i)$. This defines a computable function from $\widehat {\Psi(X)}$ to $\widehat X$, and proves the desired inequality. 
    \end{proof}

Recall that two groups $G$ and $H$ are said to be \define{commensurable} when there exist finite index subgroups $G' \leqslant G$ and $H' \leqslant H$ which are isomorphic. 

    \begin{proposition}\label{prop:commensurable_same_degrees}
        Let $G$ and $H$ be two commensurable finitely generated groups. Then $\msft{G} = \msft{H}$ and $\msof{G} = \msof{H}$.
    \end{proposition}
    \begin{proof}
        As the set of Medvedev degrees is an isomorphism invariant of groups, we assume without loss of generality that $H$ is a finite index subgroup of $G$.
        
        Let $X$ be an SFT on $H$. As $H$ has finite index, it follows by~\Cref{free-extensions} that the free extension of $X$ in $G$ is a $G$-SFT that has the same Medvedev degree as $X$. Therefore $ \msft{H} \subset \msft{G}$. 
        
        Conversely, given a $G$-SFT $X$, we have that $\Psi(X)$ is an $H$-SFT. Moreover, it has the same Medvedev degree as $X$ by~\Cref{higher-block-preserves-m}. This shows that $ \msft{G} \subset \msft{H}$. The proof that $\msof{G} = \msof{H}$ is identical. 
    \end{proof}

    \begin{proposition}\label{prop:quotients_fg}
        Consider a short exact sequence of groups $1 \to N \to G \to H \to 1$. If both $G$ and $N$ are finitely generated then \[ \msft{H} \subset \msft{G} \mbox{ and } \msof{H} \subset \msof{G}.   \]
    \end{proposition}

    \begin{proof}
        Denote by $\rho\colon G\to H$ the epimorphism in the short exact sequence and let $X\subset A^H$ be an SFT. By~\Cref{lema:medvedev_degree_of_pullback}, we have that the pullback $\rho^*(X)$ satisfies $m(\rho^*(X))=m(X)$. Therefore in order to show that $\msft{H} \subset \msft{G}$ it suffices to show that $\rho^*(X)$ is an SFT.

        As $X$ is an SFT, there exists a finite set $F\subset H$ and $L \subset A^F$ such that $x \in X$ if and only if $(hx)|_{F} \in L$ for every $h \in H$. For every $f \in F$, choose $g_f \in G$ such that $\rho(g_f) = f$ and let $U =\{g_f : f \in F \}\subset G$.
        Let $S\subset G$ be a finite symmetric generating set of $N$ which contains the identity and consider the set $E = U \cup S$. Let \[ W = \{ w \in A^{E} : (w(g_f))_{f \in F} \in L \mbox{ and } w(s) = w(1_G) \mbox{ for every } s \in S.\}.   \]

        It is immediate from the definition that $\rho^*(X) = \{y \in A^G : (gy)|_{E} \in W \mbox{ for every } g \in G\}$. In particular, this means that $\rho^*(X)$ is an SFT (a set of defining forbidden patterns is $A^{E}\setminus W$.

        Now let $Y\subset B^{H}$ be a sofic shift and let $X$ be some SFT such that $\phi \colon X \to Y$ is a topological factor map given by some local map $\Phi \colon A^{K}\to B$ for some finite $K\subset H$. As before, for every $k \in K$ choose $g_k \in G$ such that $\rho(g_k) = k$ and let $V =\{g_k : k \in K \}\subset G$. By the previous argument, we have that $\rho^*(X)$ is also an SFT. Let $\phi^*\colon \rho^*(X)\to B^{G}$ be given by the local rule $\Phi^* \colon A^V \to B$ defined by \[  \Phi^*( (a_v)_{v \in V}) = \Phi( (a_{g_k})_{k \in K}). \]

        It is clear from the definition that $\phi^*(\rho^*(X)) = \rho^*(Y)$, thus $\rho^*(Y)$ is also a sofic shift and thus $\msof{H} \subset \msof{G}$.
    \end{proof}

    \begin{remark}
        Without the assumption that $N$ is finitely generated,~\Cref{prop:quotients_fg} does not hold. For instance, for the short exact sequence $1 \to [F_2,F_2] \to F_2 \to \ZZ^2 \to 1$ we have that $\msft{F_2} = \{0_\M\}$ (\Cref{prop:virtually_free_have_0_degree}) but $\msft{\ZZ^2}$ is the set of all $\Pi_1^0$ Medvedev degrees.
    \end{remark}

    \subsection{Bounded actions and the orbit membership problem}
    
    Let $G,H$ be two finitely generated groups, endowed with word metrics. A right action $\ast$ of $H$ on $G$ is called \define{bounded} if for every $h\in H$, the map given by $g\mapsto g\ast h$ is at bounded distance from the identity function on $G$. An equivalent condition is that, for some (equivalently, every) finite generating set $S\subset H$, there exist a finite set $F\subset G$ with $g^{-1}(g\ast s)\in F$ for every $g\in G$ and $s\in S$. A right action is called \define{translation-like} if it is bounded and free. 

    The set of all bounded actions of $H$ on $G$ with fixed parameters $S$ and $F$ can be described by a subshift on $G$. Moreover, we can overlay this subshift with configurations coming from a subshift on $H$. This technique was introduced in \cite{jeandel_translationlike_2015}, and then it has been used several times to construct specific subshifts on groups
     \cite{barbieri_entropies_2021, carrasco-vargas_translationlike_2024, cohen_strongly_2021, barbieri_geometric_2019}. In what follows we describe the construction briefly, and refer the reader to \cite{jeandel_translationlike_2015} for details.
    
    Let $S$ be a finite and symmetric generating set for $H$, and $F\subset G$ a finite set. Let $B$ be the set of maps from $S$ to $F$. Given $x\in B^G$, a group element $g\in G$, and a word $w\in S^\ast$, we denote by $\Phi(g,x,w)$ the group element in $G$ obtained by interpreting $x(g)$ as an arrow from $g$ to $gx(g)$ labeled by $S$, and following these arrows as indicated by the word $w$.  More precisely, we define $\Phi(g,x,\epsilon)=g$, and then for $s\in S$ and $w\in S^\ast$ we set $\Phi(g,x,ws)=\Phi(g,x,w)\cdot x(\Phi(g,x,w))(s)$.  
    \begin{definition}
        Given $H,G,S,F$ and $B$ as above, the subshift of bounded actions $\operatorname{T} \subset B^G$ is the set of all configurations $x\in B^G$ which satisfy that $\Phi(g,x,w)=g$, for every $g\in G$ and $w\in S^\ast$ that equals the identity element in $H$.  
    \end{definition}
    Observe that for every $x\in B^G$, the free monoid $S^*$ acts (on the right) on $G$ by $g \cdot w = \Phi(g,x,w)$ for $w \in S^*$. From our definition, it is clear that $\operatorname{T}$ is precisely the subset of configurations which induce an action of $H$. With this in mind, we will extend the notation above so that given $g\in G$, $h\in H$, and $x\in \operatorname{T}$, we write $\Phi(g,x,h)$ to denote $\Phi(g,x,w)$ for any $w \in S^*$ which represents $h$. 

    \begin{definition}
        Let $X\subset A^H$ be a subshift. We define $\operatorname{T}[X]\subset(A\times B)^G$ as the set of all configurations $(x,y)\in A^G \times B^G$ such that $y\in T$ and for every $g\in G$, the configuration $\left(x(\Phi(g,y,h))\right)_{h \in H}$ lies in $X$. 
    \end{definition}

    \begin{proposition}[\cite{jeandel_translationlike_2015}, Section 2]\label{T[X]-es-SFT}
        The sets $\operatorname{T}$ and $\operatorname{T}[X]$ are $G$-subshifts. Moreover, if $H$ is finitely presented and $X$ is an $H$-SFT, then $\operatorname{T}[X]$ is a $G$-SFT. 
    \end{proposition} 

    Notice that if $Y$ is a sofic $H$-subshift and $H$ is finitely presented we have that $\operatorname{T}[Y]$ is a sofic $G$-subshift. Indeed, it suffices to take an $H$-SFT $X$ for which there exists a topological factor map $\phi \colon X \to Y$, let $W = \{ (x,y) \in X\times Y : \phi(x) = y\}$ and notice that $W$ is an $H$-SFT. Then $T[W]$ is also a $G$-SFT and $T[Y]$ can be obtained as a projection of $T[W]$.
    
    Next we will study the relation between the degrees $m(X)$ and $m(\operatorname{T}[X])$. In order to do that, we need to introduce the following notion (see also \cite{bogopolski2010orbit,carrasco-vargas_translationlike_2024}).

    \begin{definition}
        Let $G,H$ be finitely generated groups and $R$ a finite set of generators for $G$. A right action $\ast$ of $H$ on $G$ has \define{decidable orbit membership problem} if there is a Turing machine which on input $u,v \in R^*$ decides whether $\underline{u}$ and $\underline{v}$ lie in the same $H$-orbit.
    \end{definition}
    Observe that when $G$ is recursively presented and the action is computable, this notion is equivalent to the existence of a decidable set $I$ and computable sequence $(u_i)_{i \in I}$ with $u_i \in R^*$ such that the corresponding sequence of elements $(\underline{u_i})_{i \in I}$ in $G$ form a collection of representatives for the orbits of $H$. 

    \begin{proposition}\label{prop:medvedev-degrees-and-translation-like-actions}
        Let $G,H$ be finitely generated groups and let $\operatorname{T}$ be the subshift of bounded actions for parameters $S,F$. Then:
        
        \begin{enumerate}
            \item For every subshift $X\subset A^H$ we have $m(\operatorname{T}[X])\geq m(\operatorname{T})\vee m(X)$.
            \item Suppose that $G$ is recursively presented and that there is a translation-like action $H\curvearrowright G$ which is computable and has decidable orbit membership problem (for the parameters $S$ and $F$ already fixed). For every subshift $X\subset A^H$ we have $ m(\operatorname{T}[X]) = m(X)$.      
        \end{enumerate}
    \end{proposition}

    \begin{proof}
        First observe that it is possible for $\operatorname{T}[X]$ to be empty, but then the inequality holds for the trivial reason that $m(\varnothing)=1_{\M}$ is maximal. Suppose now that $m(\operatorname{T}[X])$ is nonempty. The projection to the second coordinate of the alphabet $A\times B$ yields a function $m(\widehat{\operatorname{T}[X]})\to \widehat{\operatorname{T}}$, which is computable by~\Cref{prop:basic-properties}. It follows that $m(\operatorname{T}[X])\geq m(\operatorname{T})$. We now verify the nontrivial inequality $m(\operatorname{T}[X])\geq m(X)$. 
        
        Let $R$ be a finite generating set for $G$ which contains the range of the alphabet $B$ of $\operatorname{T}$, and let $F(R)$ be the free group on $R$. Given $x\in B^{F(R)}$, $u\in F(R)$, and $w\in S^\ast$, we denote by $\widehat{\Phi}(u,x,w)$ the element in $F(R)$ defined as follows. Recursively, $\widehat{\Phi}(u,x,\epsilon)=u$, and for $s\in S$ and $w\in S^\ast$, $\Phi(u,x,ws)=\Phi(u,x,w)\cdot x(\Phi(u,x,w))(s)$. Then it is clear that $\widehat{\Phi}$ is a computable function. Also note that $\widehat{\Phi}$ and $\Phi$ are compatible, in the sense that for all $x\in \operatorname{T}$ with pullback $\widehat{x}\in \widehat{\operatorname{T}}\subset B^{F(R)}$, for all $g\in G$, $u\in F(R)$ with $\underline{u}=g$, and for all $w\in S^\ast$, we have that $\Phi(g,x,w)=\underline{\widehat \Phi (u,\widehat x,w)}$.  We now define 
        
        \[\Psi\colon (A\times B)^{F(R)}\to A^{F(S)}, \mbox{ given by } \Psi(x,y) = (u\mapsto (x(\widehat{\Phi}(1,y,u)).\]
        
        This function is computable because $\widehat{\Phi}$ is computable. Moreover, it follows from the compatibility of $\widehat{\Phi}$ and $\Phi$,  that $\Psi(\widehat{\operatorname{T}[X]})\subset \widehat{X}$. This shows that $m(\operatorname{T}[X])\geq m(X)$.
        
        We now consider the second item. Suppose there exists a translation-like action $\ast$ as in the statement. We define a configuration $y \in B^{F(R)} $ by $y(u)(s)=(\underline{u}^{-1})(\underline{u}\ast s)$, where $\underline{u}$ is the element of $G$ represented by $u \in F(R)$. As $\ast$ is a computable translation-like action and $G$ is recursively presented, it follows that $y$ is a computable element in $B^{F(R)}$. Moreover, it is clear by its definition that $y \in \widehat{T}$ thus it follows that $m(\operatorname{T})=0_{\M}$. By the first item in the statement, we obtain that $m(\operatorname{T}[X])\geq m(X)$. 
        
        Finally, we prove that $m(X)\geq m(\operatorname{T}[X])$. As the translation-like action $\ast$ has decidable orbit membership problem, we can compute a set of elements $(u_i)_{i\in I}\subset F(R)$, where $I\subset \NN$ is a decidable set, such that the corresponding elements $(\underline{u_i})_{i \in I}$ in $G$ are a collection of representatives for orbits of the action. Let $\Theta\colon \widehat{X} \to \widehat{\operatorname{T}[X]}$ be the map given by $\Theta (z)=(x,y)$, where $y$ is the computable point already defined above, and $x$ is defined by 
        \[
        x(u)=z(w) \mbox{ for } u \in F(R), \mbox{ where } w \in F(S) \mbox{ is such that } \underline{\widehat{\Phi}(u_i,y,w)}=\underline{u} \mbox{ for some } i\in I, .
        \]
        The function $\Theta$ is well defined by the above construction and is computable. Indeed, as $G$ is recursively presented, there is an algorithm which, given $u\in F(R)$, computes $i\in I$ and $w\in F(S)$ such that $\underline{u} =\underline{\widehat \Phi (g_i,y, w)}$. It follows that $m(\operatorname{T}[X]) = m(X)$.\end{proof}

    \begin{corollary}\label{prop:translation-like-actions-and-medvedev-degrees}
        Let $G,H$ be finitely generated groups, where $H$ is finitely presented and admits a translation-like action on $G$.
        \begin{enumerate}
            \item For every sofic $H$-subshift $Y$, there is a $G$-SFT $X$ with $m(X)\geq m(Y)$. In particular, if $H$ admits a sofic subshift with with nonzero Medvedev degree, then so does $G$.
            \item If $G$ is recursively presented and there is a translation-like action of $H$ on $G$ that is computable and has decidable orbit membership problem, then $\msft{H}\subset\msft{G}$ and $\msof{H}\subset\msof{G}$. 
        \end{enumerate}
    \end{corollary}
    \begin{proof}
    Consider a sofic $H$-subshift $Y$. Let X be an $H$-SFT extension. By~\Cref{prop:basic-properties} we have $m(X)\geq m(Y)$. Let $\operatorname{T}$ be the subshift of bounded actions with parameters such that it is nonempty. By~\Cref{T[X]-es-SFT} we have that $Z= \operatorname{T}[X]$ is a $G$-SFT and by~\Cref{prop:medvedev-degrees-and-translation-like-actions} it follows that $m(Z)\geq m(X)$. We conclude that $m(Z) \geq m(Y)$. 
    
    With the extra hypotheses of the second item of~\Cref{prop:medvedev-degrees-and-translation-like-actions} we have that $m(T[X]) = m(X)$ for any subshift $X$. If we take $X$ an $H$-SFT (respectively a sofic $H$-subshift) then $\operatorname{T}[X]$ is a $G$-SFT (respectively sofic $G$-subshift) and thus we obtain $\msft{H}\subset\msft{G}$ and $\msof{H}\subset\msof{G}$. 
    \end{proof}

    \begin{observation}
        The previous results can be sometimes used as obstructions for the existence of translation-like actions. For instance, the fact that $\ZZ^2$ admits SFTs with nonzero Medvedev degree, and that $F_2$ does not, implies by the previous results that $\ZZ^2$ does not act translation-like on $F_2$. While this fact is rather intuitive, a direct proof of this result is not that straightforward, see \cite[Section 3]{cohen_counterexample_2019}. 
    \end{observation}

    In these results we have shown how to use the subshift of bounded actions to obtain SFTs with certain Medvedev degrees. We now consider the case of effective subshifts. 

    \begin{proposition}\label{prop:translation-like-action-subshift-is-effective}
        Let $G$ and $H$ be finitely generated  and recursively presented groups. Suppose further that there is a translation-like action $\ast$ of $H$ on $G$. Then:
        \begin{enumerate}
            \item The subshift $\operatorname{T}\subset B^G$ is effective. 
            \item If $X\subset A^H$ is effective, then $\operatorname{T}[X]\subset (A\times B)^G$ is effective. 
        \end{enumerate}
    \end{proposition}
    \begin{proof}
        Let $S$ be a finite and symmetric finite generating set for $H$ and let $F\subset G$ be a finite set such that $g^{-1}(g\ast s)$ for every $s \in S$, $g \in G$. Let $R$ be a finite symmetric generating set for $G$ that contains $F$, and let $\widehat \Phi$ be the function defined in the proof of~\Cref{prop:medvedev-degrees-and-translation-like-actions}. Let $X_0$ be the set of all elements $x\in B^{F(R)}$ such that, for all $g\in F(R)$, and for all $w\in S^\ast$ with $\underline w=1_G$, we have that $\widehat \Phi (g,x,w)$ and $g$ coincide in $F(R)$. As the set $\{w\in S^\ast : \underline{w}=1_G\}$ is recursively enumerable, it follows that $X_0$ is a $\Pi_1^0$ of $A^{F(R)}$. Also note that by~\Cref{prop:sft-and-sofic-subshifts-are-effective} $\widehat{A^G}$ is a $\Pi_1^0$ subset of $A^{F(R)}$. As $\widehat{\operatorname{T}}$ equals the intersection of $X_0$ with $\widehat{A^G}$, it follows that $\widehat{\operatorname{T}}$ is a $\Pi_1^0$ set. Therefore $\operatorname{T}$ is an effective subshift. 

        We now prove that, if $X$ is an effective subshift on $H$, then $\operatorname{T}[X]$ is an effective subshift on $G$. For this purpose, let $\Psi\colon (A\times B)^{F(R)}\to A^{F(S)}$ be the computable function defined in the proof of \Cref{prop:medvedev-degrees-and-translation-like-actions}. We also consider the computable function $\Omega\colon (A\times B)^{F(R)}\to B^{F(R)}$ that removes the $A$ component of the alphabet. It follows from the definitions that 
        \[\widehat{\operatorname{T}[X]}=\Omega^{-1}(\operatorname{T})\cap \Psi^{-1}(X).\]

        As the preimage of an $\Pi_1^0$ set by a computable function is a $\Pi_1^0$ set, this proves that $\widehat{T[X]}$ is $\Pi_1^0$. 
    \end{proof}

    A consequence of these results and a theorem of Seward which ensures the existence of a translation-like action of $\ZZ$ on any infinite and finitely generated group~\cite[Theorem 1.4]{seward_burnside_2014} is the following.
    \begin{corollary}\label{cor:all_infinite_rp_groups_havenonzero_effective_medvedev_degree}
        Every infinite, finitely generated and recursively presented group admits a (nonempty) effective subshift which achieves the maximal $\Pi_1^0$ Medvedev degree.
    \end{corollary}
    \begin{proof}
        Let $G$ be an infinite, finitely generated and recursively presented group, and let $H=\ZZ$. Let $X$ be a nonempty effective $\ZZ$-subshift with maximal $\Pi_1^0$ Medvedev degree, whose existence was proven by Miller in~\cite{miller_two_2012}. By a result of Seward~\cite[Theorem 1.4]{seward_burnside_2014}, $G$ admits a translation-like action by $\ZZ$, so (for appropriate parameters) the subshift $\operatorname{T}[X]$ is nonempty. By~\Cref{prop:translation-like-action-subshift-is-effective}, the subshift $\operatorname{T}[X]$ is effective. Finally, by \Cref{prop:medvedev-degrees-and-translation-like-actions}, we have $m(T[X])\geq m(X)$. As $G$ is recursively presented, it follows that $m(T[X])$ is a $\Pi_1^0$ degree and thus we must have $m(T[X])=m(X)$.
    \end{proof}

    In \Cref{chap:computable-translation-like-actions} we will prove an effective version of Seward's theorem for groups with decidable word problem. This allows us to prove the following.
    \begin{theorem}\label{meff-for-groups-with-decidable-word-problem}\label{teo:nicarrasco_eff_pi01}
        Let $G$ be an infinite and finitely generated group with decidable word problem. Then $\meff{G}$ equals the set of $\Pi_1^0$ Medvedev degrees.
    \end{theorem} 
    \begin{proof}
         Let $X\subset A^\Z$ be an effective subshift. By \Cref{thm:computable-translation-like-actions-with-decidable-orbit-problem}, there is a computable translation-like action $\Z\curvearrowright G$ with decidable orbit membership problem. It follows from \Cref{prop:translation-like-actions-and-medvedev-degrees} that for suitable parameters, the subshift $T[X]$ satisfies $m(X)=m(T[X])$. Moreover, $T[X]$ it is an effective subshift by \Cref{prop:translation-like-action-subshift-is-effective}. We have proved the inclusion  $\meff{\Z}\subset \meff{G}$. Miller \cite{miller_two_2012} proved that $\meff{\Z}$ contains every $\Pi_1^0$ Medvedev degree, so the same is true for $\meff{G}$. Conversely, the degree of every effective subshift is a $\Pi_1^0$ degree, so our claim follows. 
    \end{proof}
    \subsection{Quasi-isometries}

    Let $(X,d_X), (Y,d_Y)$ be two metric spaces. A map $f\colon X\to Y$ is a \define{quasi-isometry} if there exists a constant $C\geq 1$ such that
\begin{enumerate}
    \item $f$ is a \define{quasi-isometric embedding}: for all $x,x'\in X$
    \[\frac{1}{C}d_X(x,x') - C \leq d_Y(f(x), f(x'))\leq C d_X(x,x') + C,\]
    \item $f$ is \define{relatively dense}: for all $y\in Y$ there exits $x\in X$ such that $d_Y(y, f(x))\leq C$.
\end{enumerate}

    Two metric spaces are called \define{quasi-isometric} if there is a quasi-isometry between them. Given a quasi-isometry $f\colon X \to Y$, a map $g\colon Y \to X$ is called a \define{coarse inverse} of $f$ if the composition $g\circ f$ is uniformly close to the identity map $\operatorname{Id}_X$. Coarse inverses always exist, and every coarse inverse is in particular a quasi isometry, thus being quasi-isometric is an equivalence relation (see for instance~\cite{drutu_geometric_2018}).

    Two finitely generated groups $G, H$ are \define{quasi-isometric} if there exists a quasi-isometry between $(G,d_G)$ and $(H,d_H)$ for some choice of word metrics $d_G,d_H$. It is clear that if a map is a quasi-isometry for fixed word metrics, it is also a quasi-isometry for any other choice of word metric (up to modification of the constant $C$). We also remark that if a quasi-isometry between two finitely generated groups exists, then it is necessarily bounded-to-1.

    A well-known construction of Cohen~\cite{cohen_large_2017} shows that if one is given two finitely generated groups $G,H$ and a positive integer $N$, then one can construct a subshift $\operatorname{QI}$ on $G$ whose elements encode all quasi-isometries $f\colon H \to G$ which satisfy that $|\{f^{-1}(g)\}| \leq N$ for every $g \in G$. In particular, if $H$ is quasi-isometric to $G$, this subshift $\operatorname{QI}$ is nonempty for large enough $N$. Furthermore, if one of the groups (equivalently, both\footnote{As being finitely presented is an invariant of quasi-isometry, see~\cite{drutu_geometric_2018}}) are finitely presented then $\operatorname{QI}$ turns out to be an SFT. The main observation is that in this case, for any subshift of finite type $X\subset A^H$, we can enrich $\operatorname{QI}$ with the alphabet of $X$ and add extra local rules to obtain a $G$-SFT $\operatorname{QI}[X]$ which satisfies the property that in every configuration the encoded copy of $H$ is overlaid with a configuration $x \in X$. 

    The explicit construction of Cohen involves several details which are too lengthy to include in this chapter, so we will just provide a quick overview of an equivalent formulation. 
    
    Let $S,T$ be finite symmetric sets of generators which contain the identity for $G$ and $H$ respectively. Denote $d_G$ and $d_H$ the words metrics in $G$ and $H$ with respect to these generators. Let $f\colon H \to G$ be a quasi-isometry with constant $C\in \NN$ and suppose that $\max_{g \in G}|f^{-1}(g)| = N$. Let $I = \{1,\dots,N\}$. 
    
    The subshift $\operatorname{QI}$ is defined over the alphabet $B$ whose symbols have the form $b=\left(b_1, \dots,b_N\right)$ and each $b_i$ is either a map $b_i \colon T \to S^{2C}\times I$ or the symbol $\ast$. For $q\in \operatorname{QI}$, $g \in G$ and $i \in I$, let us denote $q(g,i)$ the $i$-th coordinate of $q(g) \in B$. 

    The intuition is the following: if $q \in \operatorname{QI}$, then for every $g \in G$ the tuple \[q(g) = (q(g,1),\dots,q(g,N)) \in B\] will encode the set $f^{-1}(g)$ for some quasi-isometry $f$. As every $g \in G$ can have between $0$ and $N$ preimages, each $q(g,i)$ will either encode one of them (and be a map $q(g,i) \colon T \to S^{2C}\times I$) or encode no preimage ($q(g,i) = \ast$). 

    The maps encode only local information about $h \in f^{-1}(g)$. Namely, for every $t \in T$, $q(g,i)(t)$ encodes $f(h)^{-1}f(ht)$ (which by the upper bound on the quasi-isometry, can be represented by a word in $S$ of length at most $2C$) and the index $j \in I$ such that $q(f(ht),j)$ is encoding $ht$.

    The forbidden patterns that define $\operatorname{QI}$ can be described informally with rules as follows: 
    \begin{enumerate}
        \item  The relative dense condition in $\operatorname{QI}$ is encoded by imposing that in every ball of length $C$ in $G$ there is at least some element which has a non-$\ast$ coordinate.
        \item The directions from non-$\ast$ elements must lead to non-$\ast$ elements. More precisely, if $q(g,i)\neq \ast$, then for every $t \in T$ if $q(g,i)(t)=(w,j)$, then $q(g\underline{w},j)\neq \ast$.  
        \item The lower bound in the quasi-isometry is encoded by asking that for any two elements of $g$ with a non-$\ast$ coordinate at distance at most $2C$, there is a path of non-$\ast$ elements of bounded length linking them.
        \item The relations of $H$ are encoded by enforcing that paths associated to words $T^*$ which represent the identity are cycles. This last rule is the only one that requires the groups to be finitely presented in order to obtain an SFT.
    \end{enumerate} 

    Recall that for a word $w \in S^*$ we denote by $\underline{w}$ its corresponding element of $G$. Let $q \in \operatorname{QI}$, by the first rule of $\operatorname{QI}$, there exists $u \in S^*$ with $|u|\leq C$ and $i \in I$ such that $q(\underline{u},i)\neq \ast$. Let $(u,i)$ be lexicographically minimal among all such pairs. Consider the map $\kappa \colon T^* \to S^* \times I$ defined inductively follows. Fix $\kappa(\epsilon) = (u,i)$. Suppose for some $w\in T^*$ we have $\kappa(w) = (u',i')$ and $q(\underline{u'},i')\neq \ast$. For every $t \in T$, we compute $q(\underline{u'},i')(t) = (v,j)$ and set $\kappa(wt) = (u'v,j)$. Notice that by the second rule $q(\underline{u'v},j) \neq \ast$. Clearly the map $\kappa$ can be computed from a description of $q$.

    The rules for $\operatorname{QI}$ ensure the following two properties
    \begin{enumerate}
        \item The third rule of $\operatorname{QI}$ ensures that the range of $\kappa$ spans the set of all $(g,i)\in G \times I$ for which $q(g,i) \neq \ast$, namely, \[ \{(g,i)\in G \times I : q(g,i) \neq \ast \} = \{ (\underline{u},i)\in G \times I : \kappa(w) = (u,i) \mbox{ for some } w \in T^*\}.      \]
        \item The fourth rule ensures that the map $\kappa \colon T^* \to S^* \times I$ is well defined on $H$, meaning that if $w,w'$ are two words in $T^*$ which represent the same element in $H$, then $\kappa(w)$ and $\kappa(w')$ represent the same pair of $G \times I$. 
    \end{enumerate}

    Given an alphabet $A$ we can enrich the alphabet $B$ by replacing every non-$\ast$ coordinate by pairs of the form $(a,b_i)$ for some $a \in A$. This way from every configuration $\widetilde{q}$ with this enriched alphabet, we can obtain $\widehat{x}\colon T^*\to A$ by setting for $w \in T^*$ $\widehat{x}(w)$ as the $a$-th coordinate of $\widetilde{q}(\kappa(w))$. Furthermore, by the first property of $\kappa$, the map $\widehat{x}$ induces a configuration $x \in A^H$. 
    
    For a subshift $X\subset A^H$, we denote by $\operatorname{QI}[X]$ the subshift over the enriched alphabet where we forbid the occurrence of forbidden patterns of $X$ in the induced configurations. The fundamental property proven by Cohen in~\cite{cohen_large_2017} is that if the groups are finitely presented and $X$ is an $H$-SFT, then $\operatorname{QI}[X]$ is a $G$-SFT.

    \begin{lemma}\label{lem:QI_medvedev}
        Let $G,H$ finitely generated groups with $G$ recursively presented. Let $X\subset A^H$ be a subshift. We have \[m(\operatorname{QI}[X]) = m(\operatorname{QI})\vee m(X).\]
    \end{lemma}

    \begin{proof}
        Let $\widehat{q}\in \widehat{\operatorname{QI}[X]}$. It is clear that one can compute an element of $\widehat{\operatorname{QI}}$ from $\widehat{q}$ by just erasing the alphabet coordinate. As the map $\kappa$ defined above is computable from $\widehat{q}$, we may compute $x \in \widehat{X}$ from this map and $\widehat{q}$. Thus we have that $m(\operatorname{QI}[X]) \geq m(\operatorname{QI}) \vee m(X)$.

        Conversely, let $q \in \widehat{\operatorname{QI}}$ and $x \in \widehat{X}$. We can compute $\kappa$ from $q$ as before. For $u \in S^*$ and $i \in I$, if $q(u,i)= \ast$, set $\widehat{q}(u,i) = \ast$. Otherwise, we know that there must exist $v \in T^*$ such that $\kappa(v) = (u',i)$ with $\underline{u}=\underline{u'}$. Thus if we run the algorithm that enumerates the word problem on $u'u^{-1}$ on every pair $(u',i)$ which occurs in the range of $\kappa$ this algorithm finishes. Set $\widehat{q}(u,i) = (x(v),q(u,i))$. With this it is clear that $\widehat{q}\in \operatorname{QI}[X]$ and thus $m(\operatorname{QI})\vee m(X) \geq m(\operatorname{QI}[X])$.
    \end{proof}

    \begin{remark}\label{admitir-sft-incalculable-es-invariante-de-qi}
        It follows from this result that for finitely presented groups, admitting an SFT with nonzero Medvedev degree is an invariant of quasi-isometry. We also mention that in the previous proof, the direction $m(\operatorname{QI}[X]) \geq m(\operatorname{QI})\vee m(X)$ does not require the groups to be recursively presented. 
    \end{remark}

    \begin{proposition}\label{prop:QI-computableinverse}
        Suppose $G,H$ are finitely generated groups with decidable word problem. If there exists a computable quasi-isometry $f\colon H \to G$, then there exists a computable coarse inverse $g\colon G \to H$.
    \end{proposition}

       \begin{proof}
        Fix finite set of generators $S,T$ which induce metrics $d_G$ and $d_H$ for $G$ and $H$ respectively. Let $f \colon H \to G$ be a computable quasi-isometry with constant $C>0$ and $\widehat{f}\colon T^* \to S^*$ its computable representative. We construct $\widehat{g}\colon S^* \to T^*$ through the following algorithm.

        Given $w \in S^*$, run the algorithm which decides the word problem on $uw^{-1}$ for all words $u$ of length at most $|w|$ and replace $w$ by the lexicographically minimal $u$ which is accepted. In increasing lexicographical order, compute $d_G(u,\widehat{f}(v))$ for $v \in S^*$. Return $\widehat{g}(w) = u$ where $u$ is the first word for which $d(u,\widehat{f}(v)) \leq C$.

        As $f$ is relatively dense, this algorithm is guaranteed to stop on every input, and thus $\widehat{g}$ is a total computable map. Furthermore, the first step ensures that for two words $w,w'$ which represent the same element of $G$, then $\widehat{g}(w) = \widehat{g}(w')$ and thus it represents a map $g \colon G \to H$.

        Finally, for any $t\in G$, we have $d_G(f\circ g(t),t) \leq C$ and thus $d_{\infty}(f\circ g,\operatorname{Id}_G)$ is finite. This implies that $g$ is a coarse inverse of $f$ and thus a quasi-isometry. \end{proof}

    \begin{definition}
        We say that two finitely generated groups are \define{computably quasi-isometric} if there exists a computable quasi-isometry between them.
    \end{definition}

    \begin{remark}
        For groups with decidable word problem, we have shown that any computable quasi-isometry $f\colon G \to H$ admits a computable quasi-isometry inverse. With a dovetailing argument, the proof of~\Cref{prop:QI-computableinverse} can be generalized to the case where $G$ is recursively presented and $H$ has decidable word problem. We do not know if the result is still true if both groups are merely recursively presented.
    \end{remark}

    By~\Cref{prop:QI-computableinverse}, it follows that being computably quasi-isometric is an equivalence relation for groups with decidable word problem. Putting~\Cref{lem:QI_medvedev} with the result of Cohen stating that for finitely presented groups and $X$ an SFT then $\operatorname{QI}[X]$ is an SFT, we obtain the following:
    
    \begin{corollary}\label{coro_QI}
        Let $G,H$ be finitely presented and computably quasi-isometric groups with decidable word problem. Then $\msft{G}=\msft{H}$.
    \end{corollary}

    \begin{proof}
        As the groups have decidable word problem there are computable quasi-isometries on both directions and thus it suffices to show that $\msft{H}\subset \msft{G}$.

        Let $X$ be an $H$-subshift. From~\Cref{lem:QI_medvedev} we obtain that $m(\operatorname{QI}[X]) = m(\operatorname{QI})\vee m(X)$. Furthermore, as there is a computable quasi-isometry, one can use it to construct a computable element of $\operatorname{QI}$ and thus it follows that $m(\operatorname{QI}) = 0_{\M}$. We obtain that for any subshift $X\subset A^H$ we have $m(\operatorname{QI}[X]) = m(X)$. Finally, as $G,H$ are finitely presented, it follows that for any SFT $X\subset A^H$ then $\operatorname{QI}[X]$ is a $G$-SFT with $m(\operatorname{QI}[X]) = m(X)$, therefore $\msft{H}\subset \msft{G}$. 
    \end{proof}

    For the next lemma we will consider quasi-isometries between groups and computable metric spaces that are computable in the sense of \Cref{chap:Background}.

    \begin{observation}
        Contrary to what happens between discrete groups, for computable metric spaces $X,Y$ there may exist a computable quasi-isometry from $X$ to $Y$, but no computable quasi-isometry from $Y$ to $X$. For instance, the identity from $\ZZ$ to $\RR$ is a computable quasi-isometry, but there exist no computable quasi-isometry from $\RR$ to $\ZZ$, because computable maps are continuous and thus constant.
    \end{observation}
    
    \begin{lemma}
           Let $(X,d,(x_i)_{i \in \NN})$ be a computable metric space and let $G,H$ be two finitely generated groups with decidable word problem. If there are computable quasi-isometries $f \colon G \to X$ and $h \colon H \to X$, then $G$ and $H$ are computably quasi-isometric.
    \end{lemma}

    \begin{proof}
        Fix finite set of generators $S,T$ for $G$ and $H$ respectively. Let $\widehat{f} \colon S^* \to X$ and $\widehat{h}\colon T^* \to X$ be computable maps that represent the quasi-isometries $f$ and $g$. Let $C\in \NN$ be a quasi-isometry constant for both $f$ and $g$. We construct $\widehat{\psi}\colon S^* \to T^*$ through the following algorithm.

        Let $w \in S^*$. Using the algorithm which decides the word problem of $G$, replace $w$ for a lexicographically minimal word $u$ that represents the same element of $G$. Through a dovetailing procedure, compute approximations of the recursively open sets $\widehat{f}^{-1}(B(x_i,C))$. Let $i_{w}$ be the first index for which this procedure returns that $u \in \widehat{f}^{-1}(B(x_{i_w},C))$. Next, let $v\in T^*$ be the first word which appears in the effectively open enumeration of $\widehat{h}^{-1}(B(x_{i_w},C))$ and set $\widehat{\psi}(w) = v$.

        From this description it is clear that $\widehat{\psi}$ is total computable and represents a map $\psi\colon G \to H$. Furthermore, it satisfies the property that for any $g \in G$ \[d(f(g),h(\psi(g))\leq d(f(g),x_{i_w}) + d(x_{i_w},h(\psi(g))) \leq 2C. \] 
        Where in the above inequality $w$ is any word in $S^*$ which represents $G$. It follows that if $\overline{h}$ is a coarse inverse for $h$, then \[ d_{\infty}( \overline{h}\circ f , \psi) \mbox{ is finite}.  \]
        And therefore $\psi$ is a quasi-isometry from $G$ to $H$.
    \end{proof}

    \begin{observation}\label{obs:svarcmilnor}A well-known source of quasi-isometries between groups and metric spaces is the \v{S}varc–Milnor Lemma. Let us note that this result can also be used to obtain computable quasi-isometries. We briefly recall the statement, and refer the reader to \cite[Theorem 23]{laharpe_topics_2000} for details. Let $X$ be a metric space which is geodesic and proper, and let $G\curvearrowright X$ be a group action by isometries, such that the action is proper and the quotient $X/G$ is compact. Then $G$ is a finitely generated group, and quasi-isometric to $X$. Indeed, for all $x\in X$ the orbit map $f\colon G \to X$ given by $f(g) = g\cdot x$ is a quasi-isometry.

    Now suppose $G\curvearrowright X$ satisfies the hypothesis in the \v{S}varc–Milnor Lemma, that $(X,d,\mathcal{S})$ is a computable metric space, and that for all $g$ the map which sends $x\in X$ to $g\cdot x$ is computable. Then the orbit map of a computable point $x\in X$ provides a computable quasi-isometry from $G$ to $X$.
    \end{observation}

\section{Medvedev degrees of classes of subshifts on groups}\label{section:classification}

In this section we will employ the machinery developed in this chapter to say as much as we can about the Medvedev degrees of subshifts in groups. 

\subsection{Known results}
        Here we review the state of the art of the classification of $\msft{G}$, $\msof{G}$ and $\meff{G}$ for different groups $G$. We begin by a rephrasing of Simpson's result in our language.
        \begin{theorem}[\cite{simpson_medvedev_2014}]\label{simpson_result}
            Let $d\geq 2$. Then $\msft{\ZZ^d}=\msof{\ZZ^d}=\meff{\ZZ^d}$ equals the class of all $\Pi_1^0$ Medvedev degrees.
        \end{theorem}

        See \cite{durand_fixedpoint_2012} for a different proof. On the other hand, it is well known that every SFT and sofic subshift on $\ZZ$ have finite orbits, and thus zero Medvedev degree. While it is not true that every SFT in the free group admits a finite orbit (see~\cite{piantadosi_symbolic_2008}), it still contains computable configurations.
        \begin{proposition}\label{prop:virtually_free_have_0_degree}
            For every finitely generated virtually free group $G$ we have $\msft{G}=\msof{G}= \{0_{\M}\}$.
        \end{proposition}
        \begin{proof}
            By~\Cref{prop:commensurable_same_degrees} it suffices to consider a finitely generated free group. Let $F(S)$ be freely generated by a finite symmetric set $S$, and let $X\subset A^{F(S)}$ be a nonempty SFT. As $m(X)$ is a conjugacy invariant (\Cref{prop:basic-properties}), without loss of generality we can assume that $X$ is determined by patterns of the form $p\colon \{1,s\}\to A$, for $s\in S$ (see for instance~\cite[Proposition 1.6]{barbierilemp_thesis}). Let $A'\subset A$ be the set of all symbols in $A$ that occur in some configuration in $X$. Observe that for every $a\in A'$ and $s\in S$, there exist $b\in A'$ such that $1\mapsto a$, $s\mapsto b$ is not a forbidden pattern.
            
            We define a computable configuration $x\colon F(S)\to A'$ inductively. Fix a total order on $A'$, and let $x(1_{F(S)})$ be the minimal element in $A'$. Now, let $n >1$ and assume that we have defined $x(h)$ for all $h$ with $|h|_S < n$, and that no forbidden pattern occurs there. For $g$ with $|g|_S=n$, there is a unique $h$ with $|h|_S=n-1$ and $s \in S$ with $hs=g$. We define $x(g)$ as the minimal element in $A'$ such that the pattern $1_G\mapsto x(h),s\mapsto x(g)$ is not forbidden. Clearly no forbidden patterns can be created by this procedure because the Cayley graph of $F(S)$ with respect to $S$ is a tree, and thus we obtain that $x \in X$. This procedure is computable, and thus $m(X)=0_{\M}$. By~\Cref{prop:basic-properties} every topological factor of $X$ has at most degree $m(X) = 0_{\M}$ and the claim for sofic subshifts follows.            
        \end{proof}
        
        Let us now consider effective subshifts. Miller proved that $\meff{\ZZ}$ equals the class of all $\Pi_1^0$ Medvedev degrees \cite{miller_two_2012}. This was then generalized to all groups with decidable word problem in \Cref{meff-for-groups-with-decidable-word-problem}.

\subsection{Virtually polycyclic groups}

\begin{theorem}\label{thm:polycyclic}
    Let $G$ be a virtually polycyclic group. We have the following dichotomy
    \begin{enumerate}
        \item If $G$ is virtually cyclic, then $\msft{G} = \{0_\M\}$.
        \item If $G$ is not virtually cyclic, then $\msft{G}$ is the set of all $\Pi_1^0$ Medvedev degrees.
    \end{enumerate}
\end{theorem}

\begin{proof}
    Let $G$ be virtually cyclic. If $G$ is finite it is clear that $\msft{G}=0_{\M}$. If $G$ is infinite, then again by~\Cref{prop:commensurable_same_degrees} it follows that $\msft{G}=\msft{\ZZ} = \{0_\M\}$.

    Finally, suppose $G$ is a polycyclic group which is not virtually cyclic. Then it follows that $\ZZ^2 \leqslant G$ (see for instance~\cite{Segal_1983}). It is well known that polycyclic groups are recursively presented and have decidable membership problem (see for instance~\cite[Corollary 3.6]{Baumslag1981}), therefore by~\Cref{cor:subgroups_decidable_membership} we obtain that $\msft{\ZZ^2}\subset \msft{G}$. By Simpson's result (\Cref{simpson_result}) $\msft{\ZZ^2}$ is the set of all $\Pi_1^0$ Medvedev degrees, so $\msft{G}$ contains all $\Pi_1^0$ degrees. Conversely, every SFT on $G$ has a $\Pi_1^0$ Medvedev degree because the word problem of $G$ is decidable (\Cref{prop:msft-and-msoff-are-contained-in-Pi-1-degrees}). \end{proof}

\subsection{Direct products of groups}

We will show that in any direct product of two infinite and finitely generated groups with decidable word problem, every $\Pi_0^1$ Medvedev degree arises as the degree of an SFT.

\begin{theorem}\label{thm:direct_products_clasification_medvedev_degrees}
        Let $H,K$ be two infinite and finitely generated groups with decidable word problem. Then $\msft{H\times K}$ is the set of all $\Pi_1^0$ Medvedev degrees.
\end{theorem}

\begin{proof}
    It follows from ~\Cref{seward-calculable}  that $H\times K$ admits a computable translation-like action by $\ZZ^2$ which has decidable orbit membership problem (in fact, it is transitive, so the orbit membership problem is trivial). By \Cref{prop:translation-like-actions-and-medvedev-degrees}, we have that $\msft{\ZZ^2}$ is contained in $\msft{H\times K}$, and thus all $\Pi_1^0$ Medvedev degrees can be attained by an SFT on $H\times K$. On the other hand, the Medvedev degree of every SFT on $H\times K$ is a $\Pi_1^0$ degree because $H\times K$ has decidable word problem (\Cref{prop:msft-and-msoff-are-contained-in-Pi-1-degrees}).       
\end{proof}

    A class of groups where~\Cref{thm:direct_products_clasification_medvedev_degrees} can be applied meaningfully are branch groups. There is more than one definition in the literature of a branch group, see for instance~\cite{BartholdiGrigorchukSuni_branchgroups_2023}. However, both definitions have the following consequence: every branch group is commensurable to the a direct product of two infinite groups. In particular, if $G$ is an infinite and finitely generated branch group with decidable word problem, then it is commensurable to the direct product of two infinite and finitely generated groups with decidable word problem. Putting together~\Cref{thm:direct_products_clasification_medvedev_degrees} and~\Cref{prop:commensurable_same_degrees} we obtain the following result.

\begin{corollary}\label{cor:branch_groups}
    For every infinite and finitely generated branch group $G$ with decidable word problem $\msft{G}$ is the set of all $\Pi_1^0$ Medvedev degrees. 
\end{corollary}

Without the assumption of decidable word problem, we can only show the existence of SFTs with nonzero Medvedev degree.

\begin{theorem}\label{thm:direct_products_nontrivial_medvedev_degrees}
        For any two infinite and finitely generated groups $H,K$ the set $\msft{H\times K}$ contains a nonzero Medvedev degree. Moreover, if $H,K$ are recursively presented, then $\msft{H\times K}$ contains the maximal $\Pi_1^0$ Medvedev degree.
\end{theorem}

\begin{proof}
    By Seward's result on translation-like actions~\cite[Theorem 1.4]{seward_burnside_2014} we have that $\ZZ^2$ acts translation-like on $H \times K$. As $\ZZ^2$ is finitely presented and $\msft{\ZZ^2}$ is the set of all $\Pi_1^0$ Medvedev degrees,~\Cref{prop:translation-like-actions-and-medvedev-degrees} yields that there is a $(H \times K)$-SFT whose degree is at least the maximal $\Pi_1^0$ Medvedev degree, in particular it's nonzero. If $H,K$ are recursively presented, then so is $H\times K$ and the upper bound also holds. 
\end{proof}

\begin{corollary}
    Every infinite and finitely generated branch group $G$ admits a nonempty SFT with nonzero Medvedev degree. If $G$ is recursively presented, then $\msft{G}$ contains the maximal $\Pi_1^0$ Medvedev degree.
\end{corollary}
  
\subsection{Groups which are quasi-isometric to the hyperbolic plane}\label{subsec:QI_hyperbolic}

In this section, we consider the hyperbolic plane $\mathbb{H} = \{(x,y) \in \mathbb{R}^2 : y >0 \}$ as a computable metric space given by its Riemannian metric and with dense countable set given by a computable enumeration of $\mathbb{Q}^2 \cap \mathbb{H}$.

\begin{theorem}\label{thm:hyperbolic_qi_medvedev}
    Let $G$ be a finitely generated group that is quasi-isometric to $\mathbb{H}$. Then there exists a $G$-SFT $X$ with $m(X)>0_{\M}$. 
\end{theorem}

\begin{proof}
    Any group which is quasi-isometric to a hyperbolic metric space is a word-hyperbolic group and thus is finitely presented and has decidable word problem (see for instance~\cite[Theorem 2.6]{bridson_metric_1999}). In particular, all groups considered in the statement are finitely presented. As we already know that admitting an SFT with nonzero Medvedev degree is a quasi-isometry invariant for finitely presented groups (\Cref{admitir-sft-incalculable-es-invariante-de-qi}), it suffices to exhibit one group $G$ which is quasi-isometric to $\mathbb{H}$ and which admits a $G$-SFT $X$ with $m(X)>0_{\M}$.
    
    For this purpose, we consider the fundamental group $\pi_1(\Sigma_2)$ of a closed orientable surface of genus two. This group is finitely presented and quasi-isometric to the hyperbolic plane. In~\cite{domino-surface-groups}, the authors prove the undecidability of the domino problem for $\pi_1(\Sigma_2)$. The same construction, plus some results that are present in the literature, can be used to exhibit an SFT $X$ with $m(X)>0_{\M}$. 

    Let us proceed more precisely. In~\cite{jeandel_immortal_2012} the author shows that there exists a Turing machine which halts on every computable starting configuration, but loops for some $\Pi_1^0$ set of non-computable starting configurations. By a construction of Hooper~\cite{hooper_undecidability_1966}, this machine can be encoded by a piecewise affine map with rational coeﬃcients whose immortal points correspond to non-halting configurations. Using a beautiful construction of Kari~\cite[Section 2]{kari_theory_2005}, the set of immortal points of this piecewise affine map can be encoded as a set of colorings of a graph which represents a binary tiling of $\mathbb{H}$ given by local rules. Finally, in~\cite[Section 5]{domino-surface-groups} the authors show that every such set of colorings can be encoded in a $\pi_1(\Sigma_2)$-SFT $X$. It follows that $X$ has no computable points, and thus $m(X)>0_{\M}$. 
\end{proof}

A natural class of examples is given by some Fuchsian groups. A group $G$ is \define{Fuchsian} if it is a discrete subgroup of $\operatorname{PSL(2,\ZZ)}$, a good introductory reference is~\cite{Katok1992-nk}. Fuchsian groups act properly discontinuously on $\mathbb{H}$ by isometries (M\"obius transformations). It follows by the Milnor-\v{S}varc lemma that if a Fuchsian group is finitely generated and acts co-compactly, it is quasi-isometric to $\mathbb{H}$.

\begin{corollary}
    Every finitely generated co-compact Fuchsian group has an SFT $X$ with $m(X)>0_{\M}$.
\end{corollary}

Natural examples of finitely generated co-compact Fuchsian groups are given by the hyperbolic triangle groups $\Delta(l,m,n)$ with $l,m,n \in \NN$ and $\frac{1}{l}+\frac{1}{n}+\frac{1}{m}<1$.

Now let us turn our attention to the classification of Medvedev degrees for groups which are quasi-isometric to $\mathbb{H}$. Recall that by~\Cref{coro_QI} all groups which admit a computable quasi-isometry to $\mathbb{H}$ form a computably quasi-isometry class. It is not hard to show that $\pi_1(\Sigma_2)$ belongs to this class, thus we obtain the following result.

\begin{corollary}
    Let $\pi_1(\Sigma_2)$ be the fundamental group of the closed orientable surface of genus two. If $G$ is a finitely generated group which admits a computable quasi-isometry $f\colon G \to \mathbb{H}$, then $\msft{G}=\msft{\pi_1(\Sigma_2)}$.
\end{corollary}

We remark that while we know that $\msft{\pi_1(\Sigma_2)}$ is non-trivial, we do now know what is the precise set of $\Pi_1^0$ degrees that can be attained. A difficulty here is that while it is known that there are Turing machines whose immortal set has nonzero Medvedev degree, it is not known whether all $\Pi_1^0$ Medvedev degrees can be obtained in this manner. On the other hand, it seems reasonable that a hierarchical construction such as the one of~\cite{goodman-strauss_hierarchical_2010} can be coupled with Simpson's construction to produce all $\Pi_1^0$ Medvedev degrees.

\begin{remark}
    We do not know if every finitely generated and co-compact Fuchsian group admits a computable quasi-isometry to $\mathbb{H}$. By~\Cref{obs:svarcmilnor}, a sufficient condition for a Fuchsian group to admit a computable quasi-isometry to $\mathbb{H}$ is that their generators can be represented by matrices in $\operatorname{PSL}_2(\RR)$ with computable coefficients. We do not know if this holds in general, but it certainly holds for well-known examples, such as the co-compact hyperbolic triangle groups.
\end{remark}

\subsection{Groups obtained through simulation results}\label{subsec:simulation}

Let $\rho \colon G \to H$ be an epimorphism. Recall that we denote the pullback of a subshift $X\subset A^{H}$, by $\rho^{*}(X)$. A group $G$ \define{simulates} $H$ (through $\rho$) if the pullback of any effective $H$-subshift is a sofic $G$-subshift\footnote{In other references, for instance~\cite{barbieri_effective_2024,barbieri_groups_2022}, the notion of simulation is more restrictive and asks that the pullback of every computable action of $H$ on any $\Pi_1^0$ set is the topological factor of some $G$-SFT.}. 
\begin{proposition}\label{prop:simulation}
    Let $G,H$ be finitely generated groups such that $G$ simulates $H$. The following hold: 
    \begin{enumerate}
        \item $\msof{G} \supset\meff{H}$.
        \item If $H$ is infinite and recursively presented, then $\msft{G}$ contains a Medvedev degree which is bounded below by the maximal $\Pi_1^0$ Medvedev degree.
        \item If $H$ is infinite and has decidable word problem, $\msof{G}$ contains the set of all $\Pi_1^0$ Medvedev degrees.
    \end{enumerate}
   \end{proposition} 
    \begin{proof}
        Let $X\subset A^{H}$ be any effective subshift.  As $G$ simulates $H$, we have that $\rho^*(X)$ is sofic. By~\Cref{lema:medvedev_degree_of_pullback} we have that $m(X) = m(\rho^*(X))$ and thus $\msof{G} \supset\meff{H}$.

        Now, let us suppose that $H$ is infinite and recursively presented. By~\Cref{cor:all_infinite_rp_groups_havenonzero_effective_medvedev_degree} we have that $\meff{H}$ contains the maximal $\Pi_1^0$ Medvedev degree, and thus by the first item, so does $\msof{G}$. If we let $Y$ be a sofic $G$-shift with this degree, any $G$-SFT extension $X$ will satisfy $m(X)\geq m(Y)$ by~\Cref{prop:basic-properties}.

        Finally, suppose that $H$ is infinite and with decidable word problem. By~\Cref{teo:nicarrasco_eff_pi01}, we have that $\meff{H}$ is the set of all $\Pi_1^0$ Medvedev degrees, and thus as $\msof{G} \supset\meff{H}$ we get that $\msof{G}$ also contains the set of all $\Pi_1^0$ Medvedev degrees.\end{proof}

\section{Further remarks and questions}\label{section:questions}

\subsection{Conjectures on the Medvedev degrees of classes of subshifts}

For every finitely generated and recursively presented group $G$, we have that $\meff{G}$ is contained in the set of $\Pi_1^0$ degrees (\Cref{prop:msft-and-msoff-are-contained-in-Pi-1-degrees}).  Furthermore, \Cref{cor:all_infinite_rp_groups_havenonzero_effective_medvedev_degree} shows that it always contains the maximal $\Pi_1^0$ degree. Thus the following question arises naturally:
    \begin{question}\label{question:effective_subshifts_degrees_on_rp}
        Let $G$ be a finitely generated group that is infinite and recursively presented. Is it true that effective subshifts on $G$ attain all $\Pi_1^0$ Medvedev degrees?
    \end{question}

In fact, with our current understanding, in every recursively presented group where $\msft{G}$ is known, it is either trivial or it coincides with the class of all $\Pi_1^0$ Medvedev degrees. This leads us to propose the following conjecture.

    \begin{conjecture}\label{conjecture_allSFTSarePi10completeclasses}
    Let $G$ be an infinite, finitely generated and recursively presented group. $G$ is not virtually free if and only if $\msft{G}$ is the set of all $\Pi_1^0$ Medvedev degrees.
\end{conjecture}
    
We remark that in the case of a recursively presented group, \Cref{prop:simulation} provides us an SFTs which achieves the maximal $\Pi_1^0$ Medvedev degree, but a priori it does not give us a tool to classify $\msft{G}$. Indeed, we can only conclude that for every sofic subshift $Y$ there exists a subshift of finite type with Medvedev degree \textit{at least} $m(Y)$. Although it seems unlikely, a priori it might be the case that the lattice $\msft{G}$ may contain gaps which do not occur in $\msof{G}$.

\begin{question}\label{question:equal_degree_extension}
     Let $G$ be a recursively presented group and $Y$ a sofic $G$-subshift. Does $Y$ admit a $G$-SFT extension of equal Medvedev degree? In particular, is it true that $\msft{G}=\msof{G}$?
\end{question}

A positive answer to~\Cref{question:equal_degree_extension} would provide strong evidence towards~\Cref{conjecture_allSFTSarePi10completeclasses}.

\subsection{Immortal sets of Turing machines and hyperbolic groups}

In~\Cref{subsec:QI_hyperbolic} we mentioned a construction of Jeandel that produces a Turing machine which halts on every computable starting configuration, but loops in a $\Pi_1^0$ set of starting configurations with nonzero Medvedev degree. It is currently unknown if every $\Pi_1^0$ Medvedev degree can be obtained this way.

\begin{question}\label{question:immortality}
    Is it true that immortal sets of Turing machines attain all $\Pi_1^0$ Medvedev degrees?
\end{question}

If such a result were proven, it would imply by the same argument sketched in~\Cref{thm:hyperbolic_qi_medvedev} that the surface group $\pi_1(\Sigma_2)$ admits SFTs which attain all Medvedev degrees. In fact, it may be possible to adapt a construction of Bartholdi~\cite{bartholdi_2023_hyperbolic} to extend said result to all non virtually free hyperbolic groups. This would provide a way to settle~\Cref{conjecture_allSFTSarePi10completeclasses} for hyperbolic groups.





\subsection{Medvedev degrees, aperiodic subshifts, and the domino problem}
We now observe that for a fixed finitely generated group, the existence of SFTs with nonzero Medvedev degree implies both the existence of weakly aperiodic SFTs, and the undecidability of the domino problem.

\begin{proposition}\label{prop:med_implies_WA_conjecture}
        Let $G$ be a finitely generated group, and let $X$ be a nonempty $G$-subshift. If $m(X)>0_{\M}$, then $X$ is weakly aperiodic, that is, $G\curvearrowright X$ has no finite orbits. 
    \end{proposition}
    \begin{proof}
        We prove that if $X$ has a configuration with finite orbit, then $m(X)=0_{\M}$. Indeed, if $X$ has a configuration with finite orbit, then the same is true for the pullback $\widehat X$. Thus it suffices to prove that a configuration in $A^{F(S)}$ with finite orbit must be computable. 
        
        Indeed, if $x_0\in A^{F(S)}$ has finite orbit, then there is a finite group $K$, a group epimorphism $\phi\colon F(S)\to K$, and a configuration $y_0\in A^K$ such that the pullback $\phi^\ast (y_0)$ is equal to $x_0$ (where $\phi^\ast (y)(g)=y(\phi(g)),\ g\in F(S)$). As $A^K$ is finite it follows that $y_0$ is computable and thus that $x_0$ is computable: given $g\in F(S)$ we can compute $x_0(g)$ by first computing $\phi(g)\in K$ (recall that every homomorphism of finitely generated groups is computable), and setting $x_0(g)=y_0(\phi(g))$.\end{proof}

The domino problem of a group $G$ is the algorithmic problem of computing, from a finite set of pattern codings, if the associated $G$-SFT is nonempty. 
        
\begin{proposition}\label{prop:med_implies_domino}
    Let $G$ be a finitely generated group which admits a nonempty $G$-SFT $X$ with $m(X)>0_{\M}$. Then the domino problem for $G$ is undecidable. 
\end{proposition}
\begin{proof}
    It is well known that if $G$ has undecidable word problem, then $G$ has undecidable domino problem~\cite[Theorem 9.3.28]{aubrun_domino_2018}. Therefore we will assume that $G$ has decidable word problem and perform computations directly on patterns instead of pattern codings. We will prove that if the domino problem for $G$ is decidable, and $X$ is a nonempty $G$-SFT on alphabet $A$, then $m(X)=0_{\M}$.

    Fix a word metric on $G$ and let $B_n=\{g\in G : |g|\leq n\}$. Given $L \subset A^{B_n}$, we denote by $Y(L)$ the $G$-SFT given by \[ Y(L) = \{ y \in A^G : (gy)|_{B_n}\in L \mbox{ for every } g \in G\}.\] 

    In other words, $Y(L)$ is the $G$-SFT defined by the forbidden patterns $\mathcal{F} = A^{B_n}\setminus L$. Fix $n_0 \in \NN$ and consider a sequence $(L_n)_{n\geq n_0}$ with the following properties: $L_n\subset A^{B_n}$ is nonempty, every pattern in $L_n$ appears in $X$, and finally, for every pattern $p\in L_n$, there exists a pattern $q\in L_{n+1}$ whose restriction to $B_n$ is $p$. We remark that if a computable sequence $(L_n)_{n \geq n_0}$ as above exists, then we can extract a computable configuration $x$ by letting $x|_{B_{n_0}} = p_{n_0} \in L_{n_0}$ arbitrary, and then inductively for $n>n_0$ we choose $p_n$ as the lexicographically minimal element such that $p_n|_{B_{n-1}} = x|_{B_{n-1}}$ and set $x|_{B_n} = p_n$. It is clear that $x$ is computable and $x \in \bigcap_{n \geq n_0} Y(L_n) \subset X$. It follows that if such a computable sequence exists, then $m(X)=0_{\M}$. 

    As $X$ is a $G$-SFT, there exists a smallest $n_0\in \NN$ for which there is $D\subset A^{B_{n_0}}$ with $X = Y(D)$. Let $L_{n_0}\subset D$ be minimal (for inclusion) such that $Y(L_{n_0})$ is nonempty. Notice that $Y(L_{n_0})\subset X$ and that every pattern in $L_{n_0}$ must appear in $Y(L_{n_0})$.
    
    Now we define the sequence. Set $L_{n_0}$ as above. For $n > n_0$ we perform the following recursive step. Suppose that $L_{n-1}$ has already been computed and that it is minimal such that $Y(L_{n-1})\subset X$ is nonempty. Compute the set $D_n$ of all $p \in A^{B_n}$ whose restriction to $B_{n-1}$ lies in $L_{n-1}$. Then clearly $Y(D_n) = Y(L_{n-1})\subset X$. Using the algorithm for the domino problem of $G$, compute $L_n\subset D_n$ which is minimal such that $Y(L_n)$ is nonempty. 

    This sequence $(L_n)_{n \geq n_0}$ is computable. We claim it satisfies the required properties. By construction each $n \geq n_0$, $L_n\subset A^{B_n}$ is nonempty. Minimality implies that all patterns in $L_{n}$ appear in $Y(L_{n})\subset X$. Finally, for every $p\in L_n$, there is $q\in L_{n+1}$ whose restriction to $B_{n}$ equals $p$. Indeed, if this was not the case, then the set $R_{n}$ defined by restricting all patterns in $L_{n+1}$ to $B_n$ would be properly contained in $L_n$, and $\emptyset \neq Y(L_{n+1})\subset Y(R_n)$, contradicting the minimality of $L_{n}$. \end{proof}

These results lead us to conjecture the following:
\begin{conjecture}\label{conj:uncomputableconfig}
    Let $G$ be a finitely generated group. $G$ is virtually free if and only if every nonempty $G$-SFT has Medvedev degree $0_{\M}$.
\end{conjecture}

\begin{observation}\label{Medvedev-conjecture-is-pulent}
    Note that \Cref{conj:uncomputableconfig} implies Carroll and Penland's conjecture \cite{carroll_periodic_2015} that all infinite finitely generated groups that are not virtually cyclic admit weakly aperiodic SFTs. This simply follows from \Cref{prop:med_implies_WA_conjecture}, plus the fact that all virtually free groups that are not virtually cyclic admit weakly aperiodic SFTs \cite{piantadosi_symbolic_2008}.

    We also note that \Cref{conj:uncomputableconfig} implies Ballier and Stein's conjecture that a finitely generated infinite group has decidable domino problem if and only if it is virtually free~\cite{ballier_domino_2018}. Indeed, this follows from~\Cref{prop:med_implies_domino} and the fact that the conjecture holds for virtually free groups.
\end{observation}

\subsection{Beyond $\Pi_1^0$ Medvedev degrees}
Most of our results apply to recursively presented groups, where $\msft{G}$ is contained in the set of $\Pi_1^0$ degrees. It is natural to ask what happens beyond recursively presented groups. In this case we can show that for a group whose whose word problem is too complex, then every strongly aperiodic $G$-subshift (one such that the shift action is free) has a Medvedev degree beyond $\Pi_1^0$ degrees:
\begin{proposition}\label{prop:beyondtimeandspace}
Let $G$ be a group whose word problem is not co-recursively enumerable. For every nonempty strongly aperiodic $G$-subshift $X$ we have that $m(X)$ is not a $\Pi_1^0$ Medvedev degree.
\end{proposition}
\begin{proof}
    Suppose that $G$ admits a strongly aperiodic subshift $X$ such that $m(X)$ is a $\Pi_1^0$ degree. We will prove that then $G$ has a co-recursively enumerable word problem. 

    Let $S$ be a symmetric set of generators for $G$. We start by proving that the pullback subshift $\widehat X\subset A^{F(S)}$ contains a nonempty subset that is $\Pi_1^0$. Indeed, as $m(X)$ is a $\Pi_1^0$ degree, there is a $\Pi_1^0$ set $P\subset \{0,1\}^\NN$ and a computable function $f\colon P\to \widehat X$. As the spaces $A^\NN$ and $A^{F(S)}$ are recursively compact, the computable image of a $\Pi_1^0$ set must be a $\Pi_1^0$ set (see \Cref{prop:effective-continuous-image-of-compact-is-compact}, \Cref{recursively-compact-is-closed}, and \Cref{recursively-compact-is-closed}). It follows that $Y=f(P)$ is a $\Pi_1^0$ subset of $\widehat{X}$. 
    
    For $w \in F(S)$, consider the stabilizer $\operatorname{Fix}(w) = \{ x \in A^{F(S)} : wx = x\}$, and observe that it is a $\Pi_1^0$ set. As $X$ is strongly aperiodic, a word $w\in S^\ast$ satisfies $\underline{w}\ne 1_G$ if and only if $\operatorname{Fix}(w)\cap Y$ is empty. As both $Y$ and $\operatorname{Fix}(w)$ are $\Pi_1^0$ sets, it follows that $\operatorname{Fix}(w)\cap Y$ is $\Pi_1^0$. Finally, as $A^{F(S)}$ is recursively compact, the collection of descriptions of $\Pi_1^0$ sets which are empty is recursively enumerable (see~\cite[Remark 3.13]{barbieri_carrasco_rojas_2024_effective}) and thus this gives an algorithm to enumerate the $w \in S^{*}$ which do not represent $1_{G}$.
\end{proof}

We remark that in the case where $G$ is a finitely generated and recursively presented group which admits a strongly aperiodic effective $G$-subshift, then~\Cref{prop:beyondtimeandspace} implies that $G$ has decidable word problem. This recovers a result by Jeandel~\cite[Corollary 2.7]{jeandel_notes_subshift_groups}.

We also remark that using simulation theorems it is possible to construct groups with non co-recursively enumerable word problem and which admit strongly aperiodic SFTs. This in particular shows that there exist finitely generated groups $G$ for which $\msft{G}$ is not contained in the $\Pi_1^0$ degrees. This result will appear in an upcoming chapter of Barbieri and Salo.

%% file: Chapter_computable_analysis_on_the_space_of_subshifts.tex
The goal of this chapter is presenting some results obtained during a research stay with Mathieu Sablik, and in collaboration with Alonso H.Núñez.  These results concern the metric space $S(G)$ of all $G$-subshifts, where $G$ is a finitely generated group.

\minitoc
\section{Introduction}
How does a typical dynamical system look like? This is a classical question in topological and measurable dynamics. It can be made precise by considering a topological space whose points represent dynamical systems. In this setting,   \textit{typical} means topologically dense or generic (a dense $G_\delta$ set). In this chapter we will be interested in the space $S(G)$ of $G$-subshifts:
\begin{defn}\label{def:S(G)}
    Let $G$ be a finitely generated group with a word length $|\cdot |$. We let $S(G)$ be the collection of all $G$-subshifts whose alphabet is a finite subset of $\N$. We define a metric $D$ on $S(G)$ by \[D(X,Y)=\inf\{2^{-n}:n\in\N, \ L_n(X)=L_n(Y)\},\] where $L_n(X)$ denotes set of patterns that appear in $X$ and have support $\{g\in G : |g|\leq n\}$. 
\end{defn}
This space has been considered in several works. For instance, Frisch and Tamuz \cite{frisch_symbolic_2017} proved that for $G$ amenable, a generic subshift in $S(G)$ has zero entropy. More recently, Pavlov and Schmieding \cite{pavlov_structure_2023} gave a detailed description of the space  $S(\Z)$ and some subspaces of interest. Remarkably, they showed that isolated points constitute a dense set in $S(\Z)$. They further provided a dynamical characterization of these isolated points, and used this to prove the genericity of different dynamical properties in $S(\Z)$.

Pavlov and Schmieding mention that it is not clear whether their techniques extend to other groups beyond $\Z$ \cite[Section 1]{pavlov_structure_2023}. Here we show that the density of isolated points do not extend to $\Z^d$, $d\geq 2$, and more generally, any group with decidable word problem and which admits SFTs with nonzero Medvedev degree.
\subsection*{Results}
 Assuming that $G$ has decidable word problem, we will show that the metric space $S(G)$ can be described as an effectively closed \textit{subset} of a computable metric space. This elemental fact uncovers some interactions between the topology of $S(G)$ and computability features of $G$-SFTs:
\begin{enumerate}
    \item  All isolated points of $S(G)$ are algorithmically simple, in the sense that they have decidable language (\Cref{isolated-subshifts-have-decidable-language}). 
    \item All SFTs with nonzero Medvedev degree (algorithmically complex), have a neighborhood in $S(G)$ with no isolated points (\Cref{complex-SFTs-have-no-isolated-element-near}).
\end{enumerate}
Minimal $G$-SFTs are isolated in $S(G)$, so the first item provides a topological proof of the known fact that a minimal $G$-SFT has decidable language (\Cref{minimal-sfts-have-decidable-language}). Regarding the second item, there are many groups known to admit SFTs with nonzero Medvedev degree, including $\Z^d$ for $d\geq 2$ (see \Cref{chap:Medvedev}). It follows that for these groups, the study of generic dynamical properties in $S(G)$ can not rely on the characterization of isolated points. 

\begin{rem*}
Another proof of the non-density of isolated points in $S(\Z^d)$ for $d\geq 2$ can be found in \cite{gangloff_topological_2024}. Moreover, Ville Salo has pointed out to us that the Robinson SFT  \cite{robinson_undecidability_1971} is a simple example of an SFT having a neighborhood in $S(\Z^2)$ without isolated points. However, these observations do not apply to groups beyond $\Z^d$. For groups such as the Baumslag-Solitar group $\operatorname{BS}(2,3)$ or the Grigorchuk group, \Cref{complex-SFTs-have-no-isolated-element-near} seems the easiest way available to show the non-density of isolated points in the space of subshifts. 
\end{rem*}

For $G$ amenable, we also prove some basic results about entropy. We will prove that the function $S(G)\to\R$ that maps a subshift to its topological entropy is upper-semicomputable (\Cref{entropy-function-is-semicomputable}). We derive from this fact that an SFT which is entropy-minimal and whose entropy is a computable real number must have decidable language (\Cref{entropy-minimal-sfts-have-decidable-language}). 

After endowing $S(G)$ with an adequate computable metric space structure, our proofs are direct applications of basic results from computable analysis. We believe that a more thorough study of the space $S(G)$ from the point of view of computable analysis could be fruitful. We also mention that we have restricted our attention to groups with decidable word problem for the sake of clarity, but  these results can be extended to the case of  recursively presented groups using the tools from \Cref{subsec:pullback-subshift}.  

\section{The spaces $S(G)$ and $S_{\geq c}(G)$}
In this section we fix some notation, and review some basic facts about the topology of the metric space $S(G)$ for a finitely generated group $G$. Most of these facts are proved for $\Z$ in \cite{pavlov_structure_2023}, and the generalizations are straightforward.

Let us start by reviewing notation. Let $G$ be a finitely generated group endowed with a word length  $|\cdot|$. Given $n\in\N$ we write $B_n=\{g\in G : |g|\leq n\}$. We define a metric $d$ for $\N^G$ by $d(x,y)=\inf\{2^{-n} : n\in\N,\ x|_{B_n}=y|_{B_n}\}$. 

Now let $X$ be a subshift. We denote by $S(X)$ the collection of subsystems of $X$. Given a finite set $F\subset G$, we write $X|_F=\{x|_F : x\in X\}$, and for $n\in\N$ we denote $L_n(X)=X|_{B_n}$. We define the \textbf{cylinder set} $[X,n]=\{Y\in S(G) :  L_n(X)=L_n(Y)\}$. Thus  $[X,n]$ is equal to the ball $\{Y \in S(G) : D(X,Y)\leq 2^{-n}\}$. It follows that cylinder sets  sets are a basis for the topology on $S(G)$. Also note that cylinder sets are clopen subsets in $S(G)$.  The following result shows that $S(G)$ is a Baire space:
\begin{prop}\label{S(G)-homeomorphic-to-subset-of-baire}
    $S(G)$ is homeomorphic to a closed subset of the Baire space $\N^\N$. 
\end{prop}
\begin{proof}
    We only sketch a proof. Observe that \[S(G)=\bigcup_{\substack{A\subset\N \\\text{finite}}}S(A^G), \] so it suffices to show that every  $S(A^G)$ in this union is homeomorphic to a closed subset of the Cantor set $\{0,1\}^\N$. Indeed, it follows from Brouwer's theorem that the collection $K(A^\N)$ of compact subsets of $A^\N$ with the corresponding Hausdorff metric is homeomorphic to the Cantor set $\{0,1\}^\N$.  Moreover, $S(A^G)$ is easily seen to be a closed subset of $K(A^G)$. Joining these two facts we obtain our claim that $S(A^G)$ is homeomorphic to a closed subset of $\{0,1\}^\N$. The details that we omitted are given, with further computability considerations, in the proof of \Cref{S(G)-is-effectively-closed}. \end{proof}
We now review some basic properties of subshifts and SFTs in $S(G)$.
\begin{prop}\label{S(X)-is-clopen-iff-X-is-SFT}
    Let $X$ be a subshift in $S(G)$. Then $S(X)$ is a compact subset of $S(G)$. Moreover, $X$ is an SFT if and only if $S(X)$ is a clopen subset of $S(G)$. 
\end{prop}
\begin{proof}
Let $X\subset A^G$ be a subshift in $S(G)$. It can be easily seen that $S(X)$ is a closed subset of $S(A^G)$. We observed in the proof of \Cref{S(G)-homeomorphic-to-subset-of-baire} that $S(A^G)$ is compact, so it follows  that $S(X)$ is compact. 

Now we prove that $S(X)$ is open if and only if $X$ is an SFT. If $X$ is an SFT then there is $n\in\N$ such that $X$ can be defined by a set of forbidden patterns in $A^{B_n}$. Then we can write $S(X)=\bigcup_{Y\in S(X)}[Y,n]$. This equality proves that $S(X)$ is open. If $S(X)$ is open, then for some $n$ we have $[X,n]\subset S(X)$, as $[X,n]$ is equal to a closed ball of radius $2^{-n}$. But the inclusion  inclusion $S(X)\subset [X,n]$ always holds, so we have $S(X)\subset [X,n]$. This shows that $X$ is the SFT defined by the set of forbidden patterns $A^{B_n}\smallsetminus L_n(X)$.   
\end{proof}

\begin{prop}
    An isolated point in $S(G)$ is an SFT.
\end{prop}
\begin{proof}
    Let $X\subset A^G$ be isolated in $S(G)$, so there is $n\in \N$ with $[X,n]=\{X\}$. Let $Y$ be the SFT defined by alphabet $A$ and set of forbidden patterns $A^{B_n}\smallsetminus L_n(X)$. It is clear that $L_n(X)=L_n(Y)$, so $Y\in [X,n]$. As $[X,n]=\{X\}$, it follows that $Y=X$. This proves that $X$ is an SFT.
\end{proof}
\begin{prop}\label{minimal-sft-is-isolated}
    A minimal SFT in $S(G)$ is an isolated point.
\end{prop}
\begin{proof}
    Let $X\subset A^G$ be a minimal SFT in $S(G)$. Let $n\in\N$ so that $X$ can be defined by a set of forbidden patterns in $A^{B_n}$. We claim that $[X,n]=\{X\}$. Indeed, if $Y\in [X,n]$, then $Y\subset X$. But as  $X$ is minimal, we have that $Y=X$. It follows that $X$ is isolated.
\end{proof}

For an amenable group $G$, a relevant subspace of $S(G)$ is given by 
\[S_{\geq{c}}(G)=\{X\in S(G) : h(X)\geq c\}, \ \ \ \  c\geq 0.\]
Here $h(X)$ denotes the topological entropy of $X$. We assume that the reader is familiar with basic properties of topological entropy for amenable groups, see for instance  \cite[Chapter 9]{kerr_ergodic_2016}. The following observation appears without proof in \cite{frisch_symbolic_2017}, we provide an argument for completeness.
\begin{prop}\label{isolated-subshifts-in-S(G)-with-restricted-entropy}
        An isolated point in $S_{\geq c}(G)$ must be an SFT with entropy $c$. 
\end{prop}
\begin{proof}
    Let $X\subset A^G$ be isolated in $S_{\geq c}(G)$, and let $n$ be a natural number with $[X,n]\cap S_{\geq c}(G)=\{X\}$. Let $Y$ be the SFT defined by alphabet $A$, and set of forbidden patterns  $A^{B_n}\smallsetminus L_n(X)$. It follows from the definition that $X\subset Y$. As $h(Y)\geq h(X)$ and $h(X)\geq c$, we have that $h(Y)\geq c$. But then $Y\in [X,n]\cap S_{\geq c}(G)=\{X\}$, so we have $Y=X$. This shows that $X$ is an SFT. The fact that $h(X)=c$ follows from the fact that subshifts with entropy $c$ are dense in $S_{\geq c}(G)$, proved in \cite{frisch_symbolic_2017}.
\end{proof}
\begin{remark}
The argument used in the previous result shows a much more general fact: if $f\colon S(G)\to \R$ satisfies $X\subset Y \implies f(X)\leq f(Y)$, then every isolated point in $f^{-1}[c,\infty)$ is an SFT. This could be of interest when considering other dynamical invariants.   
\end{remark}

    A $G$-subshift $X$ with positive topological entropy is called \textbf{entropy-minimal} when every proper subshift $Y\subset X$ satisfies $h(Y)<h(X)$. Entropy-minimal subshifts provide examples of isolated points in $S_{\geq c}(G)$: 
\begin{prop}\label{entropy-minimal-sfts-are-isolated-in}
    An entropy-minimal SFT with topological entropy $c$ is isolated in $S_{\geq c}(G)$.
\end{prop}
\begin{proof}
    Let $X\in S_{\geq c}(G)$ be an entropy-minimal SFT with entropy $c$. As $X$ is an SFT, there is a natural number $n$ such that every subshift in $[X,n]$ is a subset of $X$. As $X$ is entropy-minimal, every $Y\in [X,n]$ with $Y\ne X$ satisfies $h(Y)<c$. It follows that $[X,n]\cap S_{\geq c}(G)=\{X\}$, so $X$ is isolated in $S_{\geq c}(G)$. 
\end{proof}
Let us observe that the union of two entropy-minimal SFTs with entropy $c$ and with disjoint alphabets is not entropy-minimal, but is easily seen to be isolated in  $S_{\geq c}(G)$. To our knowledge, there is no known characterization of isolated points in $S_{\geq c}(G)$

\section{A computable metric space containing  $S(G)$}\label{computable-metric-space-structure-for-S(G)}

In this section we show that for a finitely generated group $G$ with decidable word problem, the set $S(G)$ can be naturally identified with an effectively closed subset of a computable metric space. 

We start by recalling the computable metric space structure for $\N^G$ described in \Cref{chap:computable_analysis_subshifts} (see \Cref{computable-metric-space-on-A^G-infinite-alfabet} and \Cref{compatibility-of-metrics-for-A^G}). We take $\mathcal S=\{s_i : i\in \N\}$ as the collection of all elements in $\N^G$ that are constant outside a finite subset of $G$. We choose the numbering of $\mathcal S$ so that from the index $i$, we can compute a finite set $F\subset G$, a function $p\colon F\to \N$, and an element $n\in \N$, such that the function $s_i\colon G\to \N$ equals $p$ when restricted to $F$, and equals $n$ outside this set. This numbering exists because $G$ has decidable word problem.  

\begin{prop}\label{fullshift-on-infinite-symbols-is-computable-metric-space}
    $(\N^G,d, \mathcal S)$ is a computable metric space.
\end{prop}
\begin{proof}
    It is clear that $\mathcal S$ is a dense subset of $\N^G$, and it follows from our choice of $\{s_i : i\in \N\}$ that given natural numbers $(i,j)$, we can compute the rational number $d(s_i,s_j)$. 
\end{proof}

Now consider the collection $K(\N^G)$ of compact subsets of $\N^G$, endowed with the Hausdorff metric $d_H$ associated to $d$.  We also let $\mathcal S'$ be the collection of finite subsets of $\mathcal S$, numbered in a canonical manner.  
\begin{prop}\label{compact-subsets-of-fullshift-on-infinite-symbols-is-computable-metric-space}\label{Hyperspace-of-a-fullshift-is-recursively-compact}
    $(K(\N^G),d_{H}, \mathcal S')$ is a computable metric space. Moreover, for every finite $A\subset \N$, $K(A^G)$ is a recursively compact subset of $K(\N^G)$.
\end{prop}
\begin{proof}
The fact that $(K(\N^G),d_{H}, \mathcal S')$ is a computable metric space follows from \Cref{fullshift-on-infinite-symbols-is-computable-metric-space} and \Cref{Hyperspace-inherits-computable-metric-space-structure}. For the second claim we consider $A^G$ as a computable metric space, with a structure that is compatible with $(\N^G,d,\mathcal S)$. Namely, we define a metric $d_A$ as the restriction of $d$ to $A^G$, and we let $\mathcal S_A$ be the restriction of $\mathcal S$ to $A$ indexed in a computable manner. This is possible as clearly $\{n\in \N : s_i\in A^G\}$ is a decidable subset of $\N$. As $A^G$ is recursively compact,  \Cref{Hyperspace-inherits-computable-metric-space-structure} shows that $K(A^G)$ is a recursively compact metric space, and it is clear that the inclusion map $K(A^G)\to K(\N^G)$ is a computable homeomorphism onto its image.  
\end{proof}
We are interested in the space $K(\N^G)$ because it contains the space $S(G)$, as every $G$-subshift with alphabet contained in $\N$ is a compact subset of $\N^G$. Let us note that the metric $D$ defined for $S(G)$ is equal to the restriction of the Hausdorff metric $d_{H}$. 

The following result will be our basic tool to prove computability results about isolated points in $S(G)$.
\begin{prop}\label{S(G)-is-effectively-closed}
    $S(G)$ is an effectively closed subset of $K(\N^G)$. 
\end{prop}
\begin{proof}
    Let $S$ be a finite and symmetric generating set for $G$. For each $s\in S$, we define $f_s\colon \N^G\to \N^G$ by  $f_s(x)(g)=x(s^{-1}g)$, and then define $F_s\colon K(\N^G)\to K(\N^G)$ by
    \[
    X\mapsto f_s^{-1}(X).
    \] \Cref{hyperfunction-of-a-computable-function-is-computable} implies that $F_s$ is a computable map on $K(\N^G)$. \Cref{prop:fixed-points-of-computable-function} shows that the set of fixed points of $F_s$:
    \[\operatorname{Fix}(F_s)=\{X\in K(\N^G) : F_s(X)=X\}\]
    is effectively closed. Moreover, the following holds:
    \[S(G)=\bigcap _{s\in S}\operatorname{Fix}(F_s).\]
    Indeed, a $G$-subshift in $S(G)$ is the same as a compact subset of $\N^G$ which is invariant by shift translations by $G$, which is equivalent to being fixed by $F_s$ for all $s\in S$. As the intersection of finitely many effectively closed sets is also effectively closed, it follows that $S(G)$ is an effectively closed subset of $K(\N^G)$.
\end{proof}
\begin{prop}\label{subshift-with-computable-language-is-computable-point}
    A subshift $X\subset A^G$ is a computable point of the computable metric space $(K(\N^G),d_H,\mathcal S')$ if and only if it has decidable language. That is, there is an algorithm which on input a finite set $F\subset G$ and a pattern $p\colon F\to A$, decides correctly whether $p$ appears in $X$.
\end{prop}
\begin{proof}
    For the forward implication we suppose that $X$ is a computable point, and we exhibit an algorithm which on input $n\in \N$, outputs $L_n(X)$.  

    On input $n$, we compute a finite set $\{x_1,\dots,x_n\}\subset \mathcal S$ such that the Hausdorff distance between $X$ and $\{x_1,\dots,x_k\})$ is at most $ 2^{-n-1}$. This is possible by \Cref{characterization-of-computable-points-via-sequences}. Now it follows from the definition of Hausdorff distance that for every $x\in X$, there is some $x_i$, $1\leq i\leq k$, with $x|_{B_n}=x_i|_{B_n}$, and the same holds conversely. This shows that a pattern $p\colon B_n\to A$ appears in $X$ if and only if $p$ is equal to $x_i|_{B_n}$ for some  $1\leq i\leq k$.  The output of the algorithm is $\{x_1|_{B_n},\dots,x_k|_{B_n}\}$. The backward implication follows the same idea, and is left to the reader.
\end{proof}

\section{Computability and isolated points in $S(G)$}\label{sec:isolated-points-in-S(G)}
In this section we prove some computability results about isolated points in $S(G)$, under the assumption that $G$ has decidable word problem.
\begin{prop}\label{isolated-subshifts-have-decidable-language}
    Let $G$ be a group with decidable word problem. Then a subshift that is isolated in $S(G)$ has decidable language. \end{prop}
\begin{proof}
    If $X\subset A^G$ is isolated in $S(G)$, then it is also isolated in the clopen set $S(A^G)\subset S(G)$. Recall that $S(A^G)$ is recursively compact by \Cref{Hyperspace-of-a-fullshift-is-recursively-compact}. Now \Cref{isolated-points-are-computable} shows that $X$ is a computable point of the computable metric space $K(\N^G)$. This implies that $X$ has decidable language by \Cref{subshift-with-computable-language-is-computable-point}.
\end{proof}
    This result provides a different proof of the following well-known result:
\begin{cor}\label{minimal-sfts-have-decidable-language}
    Let $G$ be a group with decidable word problem. A minimal $G$-SFT has decidable language.
\end{cor}
\begin{proof}
    Let $X\subset A^G$ be a minimal $G$-SFT with $A\subset\N$. Then $X$ is isolated in $S(G)$ by \Cref{minimal-sft-is-isolated}, and then \Cref{isolated-subshifts-have-decidable-language} shows that $X$ has decidable language.
\end{proof}
\begin{prop}\label{complex-SFTs-have-no-isolated-element-near}
    Let $G$ be a group with decidable word problem, and let $X$ be an SFT in $S(G)$ with nonzero Medvedev degree. Then $X$ has a neighborhood in $S(G)$ with no isolated point.
\end{prop}
\begin{proof}
    As $X$ is an SFT, there is a natural number $n$ so that for every $Y\in [X,n]$, $Y$ is a subset of $X$. We claim that there is no isolated point in $[X,n]$. Indeed, if there was an isolated point $Y\in[X,n]$, then $Y$ would have decidable language by \Cref{isolated-subshifts-have-decidable-language}. This implies that $Y$ has computable elements. But $Y\subset X$, so $X$ would also have computable elements. As this contradicts our hypothesis, we conclude that $[X,n]$ has no isolated point. 
\end{proof}
In \cite{pavlov_structure_2023} it has been observed that certain dynamical properties of a subshift $X$ imply that $S(X)$ is a Cantor set. The same argument in \Cref{complex-SFTs-have-no-isolated-element-near} shows that this is the case for the property of having nonzero Medvedev degree.  
\begin{cor}
    Let $G$ be a group with decidable word problem. If $X$ is a nonempty SFT and $m(X)>0$, then $S(X)$ is a Cantor set. 
\end{cor}
\begin{proof}
    It suffices to note that $S(X)$ has no isolated points. If $S(X)$ contains an isolated points $Y$, then $Y$ is also isolated in $S(G)$ by \Cref{S(X)-is-clopen-iff-X-is-SFT}. Then $Y$ has decidable language by \Cref{isolated-subshifts-have-decidable-language}, and in particular, computable elements. As $Y\subset X$, this  contradicts the hypothesis on $X$.
\end{proof}
\section{Computability and isolated points in $S_{\geq c}(G)$}
In this section we prove some computability results about isolated points in $S_{\geq c}(G)$, under the assumption that $G$ is amenable and has decidable word problem.

It is proved in \cite{frisch_symbolic_2017} that the topological entropy map $h\colon S(G)\to \R$, $X\mapsto h(X)$ is upper semi-continuous. Here we prove an effective version of this result (upper-semicomputable functions are defined in \Cref{preliminaries-semicomputable}).
\begin{prop}\label{entropy-function-is-semicomputable}
    Let $G$ be a finitely generated amenable group with decidable word problem. Then the entropy map $h\colon S(G)\to\R$, $X\mapsto h(X)$ is upper-semicomputable.
\end{prop}
\begin{proof}
    We will exhibit an upper-semicomputable function $f\colon K(\N^G)\to \R$ whose restriction to $S(G)$ equals $h$. For this we introduce the following notation. Given $F\subset G$ finite and $L\subset \N^F$, we denote $[L]=\{X\in K(\N^G) : X|_F=L\}$.  Now for every finite subset $F\subset G$, we define a function $f_F\colon K(G)\to \R$ by 
    \[X\mapsto \frac{\log (X|_F)}{|F|}.\] 
    The function $f_F$ is continuous. Indeed, we can partition $K(\N^G)=\bigsqcup_{L\subset A^F}[L]$, and for each $L\subset A^F$ the value of $f_F$ is constant on $[L]$. This argument also shows that $f_F$ is computable, uniformly on $F$. We now define a function $f\colon K(G)\to\R$ by
    \[X\mapsto \inf_{\substack{F\subset G \\ \text{finite}}} f_F(X) \]
    Given a familiy $(f_i)_{i\in\N}$ of uniformly upper-semicomputable functions, the infimum $\inf_{i\in\N} f_i$ is also upper-semicomputable (\Cref{semicomputable-functions-are-closed-by-infimum}). This proves that $f$ is upper-semicomputable. Finally, it is proved in \cite{downarowicz_shearer_2016} that the topological entropy of a subshift $X$ can be computed as  
    \[h(X)=\inf _{\substack{F\subset G \\ \text{finite}}} \frac{\log |X_F|}{|F|}.\]
    Thus for every subshift $X\in S(G)$ we have $f(X)=h(X)$. This proves our claim.
\end{proof}
\begin{rem}
    If we ignore the computability in the proof, this provides an alternative proof of the upper-semicontinuity of the topological entropy function $h\colon S(G)\to \R$ for a finitely generated group. This proof is shorter and simpler than the one given in \cite{frisch_symbolic_2017}, but this is thanks to the nontrivial result from \cite{downarowicz_shearer_2016}.  
\end{rem}

We now examine some consequences of \Cref{entropy-function-is-semicomputable}. The following follows from the definition of upper-semicomputable function, and the fact that a set of the form $\{x\in\R : x<c\}$ is effectively open exactly when $c$ is lower-semicomputable. 
\begin{cor}\label{S(G)-with-restricted-entropy-is-effectively-closed}\label{subshifts-with-entorpy-at-least-c-is-effectively-closed}
    Let $G$ be a finitely generated amenable group with decidable word problem. For every lower-semicomputable real number $c$, the set $S_{\geq c}(G)$ is effectively closed. 
\end{cor}
If $c$ is lower-semicomputable but is not upper-semicomputable, then the space $S_{\geq c}(G)$ has no isolated points. This is because an isolated point of $S_{\geq c}(G)$ must be an SFT with entropy $c$ (\Cref{isolated-subshifts-in-S(G)-with-restricted-entropy}), but the topological entropy of every $G$-SFT is upper-semicomputable \cite{barbieri_entropies_2021}. On the other hand, if $c$ is both lower-semicomputable and upper-semicomputable, we have the following results (compare with  \Cref{sec:isolated-points-in-S(G)}).  

\begin{prop}\label{isolated-points-in-S(G)-with-entropy}
    Let $G$ be a finitely generated amenable group with decidable word problem, and let $c$ be a positive computable real number. Then a subshift that is isolated in $S_{\geq c}(G)$ has decidable language. 
\end{prop}
\begin{proof}

    We start observing that $S_{\geq c}(G)\cap S(A^G)$ is recursively compact. Indeed, $S_{\geq c}(G)$ is effectively closed by \Cref{S(G)-with-restricted-entropy-is-effectively-closed}, so the intersection $S_{\geq c}(G)\cap S(A^G)$ is effectively closed. Moreover, $S(A^G)$ is recursively compact by \Cref{Hyperspace-of-a-fullshift-is-recursively-compact}, and then \Cref{prop:effective-closed-subset-of-compact-is-compact} shows that $S_{\geq c}(G)\cap S(A^G)$ is recursively compact.

    If $X\subset A^G$ is isolated in $S_{\geq c}(G)$, then it is also isolated in $S_{\geq c}(G)\cap  S(A^G)$. Then \Cref{isolated-points-are-computable} shows that $X$ is a computable point of the computable metric space $K(\N^G)$. Finally, \Cref{subshift-with-computable-language-is-computable-point} shows that $X$ has decidable language.
\end{proof}
As entropy-minimal SFTs with entropy $c$ are isolated in $S_{\geq c}(G)$ (\Cref{entropy-minimal-sfts-are-isolated-in}), the following is a particular case of \Cref{isolated-points-in-S(G)-with-entropy}.
\begin{cor}\label{entropy-minimal-sfts-have-decidable-language}
    Let $G$ be a finitely generated amenable group with decidable word problem. Every $G$-SFT which is entropy-minimal and whose entropy is a computable real number has decidable language. 
\end{cor}
In general, the fact that an SFT has a computable entropy does not imply that its language is decidable. In $\Z^2$, this easily follows from the existence of SFTs with uncomputable language and zero topological entropy. This question is also considered in  \cite{hertling_shifts_2008}. Our results show that with the extra property of entropy-minimality, the real number $h(X)$ indeed allows to compute the language of $X$. These results raise the problem of classifying the entropies of entropy-minimal SFTs on a finitely generated amenable group with decidable word problem.

%% file: Chapter_rice_sfts.tex
The goal of this chapter is presenting the results obtained by the author in \cite{carrasco-vargas_rice_2024}. Let $G$ be a group with undecidable domino problem, such as $\mathbb{Z}^2$. We prove that all nontrivial dynamical properties for sofic $G$-subshifts are undecidable, that this is not true for $G$-SFTs, and an undecidability result for dynamical properties of $G$-SFTs similar to the Adian-Rabin theorem. Furthermore we prove a Rice-like result for dynamical invariants asserting that every computable real-valued invariant for $G$-SFTs that is monotone by disjoint unions and products is constant.

\minitoc
\section{Introduction}\label{sec:Introduction}

There has been a recent interest in extending the study of SFTs on $\Z$ to SFTs on other groups. An important obstruction that has arisen for this project is of an algorithmic nature. Indeed, many dynamical questions become undecidable when we move from $\Z$-SFTs to $\Z^2$-SFTs. Lind \cite{lind_multidimensional_2004} called this the  
``swamp of undecidability''. 
One way to better understand this swamp is to ask whether we
have a result analogous to Rice's theorem. This result states that \textit{every} nontrivial semantic property of computer programs is algorithmically undecidable
\cite{rice_classes_1953}. Nontrivial means that some element satisfies the property and some element does not. This result has been paradigmatic in the
sense that \emph{Rice-like theorems} have been discovered in a variety
of mathematical contexts \cite{kari_rice_1994,delacourt_rice_2011,menibus_characterization_2018,hirst_ricestyle_2009,lafitte_computability_2008,Gamard2021RiceLikeTF,delvenne_quasiperiodic_2004,rice_classes_1953,adian_decision_2000,rabin_recursive_1958,guillon_revisiting_2010,adyan_algorithmic_1955}. In this chapter we consider the following question:

\begin{question}
\label{que:main-question} Is there a Rice theorem for dynamical properties of SFTs on $\Z^2$? What about SFTs on other groups?
\end{question}
That is, the goal of this chapter is converting the metaphor \textit{swamp of undecidability} from \cite{lind_multidimensional_2004} to precise mathematical statements. 

\subsubsection*{Known results}
These results are stated in terms of sets of tilings of $\mathbb{Z}^{2}$. We recall that a \textbf{tile }(or
Wang tile) is a unit square with colored edges, a \textbf{tileset}
$\tau$ is a finite set of tiles, a \textbf{tiling} is a function $\Z^{2}\to\tau$ satisfying the rule that two tiles
that share an edge must have the same color at that edge, and a \textbf{set of tilings} is the collection of all tilings associated to a tileset. Sets of tilings of $\Z^{2}$ and $\Z^{2}$-SFTs are related by the fact that every set of tilings of $\Z^{2}$ is a $\Z^{2}$-SFT and every $\Z^{2}$-SFT
is topologically conjugate to a set of tilings \cite{aubrun_domino_2018}.
\begin{enumerate}
\item In \cite{cervelle_tilings_2004} the authors prove the
undecidability of set equality and set inclusion for sets
of tilings of $\Z^{2}$, and ask for a Rice theorem for tilings.
\item In \cite{lafitte_computability_2008} the authors consider
properties of sets of tilings of $\Z^{2}$ that are preserved by ``zoom''.
Informally two tilesets are equivalent in this sense when the
tiles from each tileset produce unique macro-tiles that satisfy the
same matching rules from the other tileset. The authors prove that
all nontrivial properties of sets of tilings  preserved by ``zoom''
are undecidable.
\item In \cite{delvenne_quasiperiodic_2004} the authors prove
the undecidability of every property of sets of tilings of $\Z^{2}$
that is stable by topological conjugacy, by direct
products among nonempty systems, and is not satisfied by the
empty set. 
\end{enumerate}

\subsubsection*{Results}
In this work we consider dynamical properties of SFTs, that is, properties preserved by topological conjugacy. We investigate the (un)decidability of these properties for SFTs on finitely generated groups. A first observation is that a Rice theorem is not possible in this setting: the property of having some fixed point is nontrivial and  decidable from presentations (\Cref{prop:decidable-fixed-points}). Thus we need to add some hypotheses to the properties considered in order to prove a Rice-like theorem.

Our undecidability results are related to the domino problem, the algorithmic problem  of determining whether an SFT is empty. Berger proved \cite{berger_undecidability_1966} that the domino
problem is undecidable for $\Z^{d}$, $d\geq2$, and this has been
recently extended to a large class of finitely generated groups \cite{aubrun_domino_2019,aubrun_tiling_2013,ballier_domino_2018,bartholdi_domino_2023,bitar_contributions_2023,cohen_large_2017,jeandel_translationlike_2015,jeandel_undecidability_2020,margenstern_domino_2008,aubrun_strongly_2023}.
In \cite{ballier_domino_2018} the authors conjecture that all non virtually free groups have
undecidable domino problem. 

Our main result states that if $G$ has undecidable domino problem, then all dynamical properties of $G$-SFTs that satisfy a mild technical condition are undecidable (\Cref{thm:adian-rabin-for-sfts}). This result can be applied to several dynamical properties of common interest, such as transitivity, minimality, and others (\Cref{sec:examples}). This result can be applied to all nontrivial dynamical properties that
are preserved by topological factor maps or topological extensions (\Cref{cor:rice-for-monotone-properties}). \Cref{thm:adian-rabin-for-sfts} exhibits an analogy with the
Adian-Rabin theorem for group properties \cite{adyan_algorithmic_1955,rabin_recursive_1958}, the ``Rice-like
theorem in group theory''. A number of analogies have been observed between group theory and symbolic dynamics \cite{jeandel_enumeration_2017,jeandel_characterization_2019,jeandel_computability_2016}. 

We also consider dynamical invariants of SFTs taking values in partially ordered
sets. We show that if $G$ has undecidable domino problem then every abstract real-valued dynamical invariant for $G$-SFTs that is nondrecreasing by disjoint unions and direct products must be constant  (\Cref{invariants-monotone-by-products-and-unions-are-constants}). This result covers topological entropy for amenable groups and other related invariants. We mention that for some amenable groups it is known a much stronger fact: the existence of SFTs whose entropy is a non-computable real number \cite{hochman_characterization_2010,barbieri_entropies_2021,bartholdi_shifts_2024}. The conclusion of \Cref{invariants-monotone-by-products-and-unions-are-constants} is weaker, but it is a much more general result. 

We also consider the larger class of sofic subshifts. We prove that if the domino problem for $G$ is undecidable, then a Rice theorem for dynamical properties of sofic subshift holds. That is, all nontrivial
dynamical properties are undecidable (\Cref{thm:Rice-theorem-sofic-subshifts}).

All our undecidability results are proved through a many-one reduction to the domino problem for $G$. In informal words our results show that most dynamical properties are harder than the domino problem. As in the original Rice's theorem our proofs are very short. They are  based on the computability of direct products and disjoint unions at the level of presentations. Despite tiles  can be defined
on Cayley graphs of finitely generated groups, the proofs of the Rice-like
theorems in \cite{lafitte_computability_2008,cervelle_tilings_2004}
are specific to $\Z^{2}$, and it is not clear whether they could
be generalized to a group that is not residually finite. 

\section{Preliminaries} 
A property of subshifts is called a \textbf{dynamical property} when it is invariant by conjugacy. We now review subshifts of finite type (SFTs) and their {presentations}. See also  \cite{aubrun_notion_2017,aubrun_domino_2018}. Let $G$ be a finitely generated group and let $S$ be a finite and symmetric generating set for $G$. 
\begin{defn}
A \textbf{pattern presentation }is a function $p\colon W\to A$, where
$W\subset S^{\ast}$ is a finite set of words and $A\subset\N$ is
a finite alphabet. We say that $p$ \textbf{appears }in $x\in A^{G}$
at $g\in G$ when $x(g\underline{w})=p(w)$ for every $w\in W$. An \textbf{SFT presentation }is a tuple $(A,\mathcal{F})$ of a finite alphabet 
$A\subset \N$ and a finite set $\mathcal{F}$ of pattern
presentations associated to $S$. It determines the subshift $X_{(A,\mathcal{F})}$
of all configurations in $A^{G}$ where no pattern presentation
from $\mathcal{F}$ appears. We call $X_{(A,\mathcal{F})}$ a \textbf{subshift of finite type} (SFT). The \textbf{domino problem }$\Dp(G)$ is the set of all presentations
$(A,\mathcal{F})$ for which the SFT $X_{(A,\mathcal{F})}^ {}$ is
empty. %
\end{defn}
A survey on the domino problem can be found in  \cite{aubrun_domino_2018}. Observe that the empty set is an SFT with our definitions. Some
authors consider the empty set as a subshift \cite{cohen_large_2017},
while others exclude it by definition \cite{aliakbar_set_2018}. Here
we follow the first convention. This is a natural choice in our setting: an algorithm
able to detect a particular property from $G$-SFT presentations should give the same output when given presentations of the empty subshift. Furthermore we can not exclude these presentations without assuming that the domino
problem for $G$ is decidable. A consequence of this convention is
that a dynamical property must assign yes/no value to the empty subshift.

When we consider the decidability of a dynamical property
$\mathscr{P}$ from presentations of $G$-SFTs we formally refer to the set $\{(A,\mathcal{F})\mid X_{(A,\mathcal{F})}\text{ satisfies }\mathscr{P}\}$. Our undecidability results will follow from many-one reductions of the form $\Dp(G)\leq_m\P$, but this notion is not required to understand the proofs.  
\section{Undecidability of dynamical properties of SFTs }

In this section we study the undecidability of dynamical properties
of SFTs. The following example shows that a Rice theorem fails in this setting. 
\begin{prop}
\label{prop:decidable-fixed-points} Let $G$ be a finitely generated
group. Then the property of $G$-SFTs of containing a fixed point is decidable
from presentations. 
\end{prop}

\begin{proof}
Recall that a fixed point of an SFT is a configuration $x$ such that $gx=x$ for all $g\in G$. Let $(A,\mathcal{F})$ be a $G$-SFT presentation, and for each $a\in A$ 
denote by $x_{a}\colon G\to A$ the configuration with constant value
$a$. Note that a fixed point in $X_{(A,\mathcal{F})}$
is equal to $x_{a}$ for some $a\in A$. Clearly a pattern
presentation $p\colon W\to A$ appears in $x_{a}$ if and only 
if $p$ has constant value $a$, and this is
a decidable property of $p$. Thus the following algorithm proves the statement: on input $(A,\mathcal{F})$ check whether for some $a\in A$ the set $\mathcal{F}$ fails
to contain a pattern presentation with constant value $a$.
\end{proof}
In order to prove our undecidability results we need to verify the computability of direct products and disjoint unions at the level of presentations. 
\begin{lem}
\label{lem:products}Let $G$ be a finitely generated group. There is an effective procedure which, given
two presentations $(A,\mathcal{F})$ and $(B,\mathcal{C})$ of $G$-SFTs,
outputs a presentation of a $G$-SFT that is topologically conjugate
to the direct product $X_{(A,\mathcal{F})}\times X_{(B,\mathcal{C})}$. 
\end{lem}

\begin{proof}
Let $\alpha\colon\N^{2}\to\N$ be a computable bijection, and let
$\pi_{1},\pi_{2}\colon\N\to\N$ be the computable functions defined
by $\pi_{i}(\alpha(n_{1},n_{2}))=n_{i}$, $i=1,2$. On input $(A,\mathcal{F})$
and $(B,\mathcal{C})$, our algorithm starts defining the alphabet of the new SFT as $C=\{\alpha(a,b)\mid a\in A,b\in B\}\subset\N$.
Then we compute a set $\mathcal{G}$ of pattern presentations as
follows. For each $p\colon W\to A$ in $\F$ (resp. $p\colon S\to B$
in $\mathcal{C}$), we add to $\mathcal{G}$ every function $q\colon W\to C$
such that $\pi_{1}\circ q=p$ (resp. $\pi_{2}\circ q=p$). The output is $(C,\mathcal{G})$. It is clear that $X_{(C,\mathcal{G})}$ is conjugate
to the direct product $X_{(A,\mathcal{F})}\times X_{(B,\mathcal{C})}$.
Indeed, the map $\phi\colon X_{(C,\mathcal{G})}\to X_{(A,\mathcal{F})}^ {}\times X_{(B,\mathcal{C})}^ {}$
given by $\phi(x)(g)=(\pi_{1}(x(g)),\pi_{2}(y(g)))$ is a topological
conjugacy. 
\end{proof}
The proof of the following result is very similar, and is left to the reader.
\begin{lem}

\label{lem:disjoint-unions}Let $G$ be a finitely generated group. There is an effective procedure which,
given two SFT presentations $(A,\mathcal{F})$ and $(B,\mathcal{C})$,
outputs a presentation of an SFT that is topologically conjugate to
the disjoint union $X_{(A,\mathcal{F})}\sqcup X_{(B,\mathcal{C})}$.
\end{lem}

\begin{rem}
According to the definitions the direct product of two SFTs is not an SFT, but a topological
dynamical system that is conjugate to an SFT. We will ignore this subtlety for the sake of clarity. 
\end{rem}

In the following definition we propose the term Berger property because of the analogy with Markov properties of finitely presented groups and the Adyan-Rabin undecidability theorem  \cite{lyndon_combinatorial_2001}. 
\begin{defn}\label{def:Berger}
A dynamical property $\mathscr{P}$ of $G$-SFTs is called a \textbf{Berger
}property if there are two $G$-SFTs $X_{-}$ and $X_{+}$ satisfying
the following:
\begin{enumerate}
\item $X_{+}$ satisfies $\mathscr{P}$.
\item Every SFT that factors onto $X_{-}$ fails to satisfy $\mathscr P$.
 
\item There is a morphism from $X_{+}$ to $X_{-}$. 
\end{enumerate}
The subshift $X_{+}$ is allowed to be empty.  
\end{defn}
The main result of this section is the following.
\begin{thm}
\label{thm:adian-rabin-for-sfts} Let $G$ be a finitely generated group with undecidable domino problem. Then every Berger property of $G$-SFTs is undecidable. 
\end{thm}

\begin{proof}
Let $\P$ be a Berger property, and let $X_{+}$, $X_{-}$ be as in \Cref{def:Berger}. Given an SFT presentation
$(A,\mathcal{F})$ we define $Z$ as the disjoint union of  $X_{+}$ and $X_{(A,\mathcal{F})}^ {}\times X_{-}$. We claim that $Z$ has property $\P$ if and only if
$X_{(A,\mathcal{F})}^ {}$ is empty. Indeed, if $X_{(A,\mathcal{F})}$
is empty then $Z$ is topologically conjugate to $X_{+}$. If $X_{(A,\mathcal{F})}^ {}$
is nonempty, then $Z$ factors over $X_{-}$. This follows from two
facts: that for $X_{(A,\mathcal{F})}$ nonempty $X_{(A,\mathcal{F})}^ {}\times X_{-}$
factors over $X_{-}$, and that there is a topological morphism from
$X_{+}$ to $X_{-}$. If $\P$ was a decidable property, then we could decide whether $X_{(A,\mathcal{F})}^ {}$
is empty by computing a presentation for $Z$, which is possible thanks to
\Cref{lem:products} and \Cref{lem:disjoint-unions}, and then
checking whether $Z$ satisfies $\P$. This contradicts the undecidability
of $\Dp(G)$.
\end{proof}
In the language of many-one reductions we proved that a Berger property $\P$ satisfies $\P\geq_m\Dp(G)$ and thus it is $\Dp(G)$-hard. Note that \Cref{thm:adian-rabin-for-sfts} applies to every nontrivial property that is preserved to topological
factors and is satisfied by the empty subshift. Since the negation of a property preserved to
topological factors is preserved to topological extensions, we obtain
the following result.
\begin{cor}
\label{cor:rice-for-monotone-properties} Let $G$ be a finitely generated group with undecidable domino problem. Every nontrivial dynamical property
for $G$-SFTs which is preserved to topological factors (resp. extensions),
and which is satisfied (resp. not satisfied) by the empty subshift,
is undecidable. 
\end{cor}
The assumption on the empty subshift is necessary: the decidable property from  \Cref{prop:decidable-fixed-points} is preserved to factors.

\subsection{Examples}\label{sec:examples}
Here we present some examples. We start observing that the decidability of a property could be altered if we include or exclude the empty subshift:

\begin{example}\label{property-that-becomes-undecidable-after-we-allow-empty-subshift}
    Let $\mathscr P$ be the property of having some fixed point. We proved that $\P$ is decidable in \Cref{prop:decidable-fixed-points}. However ``$\P$ or empty'' is a Berger property. Indeed, let $X_{+}=\emptyset$, and let $X_{-}\ne\emptyset$ have no fixed point.   
\end{example}
There is a simple class of properties where this situation is prevented:

\begin{remark}\label{criterion}
    Let $\mathscr P$ be a Berger property, and let $X_{+}$ and $X_{-}$ as in \Cref{def:Berger}. If  $X_{+}$ is nonempty, then both ``$\mathscr P$ or empty'' and ``$\mathscr P$ and nonempty'' are Berger properties. This is shown by the same pair of sets $X_{+}$ and $X_{-}$.
\end{remark}
Many commonly studied dynamical properties can be shown to satisfy this criterion, and thus they are Berger properties regardless of the value assigned to the empty subshift. Now we review a few of them. 
\begin{example}
A $G$-SFT is topologically transitive when it contains a configuration
with dense orbit. Transitivity is a Berger property. Indeed, it suffices to choose $X_{+}$ as an SFT with exactly one fixed point, and choose $X_{-}$ as an SFT with exactly two fixed points. Note that topological extensions of $X_{-}$ are not transitive. 
\end{example}
\begin{example}
A $G$-SFT is minimal if it has no proper nonempty subsystem. This is a Berger property by the same reasoning as in the previous example. 
\end{example}

\begin{example}
    A configuration $X\in A^G$ is called strongly aperiodic when $gx\ne x$ for every $g\in G$ different to the identity. Consider the property $\mathcal{AD}$ of having at least strongly aperiodic configuration. This property is known as the \textit{aperiodic domino problem}. The negation of $\mathcal{AD}$ is a Berger property: it suffices to take $X_{+}$ as an SFT with only one fixed point, and $X_{-}=\{0,1\}^G$.  
    
    The complexity of $\mathcal{AD}$ for the group $\Z^d$ is studied in detail in  \cite{callard_aperiodic_2022}. The authors prove that this problem is $\Pi_1^0$-complete for $d=2$ and $\Sigma_1^1$-complete for $d\geq 4$. Its exact complexity is not known for $d=3$.  The argument presented here only shows that $\mathcal{AD}$ is $\Pi_1^0$-hard for $d\geq 2$.
    \end{example}
\begin{example}
Let $G=\Z^{d}$, $d\geq2$. A $G$-SFT has topologically complete
positive entropy (TCPE) when every topological factor is either a
singleton with the trivial action by $G$, or has positive topological
entropy. We claim that TCPE is a Berger property. Indeed, let $X_{-}$ be the SFT $\{0,1\}^{G}\cup\{2,3\}^{G}$.
This system fails to have TCPE because it factors onto the SFT with
exactly two configurations and zero topological entropy. The same is true for all extensions of $X_{-}$. Now let $X_{+}=\{0,1\}^G$. It is well known that this system has TCPE, and thus we have proved that  TCPE is a Berger property. This argument only shows that TCPE is $\Sigma_1^0$-hard. The complexity of this property is studied in detail in \cite{westrick_topological_2020}, where it is shown that it is $\Pi_1^1$-complete.
\end{example}
\begin{example}\label{ex:conjugacy}
    Let $G$ be amenable. The following are Berger properties for every nonempty SFT $X$: 
    \begin{itemize}
        \item The property  $\mathscr C (X)$ of being conjugate to $X$.
        \item The property $\mathscr F(X)$ of being a factor of $X$.
        \item  The property $\mathscr I(X)$ of embedding into $X$.
    \end{itemize}
    In the three cases it suffices to let  $X_{+}=X$, and let  $X_{-}$ be an SFT whose topological entropy is greater than that of $X$.  Note that $\mathscr C(X)$ and $\mathscr I(X)$ are conjugacy-invariant counterparts of properties studied in \cite[Section 3]{cervelle_tilings_2004}. Admitting an embedding \textit{from} $X$ is not always a Berger property, see \Cref{sec:remarks}. 
\end{example} 
It is natural to ask what is the  complexity of the properties considered here as  \Cref{thm:adian-rabin-for-sfts}  only shows that Berger properties are $\Dp(G)$-hard.
\section{Uncomputability of dynamical invariants of SFTs}
Recall that a real number $x$ is computable when there is an algorithm providing rational approximations to $x$ with any desired precision. We say that a real-valued dynamical invariant for $G$-SFTs $\mathcal I$ is computable from presentations when there is an algorithm which given a presentation for $X$, provides rational approximations to $\mathcal I(X)$ with any desired precision.

A fundamental invariant for $\mathbb{Z}^2$-SFTs is topological entropy. This invariant is not computable from presentations since there are $\Z^{2}$-SFTs whose entropy is
a non-computable real number \cite{hochman_characterization_2010}. A similar result holds for the invariant of entropy dimension of $\Z^{2}$-SFTs  \cite{meyerovitch_growthtype_2011,gangloff_characterizing_2022}. These results have motivated the search of dynamical restrictions that imply the computability of entropy  \cite{gangloff_effect_2019,pavlov_entropies_2015}. 

The main result of this section is a Rice-like theorem for real-valued invariants satisfying mild monotony conditions:
\begin{theorem}\label{invariants-monotone-by-products-and-unions-are-constants}
Let $G$ be a finitely generated group with undecidable domino problem. Every computable dynamical invariant for $G$-SFTs that is nondecreasing by disjoint unions and products with nonempty systems is constant.   
\end{theorem}
\begin{proof}
    Let $\mathcal I$ be an invariant as in the statement, and suppose that  there are two SFTs $X_0$ and $Y_0$ with $\mathcal I(X_0)<\mathcal I(Y_0)$. Let $q$ be a rational number with $\mathcal I(X_0)<q<\mathcal I(Y_0)$. Given a possibly empty SFT $X$, we define $Z$ by 
    \[Z=X_0\sqcup Y_0\times X.\]
    By our assumptions on $\mathcal I$ we have $\mathcal I(Z)< q$ when $X=\emptyset$ and $\mathcal I(Z)> q$  when $X\ne\emptyset$. Since we can compute a presentation for $Z$ from a presentation for $X$ (\Cref{lem:products} and \Cref{lem:disjoint-unions}), and thanks to our assumption on the computability of $\mathcal I$, given $X$ we can determine in finite time whether $\mathcal I(Z)>q$ or $\mathcal I(Z)<q$. This amounts to determining whether $X=\emptyset$. This contradicts our hypothesis on $\Dp(G)$. Thus $\mathcal I(X_0)=\mathcal I (Y_0)$ and  $\mathcal I$ is constant. 
\end{proof}
It follows that for every amenable group with undecidable domino problem, topological entropy is not computable from presentations. For some groups it is known a much stronger fact, namely, the existence of SFTs whose entropy is a non-computable real number  \cite{barbieri_entropies_2021,bartholdi_shifts_2024}. Our result has a much weaker conclusion, but its proof is remarkably simple and covers many other invariants. For instance, it covers invariants designed for zero-entropy systems that are similar to entropy dimension (see \cite{kanigowski_survey_2020}). Our result also holds if the invariant is only defined for nonempty subshifts or is only assumed to be computable for nonempty subshifts:
\begin{rem}\label{computability-of-invariants-within-a-class}
    Let $\mathscr C$ be a class of $G$-SFTs such that (1) for every $X\in \mathscr C$ and nonempty SFT $Y$ we have $X\times Y\in\mathscr C$, and (2) $\mathscr C$ is closed by disjoint unions. Then  \Cref{invariants-monotone-by-products-and-unions-are-constants} holds within $\mathscr C$:  every dynamical invariant that is defined on $\mathscr C$, is computable in $\mathscr C$, and is  monotone by products and unions on $\mathscr C$, must have constant value on $\mathscr C$. The proof follows the same argument. 

    It follows that a property for $\Z^2$-SFTs that implies the computability of topological entropy can not verify (1) and (2).  
\end{rem}

We finish this section with a result for invariants taking values on abstract partially ordered sets. This result can be applied to recursion-theoretical invariants such as
the Turing degree of the language of the SFT \cite{jeandel_turing_2013},
 Muchnik degrees, and Medvedev degrees  \cite{simpson_medvedev_2014,barbieri_medvedev_2024}.
\begin{thm}
\label{thm:undecidability=000020of=000020monotone=000020invariants}Let $G$ be a finitely generated group with undecidable domino problem. Let $\h$ be a dynamical
invariant for $G$-SFTs taking values in a partially ordered set $(\mathscr{R},\leq)$,
which is non-increasing by factor maps, and whose value is minimal
at the empty subshift. Then for every $r\in\mathscr{R}$ the following
properties of a $G$-SFT $X$ are either trivial or undecidable: $\mathcal{I}(X)\geq r$, $\mathcal{I}(X)\leq r$, $\mathcal{I}(X)>r$, and $\mathcal{I}(X)<r$. 
\end{thm}

\begin{proof}
We consider in detail the property $\mathcal I(X)\leq r$, the other three cases are similar.  Let $\P$ be the property $\mathcal{I}(X)\leq r$ and observe that this property is preserved to factors. If the empty subshift verifies $\P$ then our claim follows  from \Cref{cor:rice-for-monotone-properties}. If the empty subshift does not verify $\P$ then this property is trivial: this follows from the transitivity of $\leq$, and the hypothesis that the value of $\mathcal{I}$ is minimal at the empty subshift.
\end{proof}
\section{Undecidability of dynamical properties of sofic subshifts}

In this section we prove that if $G$ is a group with undecidable
domino problem, then all nontrivial dynamical properties of sofic
$G$-subshifts are undecidable from presentations. 

We start by defining presentations for sofic subshifts. Let $S$ be a finite and symmetric generating set for $G$, and let $\pi\colon S^\ast\to G$ be defined by $\pi(w)=\underline w$. A \textbf{local function presentation} is a function $\mu\colon A^{W}\to B$,
where $W$ is a finite subset of $\subset S^{\ast}$ and $A$ and $B$
are finite subsets of $\N$. The local function $\mu_{0}$ presented
by $\mu$ is defined as follows. We set $D=\pi(W)$, and define   $\mu_{0}\colon A^{D}\to B$
 by $\mu_{0}(p)=\mu(p\circ\pi)$. A \textbf{sofic $G$-subshift presentation} is a tuple $(A,\mathcal{F},\mu,B)$,
where $A,B\subset\N$ are finite alphabets, $(A,\mathcal{F})$ is
a $G$-SFT presentation, and $\mu\colon A^{W}\to B$, $W\subset S^{*}$
is a local function presentation. The sofic subshift associated to
this presentation $Y_{(A,\mathcal{F},\mu,B)}$ is the
image of $X_{(A,\mathcal{F})}$ under the topological factor map with
local function presented by $\mu$. 
\begin{thm}
\label{thm:Rice-theorem-sofic-subshifts} Let $G$ be a finitely generated group with undecidable domino problem.  Then all nontrivial dynamical properties
of sofic $G$-subshifts are undecidable.
\end{thm}

\begin{proof}
Let $\P$ be a nontrivial dynamical property of sofic subshifts. Replacing
$\mathscr{P}$ by its negation if necessary, we can assume that the
empty subshift does not satisfy the property. As $\mathscr{P}$ is
nontrivial, there is a sofic subshift $Y_{+}$ satisfying $\P$. We
fix this subshift for the rest of the proof. We also fix a presentation
$(A_{+},\mathcal{F}_{+},\mu_{+},B_{+})$ for $Y_+$, and also $W_{+}\subset A^{\ast}$
with $\mu_{+}\colon A_{+}^{W_{+}}\to B_{+}$. We define a computable function $f$ whose input is a presentation
$(A,\mathcal{F})$ of a $G$-SFT, and whose output $f(A,\mathcal{F})$
is the presentation of a sofic $G$-subshif such that $Y_{f(A,\mathcal{F})}$ has property
$\P$ if and only if $X_{(A,\mathcal{F})}$ is nonempty. The existence of this function proves that $\mathscr{P}$ is undecidable, as otherwise it could be used to solve the domino problem for $G$.

On input $(A,\mathcal{F})$, the function $f$ starts by computing a presentation $(A',\mathcal{F}')$,
such that $X_{(A',\mathcal{F}')}$ is topologically conjugate to the
direct product $X_{(A_{+},\mathcal{F}_{+})}\times X_{(A,\mathcal{F})}^ {}$. For this we use \Cref{lem:products}. Note that the alphabet $A'$ is equal to $\{\alpha(a_{+},a)\mid a_{+}\in A_{+},\ a\in A\}$, where $\alpha$ was defined in the proof of \Cref{lem:products}, and where we also fixed
computable functions $\pi_{1},\pi_{2}$ satisfying $\pi_{i}(\alpha(n_{1},n_{2}))=n_{i}$,
$i=1,2$. Next, we define a local function presentation $\mu\colon A'{}^{W_{+}}\to B_{+}$
by $p\mapsto\mu(p)=\mu_{0}(\pi_{2}\circ p)$. Finally, the output
of the computable function $f$ is $(A',\mathcal{F}',\mu',B_{+})$. It is clear that $f$ has the mentioned properties.
\end{proof}
In the language of many one reductions we proved that every nontrivial property of sofic subshifts $\P$ that is not verified by the empty subshift satisfies $\Dp(G)\leq_m \P$.
\section{Further remarks}\label{sec:remarks}
The results presented here can be used to show the undecidability of a many dynamical properties of SFTs of common interest in the case of a group with undecidable domino problem. However, we consider that the frontier between decidability and undecidability is rather unclear as we know very little about the decidable region in the ``swamp of undecidability''. That is, we know very few decidable dynamical properties. 

A natural generalization of \Cref{prop:decidable-fixed-points} is as follows. Given an SFT $X$, let $\mathscr E(X)$ be the property of admitting an embedding from $X$.  \Cref{prop:decidable-fixed-points} shows that this property is decidable when $|X|=1$, and the proof can be easily generalized to the case where $X$ is finite. The following question raises naturally:  
\begin{question}
    Is there an infinite SFT $X$ such that $\mathscr E (X)$ is decidable?
\end{question} 
In this work we have focused on properties preserved by conjugacy, but some of the arguments presented here can be adapted to set properties of SFTs. In this context an analogous of $\mathscr E(X)$ is the property $\mathcal C(X)$ of containing $X$ for a fixed SFT $X$. It turns out that the decidability of $\mathcal C(X)$ admits a characterization. The following argument was communicated to us by J.Kari. 
\begin{prop}
    $\mathcal C (X)$ is decidable if and only if the set 
    \[L(X)=\{p\colon W\to A \ \mid \text{$W\subset S^\ast$ is finite and $p$ appears in some $x\in X$}\}\]
    is decidable.
\end{prop}
\begin{proof}
    Let $(A,\mathcal F)$ be a presentation for $X$, and suppose that $\mathcal C(X)$ is decidable. Then for every pattern presentation $p$ we have $p\not\in L(X)$ if and only if $X$ is contained in $X_{(A,\{p\})}$. This proves the forward implication. For the backward implication suppose that $L(X)$ is decidable, and let $(B,\mathcal G)$ be an SFT presentation.  It follows from the definitions that $(B,\mathcal G)$ verifies $\mathcal C(X)$ if and only if no element from $\mathcal G$ appears in $L(X)$. This is decidable by hypothesis and thus the backward implication is proved.    
\end{proof}

%% file: Chapter_translation_like_actions.tex
\renewcommand{\Gamma}{\ensuremath{\Upgamma}}
In this chapter we present some of the results obtained by the author
in \cite{carrasco-vargas_translationlike_2024}. Seward \cite{seward_burnside_2014} proved that all
infinite and finitely generated groups admit a translation-like action
by $\Z$, and that it can be chosen transitive exactly when the group
has either one or two ends. Here we generalize this result to all locally finite graphs, answering
a question left in \cite[Problem 3.5]{seward_burnside_2014}. The
translation-like actions that we construct satisfy the bound $d(v,v\ast1)\leq3$, which
is optimal with these hypotheses. 

\minitoc
\section{Introduction}

A right action $\ast$ of a group $H$ on a metric space $(X,d)$
is called a \textbf{translation-like} \textbf{action} if it is \textbf{free}\footnote{That is, $x\ast h=x$ implies $h=1_{H}$, for $x\in X$, $h\in H$.},
and for each $h\in H$, the set $\{d(x,x\ast h)|\ x\in X\}\subset\mathbb{R}$
is bounded. If $G$ is a finitely generated group endowed with the
left-invariant word metric associated to some finite set of generators,
then the action of any subgroup $H$ on $G$ by right multiplication 
$(g,h)\mapsto gh$ is a translation-like action. On the other hand,
observe that despite the action $H\curvearrowright G$ by left multiplication
is usually referred to as an action by translations, in general it
is not translation-like for a left-invariant word metric. 

Whyte \cite{whyte_amenability_1999}  proposed to consider translation-like actions as a generalization of subgroup
containment, and to replace subgroups by translation-like actions
in different questions or conjectures about groups and subgroups.
This was called a \textit{geometric reformulation}. For example, the von Neumann
Conjecture asserted that a group is nonamenable if and only if it
contains a nonabelian free subgroup. Its geometric reformulation asserts
then that a group is nonamenable if and only if it admits a translation-like
action by a nonabelian free group. While the conjecture was proven
to be false \cite{olshanskij_question_1980}, Kevin Whyte proved that
its geometric reformulation holds \cite{whyte_amenability_1999}. Similarly, Burnside's problem asked if every finitely
generated infinite group contains $\Z$ as a subgroup. While this problem was answered negatively \cite{golod_class_1964}, Seward proved that the geometric reformulation of Burnside's problem holds.
\begin{thm}[Geometric Burnside's problem, \cite{seward_burnside_2014}]
\label{thm:seward} Every finitely generated infinite group admits
a translation-like action by $\Z$. 
\end{thm}

A finitely generated infinite group with two or more ends has a subgroup
isomorphic to $\Z$, by Stalling's structure theorem. Thus, it is
the one ended case that makes necessary the use of translation-like actions. For these groups,  \Cref{thm:seward} was derived from more general result about graphs:
\begin{thm}[{\cite[Theorem 1.6]{seward_burnside_2014}}]
\label{thm:sewardgrafos} Let $\G$ be a connected and infinite graph
whose vertices have uniformly bounded degree. Then $\G$ admits a
transitive translation-like action by $\Z$ if and only if it is connected
and has either one or two ends. 
\end{thm}
The proof of this result relies strongly on
a uniform bound for the degrees of the vertices of the graph. Indeed, the actions constructed in \cite{seward_burnside_2014} have a bound on $d_{\G}(v,v\ast1)$ that depends linearly on a uniform bound for
the degree of the vertices of the graph. Here we strengthen Seward's
results by weakening the hypothesis to the locally finite case, and
improving the bound on $d_{\G}(v,v\ast1)$ to 3. 

\begin{thm}
\label{thm:t1-translation-like-action-transitivas}\label{thm:t2-translation-like-actions}
Every connected, locally finite, and infinite graph admits a translation-like action by $\Z$. Moreover, this action can be chosen transitive exactly when the graph has one or two ends. The actions
constructed satisfy that the distance between a vertex $v$ and $v\ast 1$ is at most 3.
\end{thm}

This result implies both \Cref{thm:seward} and \Cref{thm:sewardgrafos}. Moreover, this result provides a partial answer to \cite[Problem 3.5]{seward_burnside_2014}, which asks for a characterization of graphs admitting a transitive translation-like action
by $\Z$. Thus we have solved the case of locally finite graphs, and
it only remains the case of graphs with vertices of infinite degree. 

These statements about translation-like actions can also be stated
in terms of powers of graphs. Given a graph $\G$, its $n$-th power
$\G^{n}$ the graph with the same set of vertices, and
where two vertices $u,v$ are joined if their distance in $\G$ is
at most $n$. It is well-known that the cube of every finite and connected
graph is Hamiltonian \cite{zbMATH03278210,sekanina1960ordering,karaganis_cube_1968}.
 \Cref{thm:t1-translation-like-action-transitivas} generalizes
this to infinite and locally finite graphs. That is, it shows that
the cube of a locally finite and connected graph with one or two ends
admits a bi-infinite Hamiltonian path. We mention that the part of \Cref{thm:t2-translation-like-actions} that concerns non-transitive actions has been proved in \cite[Section 4]{cohen_strongly_2021}, using the same fact about cubes of finite graphs. 

\section{Preliminaries}\label{sec:Preliminaries}
Here we fix the terminology that will be used along this chapter. In what follows, all graphs are assumed to be undirected and unlabeled. Loops and
multiple edges are allowed. The vertex set of a graph $\G$ will be
denoted by $V(\G)$, and its edge set by $E(\G)$. Each edge \textbf{joins
}a pair of vertices, and is said to be \textbf{incident }to them.
Two vertices joined by an edge are called \textbf{adjacent}. The \textbf{degree}
$\deg_{\G}(v)$ of the vertex $v$ is the number of incident edges
to $v$, where loops are counted twice. A graph is said to be \textbf{finite}
when its edge set is finite, and \textbf{locally finite} when every
vertex has finite degree. 

In our constructions we will constantly consider induced subgraphs.
Given a set of vertices $V\subset V(\G)$, the \textbf{induced subgraph
}$\G[V]$ is the subgraph of $\G$ whose vertex set equals $V$, and
whose edge set is that of all edges in $E(\G)$ whose incident vertices
lie in $V$. On the other hand, $\G-V$ stands for the subgraph of
$\G$ obtained by removing from $\G$ all vertices in $V$, and all
edges incident to vertices in $V$. That is, $\G-V$ equals the induced
subgraph $\G[V(\G)-V]$. If $\L$ is a subgraph of $\G$, we denote
by $\G-\L$ the subgraph $\G-V(\L)$. 

A \textbf{path }on $\G$\textbf{ }is an injective function $f\colon[a,b]\to V(\G)$
that sends consecutive integers to adjacent vertices, where $[a,b]\subset\Z$.
We introduce now some useful terminology for dealing with paths. We
say that $f$ \textbf{joins $f(a)$ }to $f(b)$, and define its \textbf{length
}as $b-a$. We say that $f$ \textbf{visits }the vertices in its image,
and we denote this set by $V(f)$. We denote by $\G-f$ the subgraph
$\G-V(f)$. The vertices $f(a)$ and $f(b)$ are called the \textbf{initial}
and \textbf{final} vertices of $f$, respectively. When every pair
of vertices in the graph $\G$ can be joined by a path, then we say
that $\G$ is \textbf{connected}. In this case we define the \textbf{distance}
between two vertices as the length of the shortest path joining them.
This distance induces the \textbf{path-length }metric on $V(\G)$,
which we denote by $d_{\G}$. 

A \textbf{connected component }of $\G$ is a connected subgraph of
$\G$ which is maximal for the subgraph relation. The \textbf{number
of ends }of $\G$ is the supremum of the number of infinite connected
components of $\G-V$, where $V$ ranges over all finite sets of vertices
in $\G$.

\section{Proof of \Cref{thm:t2-translation-like-actions}}\label{sec:Translation-like-actions-by}

In this section we prove \Cref{thm:t1-translation-like-action-transitivas}. That is, that every
connected, locally finite, and infinite graph admits a translation
by $\Z$, and that this action can be chosen transitive exactly when
the graph has one or two ends. The actions that we construct satisfy
that the distance between a vertex $v$ and $v\ast1$ is at most $3$. Our proof goes by constructing these actions locally, and in terms
of 3-paths: 
\begin{defn}
Let $\G$ be a graph. A \textbf{3-path} on $\G$ is an injective function
$f\colon[a,b]\subset \Z\to V(\G)$ such that consecutive integers in $[a,b]$
are mapped to vertices whose distance is at most 3. A \textbf{bi-infinite
}3-path on $\G$ is an injective function $f\colon\Z\to V(\G)$ satisfying
the same condition on the vertices. A 3-path or bi-infinite 3-path
is called \textbf{Hamiltonian }when it is also a surjective function. 
\end{defn}

It is well known that every finite and connected graph admits a Hamiltonian
3-path, where we can choose its initial and final vertex \cite{zbMATH03278210,sekanina1960ordering,karaganis_cube_1968}.
Here we will need a slight refinement of this fact:
\begin{lem}
\label{lem:lema-tecnico-finite-3-paths} Let $\G$ be a graph that
is connected and finite. For every pair of different vertices $u$
and $v$, $\G$ admits a Hamiltonian $3$-path $f$ which starts at
$u$, ends at $v$, and moreover satisfies the following two conditions:
\begin{enumerate}
\item The first and last ``jump'' have length at most $2$. That is, if
$f$ visits $w$ immediately after the initial vertex $u$, then $d_{\G}(u,w)\leq2$.
Moreover, if $f$ visits $w$ immediately before the final vertex
$v$, then $d_{\G}(w,v)\leq2$. , 
\item There are no consecutive ``jumps'' of length $3$. That is, if $f$
visits $w_{1}$, $w_{2}$ and $w_{3}$ consecutively, then $d_{\G}(w_{1},w_{2})\leq2$
or $d_{\G}(w_{2},w_{3})\leq2$. 
\end{enumerate}
\end{lem}

Let us review some terminology on 3-paths before proving this result.
When dealing with 3-paths, we will use the same terms introduced for
paths in the preliminaries, such as initial vertex, final vertex,
visited vertex, etc. Let $f$ and $g$ be 3-paths. We say that $f$
extends $g$ if its restriction to the domain of $g$ equals $g$.
We will extend 3-paths by concatenation, which we define as follows.
Suppose that final vertex of $f$ is at distance at most 3 from the
initial vertex of $g$, and such that $V(f)\cap V(g)=\emptyset$.
The \textbf{concatenation of $f$, $g$ }is the 3-path that extends
$f$, and after the final vertex of $f$ visits all vertices visited
by $g$ in the same order. Finally, the \textbf{inverse }of the 3-path
$f$, denoted by $-f$, is defined by $(-f)(n)=f(-n)$. The domain of $-f$ is also determined by this expression. 

\begin{proof}[Proof of \Cref{lem:lema-tecnico-finite-3-paths}]
The proof is by induction of the cardinality of $V(\G)$. The claim
clearly holds if $|V(\G)|\leq 2$. Now assume that $\G$ is a connected
finite graph with $|V(\G)|\geq 3$, and let $u$ and $v$ be two different vertices. We
consider the connected components of the graph $\G-\{v\}$ obtained
by removing the vertex $v$ from $\G$. Let $\G_{u}$ be the finite
connected component of $\G-\{v\}$ that contains $u$, and let $\G_{v}$
be the subgraph of $\G$ induced by the set of vertices $V(\G)-V(\G_{u})$.
Thus $u\in\G_{u}$, $v\in\G_{v}$, and both $\G_{u}$ and $\G_{v}$
are connected. Let us first assume that both $\G_{u}$ and $\G_{v}$
are graphs with at least two vertices. Then we can apply the inductive
hypothesis on each one of them. Let $f_{u}$ be a Hamiltonian 3-path
on $\G_{u}$ as in the statement, whose initial vertex is $u$, and
whose final vertex $u'$ is at distance to $v$ at most 2. 
Let $f_{v}$ be a Hamiltonian
3-path on $\G_{v}$ as in the statement, whose initial vertex $v'$ is adjacent
to $v$, and whose final vertex is $v$.  
We claim that the 3-path $f$ obtained by concatenating $f_{u}$,
$f_{v}$ verifies the required conditions. It is clear that  $d_{\Gamma}(u',v')\leq 3$, and thus $f$ is a 3-path. It is also clear that $f$ verifies the first condition in the statement. Regarding the second condition, it suffices to show that the ``jump'' from $u'$ to $v'$ is between two ``jumps'' with length at most two. That is, that the vertex visited by $f$ before $u'$ is at distance at most 2 from $u'$, and that the vertex visited by $f$ after $v'$ is at distance at most $2$ from $v'$. Indeed, this follows from the fact that both $f_u$ and $f_v$ verify the first condition in the statement. This finishes the argument in the case that both $\G_u$ and $\G_v$ have at least two vertices. If $\G_u$ has one vertex and $\G_v$ has at least two vertices then we modify the previous procedure by redefining $f_u$ as the 3-path that only visits $u$. It is easy to verify that then the concatenation of $f_u$ and $f_v$ is a 3-path in $\G$ verifying the two numbered conditions. The case where $\G_v$ has one vertex and $\G_u$ has at least two vertices is symmetric, and the case where both $\G_u$ and $\G_v$ have one vertex is excluded since we assumed $|V(G)|\geq 3$.
\end{proof}
We will define bi-infinite 3-paths by extending finite ones iteratively.
The following definition will be key for this purpose:
\begin{defn}
\label{def:bi-right-extensible}Let $f$ be a 3-path on a graph $\G$.
We say that $f$ is \textbf{bi-extensible} if the following conditions
are satisfied:
\begin{enumerate}
\item $\G-f$ has no finite connected component.
\item There is a vertex $u$ in $\G-f$ at distance at most 3 from the final
vertex of $f$. 
\item There is a vertex $v\ne u$ in $\G-f$ at distance at most 3 from
the initial vertex of $f$. 
\end{enumerate}
If only the first two conditions are satisfied, we say that $f$ is
\textbf{right-extensible}.
\end{defn}

We will now prove some elementary facts about the existence of 3-paths
that are bi-extensible and right-extensible. 
\begin{lem}
\label{lem:existencia-right-extensible-1}Let $\G$ be a graph that
is infinite, connected, and locally finite. Then for every pair of
vertices $u$ and $v$ in $\G$, there is a right-extensible 3-path
whose initial vertex is $u$, and which visits $v$. 
\end{lem}

\begin{proof}
As $\G$ is connected, there is a path $f$ joining $u$ to $v$.
Now define $\L$ as the graph induced in $\G$ by the set of vertices
that are visited by $f$, or that lie in a finite connected component
of $\G-f$. Notice that as $\G$ is locally finite, there are finitely
many such connected components, and thus $\L$ is a finite and connected
graph. By construction, $\G-\L$ has no finite connected component. 

The desired 3-path will be obtained as a Hamiltonian 3-path on $\L$.
Indeed, as $\G$ is connected, there is a vertex $w$ in $\L$ that
is adjacent to some vertex in $\G-\L$. By \Cref{lem:lema-tecnico-finite-3-paths}
there is a 3-path $f'$ which is Hamiltonian on $\L$, starts at $u$
and ends in $w$. We claim that $f'$ is right-extensible. Indeed,
our choice of $\L$ ensures that $\G-f'$ has no finite connected
component, and our choice of $w$ ensures that the final vertex of
$f'$ is adjacent to a vertex in $\G-f'$. 
\end{proof}
\begin{lem}
\label{lem:existencia-bi-extensible} Let $\G$ be a graph that is
infinite, connected, and locally finite. Then for every vertex $u$
in $\G$, there is a bi-extensible 3-path in $\G$ that visits $u$. 
\end{lem}

\begin{proof}
Let $v$ be a vertex in $\G$ that is adjacent to $u$, with $v\neq u$.
Let $\L$ be the subgraph of $\G$ induced by the set of vertices
that lie in a finite connected component of $\G-\{u,v\}$, or in $\{u,v\}$.
As $\G$ is locally finite, there are finitely many such connected
components, and thus $\L$ is a finite and connected subgraph of $\G$.
By construction, $\G-\L$ has no finite connected component. 

The desired 3-path will be obtained as a Hamiltonian 3-path on $\L$.
Indeed, as $\G$ is connected there are two vertices $w\in V(\G-\L)$
and $w'\in V(\L)$, with $w$ adjacent to $w'$ in $\G$. As $\L$
has at least two vertices, we can invoke \Cref{lem:lema-tecnico-finite-3-paths}
to obtain a 3-path $f$ that is Hamiltonian on $\L$, whose initial
vertex is $w'$, and whose final vertex is adjacent to $w'$. It is
clear that then $f$ is a bi-extensible 3-path in $\G$. 
\end{proof}
Our main tool to construct bi-infinite 3-paths is the following result,
which shows that bi-extensible 3-paths can be extended to larger bi-extensible
3-paths.
\begin{lem}
\label{lem:extend-3-paths}Let $\G$ be a graph that is infinite,
connected, and locally finite. Let $f$ be a bi-extensible 3-path
on $\G$, and let $u$ and $v$ be two different vertices in $\G-f$
whose distance to the initial and final vertex of $f$ is at most
3, respectively. Let $w$ be a vertex in the same connected component
of $\G-f$ that some of $u$ or $v$. Then there is a 3-path $f'$
which extends $f$, is bi-extensible on $\G$, and visits $w$. Moreover,
we can assume that the domain of $f'$ extends that of $f$ in both
directions.
\end{lem}

\begin{proof}
If $u$ and $v$ lie in different connected components of $\G-f$,
then then the claim is easily obtained by applying \Cref{lem:existencia-right-extensible-1}
on each of these components. Indeed, by \Cref{lem:existencia-right-extensible-1}
there are two right-extensible 3-paths $g$ and $h$ in the corresponding
connected components of $\G-f$, such that the initial vertex of $g$
is $u$, the initial vertex of $h$ is $v$, and some of them visits
$w$. Then the concatenation of $-g$, $f$ and $h$ satisfies the
desired conditions. 

We now consider the case where $u$ and $v$ lie in the same connected
component of $\G-f$. This graph will be denoted $\L$. Note that
$\L$ is infinite because $f$ is bi-extensible. We claim that there
are two right-extensible 3-paths on $\L$, $g$ and $h$, satisfying
the following list of conditions: the initial vertex of $g$ is $u$,
the initial vertex of $h$ is $v$, some of them visits $w$, and
$V(g)\cap V(h)=\emptyset$. In addition, $(\L-g)-h$ has no finite
connected component, and has two different vertices $u'$ and $v'$
such that $u'$ is at distance at most 3 from the last vertex of $g$,
and $v'$ is at distance at most 3 from the last vertex of $h$. Suppose
that we have $g,h$ as before. Then we can define a 3-path $f'$ by
concatenating $-g$, $f$ and then $h$. It is clear that then $f'$
satisfies the conditions in the statement. 

We now construct $g$ and $h$. We start by taking a connected finite
subgraph $\L_{0}$ of $\L$ which contains $u,v,w$ and such that
$\L-\L_{0}$ has no finite connected component. The graph $\L_{0}$
can be obtained, for instance, as follows. As $\Lambda$ is connected,
we can take a path $f_{u}$ from $u$ to $w$, and a path $f_{v}$
from $v$ to $w$. Then define $\L_{0}$ as the graph induced by the
vertices in $V(f_{v})$, $V(f_{u})$, and all vertices in the finite
connected components of $(\L-f_{v})-f_{u}$. 

Let $p$ be a Hamiltonian 3-path on $\L_{0}$ from $u$ to $v$, as
in \Cref{lem:lema-tecnico-finite-3-paths}. The desired 3-paths
$f$ and $g$ will be obtained by ``splitting'' $p$ in two. As
$\L$ is connected, there are two vertices $u_{0}\in V(\L_{0})$,
$v'\in V(\L-\L_{0})$ such that $u_{0}$ and $v$ are adjacent in
$\L$. By the conditions in \Cref{lem:lema-tecnico-finite-3-paths},
there is a vertex $v_{0}$ in $V(\L_{0})$ whose distance from $u_{0}$
is at most 2, and such that $p$ visits consecutively $\{u_{0},v_{0}\}$.
We will assume that $p$ visits $v_{0}$ after visiting $u_{0}$,
the other case being symmetric. As $\L-\L_{0}$ has no finite connected
component, there is a vertex $u'$ in $\L-\L_{0}$ that is adjacent
to $v'$. Thus, $u_{0}$ is at distance at most 2 from $u'$, and
$v_{0}$ is at distance at most 3 from $v'$. Now we define $g$ and
$h$ by splitting $p$ at the vertex $u_{0}$. More precisely, let
$[a,c]$ be the domain of $p$, and let $b$ be such that $p(b)=u_{0}$.
Then $h$ is defined as the restriction of $p$ to $[a,b]$, and we
define $g$ by requiring $-g$ to be the restriction of $p$ to $[b+1,c].$
Thus $h$ is a 3-path from $v$ to $v_{0}$, and $g$ is a 3-path
from $u$ to $u_{0}$. By our choice of $\L_{0}$ and $p$, the 3-paths
$h$ and $g$ satisfy the mentioned list of conditions, and thus the
proof is finished. 
\end{proof}
When the graph has one or two ends, the hypotheses of \Cref{lem:extend-3-paths}
on $u,v$ and $w$ are trivially satisfied. We obtain a very simple
and convenient statement: we can extend a bi-extensible 3-path so
that it visits a vertex of our choice. 
\begin{cor}
\label{cor:one-two-ends}Let $\G$ be a graph that is infinite, connected,
locally finite, and whose number of ends is either $1$ or $2$. Then
for every bi-extensible 3-path $f$ and vertex $w$, there is a bi-extensible
3-path on $\G$ that extends $f$ and visits $w$. We can assume that
the domain of the new 3-path extends that of $f$ in both directions. 
\end{cor}

We are now in position to prove some results about bi-infinite 3-paths.
We start with the Hamiltonian case, which is obtained by iterating
\Cref{cor:one-two-ends}. When we deal with bi-infinite 3-paths,
we use the same notation and abbreviations introduced before for 3-paths,
as long as they are well defined. 
\begin{prop}
\label{prop:Hamiltonian-infinite-3paths}Let $\G$ be a graph that
is infinite, connected, locally finite, and whose number of ends is
either $1$ or $2$. Then $\G$ admits a bi-infinite Hamiltonian 3-path. 
\end{prop}

\begin{proof}
Let $(v_{n})_{n\in\N}$ be a numbering of the vertex set of $\G$.
We define a sequence of bi-extensible 3-paths $(f_{n})_{n\in\N}$
on $\G$ recursively. We define $f_{0}$ as a bi-extensible 3-path
which visits $v_{0}$. The existence of $f_{0}$ is guaranteed by
\Cref{lem:existencia-bi-extensible}. Now let $n\geq0$, and
assume that we have defined a 3-path $f_{n}$ that visits $v_{n}$.
We define $f_{n+1}$ as a bi-extensible 3-path on $\G$ which extends
$f_{n}$, visits $v_{n+1}$, and whose domain extends the domain of
$f_{n}$ in both directions. The existence such a 3-path is guaranteed
by \Cref{cor:one-two-ends}. We have obtained a sequence $(f_{n})_{n\in\N}$
such that for all $n$, $f_{n}$ visits $v_{n}$, and $f_{n+1}$ extends
$f_{n}$. With this sequence we define a bi-infinite 3-path $f\colon\Z\to V(\G)$
by setting $f(k)=f_{n}(k)$, for $n$ big enough. Note that $f$ is
well defined because $f_{n+1}$ extends $f_{n}$ as a function, and
the domains of $f_{n}$ exhaust $\Z$. By construction, $f$ visits
every vertex exactly once, and thus it is Hamiltonian. 
\end{proof}
We now proceed with the non Hamiltonian case, where there are no restrictions
on ends. We first prove that we can take a bi-infinite 3-path whose
deletion leaves no finite connected component. 

\begin{lem}
\label{lem:bi-infinte-3-path-removable}Let $\G$ be a graph that
is infinite, connected, and locally finite. Then for every vertex
$v$, there is a bi-infinite 3-path $f$ that visits $v$, and such
that $\G-f$ has no finite connected component. 
\end{lem}

\begin{proof}
By \Cref{lem:existencia-bi-extensible} and \Cref{lem:extend-3-paths},
$\G$ admits a sequence $(f_{n})_{n\in\N}$ of bi-extensible 3-paths
such that $f_{0}$ visits $v$, $f_{n+1}$ extends $f_{n}$ for all
$n\geq0$, and such that their domains exhaust $\Z$. We define a
bi-infinite 3-path $f\colon\Z\to\G$ by setting $f(k)=f_{n}(k)$,
for $n$ big enough. We claim that $f$ satisfies the condition in
the statement, that is, that $\G-f$ has no finite connected component.
We argue by contradiction. Suppose that $\G_{0}$ is a nonempty and
finite connected component of $\G-f$. Define $V_{1}$ as the set
of vertices in $\G$ that are adjacent to some vertex in $\G_{0}$,
but which are not in $\G_{0}$. Then $V_{1}$ is nonempty as otherwise
$\G$ would not be connected, and it is finite because $\G$ is locally
finite. Moreover, $f$ visits all vertices in $V_{1}$, for otherwise
$\G_{0}$ would not be a connected component of $\G-f$. As $V_{1}$
is finite, there is a natural number $n_{1}$ such that $f_{n_{1}}$
has visited all vertices in $V_{1}$. By our choice of $V_{1}$ and
$n_{1}$, $\G_{0}$ is a nonempty and finite connected component of
$\G-f_{n_{1}}$, and this contradicts the fact that $f_{n_{1}}$ is
bi-extensible.
\end{proof}
Now the proof of the following result is by iteration of \Cref{lem:bi-infinte-3-path-removable}. 
\begin{prop}
\label{prop:existencia-3-paths}Let $\G$ be a graph that is infinite,
connected, and locally finite. Then there is a collection of bi-infinite
3-paths $f_{i}\colon\Z\to\G$, $i\in I$, such that $V(\G)$ is the
disjoint union of $V(f_{i})$, $i\in I$. 
\end{prop}

\begin{proof}
By \Cref{lem:bi-infinte-3-path-removable}, $\G$ admits a bi-infinite
3-path $f_{0}$ such that $\G-f_{0}$ has no finite connected component.
Each connected component of $\G-f_{0}$ is infinite, and satisfies
the hypotheses of \Cref{lem:bi-infinte-3-path-removable}. Thus
we can apply \Cref{lem:bi-infinte-3-path-removable} on each
of these connected components. Iterating this process in a tree-like
manner, we obtain a family of 3-paths $f_{i}\colon\Z\to\G$, $i\in I$
whose vertex sets $V(f_{i})$ are disjoint. As \Cref{lem:bi-infinte-3-path-removable}
allows us to choose a vertex to be visited by the bi-infinite 3-path,
we can choose $f_{i}$ ensuring that every vertex of $\G$ is visited
by some $f_{i}$. In this manner, $V(\G)$ is the disjoint union of
$V(f_{i})$, ranging $i\in I$.
\end{proof}
We are now ready to prove \Cref{thm:t1-translation-like-action-transitivas}.

\begin{proof}[Proof of \Cref{thm:t1-translation-like-action-transitivas}]
Let $\G$ be a connected, infinite, and locally finite graph. By \Cref{prop:existencia-3-paths}, there is a collection of bi-infinite 3-paths $f_i\colon\Z\to\Gamma$, $i\in I$, such that $V(\G)$ is the disjoint union of $V(f_i)$, $i\in I$. We define a translation-like
$\ast\colon V(\G)\times\Z\to V(\G)$ by setting $v\ast n=f_i(f_i^{-1}(v)+n),\ n\in\Z$, where $v\in V(f_i)$.  This proves that $\G$ admits a translation-like action by $\Z$. Furthermore, if $\G$ has either one or two ends, then we can choose a Hamiltonian bi-infinite 3-path  $f\colon\Z\to\Gamma$ by \Cref{prop:Hamiltonian-infinite-3paths}, and define a translation-like action in the same manner. This action is transitive because the 3-path is Hamiltonian.  

It only remains to prove that a connected and locally finite graph which admits a transitive translation-like
action by $\Z$ must have either one or two ends. This is stated in
\cite[Theorem 3.3]{seward_burnside_2014} for graphs with uniformly
bounded vertex degree, but the same proof can be applied to locally
finite graphs. For completeness, we provide an alternative argument.
Let $\G$ be a connected and locally finite graph which admits a transitive
translation-like action by $\Z$, denoted $\ast$. As the action is
free, $V(\G)$ must be infinite, and thus $\G$ has at least $1$
end. Suppose now that it has at least 3 ends to obtain a contradiction.
Let $J=\max\{d_{\G}(v,v\ast1)\mid v\in V(\G)\}$. As $\G$ has at
least $3$ ends, there is a finite set of vertices $V_{0}$ such that
$\G-V_{0}$ has at least three infinite connected components, which
we denote by $\G_{1},$ $\G_{2}$ and $\G_{3}$. By enlarging $V_{0}$
if necessary, we can assume that any pair of vertices $u$ and $v$
that lie in different connected components in $\G-V_{0}$, are at
distance $d_{\G}$ at least $J+1$. Now as $V_{0}$ is finite, there
are two integers $n\leq m$ such that $V_{0}$ is contained in $\{v\ast k\mid n\leq k\leq m\}$.
By our choice of $V_{0}$, it follows that the set $\{v\ast k\mid k\geq m+1\}$
is completely contained in one of $\G_{1}$, $\G_{2}$, or $\G_{3}$.
The same holds for $\{v\ast k\mid k\leq n-1\}$, and thus one of $\G_{1}$,
$\G_{2}$, or $\G_{3}$ must be empty, a contradiction. 
\end{proof}
\begin{rem}
The proof given in this section is closely related to the characterization
of those infinite graphs that admit infinite Eulerian paths. This
is a theorem of Erdős, Grünwald, and Weiszfeld \cite{erdos_eulerian_1936}.
In the recent work \cite{carrasco-vargas_infinite_2024}, the author
of this work gave a different proof of the Erd\H os, Grünwald, and
Weiszfeld theorem, that complements the original result by also characterizing
those finite paths that can be extended to infinite Eulerian ones.
This characterization is very similar to the notion of bi-extensible
defined here. Indeed, the proofs of \Cref{prop:Hamiltonian-infinite-3paths}
and the proof of the mentioned result about Eulerian paths follow
the same iterative construction. 
\end{rem}

\begin{rem}
As we mentioned before, it is known that the cube of every finite
and connected graph is Hamiltonian \cite{zbMATH03278210,sekanina1960ordering,karaganis_cube_1968}.
\Cref{prop:Hamiltonian-infinite-3paths} can be considered as
a generalization of this fact to locally finite graphs. That is, \Cref{prop:Hamiltonian-infinite-3paths}
shows that the cube of every locally finite and connected graph with
either $1$ or $2$ ends admits a bi-infinite Hamiltonian path. 
\end{rem}

We end this section by rephrasing a problem left in \cite[Problem 3.5]{seward_burnside_2014}.
\begin{problem}
Find necessary and sufficient conditions for a connected graph to
admit a transitive translation-like action by $\Z$. 
\end{problem}

We have shown that for locally finite graphs, the answer to this problem
is as simple as possible, involving only the number of ends of the
graph. The problem is now open for graphs that are not locally finite.
We observe that beyond locally finite graphs there are different and
non-equivalent notions of ends \cite{diestel_graphtheoretical_2003},
and thus answering the problem above also requires to determine which
is the appropriate notion of ends.

%% file: Chapter_computable_translation_like_actions.tex
In this chapter we consider the problem of effectively computing translation-like actions by $\Z$ on groups and graphs. This chapter is a continuation of \Cref{chap:translation-like-actions}, and is based on the work \cite{carrasco-vargas_translationlike_2024}. 

\minitoc
\section{Introduction}
A graph is computable
if there exists an algorithm which given two vertices, determines
whether they are adjacent or not. If moreover the graph is locally
finite, and the function that maps a vertex to its degree is computable,
then the graph is said to be highly computable. This extra condition
is necessary to compute the neighborhood of a vertex. The notion of highly computable graph corresponds to the intuition that we can compute neighborhoods in the graph of any desired size.  An important example comes from group theory: if $G$ is a finitely
generated group with decidable word problem and $S$ is a finite set
of generators, then its Cayley graph with respect to $S$ is highly
computable.

Many basic results about infinite graphs are not \textit{effective} for highly computable graphs. For instance, König's infinity lemma asserts that every
infinite, connected, and locally finite graph admits an infinite path.
However, there are highly computable graphs which admit paths, but none of them is computable \cite{jockusch_pi_1972}. Another example
comes from Hall's matching theorem. There are highly computable graphs
satisfying the hypotheses in the theorem, but which admit no computable
right perfect matching \cite{manaster_effective_1972}. These two
results are used in the proof \cite[Theorem 1.4]{seward_burnside_2014}, so the translation-like
actions from this proof are not clearly computable. We say that a
translation-like action by $\Z$ on a graph is computable when there
is an algorithm which given a vertex $v$ and $n\in\Z$, computes
the vertex $v\ast n$. %

Our interest in the computability of translation-like actions is motivated by the results in \Cref{chap:Medvedev}, where we need computable translation-like actions where in addition it is
possible to distinguish in a computable manner between different orbits. 
We introduce here a general definition, though in this chapter we will only treat
the case where the acting group is $\Z$. 
\begin{defn}
\label{def:orbit-membership-problem}Let $G$ be a group, and let
$S\subset G$ be a finite set of generators. A group action of $H$
on $G$ is said to have \textbf{decidable} \textbf{orbit membership
problem} if there exists an algorithm which given two words $u$ and
$v$ in $(S\cup S^{-1})^{*}$, decides whether the corresponding group
elements $u_{G},v_{G}$ lie in the same orbit under the action. 
\end{defn}

Note that if $H$ is a subgroup of $G$, then the action $H\curvearrowright G$
by right translations has decidable orbit membership problem if and
only if $H$ has decidable subgroup membership problem (\Cref{prop:obvio}).
Thus this property can be regarded as the geometric reformulation,
in the sense of Whyte \cite{whyte_amenability_1999}, of the subgroup
property of having decidable membership problem. The orbit membership
problem has been studied for some actions by conjugacy and by group
automorphisms, see \cite{bogopolski_orbit_2009,burillo_conjugacy_2016,ventura_grouptheoretic_2014}
and references therein.

{} 

Our main result associated to computable translation-like actions
on groups is the following. 

\begin{thm}
\label{thm:computable-translation-like-actions-with-decidable-orbit-problem}\label{seward-calculable}Let
$G$ be a finitely generated infinite group with decidable word problem.
Then $G$ admits a translation-like action by $\Z$ that is computable
and has decidable orbit membership problem.
\end{thm}
The proof of \Cref{thm:computable-translation-like-actions-with-decidable-orbit-problem}
is as follows. For groups with at most two ends, we prove the existence
of a computable and transitive translation-like action. For this, we prove that the part of \Cref{thm:t1-translation-like-action-transitivas} that concerns transitive translation-like actions is effective for highly computable graphs. For groups with more than two ends, we prove the existence of a subgroup
isomorphic to $\Z$ and with decidable subgroup membership problem.
Thus for groups with two ends we provide two different proofs for
\Cref{thm:computable-translation-like-actions-with-decidable-orbit-problem}.
A group with two ends is virtually $\Z$, and it would be easy to
give a direct proof, but the intermediate statements may have independent
interest (\Cref{thm:computable-t1} and \Cref{prop:stallings-descomposicion-calculable}). 

\section{Preliminaries}
In this chapter we assume that the reader is familiar with \Cref{chap:translation-like-actions}. Moreover, we refer the reader to \Cref{chap:Background} for details about numberings and computability. We will consider the following computability notions for graphs: 

\begin{defn}
An undirected and simple graph $\G$ is \textbf{computable} if we can endow $V(\G)$ and $E(\G)$
with numberings, in such a manner that the relation of adjacency,
and the relation $\{(e,u,v)\mid e\text{ joins }u\text{ and }v\}$
are decidable. We say that $\G$ is also \textbf{highly computable}
when it is locally finite, and the vertex degree function $V(\G)\to\N$,
$v\mapsto\deg_{\G}(v)$ is computable. 
\end{defn}

We now review Cayley graphs of finitely generated groups. Let $G$ be finitely generated by $S$. We denote by $\cay(G,S)$
the (undirected, and right)\textbf{ Cayley graph} of $G$ relative
to $S$. The vertex set of $\cay(G,S)$ is $G$, and two vertices $g$ and $h$ are joined by an edge when $g=hs$ for some $s\in S\cup S^{-1}$. Note that
the distance that this graph assigns to a pair of elements in $G$
equals their distance in the left-invariant word metric associated to the same generating
set. The \textbf{number of ends }of a finitely generated group is
the number of ends of its Cayley graph, for any generating set. This
definition does not depend on the chosen generating set, and can only
be among the numbers $\{0,1,2,\infty\}$ \cite{freudenthal_ueber_1945,hopf_enden_1944}. 
We will need the following well-known fact, whose proof is left to the reader (see \Cref{prop:computabilitiy-groups-1}).
\begin{prop}
\label{prop:cayley-graph-of-computable-group-is-computable} 
Let $G$
be a finitely generated group with decidable word problem, and let
$S$ be a finite generating set. Then $\cay(G,S)$ is a highly computable
graph. 
\end{prop}
\section{Computable and transitive translation-like actions by $\protect\Z$}\label{sec:Computable-translation-like-acti}
In this section we prove that the part of \Cref{thm:t1-translation-like-action-transitivas} that concerns graphs with one or two ends, is effective for highly computable graphs:

\begin{thm}\label{thm:computable-t1} Let $\G$ be a graph that is highly computable,
connected, and has either $1$ or $2$ ends. Then $\G$ admits a computable
and transitive translation-like action by $\Z$, where the distance
between a vertex $v$ and $v\ast1$ is uniformly bounded by 3.
\end{thm}

We start by proving that the bi-extensible property (\Cref{def:bi-right-extensible})
is algorithmically decidable on highly computable graphs with one
end.
\begin{prop}
\label{prop:bi-extensible-decidible-1-end}Let $\G$ be a graph that
is highly computable, connected, and has one end. Then it is algorithmically
decidable whether a 3-path $f$ is bi-extensible. 
\end{prop}

\begin{proof}
It is clear that the second and third conditions in the definition
of bi-extensible are algorithmically decidable. For the first condition,
note that as $\G$ has one end, we can equivalently check whether
$\G-f$ is connected. This is proved to be a decidable problem in
\cite{carrasco-vargas_infinite_2024}, Lemma 5.6. Note that the mentioned
result concerns the remotion of edges instead of vertices, but indeed
this is stronger: given $f$, we compute the set $E$ of all edges
incident to a vertex in $V(f)$, and then use \cite[Lemma 5.6]{carrasco-vargas_infinite_2024}
with input $E$.
\end{proof}
For graphs with two ends we prove a similar result, but we need the extra assumption that the input 3-path leaves two infinite connected components. 
\begin{prop}
\label{prop:bi-extensible-decidible-2-end}Let $\G$ be a graph that
is highly computable, connected and has two ends. Let $f_{0}$ be
a bi-extensible 3-path on $\G$, such that $\G-f_{0}$ has two infinite
connected components. Then there is an algorithm that on input a 3-path
$f$ that extends $f_{0}$, decides whether $f$ is bi-extensible. 
\end{prop}

\begin{proof}
It is clear that the second and third conditions in the definition
of bi-extensible are algorithmically decidable. We address the first
condition. We prove the existence of a procedure that, given a 3-path
$f$ as in the statement, decides whether $\G-f$ has no finite connected
component.  In \cite[Lemma 5.5]{carrasco-vargas_infinite_2024}
there is an effective procedure that halts if and only if $\G-f$
has some finite connected component (the mentioned result mentions
the remotion of edges instead of vertices, but indeed this is stronger:
given $f$, we compute the set $E$ of all edges incident to a vertex
in $V(f)$, and then use \cite[Lemma 5.5]{carrasco-vargas_infinite_2024}
with input $E$).

Thus we need an effective procedure that halts if and only if $\G-f$
has no finite connected component. As $f$ extends $f_{0}$, this
is equivalent to ask whether $\G-f$ has at most two connected components.
The procedure is as follows: given $f$, we start by computing the
set $V_{0}$ of vertices in $\G-f$ that are adjacent to a vertex
visited by $f$. Then for every pair of vertices in $u,v\in V_{0}$,
we search exhaustively for a path that that joins them, and that never
visits vertices in $V(f)$. That is, a path in $\G-f$. Such a path
will be found if and only if the connected component of $\G-f$ that
contains $u$ equals the one that contain $v.$ We stop the procedure
once we have found enough paths to write $V_{0}$ as the disjoint
union $V_{1}\sqcup V_{2}$, where every pair of vertices in $V_{1}$
(resp. $V_{2}$) is joined by a path as described. 
\end{proof}
We can now show a computable version of \Cref{prop:Hamiltonian-infinite-3paths}. 
\begin{prop}[Computable \Cref{prop:Hamiltonian-infinite-3paths}]
\label{prop:hamiltonian-3-path-calculable}Let $\G$ be a graph that
is highly computable, connected, and has either $1$ or $2$ ends.
Then it admits a bi-infinite Hamiltonian 3-path which is computable.
\end{prop}

\begin{proof}
Let $(v_{i})_{i\in\N}$ be a numbering of the vertex set of the highly
computable graph $\G$. Now let $f_{0}$ be a 3-path which is bi-extensible
and visits $v_{0}$. If $\G$ has two ends, then we also require that
$\G-f_{0}$ has two infinite connected components. Observe that we are not claiming that  $f_{0}$ can be computed from a description
of the graph, but it exists and can be specified with finite information.
After fixing $f_{0}$, we just follow the proof of \Cref{prop:Hamiltonian-infinite-3paths},
and observe that a sequence of 3-paths $(f_{n})_{n\in\N}$ as in this
proof can be uniformly computed. That is, there is an algorithm which
given $n$, computes $f_{n}$. The algorithm proceeds recursively:
assuming that $(f_{i})_{i\leq n}$ have been computed, we can compute
$f_{n+1}$ by an exhaustive search. The search is guaranteed to stop,
and the conditions that we impose on $f_{n+1}$ are decidable thanks
to Propositions \ref{prop:bi-extensible-decidible-1-end} and \ref{prop:bi-extensible-decidible-2-end}.
Finally, let $f\colon\Z\to V(\G)$ be the Hamiltonian 3-path on $\G$
defined by $f(k)=f_{n}(k)$, for $n$ big enough. Then it is clear
that the computability of $(f_{n})_{n\in\N}$ implies that $f$ is
computable. 
\end{proof}
Now, we are ready to prove \Cref{thm:computable-t1}.

\begin{proof}[Proof of \Cref{thm:computable-t1}]
 Let $\G$ be as in the statement. By \Cref{prop:hamiltonian-3-path-calculable},
$\G$ admits a bi-infinite Hamiltonian 3-path $f\colon\Z\to V(\G)$
that is computable. Then it is clear that the translation-like action
$\ast\colon V(\G)\times\Z\to V(\G)$ defined by $v\ast n=f(f^{-1}(v)+n),\ n\in\Z$,
is computable.
\end{proof}
This readily implies \Cref{thm:computable-translation-like-actions-with-decidable-orbit-problem}
for groups with one or two ends. 
\begin{proof}[Proof of \Cref{thm:computable-translation-like-actions-with-decidable-orbit-problem}
for groups with one or two ends]
 Let $G$ be a finitely generated infinite group with one or two
ends, and with decidable word problem. Let $S\subset G$ be a finite
set of generators, and let $\G=\cay(G,S)$ be the associated Cayley
graph. As $G$ has decidable word problem, this is a highly computable
graph (\Cref{prop:cayley-graph-of-computable-group-is-computable}).
Then by \Cref{thm:computable-t1}, $\G$ admits a computable
and transitive translation-like action by $\Z$. As the vertex set
of $\G$ is $G$, this is also a computable and transitive translation-like
action on $G$. This action has decidable orbit membership problem
for the trivial reason that it has only one orbit. 
\end{proof}
\begin{rem}
As mentioned in the introduction, there is a number of results in
the theory of infinite graphs that can not have an effective counterpart
for highly computable graphs. In contrast, we have the following consequences
of Theorems \ref{thm:t1-translation-like-action-transitivas} and
\ref{thm:computable-t1}: 
\begin{enumerate}
\item A highly computable graph admits a transitive translation-like action
by $\Z$ if and only if it admits a computable one.
\item A group with decidable word problem has a Cayley graph with a bi-infinite
Hamiltonian path if and only if it has a Cayley graph with a computable
bi-infinite Hamiltonian path. 
\item The cube of a highly computable graph admits a bi-infinite Hamiltonian
path if and only if it admits a computable one.
\end{enumerate}
The third item should be compared with the following result of D.Bean:
there is a graph that is highly computable and admits infinite Hamiltonian
paths, but only uncomputable ones \cite{bean_recursive_1976}. Thus,
the third item shows that for graphs that are cubes, it is algorithmically
easier to compute infinite Hamiltonian paths. 

It follows from our results that the problem of \emph{deciding} whether
a graph admits a bi-infinite Hamiltonian path is also algorithmically
easier when we restrict ourselves to graphs that are cubes. D.Harel
proved that the problem of Hamiltonicity is analytic-complete for
highly computable graphs \cite[Theorem 2]{harel_hamiltonian_1991}.
On the other hand, it follows from \Cref{thm:t1-translation-like-action-transitivas}
that for graphs that are cubes, it suffices to check that the graph
is connected, and has either $1$ or $2$ ends. These conditions are
undecidable, but are easily seen to be arithmetical \cite{inproceedings}.
In view of these results, it is natural to ask if these problems are
easier when we restrict ourselves to graphs that are squares:
\end{rem}

\begin{question}
The problem of computing infinite Hamiltonian paths (resp. deciding
whether an infinite graph is Hamiltonian) on highly computable graphs,
is easier when we restrict to graphs that are squares?
\end{question}

\section{Computable normal forms and Stalling's theorem}

Here we prove \Cref{thm:computable-translation-like-actions-with-decidable-orbit-problem}
for groups with two or more ends. It follows from Stalling's structure
therem on ends of groups that a group with two or more ends has a
subgroup isomorphic to $\Z$. We will prove that, if the group has
solvable word problem, then this subgroup has decidable membership
problem. This will be obtained from normal forms associated to HNN
extensions and amalgamated products. We now recall well known facts
about these constructions, the reader is referred to \cite[Chapter IV]{lyndon_combinatorial_2001}. 

HNN extensions are defined from a group $H=\langle S_{H}\mid R_{H}\rangle$,
a symbol $t$ not in $S_{H}$, and an isomorphism $\phi\colon A\to B$
between subgroups of $H$. The \textbf{HNN extension} relative to
$H$ and $\phi$ is the group with presentation $H\ast_{\phi}=\langle S_{H},t\mid\ R_{H},\ tat^{-1}=\phi(a),\ \forall a\in A\rangle.$
Now let $T_{A}\subset H$ and $T_{B}\subset H$ be sets of representatives
for equivalence classes of $H$ modulo $A$ and $B$, respectively.
The group $H\ast_{\phi}$ admits a normal form associated to the sets
$T_{A}$ and $T_{B}$. The sequence of group elements $h_{0},t^{\epsilon_{1}},h_{1},\dots,t^{\epsilon_{n}},h_{n}$,
$\epsilon_{i}\in\{1,-1\}$, is in \textbf{normal form} if (1) $h_{0}\in H$,
(2) if $\epsilon_{i}=-1$, then $h_{i}\in T_{A}$, (3) if $\epsilon_{i}=1$,
then $h_{i}\in T_{B}$, and (4) there is no subsequence of the form
$t^{\epsilon},1_{H},t^{-\epsilon}$. For every $g\in H\ast_{\phi}$
there exists a unique sequence in normal form whose product equals
$g$ in $H\ast_{\phi}$. 

Amalgamated products are defined from two groups $H=\langle S_{H}\mid R_{H}\rangle$
and $K=\langle S_{K}\mid R_{K}\rangle$, and a group isomorphism $\phi\colon A\to B$,
with $A\leqslant H$ and $B\leqslant K$. The \textbf{amalgamated
product} of $H$ and $K$ relative to $\phi$, is the group with presentation
$H\ast_{\phi}K=\langle S_{H},S_{K}\mid R_{H},R_{K},\ a=\phi(a),\ \forall a\in A\rangle.$
Now let $T_{A}\subset H$ be a set of representatives for $H$ modulo
$A$, and let $T_{B}\subset K$ be a set of representatives for $K$
modulo $B$. The group $H\ast_{\phi}K$ admits a normal form associated
to the sets $T_{A}$ and $T_{B}$. A sequence of group elements $c_{0},c_{1},\dots,c_{n}$
is in \textbf{normal form }if (1) $c_{0}$ lies in $A$ or $B$, (2)
$c_{i}$ is in $T_{A}$ or $T_{B}$ for $i\geq1$, (3) $c_{i}\neq1$
for $i\geq1$, and (4) successive $c_{i}$ alternate between $T_{A}$
and $T_{B}$. For each element $g\in H\ast_{\phi}K$, there exist
a unique sequence in normal form whose product equals $g$ in $H\ast_{\phi}K$. 

Stalling's structure theorem relates ends of groups with HNN extensions
and amalgamated products \cite{dunwoody_cutting_1982}. This result
asserts that every finitely generated group $G$ with two or more
ends is either isomorphic to an HNN extension $H\ast_{\phi}$, or
isomorphic to an amalgamated product $H\ast_{\phi}K$. In both cases,
the corresponding isomorphism $\phi$ is between finite and proper
subgroups, and the groups $H$, or $H$ and $K$, are finitely generated
(see \cite[pages 34 and 43]{cohen_combinatorial_1989}). We will now
prove that when $G$ has decidable word problem, then the associated
normal forms are computable. This means that there is an algorithm
which given a word representing a group element $g$, computes a sequence
of words such that the corresponding sequence of group elements, is
a normal form for $g$. The proof is direct, but we were unable to
find this statement in the literature.
\begin{prop}
\label{prop:stallings-descomposicion-calculable} Let $G$ be a finitely
generated group with two or more ends and decidable word problem.
Then the normal form associated to the decomposition of $G$ as HNN
extension or amalgamated product is computable.
\end{prop}

\begin{proof}
Let us assume first that we are in the first case, so there is a finitely
generated group $H=\langle S_{H}\mid R_{H}\rangle$, and an isomorphism
$\phi\colon A\to B$ between finite subgroups of $H$, such that $G$
is isomorphic to the HNN extension $H\ast_{\phi}=\langle S_{H},t\mid\ R_{H},\ tat^{-1}=\phi(a),\ a\in A\rangle$.
A preliminary observation is that $H$ has decidable word problem.
Indeed, this property is inherited by finitely generated subgroups,
and $G$ has decidable word problem by hypothesis. The computability
of the normal form will follow from two simple facts:

First, observe that the finite group $A=\{a_{1},\dots,a_{n}\}$ has
decidable membership problem in $H$. Indeed, given a word $w\in(S_{H}\cup S_{H}^{-1})^{\ast}$,
we can decide if $\underline{w}\in A$ by checking if $\underline w =\underline{a_{i}}$ for $i=1,\dots,m$.
This is an effective procedure as the word problem of $H$ is decidable,
and is guaranteed to stop as $A$ is a finite set. As a consequence
of this, we can also decide if $\underline{u}\in A\underline{v}$ for any pair of
words $u,v\in(S_{H}\cup S_{H}^{-1})^{*}$, as this is equivalent to
decide if $(\underline{uv^{-1}})$ lies in $A$. Observe that $v^{-1}$ denotes the formal inverse to $v$. The same is true for $B$.

Second, there is a computably enumerable set $W_{A}\subset(S_{H}\cup S_{H}^{-1})^{*}$
such that the corresponding set $T_{A}$ of group elements in $H$
constitutes a collection of representatives for $H$ modulo $A$. We
sketch an algorithm that computably enumerates $W_{A}$ as a computable
sequence of words. Set $u_{0}$ to be the empty word. Now assume that
words $u_{0},\dots,u_{n}$ have been selected, and search for a word
$u_{n+1}\in(S_{H}\cup S_{H}^{-1})^{*}$ such that $\underline{u_{n+1}}$
does not lie in $A\underline{u_{0}},\dots,A\underline{u_{n}}$. The condition
that we impose to $u_{n+1}$ is decidable by the observation in the
previous paragraph, and thus an exhaustive search is guaranteed to
find a word as required. It is clear that the set $W_{A}$ that we
obtain is computably enumerable, and that the set $T_{A}$ of the
group elements of $H$ corresponding to these words is a set of representatives
for $H$ modulo $A$. A set $W_{B}$ corresponding to $T_{B}$ can
be enumerated analogously. 

Finally, we note that we can computably enumerate sequences of words
$w_{0},\dots,w_{n}$ that represent normal forms (with respect to
$T_{A}$ and $T_{B})$ for all group elements. Indeed, using the fact
$W_{A}$ and $W_{B}$ are computably enumerable sets, we just have
to enumerate sequences of words $w_{0},\dots,w_{n}$ such that $w_{0}$
is an arbitrary element of $(S_{H}\cup S_{H}^{-1})^{*}$, and the
rest are words from $W_{A},$ $W_{B}$, or $\{t,t^{-1}\}$ that alternate
as in the definition of normal form. In order to compute the normal
form of a group element $\underline{w}$ given by a word $w$, we just enumerate
these sequences $w_{1},\dots,w_{n}$ until we find one satisfying
$\underline{w}=\underline{w_{1}\dots w_{n}}$, which is a decidable question as $G$ has
decidable word problem. We have proved the computability of normal
forms as in the statement, in the case where $G$ is a (isomorphic
to) HNN extension. 

If $G$ is not isomorphic to an HNN extension, then it must be isomorphic
to an amalgamated product. Then there are two finitely generated groups
$H=\langle S_{H}\mid \ R_{H}\rangle$ and $K=\langle S_{K}\mid R_{K}\rangle$,
and a group isomorphism $\phi\colon A\to B$, with $A\leqslant H$
and $B\leqslant K$ finite groups, such that $G$ is isomorphic to
$\langle S_{H},S_{K}\mid R_{H},R_{K},\ a=\phi(a),\ \forall a\in A\rangle$.
Now the argument is the same as the one given for HNN extensions.
That is, $A$ and $B$ have decidable membership problem because they
are finite, and there are two computably enumerable sets of words
$W_{A}$ and $W_{B}$ corresponding to sets $T_{A}$ and $T_{B}$
as in the definition of normal form for amalgamated products. This,
plus the decidability of the word problem, is sufficient to compute
the normal form of a group element given as a word, by an exhaustive
search. 
\end{proof}
We obtain the following result from the computability of these normal
forms.

\begin{prop}
\label{prop:containment-Z-decidable-membership-problem}Let $G$ be
a finitely generated group with two or more ends and decidable word
problem. Then it has a subgroup isomorphic to $\Z$ with decidable
subgroup membership problem.
\end{prop}
\begin{proof}
By Stalling's structure theorem, either $G$ is isomorphic to an HNN
extensions, or $G$ is isomorphic to an amalgamated product. We first
suppose that $G$ is isomorphic to an HNN extension $H\ast_{\phi}=\langle S_{H},t\mid \ R_{H},tat^{-1}=\phi(a),a\in A\rangle$.
Without loss of generality, we will assume that $G$ is equal to this
group instead of isomorphic, as the decidability of the membership
problem of an infinite cyclic subgroup is preserved by group isomorphisms.
We claim that the subgroup of $G$ generated by $t$ has decidable
membership problem. Indeed, a group element $g$ lies in this subgroup
if and only if the normal form of $g$ or $g^{-1}$ is $1,t,1\dots,t,1$.
By \Cref{prop:stallings-descomposicion-calculable}, this normal
form is computable, and thus we obtain a procedure to decide membership
in the subgroup of $G$ generated by $t$.

We now consider the case where $G$ is isomorphic to an amalgamated
product. Then there are two finitely generated groups $H=\langle S_{H}\mid \ R_{H}\rangle$
and $K=\langle S_{K}\mid \ R_{K}\rangle$, and a group isomorphism $\phi:A\to B$,
with $A\leqslant H$ and $B\leqslant K$ finite groups, such that
$G$ is isomorphic to $\langle S_{H},S_{K}\mid R_{H},R_{K},\ a=\phi(a),\ \forall a\in A\rangle$.
As before, we will assume without loss of generality that $G$ is
indeed equal to this group. Now let $T_{A}$ and $T_{B}$ be the sets
defined in \Cref{prop:stallings-descomposicion-calculable} that
are associated to the computable normal form, and let $u\in T_{A}$,
$v\in T_{B}$ be both non trivial elements. We claim that the subgroup
subgroup of $G$ generated by $uv$ is isomorphic to $\Z$, and has
decidable membership problem. Indeed, a group element $g$ lies in
this subgroup if and only if the normal form of $g$ or $g^{-1}$
is $u,v,\dots,u,v$. This is a decidable question by \Cref{prop:stallings-descomposicion-calculable}.
\end{proof}
We now verify the fact that for translation-like actions coming from
subgroups, the properties of decidable orbit membership problem and
decidable subgroup membership problem are equivalent. 

\begin{prop}
\label{prop:obvio} Let $H\leqslant G$ be finitely generated groups.
Then $H$ has decidable membership problem in $G$ if and only if
the action of $H$ on $G$ by right translations has decidable orbit
membership problem. 
\end{prop}

\begin{proof}
Let $\ast$ be the action defined by $G\times H\to G$, $(g,h)\mapsto gh$.
The claim follows from the fact that two elements $g_{1},g_{2}\in G$
lie in the same $\ast$ orbit if and only if $g_{1}g_{2}{}^{-1}\in H$,
and an element $g\in G$ lies in $H$ if and only if it lies in the
same $\ast$ orbit as $1_{G}$.

It is clear how to rewrite this in terms of words, but we fill the
details for completeness. For the forward implication, let $u,v\in(S\cup S^{-1})^{*}$
be two words, for which we want to decide whether $u_{G}$, $v_{G}$
lie in the same orbit. We start by computing the formal inverse of
$v$, denoted $v^{-1}$, and then check whether the word $uv^{-1}$
lies in $\{w\in(S\cup S^{-1})^{*}\mid \ \underline{w}\in H\}$. This set is decidable
by hypothesis. For the reverse implication, assume that the action
has decidable orbit membership problem. The set $\{w\in(S\cup S^{-1})^{*}\mid \ \underline{w}\in H\}$
equals the set of words $w\in(S\cup S^{-1})^{*}$ such that $\underline{w}$
and $1_{G}$ lie in the same orbit, which is a decidable set by hypothesis.
It follows that $H$ has decidable subgroup membership problem in
$G$. 
\end{proof}
We can now finish the proof of \Cref{thm:computable-translation-like-actions-with-decidable-orbit-problem}. 
\begin{proof}[Proof of \Cref{thm:computable-translation-like-actions-with-decidable-orbit-problem}
for groups with two or more ends]
Let $G$ be a finitely generated infinite group with decidable word
problem and at least two ends. By \Cref{prop:containment-Z-decidable-membership-problem}
there is an element $c\in G$ such that $\langle c\rangle$ is isomorphic
to $\Z$, and has decidable subgroup membership problem in $G$. The
right action $\Z\curvearrowright G$ defined by $g\ast n=gc^{n}$
has decidable orbit membership problem by \Cref{prop:obvio}. 

It only remains to verify that the function $G\times\Z\to G$, $(g,n)\mapsto g\ast n$
is computable in the sense of 
\Cref{preliminaries-computability}. This is clear, but we write the details for completeness. The group
operation $f_{1}\colon G\times G\to G$ is computable by \Cref{prop:computabilitiy-groups-1}.
Moreover, it is clear that the function $f_{2}\colon\Z\to G$, $n\mapsto c^{n}$
is computable. Then it follows that the function $f_{3}\colon G\times\Z\to G$,
$(g,n)\mapsto f_{1}(g,f_{2}(n))$ is computable, being the composition
of computable functions. But $f_{3}(g,n)=g\ast n$, and thus $\ast$
is a computable group action. 
\end{proof}

%% file: Chapter_conclusion.tex
This thesis work is mainly about subshifts on groups and their recursive properties. In relation to the interplay between dynamics and recursion theory, let me recall here the following important feature of subshifts: all morphisms of subshifts are computable. This essentially follows from Curtis Hedlund Lindon theorem, which shows that these morphisms are determined by the finite information of a local rule.  Since all morphisms are computable we have in particular that if two subshifts are topologically conjugate, they are also recursively homeomorphic. This fact is behind the robustness of different recursive properties for subshifts. In other words, many computability properties of subshifts are in particular dynamical properties, as they are preserved by topological conjugacy. It is surprising that conversely,  fundamental questions about subshifts on groups that are posed solely in dynamical terms have recursion theoretical answers. One instance is the classification of topological entropies of $\Z^2$-SFTs, but there are many other examples (see the introduction). The investigation of recursive properties of subshifts on groups has intensified after these findings, and this research trend is the mathematical context for this thesis.

In this work we have taken an approach based on the tools from computable analysis, which provides the basis of recursion theory on separable metric spaces like $\{0,1\}^{\N}$, $\R$, or $S^1$. Concepts such as continuity and compactness have effective counterparts, and we have considered topological dynamical systems at the light of these concepts. This approach has precedents in the case of $\Z$-actions \cite{galatolo_dynamics_2011}. When we consider actions of finitely generated groups new phenomena appear, and algorithmic aspects of finitely generated groups interact with those of their topological dynamical systems  (see for instance \Cref{example:fullshift_RP_is_effective}).

In \Cref{chap:computable_analysis_subshifts}  we have revisited the notion of effective subshift. This class of subshifts was previously defined by a computability property of a defining set of forbidden patterns. Here we have considered an alternative definition that refers to the recursive compactness of the space. The two notions coincide for recursively presented groups, and only for these groups. 

In \Cref{chap:EDS} we have studied the class of effective dynamical systems, an effective counterpart of the classical topological definition. For instance, subshifts of finite type on a recursively presented groups and circle rotations with a computable angle are effective dynamical systems. Our main result is that within the class of effective dynamical systems given by a recursively presented group, every element admits an extension with topological dimension zero. A surprising consequence of this result and existing simulation results is that for some groups, one can show that a system is the factor of an SFT by simply showing that it is effective. 

In \Cref{chap:EDS}  we constantly consider factors and extensions of effective dynamical systems. We recall here that if $X$ is an effective subshift, then this property is transferred to every topological factor of $X$ that is also a subshift. This robustness property is exclusive of subshifts, it is not true in the general case. A problem left in this work is that of characterizing topological factors of effective dynamical systems (see \Cref{sec:factors}).

In \Cref{chap:Medvedev}  we have studied Medvedev degrees of subshifts. For several classes of groups, such as virtually poly-cyclic groups, we have proved a full classification of the degrees of SFTs. This classification problem is closely related to the well known domino problem for groups, and is also related to the well known problem of classifying those groups with weakly and strongly aperiodic SFTs. These problems are related in the sense that they share proof techniques. That is, a proof technique introduced to solve one of them for some fixed group have often lead to a solution of the other two problems. We have proved in \Cref{prop:med_implies_domino} that the existence of $G$-SFTs with nonzero Medvedev degree implies the undecidability of the domino problem for $G$. We hope that this result can be useful in the future to prove the undecidability of the domino problem for groups. Current  proofs typically involve the construction of family of SFTs with Turing machines and related constructions embedded, and uniformity considerations can be quite involved.  \Cref{prop:med_implies_domino} shows that it is sufficient to embed computation on a single SFT to conclude that the domino problem is undecidable.

The problem of classifying the values of Medvedev degrees of subshifts poses extra challenges, in the sense that some infinite constructions need to be effectivized. For instance, it was proved in \cite{cohen_large_2017} that quasi-isometries between finitely presented groups preserve strong aperiodicity, and preserve  the undecidability of the domino problem. However, one needs a \textit{computable quasi-isometry} of finitely presented groups to preserve Medvedev degrees. Thus the problem of classifying Medvedev degrees appears as a motivation to effectivize different results about infinite graphs. With this application in mind we proved in \Cref{chap:computable-translation-like-actions} an effective version of Seward's theorem about translation-like actions $\Z$ on finitely generated groups.   

In \Cref{Chap:computable-analysis-on-S(G)}  we have observed a basic interaction between Medvedev degrees of subshifts, and the topology of the space of subshifts $S(G)$ for a group $G$ with decidable word problem. It was proved in \cite{pavlov_structure_2023} that isolated points are dense in $S(\mathbb{Z})$, and thus that genericity results for $\mathbb{Z}$-subshifts can be proved by studying these isolated points. In contrast, we observed that that property of having an SFT with nonzero Medvedev degree creates a neighborhood in $S(G)$ without isolated points. The proof of these results are direct after one embeds $S(G)$ in a computable metric space. 

In \Cref{Chap:rice} we investigated the existence of a Rice theorem for dynamical properties of SFTs. The goal of this chapter was converting the metaphor \textit{swamp of undecidability} from \cite{lind_multidimensional_2004} to precise mathematical statements. We considered SFTs on a group with undecidable domino problem. We found that although a Rice theorem for dynamical properties of SFTs does not hold in a literal sense, it does for the larger class of sofic subshifts. For SFTs we proved a result very similar to the Adian-Rabin undecidability theorem for finitely presented groups. This result covers many dynamical properties of common interest, such as minimality, transitivity, and having zero entropy. The frontier of decidability for dynamical properties of SFTs remains unclear. The property of having a fixed point is decidable, and we do not know other examples besides trivial modifications. We have the hope that further research might clarify the boundary of the decidable zone in the {swamp of undecidability}.

In \Cref{chap:translation-like-actions} we presented a different proof of Seward's result asserting that every finitely generated infinite group admits a translation-like action by $\Z$. We generalize this result to all locally finite graphs, and prove that the action can be taken transitive exactly when the graph has either 1 or 2 ends. The proof of this result is closely related to the characterization of infinite Eulerian graphs, and the effective proof of this result presented by the author in  \cite{carrasco-vargas_infinite_2024,}. 

In \Cref{chap:computable-translation-like-actions} we proved an effective version Seward's result. Under the extra assumption that the group has decidable word problem, we proved that the action can be taken computable and with decidable orbit membership problem. This result was required to prove our results about Medvedev degrees of subshifts. Some of the ideas from this chapter have been latter refined and improved in the work \cite{carrasco-vargas_complexity_2024} about computability aspects of ends of graphs.

Many problems considered in this thesis seem far from being solved. We only mention the problem of characterizing topological factors of effective dynamical systems, and the problem of understanding Medvedev degrees of SFTs. I hope that the results presented in this text will be useful to those interested in computability theory on uncountable spaces and its interactions with other subjects.

%% file: Medvedev.bib
@article{bartholdi_salo_shift_lamplighter_2024,
Author = {Laurent Bartholdi and Ville Salo},
Title = {Shifts on the lamplighter group},
Year = {2024},
journal = {arXiv:2402.14508},
}

@article{bartholdi_2023_hyperbolic,
Author = {Laurent Bartholdi},
Title = {The domino problem for hyperbolic groups},
Year = {2023},
journal = {arXiv:2305.06952},
}

@book{Segal_1983, place={Cambridge}, series={Cambridge Tracts in Mathematics}, title={Polycyclic Groups}, publisher={Cambridge University Press}, author={Segal, Daniel}, year={1983}, collection={Cambridge Tracts in Mathematics}}

@unpublished{jeandel_notes_subshift_groups,
  TITLE = {{Some Notes about Subshifts on Groups}},
  AUTHOR = {Jeandel, Emmanuel},
  URL = {https://inria.hal.science/hal-01110211},
  NOTE = {working paper or preprint, hal-01110211},
  YEAR = {2015},
  MONTH = Jan,
  HAL_ID = {hal-01110211},
  HAL_VERSION = {v1},
}

@incollection{Cenzer_1999_survey_in_Pi_10_degrees,
title = {{$\Pi_1^0$} Classes in Computability Theory},
editor = {Edward R. Griffor},
series = {Studies in Logic and the Foundations of Mathematics},
publisher = {Elsevier},
volume = {140},
pages = {37-85},
year = {1999},
booktitle = {Handbook of Computability Theory},
issn = {0049-237X},
doi = {https://doi.org/10.1016/S0049-237X(99)80018-4},
url = {https://www.sciencedirect.com/science/article/pii/S0049237X99800184},
author = {Douglas Cenzer}
}

@phdthesis{barbierilemp_thesis,
TITLE = {{Shift spaces on groups : computability and dynamics}},
AUTHOR = {Barbieri, Sebasti{\'{a}}n},
URL = {https://tel.archives-ouvertes.fr/tel-01563302},
NUMBER = {2017LYSEN021},
SCHOOL = {{Universit{\'e} de Lyon}},
YEAR = {2017},
MONTH = Jun,
TYPE = {Theses},
}

@article{Lind1984,
    AUTHOR = {Lind, D. A.},
     TITLE = {The entropies of topological {M}arkov shifts and a related
              class of algebraic integers},
   JOURNAL = {Ergodic Theory Dynam. Systems},
  FJOURNAL = {Ergodic Theory and Dynamical Systems},
    VOLUME = {4},
      YEAR = {1984},
    NUMBER = {2},
     PAGES = {283--300},
      ISSN = {0143-3857,1469-4417},
   MRCLASS = {58F11 (15A48 28D20)},
  MRNUMBER = {766106},
MRREVIEWER = {W.\ Krieger},
       DOI = {10.1017/S0143385700002443},
       URL = {https://doi.org/10.1017/S0143385700002443},
}

@BOOK{Katok1992-nk,
  title     = "Fuchsian Groups",
  author    = "Katok, Svetlana",
  publisher = "University of Chicago Press",
  series    = "Chicago Lectures in Mathematics",
  month     =  jun,
  year      =  1992,
  address   = "Chicago, IL"
}

@incollection{BartholdiGrigorchukSuni_branchgroups_2023,
title = {Branch groups},
editor = {M. Hazewinkel},
series = {Handbook of Algebra},
publisher = {North-Holland},
volume = {3},
pages = {989-1112},
year = {2003},
issn = {1570-7954},
doi = {https://doi.org/10.1016/S1570-7954(03)80078-5},
url = {https://www.sciencedirect.com/science/article/pii/S1570795403800785},
author = {Laurent Bartholdi and Rostislav I. Grigorchuk and Zoran Šuni},
}

@article{Baumslag1981,
  title = {Computable algebra and group embeddings},
  volume = {69},
  ISSN = {0021-8693},
  url = {http://dx.doi.org/10.1016/0021-8693(81)90138-1},
  DOI = {10.1016/0021-8693(81)90138-1},
  number = {1},
  journal = {Journal of Algebra},
  publisher = {Elsevier BV},
  author = {Baumslag,  Gilbert and Cannonito,  Frank B and Miller,  Charles F},
  year = {1981},
  month = mar,
  pages = {186–212}
}

@article{ballier2013universality,
      title={Universality in symbolic dynamics constrained by Medvedev degrees}, 
      author={Alexis Ballier},
      year={2013},
      journal={arXiv:1304.5418},
      archivePrefix={arXiv},
      primaryClass={math.DS}
}

@InProceedings{domino-surface-groups,
  author =	{Aubrun, Nathalie and Barbieri, Sebasti\'{a}n and Moutot, Etienne},
  title =	{{The Domino Problem is Undecidable on Surface Groups}},
  booktitle =	{44th International Symposium on Mathematical Foundations of Computer Science (MFCS 2019)},
  pages =	{46:1--46:14},
  series =	{Leibniz International Proceedings in Informatics (LIPIcs)},
  ISBN =	{978-3-95977-117-7},
  ISSN =	{1868-8969},
  year =	{2019},
  volume =	{138},
  editor =	{Rossmanith, Peter and Heggernes, Pinar and Katoen, Joost-Pieter},
  publisher =	{Schloss Dagstuhl -- Leibniz-Zentrum f{\"u}r Informatik},
  address =	{Dagstuhl, Germany},
  URL =		{https://drops-dev.dagstuhl.de/entities/document/10.4230/LIPIcs.MFCS.2019.46},
  URN =		{urn:nbn:de:0030-drops-109900},
  doi =		{10.4230/LIPIcs.MFCS.2019.46}
}

@article{bogopolski2010orbit,
  title={Orbit decidability and the conjugacy problem for some extensions of groups},
  author={Bogopolski, Oleg and Martino, Armando and Ventura, Enric},
  journal={Transactions of the American Mathematical Society},
  volume={362},
  number={4},
  pages={2003--2036},
  year={2010}
}


%% file: MyLibrary.bib
@article{adian_decision_2000,
  title = {Decision Problems for Groups and Semigroups},
  author = {Adian, S. I. and Durnev, V. G.},
  year = {2000},
  month = apr,
  journal = {Russian Mathematical Surveys},
  volume = {55},
  number = {2},
  pages = {207},
  publisher = {IOP Publishing},
  issn = {0036-0279},
  doi = {10.1070/RM2000v055n02ABEH000267},
  urldate = {2022-05-10},
  langid = {english}
}

@article{adyan_algorithmic_1955,
  title = {{Algorithmic undecidability of problems of recognition of certain properties of groups}},
  author = {Adyan, S. I.},
  year = {1955},
  journal = {Doklady Akademii Nauk SSSR},
  volume = {103},
  pages = {533--535},
  issn = {0002-3264},
  langid = {russian}
}

@article{aliakbar_set_2018,
  title = {Set of Periods of a Subshift},
  author = {Ali Akbar, K. and Kannan, V.},
  year = {2018},
  month = oct,
  journal = {Proceedings - Mathematical Sciences},
  volume = {128},
  number = {5},
  pages = {63},
  issn = {0973-7685},
  doi = {10.1007/s12044-018-0431-2},
  urldate = {2024-01-04},
  langid = {english}
}

@article{amir_strong_2023,
  title = {Strong Computable Type},
  author = {Amir, Djamel Eddine and Hoyrup, Mathieu},
  year = {2023},
  month = jan,
  journal = {Computability},
  volume = {12},
  number = {3},
  pages = {227--269},
  publisher = {IOS Press},
  issn = {2211-3568},
  doi = {10.3233/COM-220430},
  urldate = {2023-12-02},
  langid = {english}
}

@incollection{aubrun_domino_2018,
  title = {About the {{Domino Problem}} for {{Subshifts}} on {{Groups}}},
  booktitle = {Sequences, {{Groups}}, and {{Number Theory}}},
  author = {Aubrun, Nathalie and Barbieri, Sebastián and Jeandel, Emmanuel},
  editor = {Berthé, Valérie and Rigo, Michel},
  year = {2018},
  pages = {331--389},
  publisher = {Springer International Publishing},
  address = {Cham},
  doi = {10.1007/978-3-319-69152-7_9},
  urldate = {2021-11-25},
  isbn = {978-3-319-69151-0 978-3-319-69152-7},
  langid = {english}
}

@article{aubrun_domino_2019,
  title = {The {{Domino Problem}} Is {{Undecidable}} on {{Surface Groups}}},
  author = {Aubrun, Nathalie and Barbieri, Sebastián and Moutot, Etienne},
  year = {2019},
  pages = {14 pages},
  publisher = {Schloss Dagstuhl - Leibniz-Zentrum fuer Informatik GmbH, Wadern/Saarbruecken, Germany},
  doi = {10.4230/LIPICS.MFCS.2019.46},
  urldate = {2021-09-08},
  collaborator = {Wagner, Michael},
  copyright = {Creative Commons Attribution 3.0 Unported license (CC-BY 3.0)},
  langid = {english}
}

@article{aubrun_notion_2017,
  title = {A Notion of Effectiveness for Subshifts on Finitely Generated Groups},
  author = {Aubrun, Nathalie and Barbieri, Sebastián and Sablik, Mathieu},
  year = {2017},
  journal = {Theoretical Computer Science},
  volume = {661},
  pages = {35--55},
  issn = {0304-3975},
  doi = {10.1016/j.tcs.2016.11.033},
  langid = {english}
}

@article{aubrun_simulation_2013,
  title = {Simulation of Effective Subshifts by Two-Dimensional Subshifts of Finite Type},
  author = {Aubrun, Nathalie and Sablik, Mathieu},
  year = {2013},
  journal = {Acta Applicandae Mathematicae},
  volume = {126},
  number = {1},
  pages = {35--63},
  issn = {0167-8019},
  doi = {10.1007/s10440-013-9808-5},
  langid = {english}
}

@article{aubrun_strongly_2023,
  title = {Strongly Aperiodic {{SFTs}} on Generalized {{Baumslag}}–{{Solitar}} Groups},
  author = {Aubrun, Nathalie and Bitar, Nicolás and {Huriot-Tattegrain}, Sacha},
  year = {2023},
  month = jun,
  journal = {Ergodic Theory and Dynamical Systems},
  pages = {1--30},
  issn = {0143-3857, 1469-4417},
  doi = {10.1017/etds.2023.44},
  urldate = {2023-10-14},
  langid = {english}
}

@article{aubrun_tiling_2013,
  title = {Tiling {{Problems}} on {{Baumslag-Solitar}} Groups},
  author = {Aubrun, Nathalie and Kari, Jarkko},
  year = {2013},
  month = sep,
  journal = {Electronic Proceedings in Theoretical Computer Science},
  volume = {128},
  pages = {35--46},
  issn = {2075-2180},
  doi = {10.4204/EPTCS.128.12},
  urldate = {2021-11-25}
}

@article{ballier_domino_2018,
  title = {The Domino Problem on Groups of Polynomial Growth},
  author = {Ballier, Alexis and Stein, Maya},
  year = {2018},
  journal = {Groups, Geometry, and Dynamics},
  volume = {12},
  number = {1},
  pages = {93--105},
  issn = {1661-7207},
  doi = {10.4171/GGD/439},
  langid = {english}
}

@article{barbieri_effective_2024,
  title = {Effective Dynamical Systems beyond Dimension Zero and Factors of {{SFTs}}},
  author = {Barbieri, Sebastián and {Carrasco-Vargas}, Nicanor and Rojas, Cristóbal},
  year = {2024},
  month = oct,
  journal = {Ergodic Theory and Dynamical Systems},
  pages = {1--41},
  issn = {0143-3857, 1469-4417},
  doi = {10.1017/etds.2024.79},
  urldate = {2024-10-15},
  langid = {english}
}

@article{barbieri_entropies_2021,
  title = {On the Entropies of Subshifts of Finite Type on Countable Amenable Groups},
  author = {Barbieri, Sebastián},
  year = {2021},
  journal = {Groups, Geometry, and Dynamics},
  volume = {15},
  number = {2},
  pages = {607--638},
  issn = {1661-7207},
  doi = {10.4171/GGD/608},
  langid = {english}
}

@article{barbieri_generalization_2019,
  title = {A Generalization of the Simulation Theorem for Semidirect Products},
  author = {Barbieri, Sebastián and Sablik, Mathieu},
  year = {2019},
  month = dec,
  journal = {Ergodic Theory and Dynamical Systems},
  volume = {39},
  number = {12},
  pages = {3185--3206},
  publisher = {Cambridge University Press},
  issn = {0143-3857, 1469-4417},
  doi = {10.1017/etds.2018.21},
  urldate = {2021-09-24},
  langid = {english}
}

@article{barbieri_geometric_2019,
  title = {A Geometric Simulation Theorem on Direct Products of Finitely Generated Groups},
  author = {Barbieri, Sebastián},
  year = {2019},
  month = jun,
  journal = {Discrete Analysis},
  doi = {10.19086/da.8820}
}

@misc{barbieri_groups_2022,
  title = {Groups with Self-Simulable Zero-Dimensional Dynamics},
  author = {Barbieri, Sebastián and Sablik, Mathieu and Salo, Ville},
  year = {2022},
  month = nov,
  number = {arXiv:2104.05141},
  eprint = {2104.05141},
  primaryclass = {math},
  publisher = {arXiv},
  doi = {10.48550/arXiv.2104.05141},
  urldate = {2023-05-17},
  archiveprefix = {arXiv}
}

@misc{barbieri_medvedev_2024,
  title = {Medvedev Degrees of Subshifts on Groups},
  author = {Barbieri, Sebastián and {Carrasco-Vargas}, Nicanor},
  year = {2024},
  month = jun,
  number = {arXiv:2406.12777},
  eprint = {2406.12777},
  primaryclass = {math},
  publisher = {arXiv},
  doi = {10.48550/arXiv.2406.12777},
  urldate = {2024-06-19},
  archiveprefix = {arXiv}
}

@misc{barbieri_soficity_2023,
  title = {Soficity of Free Extensions of Effective Subshifts},
  author = {Barbieri, Sebastián and Sablik, Mathieu and Salo, Ville},
  year = {2023},
  month = oct,
  number = {arXiv:2309.02620},
  eprint = {2309.02620},
  primaryclass = {math},
  publisher = {arXiv},
  doi = {10.48550/arXiv.2309.02620},
  urldate = {2023-10-23},
  archiveprefix = {arXiv}
}

@article{bartholdi_domino_2023,
  title = {The Domino Problem for Hyperbolic Groups},
  author = {Bartholdi, Laurent},
  year = {2023},
  month = may,
  doi = {10.48550/arXiv.2305.06952},
  urldate = {2023-11-04}
}

@misc{bartholdi_shifts_2024,
  title = {Shifts on the Lamplighter Group},
  author = {Bartholdi, Laurent and Salo, Ville},
  year = {2024},
  month = feb,
  number = {arXiv:2402.14508},
  eprint = {2402.14508},
  primaryclass = {cs, math},
  publisher = {arXiv},
  doi = {10.48550/arXiv.2402.14508},
  urldate = {2024-03-14},
  archiveprefix = {arXiv}
}

@article{bean_recursive_1976,
  title = {Recursive {{Euler}} and {{Hamilton}} Paths},
  author = {Bean, Dwight R.},
  year = {1976},
  journal = {Proceedings of the American Mathematical Society},
  volume = {55},
  pages = {385--394},
  issn = {0002-9939},
  doi = {10.2307/2041731},
  langid = {english}
}

@book{berger_undecidability_1966,
  title = {The Undecidability of the Domino Problem},
  author = {Berger, Robert},
  year = {1966},
  series = {Memoirs of the {{American Mathematical Society}}},
  publisher = {American Mathematical Society},
  urldate = {2021-09-04},
  isbn = {978-0-8218-1266-2 978-1-4704-0013-2},
  langid = {english}
}

@incollection{Birman_Brendle_2005_braidgroupssurvey,
  title = {Chapter 2 - Braids: A Survey},
  booktitle = {Handbook of Knot Theory},
  author = {Birman, Joan S. and Brendle, Tara E.},
  editor = {Menasco, William and Thistlethwaite, Morwen},
  year = {2005},
  pages = {19--103},
  publisher = {Elsevier Science},
  address = {Amsterdam},
  doi = {10.1016/B978-044451452-3/50003-4},
  isbn = {978-0-444-51452-3}
}

@misc{bitar_contributions_2023,
  title = {Contributions to the {{Domino Problem}}: {{Seeding}}, {{Recurrence}} and {{Satisfiability}}},
  shorttitle = {Contributions to the {{Domino Problem}}},
  author = {Bitar, Nicolás},
  year = {2023},
  month = dec,
  number = {arXiv:2312.08911},
  eprint = {2312.08911},
  primaryclass = {cs, math},
  publisher = {arXiv},
  doi = {10.48550/arXiv.2312.08911},
  urldate = {2024-01-05},
  archiveprefix = {arXiv}
}

@article{bogopolski_orbit_2009,
  title = {Orbit Decidability and the Conjugacy Problem for Some Extensions of Groups},
  author = {Bogopolski, O. and Martino, A. and Ventura, E.},
  year = {2009},
  month = nov,
  journal = {Transactions of the American Mathematical Society},
  volume = {362},
  number = {4},
  pages = {2003--2036},
  issn = {0002-9947, 1088-6850},
  doi = {10.1090/S0002-9947-09-04817-X},
  urldate = {2024-03-31},
  copyright = {https://www.ams.org/publications/copyright-and-permissions},
  langid = {english}
}

@book{bowen_axiom_1978,
  title = {On Axiom a Diffeomorphisms: Expository Lectures from the {{CBMS Regional Conference}} Held at {{North Dakota State Univ}}., {{June}} 20-24, 1977},
  shorttitle = {On Axiom a Diffeomorphisms},
  author = {Bowen, Rufus},
  year = {1978},
  series = {Conference Board of the Mathematical Sciences. {{Regional}} Conference Series in Mathematics},
  number = {35},
  publisher = {Conference Board of the Mathematical Sciences by the American Mathematical Society},
  address = {Providence, R.I},
  isbn = {978-0-8218-1685-1},
  langid = {english}
}

@book{brattka_handbook_2021,
  title = {Handbook of Computability and Complexity in Analysis},
  editor = {Brattka, Vasco and Hertling, Peter},
  year = {2021},
  series = {Theory {{Appl}}. {{Comput}}.},
  publisher = {Cham: Springer},
  issn = {2190-619X},
  doi = {10.1007/978-3-030-59234-9},
  isbn = {978-3-030-59233-2 978-3-030-59236-3 978-3-030-59234-9},
  langid = {english}
}

@incollection{brattka_tutorial_2008,
  title = {A {{Tutorial}} on {{Computable Analysis}}},
  booktitle = {New {{Computational Paradigms}}: {{Changing Conceptions}} of {{What}} Is {{Computable}}},
  author = {Brattka, Vasco and Hertling, Peter and Weihrauch, Klaus},
  editor = {Cooper, S. Barry and Löwe, Benedikt and Sorbi, Andrea},
  year = {2008},
  pages = {425--491},
  publisher = {Springer},
  address = {New York, NY},
  doi = {10.1007/978-0-387-68546-5_18},
  urldate = {2024-07-03},
  isbn = {978-0-387-68546-5},
  langid = {english}
}

@book{bridson_metric_1999,
  title = {Metric Spaces of Non-Positive Curvature},
  author = {Bridson, Martin R. and Haefliger, André},
  year = {1999},
  series = {Grundlehren Der Mathematischen {{Wissenschaften}}},
  number = {319},
  publisher = {Springer},
  address = {Berlin ; New York},
  isbn = {978-3-540-64324-1},
  lccn = {QA611.28 .B75 1999}
}

@article{burillo_conjugacy_2016,
  title = {The Conjugacy Problem in Extensions of {{Thompson}}’s Group {{F}}},
  author = {Burillo, José and Matucci, Francesco and Ventura, Enric},
  year = {2016},
  month = oct,
  journal = {Israel Journal of Mathematics},
  volume = {216},
  number = {1},
  pages = {15--59},
  issn = {1565-8511},
  doi = {10.1007/s11856-016-1403-9},
  urldate = {2024-03-31},
  langid = {english}
}

@inproceedings{callard_aperiodic_2022,
  title = {The {{Aperiodic Domino Problem}} in {{Higher Dimension}}},
  booktitle = {39th {{International Symposium}} on {{Theoretical Aspects}} of {{Computer Science}} ({{STACS}} 2022)},
  author = {Callard, Antonin and {Hellouin de Menibus}, Benjamin},
  editor = {Berenbrink, Petra and Monmege, Benjamin},
  year = {2022},
  series = {Leibniz {{International Proceedings}} in {{Informatics}} ({{LIPIcs}})},
  volume = {219},
  pages = {19:1--19:15},
  publisher = {Schloss Dagstuhl – Leibniz-Zentrum für Informatik},
  address = {Dagstuhl, Germany},
  issn = {1868-8969},
  doi = {10.4230/LIPIcs.STACS.2022.19},
  urldate = {2023-10-14},
  isbn = {978-3-95977-222-8}
}

@misc{carrasco-vargas_complexity_2024,
  title = {On the Complexity of the {{Eulerian}} Path Problem for Infinite Graphs},
  author = {{Carrasco-Vargas}, Nicanor and Rose, Valentino Delle and Rojas, Cristóbal},
  year = {2024},
  month = sep,
  number = {arXiv:2409.03113},
  eprint = {2409.03113},
  publisher = {arXiv},
  doi = {10.48550/arXiv.2409.03113},
  urldate = {2024-10-15},
  archiveprefix = {arXiv}
}

@misc{carrasco-vargas_infinite_2024,
  title = {Infinite {{Eulerian}} Trails Are Computable on Graphs with Vertices of Infinite Degree.},
  author = {{Carrasco-Vargas}, Nicanor},
  year = {2024},
  month = jan,
  doi = {10.48550/arXiv.2305.17998},
  urldate = {2024-01-23}
}

@article{carrasco-vargas_rice_2024,
	title = {On a Rice theorem for dynamical properties of SFTs on groups},
	doi = {https://doi.org/10.1007/s00013-025-02125-x},
	journal = {Archiv der Mathematik},
	author = {Carrasco-Vargas, Nicanor},
	year = {2025},
	volume = {124},
	issue = {6},
	number = {591-603}
}

@article{carrasco-vargas_translationlike_2024,
  title = {Translation-like Actions by $\mathbb{Z}$, the subgroup membership problem, and Medvedev degrees of effective subshifts},
  author = {{Carrasco-Vargas}, Nicanor},
  year = {2024},
  month = aug,
  journal = {Groups, Geometry, and Dynamics},
  issn = {1661-7207},
  doi = {10.4171/ggd/817},
  urldate = {2024-10-15},
  langid = {english}
}

@article{carroll_periodic_2015,
  title = {Periodic Points on Shifts of Finite Type and Commensurability Invariants of Groups},
  author = {Carroll, David and Penland, Andrew},
  year = {2015},
  journal = {The New York Journal of Mathematics},
  volume = {21},
  pages = {811--822},
  issn = {1076-9803},
  langid = {english}
}

@book{ceccherini-silberstein_cellular_2010,
  title = {Cellular Automata and Groups},
  author = {{Ceccherini-Silberstein}, Tullio and Coornaert, Michel},
  year = {2010},
  series = {Springer {{Monogr}}. {{Math}}.},
  publisher = {Springer},
  address = {Berlin},
  isbn = {978-3-642-14033-4 978-3-642-14034-1},
  langid = {english}
}

@incollection{ceccherini-silberstein_cellular_2018,
  title = {Cellular Automata and Groups},
  booktitle = {Cellular Automata},
  author = {{Ceccherini-Silberstein}, Tullio and Coornaert, Michel},
  year = {2018},
  edition = {2nd edition},
  pages = {221--238},
  publisher = {New York, NY: Springer},
  doi = {10.1007/978-1-4939-8700-9_52},
  isbn = {978-1-4939-8699-6 978-1-4939-8700-9 978-1-4939-8701-6},
  langid = {english}
}

@article{cenzer_computable_2008,
  title = {Computable Symbolic Dynamics},
  author = {Cenzer, Douglas and Dashti, S. Ali and King, Jonathan L. F.},
  year = {2008},
  journal = {Mathematical Logic Quarterly},
  volume = {54},
  number = {5},
  pages = {460--469},
  issn = {0942-5616},
  doi = {10.1002/malq.200710066},
  langid = {english}
}

@article{cervelle_tilings_2004,
  title = {Tilings: Recursivity and Regularity},
  shorttitle = {Tilings},
  author = {Cervelle, Julien and Durand, Bruno},
  year = {2004},
  month = jan,
  journal = {Theoretical Computer Science},
  volume = {310},
  number = {1},
  pages = {469--477},
  issn = {0304-3975},
  doi = {10.1016/S0304-3975(03)00242-1},
  urldate = {2022-01-10},
  langid = {english}
}

@article{chung_2013,
  title = {Topological Pressure and the Variational Principle for Actions of Sofic Groups},
  author = {Chung, Nhan-Phu},
  year = {2013},
  journal = {Ergodic Theory and Dynamical Systems},
  volume = {33},
  number = {5},
  pages = {1363--1390},
  publisher = {Cambridge University Press},
  doi = {10.1017/S0143385712000429}
}

@book{cohen_combinatorial_1989,
  title = {Combinatorial Group Theory: A Topological Approach},
  shorttitle = {Combinatorial Group Theory},
  author = {Cohen, Daniel E.},
  year = {1989},
  series = {London {{Mathematical Society}} Student Texts},
  number = {14},
  publisher = {Cambridge University Press},
  address = {Cambridge ; New York},
  isbn = {978-0-521-34133-2 978-0-521-34936-9},
  lccn = {QA171 .C672 1989}
}

@article{cohen_counterexample_2019,
  title = {A Counterexample to the Easy Direction of the Geometric {{Gersten}} Conjecture.},
  author = {Cohen, David},
  year = {2019},
  month = feb,
  journal = {Pacific Journal of Mathematics},
  volume = {298},
  number = {1},
  pages = {27--31},
  issn = {0030-8730, 0030-8730},
  doi = {10.2140/pjm.2019.298.27},
  urldate = {2024-04-03},
  langid = {english}
}

@article{cohen_large_2017,
  title = {The Large Scale Geometry of Strongly Aperiodic Subshifts of Finite Type},
  author = {Cohen, David Bruce},
  year = {2017},
  month = feb,
  journal = {Advances in Mathematics},
  volume = {308},
  pages = {599--626},
  issn = {0001-8708},
  doi = {10.1016/j.aim.2016.12.016},
  urldate = {2023-03-18},
  langid = {english}
}

@article{cohen_strongly_2021,
  title = {Strongly Aperiodic Subshifts of Finite Type on Hyperbolic Groups},
  author = {Cohen, David B. and {Goodman-Strauss}, Chaim and Rieck, Yo’av},
  year = {2021},
  month = sep,
  journal = {Ergodic Theory and Dynamical Systems},
  pages = {1--44},
  publisher = {Cambridge University Press},
  issn = {0143-3857, 1469-4417},
  doi = {10.1017/etds.2021.70},
  urldate = {2021-11-24},
  langid = {english}
}

@book{coornaert_symbolic_1993,
  title = {Symbolic Dynamics and Hyperbolic Groups},
  author = {Coornaert, Michel and Papadopoulos, Athanase},
  year = {1993},
  series = {Lecture Notes in Mathematics},
  number = {1539},
  publisher = {Springer-Verlag},
  address = {Berlin},
  isbn = {978-3-540-47573-6},
  langid = {english},
  lccn = {516.362}
}

@incollection{delacourt_rice_2011,
  title = {Rice's Theorem for $\mu$-limit sets of cellular automata},
  booktitle = {Automata, Languages and Programming. 38th International Colloquium, {{ICALP}} 2011, {{Zurich}}, {{Switzerland}}, {{July}} 4–8, 2011. {{Proceedings}}, {{Part II}}},
  author = {Delacourt, Martin},
  year = {2011},
  pages = {89--100},
  publisher = {Berlin: Springer},
  isbn = {978-3-642-22011-1},
  langid = {english}
}

@article{delvenne_quasiperiodic_2004,
  title = {Quasi-Periodic Configurations and Undecidable Dynamics for Tilings, Infinite Words and {{Turing}} Machines},
  author = {Delvenne, Jean-Charles and Blondel, Vincent D.},
  year = {2004},
  month = jun,
  journal = {Theoretical Computer Science},
  series = {Combinatorics of the {{Discrete Plane}} and {{Tilings}}},
  volume = {319},
  number = {1},
  pages = {127--143},
  issn = {0304-3975},
  doi = {10.1016/j.tcs.2004.02.018},
  urldate = {2022-12-05},
  langid = {english}
}

@article{diestel_graphtheoretical_2003,
  title = {Graph-Theoretical versus Topological Ends of Graphs.},
  author = {Diestel, Reinhard and Kühn, Daniela},
  year = {2003},
  journal = {Journal of Combinatorial Theory. Series B},
  volume = {87},
  number = {1},
  pages = {197--206},
  issn = {0095-8956},
  doi = {10.1016/S0095-8956(02)00034-5},
  langid = {english}
}

@book{drutu_geometric_2018,
  title = {Geometric Group Theory},
  author = {Druţu, Cornelia and Kapovich, Michael and Nica, Bogdan},
  year = {2018},
  series = {Colloquium Publications},
  number = {volume 63},
  publisher = {American Mathematical Society},
  address = {Providence, Rhode Island},
  isbn = {978-1-4704-1104-6},
  langid = {english}
}

@article{dunwoody_cutting_1982,
  title = {Cutting up Graphs},
  author = {Dunwoody, M. J.},
  year = {1982},
  journal = {Combinatorica},
  volume = {2},
  pages = {15--23},
  issn = {0209-9683},
  doi = {10.1007/BF02579278},
  langid = {english}
}

@incollection{durand_effective_2010,
  title = {Effective Closed Subshifts in {{1D}} Can Be Implemented in {{2D}}},
  booktitle = {Fields of Logic and Computation. {{Essays}} Dedicated to {{Yuri Gurevich}} on the Occasion of His 70th Birthday},
  author = {Durand, Bruno and Romashchenko, Andrei and Shen, Alexander},
  year = {2010},
  pages = {208--226},
  publisher = {Springer},
  address = {Berlin},
  isbn = {978-3-642-15024-1},
  langid = {english}
}

@article{durand_fixedpoint_2012,
  title = {Fixed-Point Tile Sets and Their Applications},
  author = {Durand, Bruno and Romashchenko, Andrei and Shen, Alexander},
  year = {2012},
  month = may,
  journal = {Journal of Computer and System Sciences},
  series = {In {{Commemoration}} of {{Amir Pnueli}}},
  volume = {78},
  number = {3},
  pages = {731--764},
  issn = {0022-0000},
  doi = {10.1016/j.jcss.2011.11.001},
  urldate = {2023-10-16}
}

@article{erdos_eulerian_1936,
  title = {{On Eulerian lines in infinite graphs}},
  author = {Erdős, Pál and Grünwald, T. and Weiszfeld, E.},
  year = {1936},
  journal = {Matematikai és Fizikai Lapok},
  volume = {43},
  pages = {129--140},
  issn = {0302-7317},
  langid = {hungarian}
}

@article{esnay_parametrization_2023,
  title = {Parametrization by Horizontal Constraints in the Study of Algorithmic Properties of  \$ \textbackslash mathbb\{\vphantom\}{{Z}}\vphantom\{\}\textasciicircum 2 \$-Subshifts of Finite Type},
  author = {Esnay, Solène J. and Sablik, Mathieu},
  year = {2023},
  journal = {Discrete and Continuous Dynamical Systems},
  volume = {43},
  number = {5},
  pages = {2002--2046},
  issn = {1078-0947, 1553-5231},
  doi = {10.3934/dcds.2023001},
  urldate = {2024-06-24}
}

@article{freudenthal_ueber_1945,
  title = {{Über die Enden diskreter Räume und Gruppen}},
  author = {Freudenthal, Hans},
  year = {1945},
  journal = {Commentarii Mathematici Helvetici},
  volume = {17},
  pages = {1--38},
  issn = {0010-2571},
  doi = {10.1007/BF02566233},
  langid = {german}
}

@article{frisch_symbolic_2017,
  title = {Symbolic Dynamics on Amenable Groups: The Entropy of Generic Shifts},
  shorttitle = {Symbolic Dynamics on Amenable Groups},
  author = {Frisch, Joshua and Tamuz, Omer},
  year = {2017},
  month = jun,
  journal = {Ergodic Theory and Dynamical Systems},
  volume = {37},
  number = {4},
  eprint = {1503.06251},
  pages = {1187--1210},
  issn = {0143-3857, 1469-4417},
  doi = {10.1017/etds.2015.84},
  urldate = {2022-04-07},
  archiveprefix = {arXiv}
}

@article{galatolo_dynamics_2011,
  title = {Dynamics and Abstract Computability: {{Computing}} Invariant Measures},
  shorttitle = {Dynamics and Abstract Computability},
  author = {Galatolo, Stefano and Hoyrup, Mathieu and Rojas, Cristóbal and {,Dipartimento di Matematica Applicata, Universita di Pisa} and {,LORIA, Vandoeuvre-lès-Nancy} and {,Fields Institute, Toronto}},
  year = {2011},
  journal = {Discrete \& Continuous Dynamical Systems - A},
  volume = {29},
  number = {1},
  pages = {193--212},
  issn = {1553-5231},
  doi = {10.3934/dcds.2011.29.193},
  urldate = {2021-11-02},
  langid = {english}
}

@inproceedings{Gamard2021RiceLikeTF,
  title = {Rice-like Theorems for Automata Networks},
  booktitle = {Symposium on Theoretical Aspects of Computer Science},
  author = {Gamard, Guilhem and Guillon, Pierre and Perrot, Kévin and Theyssier, Guillaume},
  year = {2021}
}

@article{gangloff_characterizing_2022,
  title = {Characterizing Entropy Dimensions of Minimal Mutidimensional Subshifts of Finite Type},
  author = {Gangloff, Silvère},
  year = {2022},
  journal = {Discrete \& Continuous Dynamical Systems},
  volume = {42},
  number = {2},
  pages = {931},
  issn = {1078-0947, 1553-5231},
  doi = {10.3934/dcds.2021143},
  urldate = {2024-06-24}
}

@article{gangloff_effect_2019,
  title = {Effect of Quantified Irreducibility on the Computability of Subshift Entropy},
  author = {Gangloff, Silvère and de Menibus, Benjamin Hellouin},
  year = {2019},
  journal = {Discrete and Continuous Dynamical Systems},
  volume = {39},
  number = {4},
  doi = {10.3934/dcds.2019083},
  urldate = {2022-06-02},
  copyright = {http://creativecommons.org/licenses/by/3.0/},
  langid = {english}
}

@article{gangloff_quantified_2021,
  title = {Quantified Block Gluing for Multidimensional Subshifts of Finite Type: Aperiodicity and Entropy},
  shorttitle = {Quantified Block Gluing for Multidimensional Subshifts of Finite Type},
  author = {Gangloff, Silvère and Sablik, Mathieu},
  year = {2021},
  month = dec,
  journal = {Journal d'Analyse Mathématique},
  volume = {144},
  number = {1},
  pages = {21--118},
  issn = {1565-8538},
  doi = {10.1007/s11854-021-0172-5},
  urldate = {2023-11-08},
  langid = {english}
}

@misc{gangloff_topological_2024,
  title = {The Topological Structure of Isolated Points in the Space of \$\textbackslash mathbb\{\vphantom\}{{Z}}\vphantom\{\}\textasciicircum d\$-Shifts},
  author = {Gangloff, Silvère and Núñez, Alonso},
  year = {2024},
  publisher = {arXiv},
  doi = {10.48550/ARXIV.2401.17119},
  urldate = {2024-06-14},
  copyright = {arXiv.org perpetual, non-exclusive license}
}

@article{golod_class_1964,
  title = {{On the class field tower}},
  author = {Golod, E. S. and Shafarevich, I. R.},
  year = {1964},
  journal = {Izvestiya Akademii Nauk SSSR. Seriya Matematicheskaya},
  volume = {28},
  pages = {261--272},
  issn = {0373-2436},
  langid = {russian}
}

@article{goodman-strauss_hierarchical_2010,
  title = {A Hierarchical Strongly Aperiodic Set of Tiles in the Hyperbolic Plane},
  author = {{Goodman-Strauss}, C.},
  year = {2010},
  month = feb,
  journal = {Theoretical Computer Science},
  volume = {411},
  number = {7},
  pages = {1085--1093},
  issn = {0304-3975},
  doi = {10.1016/j.tcs.2009.11.018},
  urldate = {2022-10-28},
  langid = {english}
}

@inproceedings{guillon_revisiting_2010,
  title = {Revisiting the {{Rice Theorem}} of {{Cellular Automata}}},
  booktitle = {Symposium on Theoretical Aspects of Computer Science},
  author = {Guillon, Pierre and Richard, Gaétan},
  year = {2010},
  eprint = {1001.0253},
  primaryclass = {cs},
  urldate = {2022-10-18},
  archiveprefix = {arXiv}
}

@article{hadamard_surfaces_1898,
  title = {{Les surfaces à courbures opposées et leurs lignes géodésiques}},
  author = {Hadamard},
  year = {1898},
  journal = {Journal de Mathématiques Pures et Appliquées},
  volume = {4},
  pages = {27--74},
  issn = {0021-7874},
  urldate = {2024-05-14},
  langid = {fra}
}

@article{hanf_nonrecursive_1974,
  title = {Nonrecursive {{Tilings}} of the {{Plane}}. {{I}}},
  author = {Hanf, William},
  year = {1974},
  journal = {The Journal of Symbolic Logic},
  volume = {39},
  number = {2},
  eprint = {2272640},
  eprinttype = {jstor},
  pages = {283--285},
  publisher = {Association for Symbolic Logic},
  issn = {0022-4812},
  doi = {10.2307/2272640},
  urldate = {2022-04-14}
}

@article{harel_hamiltonian_1991,
  title = {Hamiltonian Paths in Infinite Graphs},
  author = {Harel, David},
  year = {1991},
  month = oct,
  journal = {Israel Journal of Mathematics},
  volume = {76},
  number = {3},
  pages = {317--336},
  issn = {1565-8511},
  doi = {10.1007/BF02773868},
  urldate = {2023-12-01},
  langid = {english}
}

@article{hertling_shifts_2008,
  title = {Shifts with Decidable Language and Non-Computable Entropy},
  author = {Hertling, Peter and Spandl, Christoph},
  year = {2008},
  journal = {Discrete Mathematics \& Theoretical Computer Science},
  volume = {10},
  number = {3},
  pages = {75--94},
  issn = {1365-8050},
  doi = {10.46298/dmtcs.425},
  urldate = {2023-10-24},
  langid = {english}
}

@article{hinman_survey_2012,
  title = {A Survey of {{Mučnik}} and {{Medvedev}} Degrees},
  author = {Hinman, Peter G.},
  year = {2012},
  journal = {The Bulletin of Symbolic Logic},
  volume = {18},
  number = {2},
  pages = {161--229},
  issn = {1079-8986},
  doi = {10.2178/bsl/1333560805},
  langid = {english}
}

@article{hirst_ricestyle_2009,
  title = {A {{Rice-style}} Theorem for Parallel Automata},
  author = {Hirst, Tirza},
  year = {2009},
  month = jan,
  journal = {Information and Computation},
  volume = {207},
  number = {1},
  pages = {1--13},
  issn = {0890-5401},
  doi = {10.1016/j.ic.2008.10.004},
  urldate = {2024-01-06}
}

@article{hochman_characterization_2010,
  title = {A Characterization of the Entropies of Multidimensional Shifts of Finite Type},
  author = {Hochman, Michael and Meyerovitch, Tom},
  year = {2010},
  month = may,
  journal = {Annals of Mathematics},
  volume = {171},
  number = {3},
  pages = {2011--2038},
  issn = {0003-486X},
  doi = {10.4007/annals.2010.171.2011},
  urldate = {2021-09-24},
  langid = {english}
}

@article{hochman_dynamics_2009,
  title = {On the Dynamics and Recursive Properties of Multidimensional Symbolic Systems},
  author = {Hochman, Michael},
  year = {2009},
  month = apr,
  journal = {Inventiones mathematicae},
  volume = {176},
  number = {1},
  pages = {131--167},
  publisher = {Springer-Verlag},
  issn = {1432-1297},
  doi = {10.1007/s00222-008-0161-7},
  urldate = {2021-09-24},
  copyright = {2008 Springer-Verlag},
  langid = {english}
}

@article{hochman_note_2009,
  title = {A Note on Universality in Multidimensional Symbolic Dynamics},
  author = {Hochman, Michael},
  journal = {Discrete and Continuous Dynamical Systems - S},
  volume = {2},
  number = {2},
  pages = {301--314},
  publisher = {{Discrete and Continuous Dynamical Systems - S}},
  issn = {1937-1632},
  doi = {10.3934/dcdss.2009.2.301},
  urldate = {2024-05-03},
  copyright = {http://creativecommons.org/licenses/by/3.0/},
  langid = {english}
}

@article{hooper_undecidability_1966,
  title = {The Undecidability of the {{Turing}} Machine Immortality Problem},
  author = {Hooper, Philip K.},
  year = {1966},
  month = jun,
  journal = {Journal of Symbolic Logic},
  volume = {31},
  number = {2},
  pages = {219--234},
  issn = {0022-4812, 1943-5886},
  doi = {10.2307/2269811},
  urldate = {2022-12-06},
  langid = {english}
}

@article{hopf_enden_1944,
  title = {{Enden offener Räume und unendliche diskontinuierliche Gruppen}},
  author = {Hopf, Heinz},
  year = {1944},
  journal = {Commentarii Mathematici Helvetici},
  volume = {16},
  pages = {81--100},
  issn = {0010-2571},
  doi = {10.1007/BF02568567},
  langid = {german}
}

@inproceedings{inproceedings,
  title = {Euler Paths and Ends in Automatic and Recursive Graphs.},
  booktitle = {International {{Conference}} on {{Automata}} and {{Formal Languages}}},
  author = {Kuske, Dietrich and Lohrey, Markus},
  year = {2008},
  pages = {245--256}
}

@article{jeandel_characterization_2019,
  title = {A {{Characterization}} of {{Subshifts}} with {{Computable Language}}},
  author = {Jeandel, Emmanuel and Vanier, Pascal},
  year = {2019},
  journal = {LIPIcs, Volume 126, STACS 2019},
  volume = {126},
  pages = {40:1-40:16},
  publisher = {Schloss Dagstuhl – Leibniz-Zentrum für Informatik},
  issn = {1868-8969},
  doi = {10.4230/LIPICS.STACS.2019.40},
  urldate = {2024-05-25},
  collaborator = {Niedermeier, Rolf and Paul, Christophe},
  copyright = {Creative Commons Attribution 3.0 Unported license, info:eu-repo/semantics/openAccess},
  isbn = {9783959771009},
  langid = {english}
}

@inproceedings{jeandel_computability_2016,
  title = {Computability in {{Symbolic Dynamics}}},
  booktitle = {Pursuit of the {{Universal}}},
  author = {Jeandel, Emmanuel},
  editor = {Beckmann, Arnold and Bienvenu, Laurent and Jonoska, Nataša},
  year = {2016},
  series = {Lecture {{Notes}} in {{Computer Science}}},
  pages = {124--131},
  publisher = {Springer International Publishing},
  address = {Cham},
  doi = {10.1007/978-3-319-40189-8_13},
  isbn = {978-3-319-40189-8},
  langid = {english}
}

@inproceedings{jeandel_enumeration_2017,
  title = {Enumeration Reducibility in Closure Spaces with Applications to Logic and Algebra},
  booktitle = {2017 32nd {{Annual ACM}}/{{IEEE Symposium}} on {{Logic}} in {{Computer Science}} ({{LICS}})},
  author = {Jeandel, Emmanuel},
  year = {2017},
  month = jun,
  pages = {1--11},
  publisher = {IEEE},
  address = {Reykjavik, Iceland},
  doi = {10.1109/LICS.2017.8005086},
  urldate = {2024-05-21},
  isbn = {978-1-5090-3018-7}
}

@inproceedings{jeandel_immortal_2012,
  title = {On {{Immortal Configurations}} in {{Turing Machines}}},
  booktitle = {How the {{World Computes}}},
  author = {Jeandel, Emmanuel},
  editor = {Cooper, S. Barry and Dawar, Anuj and Löwe, Benedikt},
  year = {2012},
  series = {Lecture {{Notes}} in {{Computer Science}}},
  pages = {334--343},
  publisher = {Springer},
  address = {Berlin, Heidelberg},
  doi = {10.1007/978-3-642-30870-3_34},
  isbn = {978-3-642-30870-3},
  langid = {english}
}

@article{jeandel_translationlike_2015,
  title = {Translation-like {{Actions}} and {{Aperiodic Subshifts}} on {{Groups}}},
  author = {Jeandel, Emmanuel},
  year = {2015},
  month = aug,
  journal = {arXiv:1508.06419 [cs, math]},
  eprint = {1508.06419},
  primaryclass = {cs, math},
  urldate = {2021-11-25},
  archiveprefix = {arXiv}
}

@article{jeandel_turing_2013,
  title = {Turing Degrees of Multidimensional {{SFTs}}},
  author = {Jeandel, Emmanuel and Vanier, Pascal},
  year = {2013},
  month = sep,
  journal = {Theoretical Computer Science},
  series = {Theory and {{Applications}} of {{Models}} of {{Computation}} 2011},
  volume = {505},
  pages = {81--92},
  issn = {0304-3975},
  doi = {10.1016/j.tcs.2012.08.027},
  urldate = {2021-12-24},
  langid = {english}
}

@article{jeandel_undecidability_2020,
  title = {The {{Undecidability}} of the {{Domino Problem}}},
  author = {Jeandel, Emmanuel and Vanier, Pascal},
  year = {2020},
  journal = {Substitution and Tiling Dynamics: Introduction to Self-inducing Structures},
  pages = {293--357},
  publisher = {Springer, Cham},
  doi = {10.1007/978-3-030-57666-0_6},
  urldate = {2021-09-04},
  langid = {english}
}

@article{jockusch_pi_1972,
  title = {\$\textbackslash{{Pi}} \textasciicircum 0\_1\$ Classes and Degrees of Theories},
  author = {jun. Jockusch, Carl G. and Soare, Robert I.},
  year = {1972},
  journal = {Transactions of the American Mathematical Society},
  volume = {173},
  pages = {33--56},
  issn = {0002-9947},
  doi = {10.2307/1996261},
  langid = {english}
}

@misc{kanigowski_survey_2020,
  title = {Survey on Entropy-Type Invariants of Sub-Exponential Growth in Dynamical Systems},
  author = {Kanigowski, Adam and Katok, Anatole and Wei, Daren},
  year = {2020},
  month = apr,
  number = {arXiv:2004.04655},
  eprint = {2004.04655},
  primaryclass = {math},
  publisher = {arXiv},
  urldate = {2024-08-27},
  archiveprefix = {arXiv},
  langid = {english}
}

@article{karaganis_cube_1968,
  title = {On the {{Cube}} of a {{Graph}}},
  author = {Karaganis, Jerome J.},
  year = {1968},
  month = jun,
  journal = {Canadian Mathematical Bulletin},
  volume = {11},
  number = {2},
  pages = {295--296},
  publisher = {Cambridge University Press},
  issn = {0008-4395, 1496-4287},
  doi = {10.4153/CMB-1968-037-0},
  urldate = {2023-03-09},
  langid = {english}
}

@article{kari_rice_1994,
  title = {Rice's Theorem for the Limit Sets of Cellular Automata},
  author = {Kari, Jarkko},
  year = {1994},
  month = may,
  journal = {Theoretical Computer Science},
  volume = {127},
  number = {2},
  pages = {229--254},
  issn = {03043975},
  doi = {10.1016/0304-3975(94)90041-8},
  urldate = {2021-06-17},
  langid = {english}
}

@article{kari_theory_2005,
  title = {Theory of Cellular Automata: {{A}} Survey},
  shorttitle = {Theory of Cellular Automata},
  author = {Kari, Jarkko},
  year = {2005},
  month = apr,
  journal = {Theoretical Computer Science},
  volume = {334},
  number = {1},
  pages = {3--33},
  issn = {0304-3975},
  doi = {10.1016/j.tcs.2004.11.021},
  urldate = {2022-06-07},
  langid = {english}
}

@book{kerr_ergodic_2016,
  title = {Ergodic {{Theory}}: {{Independence}} and {{Dichotomies}}},
  shorttitle = {Ergodic {{Theory}}},
  author = {Kerr, David and Li, Hanfeng},
  year = {2016},
  series = {Springer {{Monographs}} in {{Mathematics}}},
  edition = {1st ed. 2016},
  publisher = {Springer International Publishing : Imprint: Springer},
  address = {Cham},
  doi = {10.1007/978-3-319-49847-8},
  isbn = {978-3-319-49847-8},
  lccn = {515.39}
}

@inproceedings{lafitte_computability_2008,
  title = {Computability of {{Tilings}}},
  booktitle = {Fifth {{Ifip International Conference On Theoretical Computer Science}} – {{Tcs}} 2008},
  author = {Lafitte, Grégory and Weiss, Michael},
  editor = {Ausiello, Giorgio and Karhumäki, Juhani and Mauri, Giancarlo and Ong, Luke},
  year = {2008},
  series = {{{IFIP International Federation}} for {{Information Processing}}},
  pages = {187--201},
  publisher = {Springer US},
  address = {Boston, MA},
  doi = {10.1007/978-0-387-09680-3_13},
  isbn = {978-0-387-09680-3},
  langid = {english}
}

@book{laharpe_topics_2000,
  title = {Topics in Geometric Group Theory},
  author = {de La Harpe, Pierre},
  year = {2000},
  series = {Chicago Lectures in Mathematics},
  publisher = {University of Chicago Press},
  address = {Chicago},
  isbn = {978-0-226-31719-9 978-0-226-31721-2},
  lccn = {QA183 .L3 2000}
}

@article{lewis_topological_2011,
  title = {Topological Aspects of the {{Medvedev}} Lattice},
  author = {Lewis, Andrew E. M. and Shore, Richard A. and Sorbi, Andrea},
  year = {2011},
  month = may,
  journal = {Archive for Mathematical Logic},
  volume = {50},
  number = {3},
  pages = {319--340},
  issn = {1432-0665},
  doi = {10.1007/s00153-010-0215-6},
  urldate = {2021-12-15},
  langid = {english}
}

@article{lind_homoclinic_1999,
  title = {Homoclinic Points of Algebraic {{Zᵈ-actions}}},
  author = {Lind, Douglas and Schmidt, Klaus},
  year = {1999},
  journal = {Journal of the American Mathematical Society},
  volume = {12},
  number = {4},
  eprint = {2646094},
  eprinttype = {jstor},
  pages = {953--980},
  publisher = {American Mathematical Society},
  issn = {08940347, 10886834}
}

@book{lind_introduction_1995,
  title = {An {{Introduction}} to {{Symbolic Dynamics}} and {{Coding}}},
  author = {Lind, Douglas and Marcus, Brian},
  year = {1995},
  publisher = {Cambridge University Press},
  address = {Cambridge},
  doi = {10.1017/CBO9780511626302},
  urldate = {2021-06-25}
}

@incollection{lind_multidimensional_2004,
  title = {Multi-Dimensional Symbolic Dynamics},
  booktitle = {Symbolic Dynamics and Its Applications},
  author = {Lind, Douglas},
  year = {2004},
  series = {Proc. {{Sympos}}. {{Appl}}. {{Math}}.},
  volume = {60},
  pages = {61--79},
  publisher = {Amer. Math. Soc., Providence, RI},
  doi = {10.1090/psapm/060/2078846},
  isbn = {978-0-8218-3157-1},
  mrnumber = {2078846}
}

@article{Lind1990,
  title = {Mahler Measure and Entropy for Commuting Automorphisms of Compact Groups},
  author = {Lind, Douglas and Schmidt, Klaus and Ward, Tom},
  year = {1990},
  month = dec,
  journal = {Inventiones Mathematicae},
  volume = {101},
  number = {1},
  pages = {593--629},
  publisher = {{Springer Science and Business Media LLC}},
  issn = {1432-1297},
  doi = {10.1007/bf01231517}
}

@book{lyndon_combinatorial_2001,
  title = {Combinatorial Group Theory},
  author = {Lyndon, Roger C and Schupp, Paul E},
  year = {2001},
  urldate = {2021-11-18},
  isbn = {978-3-642-61896-3},
  langid = {english}
}

@article{manaster_effective_1972,
  title = {Effective Matchmaking (Recursion Theoretic Aspects of a Theorem of {{Philip Hall}})},
  author = {Manaster, Alfred B. and Rosenstein, Joseph G.},
  year = {1972},
  journal = {Proceedings of the London Mathematical Society. Third Series},
  volume = {25},
  pages = {615--654},
  issn = {0024-6115},
  doi = {10.1112/plms/s3-25.4.615},
  langid = {english}
}

@article{margenstern_domino_2008,
  title = {The Domino Problem of the Hyperbolic Plane Is Undecidable},
  author = {Margenstern, Maurice},
  year = {2008},
  month = nov,
  journal = {Theoretical Computer Science},
  volume = {407},
  number = {1-3},
  pages = {29--84},
  issn = {03043975},
  doi = {10.1016/j.tcs.2008.04.038},
  urldate = {2022-11-02},
  langid = {english}
}

@article{medvedev_degrees_1955,
  title = {Degrees of Difficulty of the Mass Problem},
  author = {Medvedev, {\relax Yu}. T.},
  year = {1955},
  journal = {Dokl. Akad. Nauk SSSR (N.S.)},
  volume = {104},
  pages = {501--504},
  mrnumber = {73542}
}

@article{menibus_characterization_2018,
  title = {Characterization of Sets of Limit Measures of a Cellular Automaton Iterated on a Random Configuration},
  author = {Menibus, Benjamin Hellouin De and Sablik, Mathieu},
  year = {2018},
  month = apr,
  journal = {Ergodic Theory and Dynamical Systems},
  volume = {38},
  number = {2},
  pages = {601--650},
  issn = {0143-3857, 1469-4417},
  doi = {10.1017/etds.2016.46},
  urldate = {2023-10-20},
  langid = {english}
}

@article{meyerovitch_growthtype_2011,
  title = {Growth-Type Invariants for  $\mathbb{Z}^d$ subshifts of finite type and arithmetical classes of real numbers},
  author = {Meyerovitch, Tom},
  year = {2011},
  journal = {Inventiones Mathematicae},
  volume = {184},
  number = {3},
  pages = {567--589},
  issn = {0020-9910},
  doi = {10.1007/s00222-010-0296-1},
  langid = {english}
}

@article{miller_two_2012,
  title = {Two Notes on Subshifts},
  author = {Miller, Joseph S.},
  year = {2012},
  journal = {Proceedings of the American Mathematical Society},
  volume = {140},
  number = {5},
  pages = {1617--1622},
  issn = {0002-9939},
  doi = {10.1090/S0002-9939-2011-11000-1},
  langid = {english}
}

@article{morse_symbolic_1938,
  title = {Symbolic {{Dynamics}}},
  author = {Morse, Marston and Hedlund, Gustav A.},
  year = {1938},
  month = oct,
  journal = {American Journal of Mathematics},
  volume = {60},
  number = {4},
  eprint = {2371264},
  eprinttype = {jstor},
  pages = {815},
  issn = {00029327},
  doi = {10.2307/2371264},
  urldate = {2024-05-14}
}

@article{myers_nonrecursive_1974,
  title = {Nonrecursive Tilings of the Plane. {{II}}},
  author = {Myers, Dale},
  year = {1974},
  month = jun,
  journal = {Journal of Symbolic Logic},
  volume = {39},
  number = {2},
  pages = {286--294},
  issn = {0022-4812, 1943-5886},
  doi = {10.2307/2272641},
  urldate = {2021-12-10},
  langid = {english}
}

@article{olshanskij_question_1980,
  title = {{On the question of the existence of an invariant measure on a group}},
  author = {Ol'shanskij, A. {\relax Yu}.},
  year = {1980},
  journal = {Uspekhi Matematicheskikh Nauk [N. S.]},
  volume = {35},
  number = {4(214)},
  pages = {199--200},
  issn = {0042-1316},
  langid = {russian}
}

@article{pavlov_entropies_2015,
  title = {Entropies Realizable by Block Gluing  $\mathbb{Z}^d$ shifts of finite type},
  author = {Pavlov, Ronnie and Schraudner, Michael},
  year = {2015},
  journal = {Journal d'Analyse Mathématique},
  volume = {126},
  pages = {113--174},
  issn = {0021-7670},
  doi = {10.1007/s11854-015-0014-4},
  langid = {english}
}

@article{pavlov_structure_2023,
  title = {On the Structure of Generic Subshifts},
  author = {Pavlov, Ronnie and Schmieding, Scott},
  year = {2023},
  month = sep,
  journal = {Nonlinearity},
  volume = {36},
  number = {9},
  pages = {4904--4953},
  issn = {0951-7715, 1361-6544},
  doi = {10.1088/1361-6544/acea25},
  urldate = {2024-02-21}
}

@article{piantadosi_symbolic_2008,
  title = {Symbolic Dynamics on Free Groups},
  author = {Piantadosi, Steven T.},
  year = {2008},
  journal = {Discrete \& Continuous Dynamical Systems},
  volume = {20},
  number = {3},
  pages = {725},
  publisher = {American Institute of Mathematical Sciences},
  doi = {10.3934/dcds.2008.20.725},
  urldate = {2021-11-24},
  copyright = {http://creativecommons.org/licenses/by/3.0/},
  langid = {english}
}

@article{rabin_computable_1960,
  title = {Computable Algebra, General Theory and Theory of Computable Fields},
  author = {Rabin, M. O.},
  year = {1960},
  journal = {Transactions of the American Mathematical Society},
  volume = {95},
  pages = {341--360},
  issn = {0002-9947},
  doi = {10.2307/1993295},
  langid = {english}
}

@article{rabin_recursive_1958,
  title = {Recursive Unsolvability of Group Theoretic Problems},
  author = {Rabin, Michael O.},
  year = {1958},
  journal = {Annals of Mathematics. Second Series},
  volume = {67},
  pages = {172--194},
  issn = {0003-486X},
  doi = {10.2307/1969933},
  langid = {english}
}

@article{rice_classes_1953,
  title = {Classes of Recursively Enumerable Sets and Their Decision Problems},
  author = {Rice, H. G.},
  year = {1953},
  journal = {Transactions of the American Mathematical Society},
  volume = {74},
  number = {2},
  pages = {358--366},
  issn = {0002-9947, 1088-6850},
  doi = {10.1090/S0002-9947-1953-0053041-6},
  urldate = {2023-10-14},
  langid = {english}
}

@article{robinson_undecidability_1971,
  title = {Undecidability and Nonperiodicity for Tilings of the Plane},
  author = {Robinson, Raphael M.},
  year = {1971},
  month = sep,
  journal = {Inventiones Mathematicae},
  volume = {12},
  number = {3},
  pages = {177--209},
  issn = {0020-9910, 1432-1297},
  doi = {10.1007/BF01418780},
  urldate = {2022-12-19},
  langid = {english}
}

@book{Schmidt1995,
  title = {Dynamical Systems of Algebraic Origin},
  author = {Schmidt, Klaus},
  year = {1995},
  publisher = {Springer Nature},
  doi = {10.1007/978-3-0348-0277-2}
}

@article{sekanina1960ordering,
  title = {On an Ordering of the Set of Vertices of a Connected Graph},
  author = {Sekanina, Milan},
  year = {1960},
  journal = {Spisy Přírod. Fak. Univ. Brno},
  volume = {412},
  pages = {137--142}
}

@article{seward_burnside_2014,
  title = {Burnside's Problem, Spanning Trees and Tilings.},
  author = {Seward, Brandon},
  year = {2014},
  journal = {Geometry \& Topology},
  volume = {18},
  number = {1},
  pages = {179--210},
  issn = {1465-3060},
  doi = {10.2140/gt.2014.18.179},
  langid = {english}
}

@article{simpson_medvedev_2014,
  title = {Medvedev Degrees of Two-Dimensional Subshifts of Finite Type},
  author = {Simpson, Stephen G.},
  year = {2014},
  journal = {Ergodic Theory and Dynamical Systems},
  volume = {34},
  number = {2},
  pages = {679--688},
  issn = {0143-3857},
  doi = {10.1017/etds.2012.152},
  langid = {english}
}

@incollection{sorbi_medvedev_1996,
  title = {The {{Medvedev}} Lattice of Degrees of Difficulty},
  booktitle = {Computability, {{Enumerability}}, {{Unsolvability}}: {{Directions}} in {{Recursion Theory}}},
  author = {Sorbi, Andrea},
  editor = {Cooper, S. B. and Wainer, S. S. and Slaman, T. A.},
  year = {1996},
  series = {London {{Mathematical Society Lecture Note Series}}},
  pages = {289--312},
  publisher = {Cambridge University Press},
  address = {Cambridge},
  doi = {10.1017/CBO9780511629167.015},
  urldate = {2024-03-14},
  isbn = {978-0-521-55736-8}
}

@article{turing_computable_1937,
  title = {On {{Computable Numbers}}, with an {{Application}} to the {{Entscheidungsproblem}}},
  author = {Turing, Alan},
  year = {1937},
  journal = {Proceedings of the London Mathematical Society},
  volume = {s2-42},
  number = {1},
  pages = {230--265},
  issn = {1460-244X},
  doi = {10.1112/plms/s2-42.1.230},
  urldate = {2021-11-25},
  langid = {english}
}

@article{ventura_grouptheoretic_2014,
  title = {Group-Theoretic Orbit Decidability},
  author = {Ventura, Enric},
  year = {2014},
  month = nov,
  journal = {Groups Complexity Cryptology},
  volume = {6},
  number = {2},
  pages = {133--148},
  publisher = {De Gruyter},
  issn = {1869-6104},
  doi = {10.1515/gcc-2014-0012},
  urldate = {2024-03-31},
  copyright = {De Gruyter expressly reserves the right to use all content for commercial text and data mining within the meaning of Section 44b of the German Copyright Act.},
  langid = {english}
}

@article{wang_proving_1961,
  title = {Proving Theorems by Pattern Recognition {{II}}},
  author = {Wang, Hao},
  year = {1961},
  journal = {The Bell System Technical Journal},
  volume = {40},
  number = {1},
  pages = {1--41},
  doi = {10.1002/j.1538-7305.1961.tb03975.x}
}

@book{weihrauch_computable_2000,
  title = {Computable Analysis: An Introduction},
  shorttitle = {Computable Analysis},
  author = {Weihrauch, Klaus},
  year = {2000},
  series = {Texts in Theoretical Computer Science},
  publisher = {Springer Verlag},
  address = {Berlin ; New York},
  isbn = {978-3-540-66817-6},
  langid = {english},
  lccn = {QA9.59 .W46 2000}
}

@misc{westrick_topological_2020,
  title = {Topological Completely Positive Entropy Is No Simpler in \$\textbackslash mathbb \{\vphantom\}{{Z}}\vphantom\{\}\textasciicircum 2\$-{{SFTs}}},
  author = {Westrick, Linda},
  year = {2020},
  month = may,
  number = {arXiv:1904.11444},
  eprint = {1904.11444},
  primaryclass = {math},
  publisher = {arXiv},
  doi = {10.48550/arXiv.1904.11444},
  urldate = {2023-04-20},
  archiveprefix = {arXiv}
}

@article{whyte_amenability_1999,
  title = {Amenability, Bilipschitz Equivalence, and the von {{Neumann}} Conjecture},
  author = {Whyte, Kevin},
  year = {1999},
  month = jul,
  journal = {Duke Mathematical Journal},
  volume = {99},
  number = {1},
  issn = {0012-7094},
  doi = {10.1215/S0012-7094-99-09904-0},
  urldate = {2022-06-16}
}

@article{zbMATH03278210,
  title = {The Cube of Every Connected Graph Is 1-Hamiltonian},
  author = {Chartrand, G. and Kapoor, S. F.},
  year = {1969},
  journal = {J. Res. Natl. Bur. Stand., Sect. B},
  volume = {73},
  pages = {47--48},
  issn = {0098-8979},
  doi = {10.6028/jres.073B.007},
  fjournal = {Journal of Research of the National Bureau of Standards},
  langid = {english},
  zbmath = {3278210},
  zmnumber = {0174.26802}
}
